\documentclass[11pt,leqno]{amsart}
\usepackage{hyperref}
\usepackage{epsfig}
\usepackage{enumitem}
\usepackage{amsmath, amssymb}
\usepackage{amscd}
\usepackage[all,cmtip,matrix,arrow]{xy}
\usepackage{mathrsfs}
\usepackage{graphicx}
\usepackage{tikz-cd}
\usepackage{mathabx}
\usepackage{array}
\usepackage{longtable}
\xyoption{rotate}
\xyoption{pdf}
\usepackage{xparse}

\newcommand{\Implies}[2]{$\text{\ref{#1}}\implies\text{\ref{#2}}$}
\newcommand{\Iff}[2]{$\text{\ref{#1}}\iff\text{\ref{#2}}$}
\newcommand{\Ifff}[3]{$\text{\ref{#1}}\iff\text{\ref{#2}}\iff\text{\ref{#3}}$}

\makeatletter
\DeclareRobustCommand*{\bfseries}{%
  \not@math@alphabet\bfseries\mathbf
  \fontseries\bfdefault\selectfont
  \boldmath
}
\makeatother

\AtBeginEnvironment{tikzcd}{\setlength{\mathsurround}{0pt}}

\newtheorem{theo}{Theorem}[section]
\newtheorem{lemma}[theo]{Lemma}
\newtheorem{defi}[theo]{Definition}
\newtheorem{prop}[theo]{Proposition}

\newtheorem{cor}[theo]{Corollary}
\newtheorem{remark}[theo]{Remark}

\newtheorem{example}[theo]{Example}

\numberwithin{equation}{section}

\mathchardef\mhyphen="2D

\def\N{\mathbb{N}}
\def\bL{\mathbb{L}}

\def\R{\mathbb{R}}
\def\bS{\mathbb{S}}

\def\Z{\mathbb{Z}}
\def\Q{\mathbb{Q}}

\def\coh{\operatorname{coh}}

\def\wt{\widetilde}

\def\bR{{\mathbf R}}
\def\bL{{\mathbf L}}

\def\FF{{\mathbb F}}

\def\pre-tr{\operatorname{pre-tr}}
\def\h{\operatorname{h}}

\def\Hom{\operatorname{Hom}}
\def\Map{\operatorname{Map}}
\def\End{\operatorname{End}}

\DeclareMathOperator*{\colim}{colim}

\newcommand{\tens}[1]{%
  \mathbin{\mathop{\otimes}\displaylimits_{#1}}%
}

\newcommand{\Ltens}[1]{%
  \mathbin{\mathop{\otimes}\displaylimits^{\bL}_{#1}}%
}

\newcommand{\hy}{\mhyphen}

\newcommand{\indlim}[1][]{\mathop{\varinjlim}\limits_{#1}}
\newcommand{\inddlim}[1][]{{``{\indlim[#1]}"}}
\newcommand{\prolim}[1][]{\mathop{\varprojlim}\limits_{#1}}
\newcommand{\proolim}[1][]{{``{\prolim[#1]}"}}
\newcommand{\biggplus}[1][]{\mathop{\bigoplus}\limits_{#1}}

\newcommand{\prodd}[1][]{\mathop{\prod}\limits_{#1}}

\newcommand{\bbar}{\overline}
\newcommand{\hhat}{\widehat}
\newcommand{\toto}{\rightrightarrows}
\newcommand{\xto}{\xrightarrow}

\newcommand{\hto}{\hookrightarrow}
\newcommand{\leftto}{\leftarrow}

\newcommand{\bnd}{\operatorname{bnd}}
\newcommand{\rex}{\operatorname{rex}}
\newcommand{\rig}{\operatorname{rig}}
\newcommand{\oplax}{\operatorname{oplax}}
\newcommand{\lax}{\operatorname{lax}}

\newcommand{\Mot}{\operatorname{Mot}}

\newcommand{\Nuc}{\operatorname{Nuc}}
\newcommand{\Solid}{\operatorname{Solid}}

\newcommand{\IndCoh}{\operatorname{IndCoh}}
\newcommand{\Proj}{\operatorname{Proj}}

\newcommand{\Cat}{\operatorname{Cat}}

\newcommand{\Cone}{\operatorname{Cone}}
\newcommand{\Fiber}{\operatorname{Fiber}}

\newcommand{\loc}{\operatorname{loc}}
\newcommand{\cont}{\operatorname{cont}}

\newcommand{\acc}{\operatorname{acc}}

\newcommand{\st}{\operatorname{st}}
\newcommand{\cg}{\operatorname{cg}}
\newcommand{\Prr}{\operatorname{Pr}}
\newcommand{\dual}{\operatorname{dual}}

\newcommand{\deff}{\operatorname{def}}
\newcommand{\Fil}{\operatorname{Fil}}

\newcommand{\bbD}{{\mathbb D}}

\newcommand{\bE}{{\mathbb E}}

\newcommand{\mk}{\mathrm k}
\newcommand{\ev}{\mathrm ev}

\newcommand{\cJ}{{\mathcal J}}

\newcommand{\cF}{{\mathcal F}}
\newcommand{\cG}{{\mathcal G}}
\newcommand{\cO}{{\mathcal O}}

\newcommand{\cD}{{\mathcal D}}

\newcommand{\cA}{{\mathcal A}}
\newcommand{\cB}{{\mathcal B}}
\newcommand{\cI}{{\mathcal I}}
\newcommand{\cC}{{\mathcal C}}
\newcommand{\cE}{{\mathcal E}}

\newcommand{\cY}{{\mathcal Y}}
\newcommand{\cU}{{\mathcal U}}

\newcommand{\cS}{{\mathcal S}}
\newcommand{\cT}{{\mathcal T}}

\newcommand{\veps}{\varepsilon}

\newcommand{\un}{\underline}

\newcommand{\Eq}{\operatorname{Eq}}
\newcommand{\LEq}{\operatorname{LEq}}

\newcommand{\THH}{\operatorname{THH}}
\newcommand{\HH}{\operatorname{HH}}
\newcommand{\THC}{\operatorname{THC}}
\newcommand{\CAlg}{\operatorname{CAlg}}

\newcommand{\Fun}{\operatorname{Fun}}

\newcommand{\topp}{\operatorname{top}}

\newcommand{\Ab}{\operatorname{Ab}}

\newcommand{\Perf}{\operatorname{Perf}}

\newcommand{\perf}{\operatorname{perf}}

\newcommand{\Kos}{\operatorname{Kos}}

\newcommand{\Cosh}{\operatorname{Cosh}}

\newcommand{\Shv}{\operatorname{Shv}}

\newcommand{\im}{\operatorname{Im}}

\newcommand{\cha}{\operatorname{char}}

\newcommand{\Rep}{\operatorname{Rep}}

\newcommand{\Tor}{\operatorname{Tor}}

\newcommand{\tr}{\operatorname{tr}}

\newcommand{\Sp}{\operatorname{Sp}}

\newcommand{\Ind}{\operatorname{Ind}}
\newcommand{\Pro}{\operatorname{Pro}}
\newcommand{\Calk}{\operatorname{Calk}}

\newcommand{\Spf}{\operatorname{Spf}}

\newcommand{\Stab}{\rm Stab}

\newcommand{\Alg}{\operatorname{Alg}}

\newcommand{\id}{\operatorname{id}}

\newcommand{\coev}{\operatorname{coev}}

\newcommand{\tors}{\operatorname{tors}}

\newcommand{\compl}{\operatorname{compl}}

\newcommand{\pt}{\operatorname{pt}}

\newcommand{\Mod}{\operatorname{Mod}}

\newcommand{\mult}{\operatorname{mult}}
\newcommand{\trdeg}{\operatorname{trdeg}}

\newcommand{\bs}{\blacksquare}

\usepackage{epsf}
\usepackage{amscd}

\title[Localizing invariants of inverse limits]
{Localizing invariants of inverse limits}

\author{Alexander I. Efimov}
\address{The Hebrew University of Jerusalem}
\email{efimov@mccme.ru}

\thanks{The author was partially supported by the European Research Council (ERC, CurveArithmetic, 101078157).}

\begin{document}

\begin{abstract} In this paper we study the category of nuclear modules on an affine formal scheme as defined by Clausen and Scholze \cite{CS20}. We also study related constructions in the framework of dualizable and rigid monoidal categories. We prove that the $K$-theory (in the sense of \cite{E24}) of the category of nuclear modules on $\Spf(R^{\wedge}_I)$ is isomorphic to the classical continuous $K$-theory, which in the noetherian case is given by the limit $\prolim[n]K(R/I^n).$ This isomorphism was conjectured previously by Clausen and Scholze.

More precisely, we study two versions of the category of nuclear modules: the original one defined in \cite{CS20} and a different version, which contains the original one as a full subcategory. For our category $\Nuc(R^{\wedge}_I)$ we give three equivalent definitions. The first definition is by taking the internal $\Hom$ in the category $\Cat_R^{\dual}$ of $R$-linear dualizable categories. The second definition is by taking the rigidification of the usual $I$-complete derived category of $R.$ The third definition is by taking an inverse limit in $\Cat_R^{\dual}.$ For each of the three approaches we prove that the corresponding construction is well-behaved in a certain sense. 

Moreover, we prove that the two versions of the category of nuclear modules have the same $K$-theory, and in fact the same finitary localizing invariants.
\end{abstract}


\maketitle

\tableofcontents

\section{Introduction}
\label{sec:intro}

 Consider an affine noetherian formal scheme $\Spf(R^{\wedge}_I).$ It is classically understood that the ``naive'' $K$-theory spectrum $K(R^{\wedge}_I)=K(\Perf(R^{\wedge}_I))$ is not a reasonable invariant from the point of view of formal geometry. The reason is that the completion does not commute with the localization. More precisely, for an element $f\in R$ the functor between the categories of perfect complexes $\Perf(R^{\wedge}_I)\to \Perf(R[f^{-1}]^{\wedge}_I)$ is in general not a Verdier quotient functor (even up to idempotent completion). 
 
 A much better notion is the so-called continuous $K$-theory of $\Spf(R^{\wedge}_I),$ which by definition is the inverse limit $\prolim[n]K(R/I^n)$ (in fact, with some modification it can be defined in the non-noetherian case as well, see below). It was studied by many authors, in particular we mention the papers \cite{Wag76, BEK14a, BEK14b, Bei14, Mor18, KST19, KST23}. 
 
 However, the above definition of continuous $K$-theory of $\Spf(R^{\wedge}_I)$ is in some sense ad hoc: it is not clear if this inverse limit is the $K$-theory of some natural category associated with the formal scheme. It was generally assumed that such a (natural) categorification does not exist.
 
 A candidate for such a category was constructed recently by Clausen and Scholze in the framework of condensed mathematics \cite{CS20}. In loc. cit. the authors define the category of the so-called nuclear solid modules, which is a presentable stable ($R$-linear) category. We will denote it by $\Nuc^{CS}(R^{\wedge}_I).$ They proved that this category is in fact dualizable, and its full subcategory of compact objects is equivalent to $\Perf(R^{\wedge}_I).$ In particular, we can take the continuous $K$-theory $K^{\cont}$ (in the sense of \cite{E24}) of the category $\Nuc^{CS}(R^{\wedge}_I)$ . Here $K^{\cont}:\Cat_{\st}^{\dual}\to \Sp$ is a functor defined on the category of dualizable categories, where the $1$-morphisms are the functors which have a colimit-preserving right adjoint. More generally, this approach leads to a reasonable definition of $K$-theory of adic spaces, which was studied in \cite{And23}. We also refer to \cite{AM24} for a nice exposition of the general theory of nuclear solid modules. 
 
 One of the main results of the present paper is the following theorem, which was conjectured by Clausen and Scholze.
 
\begin{theo}\label{th:K_theory_of_Nuc^CS_intro}
Within the above notation we have an isomorphism 
\begin{equation}
 	K^{\cont}(\Nuc^{CS}(R^{\wedge}_I))\cong \prolim[n] K(R/I^n).
\end{equation}
\end{theo}

In fact, we prove a more general statement (Theorem \ref{th:loc_invar_of_Nuc_CS}), which describes an arbitrary localizing invariant of the category $\Nuc^{CS}(R^{\wedge}_I).$ Moreover, the result holds for a non-noetherian commutative ring $R$ with a finitely generated ideal $I=(a_1,\dots,a_m)\subset R.$ Here the quotients $R/I^n$ are replaced with the DG $R$-algebras $R_n=\Kos(R;a_1^n,\dots,a_m^n),$ given by the Koszul complexes. This is indeed a generalization: if $R$ is noetherian, then the inverse system $(R/I^n)_{n\geq 1}$ is pro-equivalent to the inverse system $(R_n)_{n\geq 1}$ in the $\infty$-category of (associative) DG $R$-algebras, see Lemma \ref{lem:pro_equivalence_noetherian}.
 
We introduce a closely related, but different version of the category of nuclear modules, which we denote by $\Nuc(R^{\wedge}_I).$ It is also dualizable, and we have a fully faithful inclusion $\Nuc^{CS}(R^{\wedge}_I)\to \Nuc(R^{\wedge}_I).$ Our version is in some sense more natural from the categorical perspective: it has three different descriptions via a universal property in the world of dualizable categories, which are studied respectively in Sections \ref{sec:dualizable_Hom}, \ref{sec:rig_of_locally_rigid} and \ref{sec:Mittag-Leffler}. The latter description is probably the most natural from the $K$-theoretic point of view:
\begin{equation*}
 	\Nuc(R^{\wedge}_I)=\prolim[n]^{\dual}D(R/I^n).
\end{equation*}
Here the limit is taken in the category $\Cat_{\st}^{\dual}$ (or equivalently in the category $\Cat_R^{\dual}$ of $R$-linear dualizable categories). This limit is much bigger (and much more complicated) than the usual limit $\prolim[n]D(R/I^n),$ which is equivalent to the $I$-complete derived category $D_{I\hy\compl}(R).$ Note that the full subcategory of compact objects in the category $\Nuc(R^{\wedge}_I)$ is given by the usual limit of the categories of perfect complexes, so we have
\begin{equation}\label{eq:Nuc_as_limit_intro}
\Nuc(R^{\wedge}_I)^{\omega}\simeq \prolim[n]\Perf(R/I^n)\simeq \Perf(R^{\wedge}_I).
\end{equation}

An ``explicit'' way to describe the category $\Nuc(R^{\wedge}_I),$ which is discussed in Section \ref{sec:rig_of_locally_rigid}, is the following. The category $\Nuc(R^{\wedge}_I)\subset\Ind(D_{I\hy\compl}(R)^{\omega_1})$ is the full subcategory generated via colimits by the formal sequential colimits $\inddlim[n]M_n,$ where the transition maps $M_n\to M_{n+1}$ are trace-class in $D_{I\hy\compl}(R)$ in the sense of \cite[Definition 13.11]{CS20}. 

We prove the following result, which is closely related with Theorem \ref{th:K_theory_of_Nuc^CS_intro}.

\begin{theo}\label{th:K_theory_of_Nuc_intro}
Within the above notation we have an isomorphism
\begin{equation*}
K^{\cont}(\Nuc(R^{\wedge}_I))\cong K^{\cont}(\prolim[n]^{\dual}D(R/I^n))\xto{\sim} \prolim[n] K^{\cont}(D(R/I^n))\cong \prolim[n]K(R/I^n).
\end{equation*}
\end{theo}

In fact, we prove a much more general result (Theorem \ref{th:local_invar_of_inverse_limits}, Corollary \ref{cor:K_theory_of_limit_of_ML}) which describes the localizing invariants of inverse limits of sufficiently nice sequences in $\Cat_{\st}^{\dual}.$ This class of inverse sequences includes the so-called strongly Mittag-Leffler sequences which we define and study in Section \ref{sec:Mittag-Leffler}. In particular, for an arbitrary commutative ring $R$ and a collection of elements $a_1,\dots,a_m\in R$ the inverse system $(D(R_n))_{n\geq 1}$ as above is strongly Mittag-Leffler. 
 
Note that by Theorems \ref{th:K_theory_of_Nuc^CS_intro} and \ref{th:K_theory_of_Nuc_intro} the categories $\Nuc^{CS}(R^{\wedge}_I)$ and $\Nuc(R^{\wedge}_I)$ have the same continuous $K$-theory. In fact, using the results of \cite{Cor} we can prove that these categories have the same finitary localizing invariants, see Corollary \ref{cor:U_loc_of_Nuc^CS_and_Nuc} (``finitary'' means ``commuting with filtered colimits''). 

We now explain another description of the category $\Nuc(R^{\wedge}_I),$ which is the most fundamental from our perspective. It is given in Section \ref{sec:dualizable_Hom}. We consider the full subcategory $D_{I\hy\tors}(R)\subset D(R)$ of objects with locally $I$-torsion homology. This category is well-known to be equivalent to the category $D_{I\hy\compl}(R)$ above. It is generated by a single compact object $\Kos(R;a_1,\dots,a_n),$ where $I=(a_1,\dots,a_m).$ We have an equivalence
\begin{equation*}
\Nuc(R^{\wedge}_I)\simeq\un{\Hom}_R^{\dual}(D_{I\hy\tors}(R),D(R)),
\end{equation*}
where the internal $\Hom$ is taken in the category $\Cat_R^{\dual}$ of $R$-linear dualizable categories. Again, this internal $\Hom$ is much bigger (and much more complicated) than the internal $\Hom$ in $\Pr^L_R$ (the category of $R$-linear presentable stable categories) which in this case is simply the category $D_{I\hy\tors}(R)\simeq D_{I\hy\compl}(R).$ 

A special case of the main result of Section \ref{sec:dualizable_Hom} (Theorem \ref{th:internal_projectivity}) states that the
functor
\begin{equation*}
\un{\Hom}_R^{\dual}(D_{I\hy\tors}(R),-):\Cat_R^{\dual}\to \Cat_R^{\dual}
\end{equation*}
takes short exact sequences to short exact sequences. In other words, the category $D_{I\hy\tors}(R)$ is internally projective in $\Cat_R^{\dual}.$ More generally, we can replace the category $D_{I\hy\tors}(R)$ with an arbitrary proper and $\omega_1$-compact $R$-linear dualizable category. Moreover, the base category $D(R)$ can be replaced with an arbitrary rigid $\bE_1$-monoidal category in the sense of Gaitsgory and Rozenblyum (see Definition \ref{def:rigid_category}). In this generality (if the base is only $\bE_1$-monoidal) Theorem \ref{th:internal_projectivity} states the relative internal projectivity over $\Cat_{\st}^{\dual}$ (if the base $\cE$ is rigid symmetric monoidal, this is equivalent to the internal projectivity in $\Cat_{\cE}^{\dual}$). In Subsection \ref{ssec:analogy_Raynaud_Gruson} we explain a surprising analogy between Theorem \ref{th:internal_projectivity} and the Raynaud-Gruson criterion of projectivity of modules over (usual) associative rings \cite{RayGru}.

This internal projectivity of the category $D_{I\hy\tors}(R)$ is essentially the main reason for the possibility of computation of localizing invariants of the category $\Nuc(R^{\wedge}_I).$  For a more precise statement, see Corollary \ref{cor:internal_Hom_in_Mot^loc_kappa} and the discussion after its proof.

We mention a surprising application of our results to Hochschild homology of the categories of nuclear modules, which uses the results of \cite{Cor}. To fix the ideas, we consider the following special case. For a prime $p,$ the categories $\Nuc(\Z_p)$ and $\Nuc^{CS}(\Z_p)$ are $\Z$-linear dualizable categories (equivalently, dualizable DG categories). We have the following isomorphisms:
\begin{equation}\label{eq:HH_Nuc_Z_p_intro}
\HH(\Nuc(\Z_p)/\Z)\cong\HH(\Nuc^{CS}(\Z_p)/\Z)\cong \Z_p,
\end{equation}
see Corollary \ref{cor:HH_Nuc_examples}. Note that in contrast the Hochschild homology $\HH(\Z_p/\Z)$ is much larger than $\Z_p:$ the cotangent complex $L_{\Z_p/\Z}\cong \Omega^1_{\Q_p/\Q}$ is a very large $\Q_p$-vector space (of dimension $2^{\aleph_0}$). We remark that \eqref{eq:HH_Nuc_Z_p_intro} implies the following: for the non-unital algebra $J$ of compact operators on the separable Banach $\Q_p$-vector space $(\biggplus[\N]\Z)^{\wedge}_p [p^{-1}]$ we have
\begin{equation*}
\HH(J/\Q)\cong \Q_p.
\end{equation*}

The paper is organized as follows.

In Section \ref{sec:preliminaries_rigid_and_dualizable} we recall some notions and results about (presentable stable) rigid monoidal categories in the sense of \cite[Definition 9.1.2]{GaRo17}, and dualizable modules over such categories. Most of the material in this section is already known, but some results seem to be new. In particular, we prove that for a rigid $\bE_1$-monoidal category $\cE$ the category $\Cat_{\cE}^{\cg}$ of relatively compactly generated left $\cE$-modules is compactly assembled. We also discuss in some detail the rigidification of a compactly generated symmetric monoidal category and its equivalence with the rigidification of the dual category. This is relevant for understanding the two versions of the category of nuclear modules discussed above. We also discuss the smoothness and properness for dualizable categories over a rigid base category, and their characterizations via right trace-class morphisms of functors.

In Section \ref{sec:lim_colim_and_filtered} we formulate certain non-trivial statements about the combinations of limits and colimits, and about filtered $\infty$-categories. An especially difficult Proposition \ref{prop:quadrofunctors}, involving tricky permutations of limits and colimits, is proved in Appendix \ref{app:limits_colimits} as a consequence of a more general result (Theorem \ref{th:pullback_square_quadrofunctors}).  A quite technical but intuitive statement about lax equalizers of functors between infinite products of filtered $\infty$-categories (Lemma \ref{lem:lax_equalizer_of_filtered_properties}) is proved in Appendix \ref{app:proof_of_lemma_on_filtered}.

In Section \ref{sec:dualizable_Hom} we study the (relative) internal $\Hom$ in the category of dualizable modules over some rigid base category. We prove Theorem \ref{th:internal_projectivity} on the internal projectivity of proper $\omega_1$-compact dualizable modules. The proof is quite difficult, and it eventually reduces to Proposition \ref{prop:quadrofunctors}. We discuss some examples and give the definition of the category $\Nuc(R^{\wedge}_I)$ as a dualizable internal $\Hom.$ Among other things, we explain the analogy between Theorem \ref{th:internal_projectivity} and the Raynaud-Gruson criterion of projectivity for modules over an associative ring. 

In Section \ref{sec:rig_of_locally_rigid} we use Theorem \ref{th:internal_projectivity} and its proof to prove certain good properties of the rigidification $\cE^{\rig}$ of a locally rigid  symmetric monoidal category $\cE,$ assuming that the unit object $1\in\cE$ is $\omega_1$-compact. Loosely speaking, we prove that any trace-class map in $\cE$ is a composition of two trace-class maps. A more precise statement is Theorem \ref{th:rigidification_of_locally_rigid_countable}. Among other things, it states that the natural functor $\cE^{\rig}\to\cE$ has a symmetric monoidal right adjoint (a priori it is only lax monoidal), and the theorem also gives a description of this right adjoint. We obtain an equivalent description of the category $\Nuc(R^{\wedge}_I)$ as a rigidification of the (locally rigid) $I$-complete derived category $D_{I\hy\compl}(R).$
 
In Section \ref{sec:Mittag-Leffler} we define the notion of a strongly Mittag-Leffler sequence of dualizable categories over some rigid base. The main result is Theorem \ref{th:hom_epi_for_ML} which informally speaking states the ``vanishing of $\bR^1\lim$'' for strongly Mittag-Leffler sequences. More precisely, for such an inverse sequence $(\cC_n)_{n\geq 0}$ we prove that for any uncountable regular cardinal $\kappa$ the functor
\begin{equation}\label{eq:hom_epi_intro}
\prolim[n]\cC_n^{\kappa}\to\prolim[n]\Calk_{\kappa}^{\cont}(\cC_n),
\end{equation} 
is a homological epimorphism, i.e. the induced functor on the ind-completions is a quotient functor. Here $\Calk_{\kappa}^{\cont}(\cC_n)$ is the Calkin category defined in \cite{E24}. This is a surprising statement, which in particular implies the essential surjectivity (up to retracts) of the functor \eqref{eq:hom_epi_intro}. It is closely related with Theorem \ref{th:internal_projectivity} on the internal projectivity, and we explain this relation. We obtain the description of the category $\Nuc(R^{\wedge}_I)$ as a dualizable inverse limit \eqref{eq:Nuc_as_limit_intro}, and a more general version for non-noetherian rings.

In Section \ref{sec:loc_invar_inverse_limits} we prove Theorem \ref{th:local_invar_of_inverse_limits} about localizing invariants of dualizable inverse limits of sequences $(\cC_n)_{n\geq 0}$ such that the functor \eqref{eq:hom_epi_intro} is a homological epimorphism at least for some uncountable regular $\kappa$ (in practice, we can restrict to $\kappa=\omega_1$). As a corollary, we prove Theorem \ref{th:K_theory_of_Nuc_intro}.

In Section \ref{sec:original_Nuc} we study the original version of the category of nuclear (solid) modules and prove Theorem \ref{th:loc_invar_of_Nuc_CS} about its localizing invariants. As a corollary, we obtain Theorem \ref{th:K_theory_of_Nuc^CS_intro}. We also explain the application to Hochschild homology both for $\Nuc(R^{\wedge}_I)$ and $\Nuc^{CS}(R^{\wedge}_I),$ using the results of \cite{Cor}. 

{\noindent {\bf Acknowledgements.}} I am grateful to Ko Aoki, Alexander Beilinson, Dustin Clausen, Adriano C\'ordova Fedeli, Vladimir Drinfeld, Dennis Gaitsgory, Dmitry Kaledin, David Kazhdan, Akhil Mathew, Thomas Nikolaus, Maxime Ramzi, Peter Scholze and Georg Tamme for useful discussions. Part of this work was done while I was a visitor in the Max Planck Institute for Mathematics in Bonn from December 2022 till February 2024, and I am grateful to the institute for their hospitality and support.  

\section{Preliminaries on rigid monoidal categories and dualizable modules}
\label{sec:preliminaries_rigid_and_dualizable}

\subsection{Notation and terminology}

We will freely use the theory of $\infty$-categories, as developed in \cite{Lur09, Lur17}. We will only deal with $(\infty,1)$-categories, unless otherwise stated. Moreover, we will simply say ``category'' instead of ``$\infty$-category'' if the meaning is clear from the context. We denote by $\cS$ the $\infty$-category of spaces, which is freely generated by one object via colimits. We denote by $\Sp$ the category of spectra.

We will mostly use the same notation as in \cite{E24} for various categories of stable categories. In particular, we denote by $\Cat^{\perf}$ the category of idempotent-complete stable categories and exact functors between them. We denote by $\Pr^L_{\st}$ the category of presentable stable categories and continuous (i.e. colimit-preserving) functors between them. We will say that a continuous functor $F:\cC\to\cD$ is strongly continuous if its right adjoint commutes with colimits. Given $\cC,\cD\in\Pr^L_{\st},$ we denote by $\Fun^L(\cC,\cD)$ resp. $\Fun^{LL}(\cC,\cD)$ the category of continuous resp. strongly continuous functors. For a regular cardinal $\kappa,$ we denote by $\Pr^L_{\st,\kappa}\subset \Pr^L_{\st}$ the non-full subcategory of $\kappa$-presentable stable categories, with $1$-morphisms being continuous functors which preserve $\kappa$-compact objects. The assignment $\cC\mapsto \cC^{\kappa}$ defines an equivalence $\Pr^L_{\st,\kappa}\xto{\sim}\Cat^{\kappa\hy\rex}_{\st},$ where $\Cat^{\kappa\hy\rex}_{\st}$ is the category of small stable categories with $\kappa$-small colimits, and exact functors which commute with $\kappa$-small colimits (in the case $\kappa=\omega$ we also require the idempotent-completeness).

We will say that an exact functor $F:\cA\to\cB$ between idempotent-complete small stable categories is a homological epimorphism if the functor $\Ind(F):\Ind(\cA)\to\Ind(\cB)$ is a quotient functor, or equivalently if its right adjoint is fully faithful. Equivalently, this means that $F$ is an epimorphism in $\Cat^{\perf}.$  

For a functor $F:\cC\to\cD$ we will denote by $F^R:\cD\to\cC$ the right adjoint functor, if it exists. Similarly, we denote by $F^L:\cD\to\cC$ the left adjoint functor.

We denote by $\Cat_{\st}^{\dual}\subset \Pr^L_{\st}$ the (non-full) subcategory of dualizable categories and strongly continuous functors between them. Recall that the ind-completion functor $\Ind:\Cat^{\perf}\to \Cat_{\\st}^{\dual}$ is fully faithful. Its essential image is the category $\Cat^{\cg}_{\st}\simeq \Pr^L_{\st,\omega}\subset\Cat_{\st}^{\dual}$ of compactly generated categories. We will consider the duality functor $(-)^{\vee}:\Cat_{\st}^{\dual}\to\Cat_{\st}^{\dual}$ as a (symmetric monoidal) covariant involution. It sends $\cC$ to the dual category $\cC^{\vee},$ and a strongly continuous functor $F:\cC\to\cD$ is sent to $F^{\vee,L}\cong (F^R)^{\vee}.$ 

Recall that any dualizable category is $\omega_1$-presentable \cite[Corollary 1.21]{E24}. For $\cC\in\Cat_{\st}^{\dual}$ we will denote by $\hat{\cY}_{\cC}:\cC\to\Ind(\cC^{\omega_1})$ (or simply $\hat{\cY}$) the left adjoint to the colimit functor. The same applies to more general compactly assembled categories (not necessarily stable).

We will use the following convention: a functor $p:I\to J$ between small $\infty$-categories is cofinal if for any $j\in J$ the category $I_{j/}=I\times_J J_{j/}$ is weakly contractible. Equivalently, for any $\infty$-category $\cC$ and for any functor $F:J\to\cC$ we have
\begin{equation*}
\indlim(J\xto{F}\cC)\cong \indlim(I\xto{p}J\xto{F}\cC),
\end{equation*}
assuming that one of the colimits exists.

We will mostly use directed posets instead of general filtered $\infty$-categories. Recall that by \cite[Proposition 5.3.1.18]{Lur09} for any $\kappa$-filtered $\infty$-category $I$ there exists a $\kappa$-directed poset $J$ with a cofinal functor $J\to I.$ However, we will need certain statements about filtered $\infty$-categories which do not formally reduce to the case of directed posets, see Subsection \ref{ssec:lemmas_about_filtered}. For an $\infty$-category $\cA$ we will typically write the objects of $\Ind(\cA)$ resp. $\Pro(\cA)$ as $\inddlim[i\in I]x_i$ resp. $\proolim[i\in I]x_i,$ where $I$ is directed resp. codirected, and we have a functor $I\to\cA,$ $i\mapsto x_i.$

By default, all the functors between stable categories will be assumed to be exact. A stable category $\cA$ is enriched over $\Sp,$ and we will denote by $\Hom_{\cA}(x,y)$ or $\cA(x,y)$ the spectrum of morphisms from $x$ to $y.$ We will frequently use the identification $\Ind(\cA)\simeq\Fun(\cA^{op},\Sp)$ in the case when $\cA$ is small stable. We denote the evaluation functor for $\Ind(\cA)$ by
\begin{equation*}
-\tens{\cA}-:\Ind(\cA)\otimes\Ind(\cA^{op})\to\Sp,
\end{equation*}
so we have
\begin{equation*}
(\inddlim[i]x_i)\tens{\cA}(\inddlim[j]y_j)\cong \indlim[i,j]\cA(y_j,x_i).
\end{equation*}

If $\cE$ is an $\bE_1$-monoidal category, we have in general two different $\bE_1$-monoidal categories $\cE^{op}$ and $\cE^{mop}.$ Namely, $\cE^{op}$ has the opposite underlying category, but the order of multiplication is the same. On the other hand, $\cE^{mop}$ has the same underlying category, but the order of multiplication is reversed. In particular, left $\cE$-modules are the same as right $\cE^{mop}$-modules.

Given a left $\cE$-module $\cC$ and two objects $x,y\in\cC,$ we denote by $\un{\Hom}_{\cC/\cE}(x,y)\in\cE$ the relative internal $\Hom$-object. If it exists, it is uniquely determined by the universal property
\begin{equation*}
\Map_{\cE}(z,\un{\Hom}_{\cC/\cE}(x,y))\simeq \Map_{\cC}(z\otimes x,y).
\end{equation*}

For an $\bE_1$-ring $A,$ we will denote by $\Mod\hy A$ the category of right $A$-modules. If $A$ is connective, we will sometimes write $D(A)$ instead of $\Mod\hy A.$ When we discuss bounded below or bounded above modules, we will usually consider the homological grading. In particular, if $A$ is connective, then we denote by $D_{\geq 0}(A)$ the category of connective right $A$-modules.

If $A$ is an algebra in an $\bE_1$-monoidal category $\cE,$ then we also denote by $\Mod\hy A$ the category of right $A$-modules in $\cE.$ This should not lead to confusion, because the corresponding monoidal category will be clear from the context.

Finally, we recall certain Grothendieck's axioms for abelian categories \cite{Gro} which make sense for a presentable $\infty$-category $\cC.$ First, we say that $\cC$ satisfies strong (AB5) if filtered colimits in $\cC$ commute with finite limits. We say that $\cC$ satisfies (AB6) if for any set $I,$ for any collection of directed posets $(J_i)_{i\in I}$ and for any family of functors $J_i\to \cC,$ $j_i\mapsto x_{j_i},$ the following natural map is an isomorphism:
\begin{equation*}
\indlim[(j_i)_i\in\prod\limits_{i\in I}J_i] \prodd[i]x_{j_i}\xto{\sim} \prodd[i]\indlim[j_i\in J_i]x_{j_i}.
\end{equation*}  
If $\cC$ is stable, then it automatically satisfies strong (AB5), and the axiom (AB6) is equivalent to the dualizability of $\cC$ \cite[Proposition 1.53]{E24}. In general, $\cC$ is compactly assembled if and only if it satisfies both strong (AB5) and (AB6).

\subsection{Rigid monoidal categories}
\label{ssec:rigid_cats}

We recall the following definition due to Gaitsgory and Rozenblyum.

\begin{defi}\cite[Definition 9.1.2]{GaRo17}\label{def:rigid_category}
Let $\cE\in\Alg_{\bE_1}(\Pr^L_{\st})$ be a presentable stable $\bE_1$-monoidal category. Then $\cE$ is called rigid if the following conditions are satisfied.

\begin{enumerate}[label=(\roman*), ref=(\roman*)]
	\item The unit object $1_{\cE}$ is compact.
	
	\item The multiplication functor $\mult:\cE\otimes\cE\to\cE$ has a colimit-preserving right adjoint $\mult^R:\cE\to\cE\otimes\cE,$ which is $\cE\hy\cE$-linear (i.e. $\mult^R$ is a bimodule functor).
\end{enumerate}
\end{defi}

The relation with the classical notion of rigidity for small monoidal categories is given by the following basic statement.

\begin{prop}\cite{GaRo17}
Let $\cE$ be a presentable stable $\bE_1$-monoidal category which is compactly generated. Then the following are equivalent.
\begin{enumerate}[label=(\roman*), ref=(\roman*)]
	\item $\cE$ is rigid.
	
	\item The unit object $1_{\cE}$ is compact, and every compact object is left and right dualizable.
	
	\item An object of $\cE$ is right dualizable iff it is left dualizable iff it is compact.
	
	\item We have a monoidal equivalence $\cE\simeq\Ind(\cA),$ where $\cA$ is a small monoidal stable category, such that each object of $\cA$ is left and right dualizable in $\cA.$
\end{enumerate}
\end{prop}

\begin{proof}
This is essentially a reformulation of \cite[Lemma 9.1.5, Corollary 9.1.7]{GaRo17}. Note that if the unit object is compact, then any (left or right) dualizable object is compact. 
\end{proof}

We recall the natural self-duality for rigid categories.

\begin{prop}\label{prop:rigid_categories_basics}\cite[Section 9.2.1]{GaRo17}
Any rigid $\bE_1$-monoidal category $\cE$ is dualizable. Moreover, we have an equivalence $\cE^{\vee}\simeq\cE,$ such that the evaluation and coevaluation functors are given by
\begin{equation*}
\ev_{\cE}:\cE\otimes\cE^{\vee}\simeq\cE\otimes\cE\xto{\mult}\cE\xto{\Hom(1,-)}\Sp,
\end{equation*}
\begin{equation*}
\coev_{\cE}:\Sp\xto{1}\cE\xto{\mult^R}\cE\otimes\cE\simeq\cE^{\vee}\otimes\cE.
\end{equation*} 
\end{prop}

We recall some basic examples.

\begin{example}
\begin{enumerate}[label=(\roman*), ref=(\roman*)]
	\item Let $X$ be a compact Hausdorff space. Then the category $\Shv(X;\Sp)$ is rigid symmetric monoidal.
	\item Let $\mk$ be am $\bE_2$-ring. Then the category $\Mod\hy \mk$ is rigid $\bE_1$-monoidal.
	\item Let $X$ be a quasi-compact quasi-separated scheme. Then the category $D_{qc}(X)$ is rigid symmetric monoidal.
\end{enumerate}
\end{example}

We recall the following notion of left and right trace-class morphisms, due to Clausen and Scholze \cite[Definition 13.11]{CS20} (more precisely, we consider the straightforward generalization to $\bE_1$-monoidal categories). Recall that for two objects $x,y$ in an $\bE_1$-monoidal category $\cE,$ the left resp. right internal $\Hom$ (if it exists) is defined by the following universal property:
\begin{equation*}
\Map_{\cE}(z,\un{\Hom}^l(x,y))\simeq \Map_{\cE}(z\otimes x,y)\quad \Map_{\cE}(z,\un{\Hom}^r(x,y))\simeq \Map_{\cE}(x\otimes z,y).
\end{equation*}

\begin{defi}\label{def:trace_class} Let $\cE$ be a (left and right) closed $\bE_1$-monoidal category. A morphism $x\to y$ in $\cE$ is said to be left (resp. right) trace-class if the corresponding morphism $1\to\un{\Hom}^l(x,y)$ (resp. $1\to\un{\Hom}^r(x,y)$) factors through $y\otimes \un{\Hom}^l(x,1)$ (resp. $\un{\Hom}^r(x,1)\otimes y$).
\end{defi}

For brevity we will use the notation $\bbD^l(x)=x^{l\vee}=\un{\Hom}^l(x,1),$ $\bbD^r(x)=x^{r\vee}=\un{\Hom}^r(x,1)$ for $x\in\cE,$ and call the object $x^{l\vee}$ resp. $x^{r\vee}$ the left resp. right predual of $x.$ The following proposition gives an alternative characterization of rigid $\bE_1$-monoidal categories, which will be very useful in the proof of rigidity of the category of localizing motives \cite{E}.

\begin{prop}\label{prop:rigidity_criterion}\cite{Ram24b}
Let $\cE$ be a presentable stable $\bE_1$-monoidal category. Then $\cE$ is rigid if and only if the following two conditions hold.
\begin{enumerate}[label=(\roman*), ref=(\roman*)]
\item The unit object of $\cE$ is compact.
\item $\cE$ is generated via colimits by the objects of the form $x=\indlim[n\in\N]x_n,$ where each map $x_n\to x_{n+1}$ is both left and right trace-class. 
\end{enumerate} 
\end{prop}

\begin{proof}
This is completely analogous to the symmetric monoidal case, for which we refer to \cite[Corollary 4.57]{Ram24b}.
\end{proof}

We refer to \cite[Section 1.8]{E24} for the notion of a compact map in a presentable stable category. If $\cC$ is a dualizable category, then a map $f:x\to y$ is compact in $\cC$ if and only if $f$ is in the image of the map $\pi_0\Map(\cY(x),\hat{\cY}(y))\to\pi_0\Map(x,y).$

\begin{prop}\cite{Ram24b}\label{prop:trace_class_iff_compact}
Let $\cE\in\Alg_{\bE_1}(\Pr^L_{\st})$ be a presentable stable $\bE_1$-monoidal category, such that $\cE$ is dualizable and the unit object $1_{\cE}$ is compact (we do not require that the multiplication functor is strongly continuous). Then for $x,y\in\cE$ we have natural maps
\begin{equation}\label{eq:from_trace_class_to_compact}
\Hom_{\cE}(1,y\otimes x^{l\vee})\to \Hom_{\Ind(\cE)}(\cY(x),\hat{\cY}(y)),\quad \Hom_{\cE}(1,x^{r\vee}\otimes y)\to \Hom_{\Ind(\cE)}(\cY(x),\hat{\cY}(y)).
\end{equation}
If $\cE$ is rigid, then the maps \eqref{eq:from_trace_class_to_compact} are isomorphisms. In particular, a map $f:x\to y$ in $\cE$ is left trace-class iff it is right trace-class iff it is compact.
\end{prop}

\begin{proof}
Again, this is analogous to the symmetric monoidal case, for which we refer to \cite[Corollary 4.52]{Ram24b}.
\end{proof}

The canonical monoidal autoequivalence described in the next proposition is the inverse to the equivalence discussed in \cite[Section 9.2.7]{GaRo17}, see Remark \ref{rem:relative_and_absolute_duality} below. It $\cE$ is symmetric monoidal, then this autoequivalence is canonically isomorphic to the identity functor.

\begin{prop}\label{prop:autoequivalence_rigid_monoidal}
Let $\cE$ be a rigid $\bE_1$-monoidal category.
\begin{enumerate}[label=(\roman*), ref=(\roman*)]
\item The composition
\begin{equation}\label{eq:D^r_is_monoidal_functor}
\cE\xto{\hat{\cY}}\Ind(\cE^{\omega_1})\xto{\Ind(\bbD^r)}\Ind(\cE^{mop,op})
\end{equation}
(which is a priori oplax monoidal) is monoidal. \label{oplax_monoidal_is_monoidal}
\item The composition
\begin{equation*}
	\cE\xto{\hat{\cY}}\Ind(\cE^{\omega_1})\xto{\Ind((\bbD^r)^2)}\Ind(\cE)
\end{equation*}
(which because of \ref{oplax_monoidal_is_monoidal} is a priori lax monoidal) is monoidal.
\label{lax_monoidal_is_monoidal}
\item We obtain the monoidal endofunctor
\begin{equation*}
\Phi_{\cE}:\cE\xto{\hat{\cY}}\Ind(\cE^{\omega_1})\xto{\Ind((\bbD^r)^2)}\Ind(\cE)\xto{\colim}\cE.
\end{equation*}
The functor $\Phi_{\cE}$ is an autoequivalence, and its inverse has a similar description as a composition
\begin{equation*}
\Phi^{-1}_{\cE}:\cE\xto{\hat{\cY}}\Ind(\cE^{\omega_1})\xto{\Ind((\bbD^l)^2)}\Ind(\cE)\xto{\colim}\cE.
\end{equation*} \label{autoequivalence}
\end{enumerate}
\end{prop}

\begin{proof}
We prove \ref{oplax_monoidal_is_monoidal}, and the rest is analogous. Suppose that $f:x\to x',$ $g:y\to y'$ are compact maps in $\cE,$ which are in particular right trace-class by Proposition \ref{prop:trace_class_iff_compact}. The choice of right trace-class witnesses $\wt{f}:1\to x^{r\vee}\otimes x',$ $\wt{g}:1\to y^{r\vee}\otimes y'$ allows to define the composition map 
\begin{equation*}
(x'\otimes y')^{r\vee}\to y^{r\vee}\otimes y'\otimes (x'\otimes y')^{r\vee}\to y^{r\vee}\otimes x^{r\vee}\otimes x'\otimes y'\otimes (x'\otimes y')^{r\vee}\to y^{r\vee}\otimes x^{r\vee},  
\end{equation*}
such that the following diagram commutes
\begin{equation*}
\begin{tikzcd}
y^{'r\vee}\otimes x^{'r\vee}\ar[r]\ar[d] & y^{r\vee}\otimes x^{r\vee}\ar[d]\\
(x'\otimes y')^{r\vee}\ar[r]\ar[ru] & (x\otimes y)^{r\vee}.
\end{tikzcd}
\end{equation*}
Now if $\hat{\cY}(x)=\inddlim[i] x_i,$ $\hat{\cY}(y)=\indlim[j]y_j,$ then we conclude that the map
\begin{equation*}
\proolim[i,j](y_j^{r\vee}\otimes x_i^{r\vee})\to \proolim[i,j](x_i\otimes y_j)^{r\vee}
\end{equation*}
is an isomorphism in $\Pro(\cE),$ which means that the functor \eqref{eq:D^r_is_monoidal_functor} is monoidal.
\end{proof}

We mention the following statement on monoidal functors from monoidal rigid categories.

\begin{prop}
Let $\cE,\cE'\in\Alg_{\bE_1}(\Pr^L_{\st}),$ and suppose that $\cE$ is rigid. Suppose that the unit object of $\cE'$ is compact. Let $F:\cE\to\cE'$ be a continuous monoidal functor.
\begin{enumerate}[label=(\roman*),ref=(\roman*)]
\item The functor $F$ is strongly continuous. \label{strcont}
\item Denote by $\im(F)\subset \cE'$ the full subcategory generated by the essential image of $F$ by colimits. Then $\im(F)$ is a rigid $\bE_1$-monoidal category. \label{image_is_rigid}
\end{enumerate}
\end{prop}

\begin{proof}
A version of Proposition \ref{prop:trace_class_iff_compact} (with the same proof) implies that any (left or right) trace-class map in $\cE'$ is compact. This implies part \ref{strcont}: $F$ sends compact maps in $\cE,$ which are the same as right trace-class maps, to right trace-class maps in $\cE',$ which are compact.

Part \ref{image_is_rigid} follows from the criterion of rigidity given by Proposition \ref{prop:rigidity_criterion}. Namely, by assumption the unit object of $\im(F)$ is compact. For an object $x\in\cE$ of the form $\indlim[n]x_n,$ where the maps $x_n\to x_{n+1}$ are both left and right trace-class, we see that the object $F(x)\cong\indlim[n]F(x_n)$ has the same property. Such objects $F(x)$ generate the category $\im(F)$ by colimits, hence it is rigid.  
\end{proof}

\subsection{Dualizable modules over rigid categories}

We recall the following notions for a general presentable stable $\bE_1$-monoidal category, but we will only use them in the rigid case.

\begin{defi}\label{def:dualizable_over_E_1_monoidal}
Let $\cE\in\Alg_{\bE_1}(\Pr^L_{\st})$ be a presentable stable $\bE_1$-monoidal category.
\begin{enumerate}[label=(\roman*), ref=(\roman*)]
\item We denote by $\Pr^L_{\cE}$ the category $\Mod_{\cE}(\Pr^L_{\st})$ of left $\cE$-modules in $\Pr^L_{st}.$ If $\kappa$ is a regular cardinal such that $\cE\in \Alg_{\bE_1}(\Pr^L_{\st,\kappa})$ (i.e. $\cE$ is $\kappa$-presentable, the unit object is $\kappa$-compact and the multiplication functor preserves $\kappa$-compact objects), then we denote by $\Pr^L_{\cE,\kappa}$ the category of left $\cE$-modules in $\Pr^L_{\st,\kappa}.$
\item A left $\cE$-module $\cC$ is dualizable over $\cE$ if there exists a right $\cE$-module $\cC^{\vee},$ together with the $\cE\hy\cE$-linear continuous functor
\begin{equation*}
\ev_{\cC/\cE}:\cC\otimes\cC^{\vee}\to \cE
\end{equation*}
and a continuous functor
\begin{equation*}
\coev_{\cC/\cE}:\Sp\to\cC^{\vee}\tens{\cE}\cC,
\end{equation*}
such that the compositions
\begin{equation*}
	\cC\simeq \cC\otimes\Sp\xto{\id\boxtimes\coev_{\cC/\cE}} \cC\otimes \cC^{\vee}\tens{\cE}\cC\xto{\ev_{\cC/\cE}\boxtimes\id}\cE\tens{\cE}\cC\simeq \cC,
\end{equation*}
\begin{equation*}
	\cC^{\vee}\simeq\Sp\otimes\cC^{\vee}\xto{\coev_{\cC/\cE}\boxtimes\id} \cC^{\vee}\tens{\cE}\cC\otimes \cC^{\vee}\xto{\id\boxtimes \ev_{\cC/\cE}}\cC^{\vee}\tens{\cE}\cE\simeq\cC^{\vee}.
\end{equation*}
are isomorphic to the identity functors (in the category of $\cE$-linear endofunctors of $\cC$ resp. $\cC^{\vee}$).
\item We denote by $\Cat_{\cE}^{\dual}\subset \Pr^L_{\cE}$ the non-full subcategory in which the objects are dualizable left $\cE$-modules and the morphisms are the $\cE$-linear functors with $\cE$-linear right adjoint.
\item For a presentable stable left $\cE$-module $\cC,$ an object $c\in\cC$ is relatively compact over $\cE$ if the (unique) $\cE$-linear functor $\cE\to\cC$ sending $1_{\cE}$ to $c$ has a continuous $\cE$-linear right adjoint. We say that $\cC$ is relatively compactly generated over $\cE$ if $\cC$ is generated via colimits by the objects $x\otimes c,$ where $c\in\cC$ is relatively compact and $x\in\cE.$ We denote by $\Cat_{\cE}^{\cg}\subset\Cat_{\cE}^{\dual}$ the full subcategory of relatively compactly generated categories.
\end{enumerate}
\end{defi}

\begin{remark}\label{rem:Cat^perf_Cat^cg}
Suppose that the category $\cE$ in Definition \ref{def:dualizable_over_E_1_monoidal} is compactly generated, the unit object is compact and the tensor product of compact objects is compact. Then we have an equivalence $\Ind(-):\Cat_{\cE^{\omega}}^{\perf}\xto{\sim}\Cat_{\cE}^{\cg},$ where $\Cat_{\cE^{\omega}}^{\perf}$ is the category of left $\cE^{\omega}$-modules in $\Cat^{\perf}.$
\end{remark}

In the case when $\cE$ is rigid, the above notions are directly related to the case of the absolute base. 

\begin{prop}\label{prop:dualizable_over_E_same_as_modules}\cite{GaRo17}
Let $\cE$ be a rigid $\bE_1$-monoidal category, and let $\cC$ be a presentable stable left $\cE$-module. 
\begin{enumerate}[label=(\roman*), ref=(\roman*)]
\item An object $c\in\cC$ is relatively compact over $\cE$ if and only if $c$ is compact. \label{rel_compact_abs_compact}
\item The $\cE$-action functor 
$$\mu_{\cC}:\cE\otimes\cC\to\cC$$ is strongly continuous. \label{action_strcont}
\item $\cC$ is dualizable over $\cE$ if and only if it is dualizable over $\Sp.$ \label{dualizable_relative_absolute}
\item Let $\cD$ be another presentable stable left $\cE$-module, and Let $F:\cC\to\cD$ be a continuous lax $\cE$-linear functor. Then $F$ is $\cE$-linear. In particular, if $F$ is a continuous right adjoint to an $\cE$-linear functor, then $F$ is $\cE$-linear. \label{lax_is_strict}
\item The induced functor $\Cat_{\cE}^{\dual}\to\Mod_{\cE}(\Cat_{\st}^{\dual})$ is an equivalence. \label{relative_Cat^dual_is_category_of_modules}
\end{enumerate}
\end{prop}

\begin{proof}
The statements \ref{rel_compact_abs_compact}, \ref{action_strcont}, \ref{dualizable_relative_absolute}, \ref{lax_is_strict} are respectively \cite[Corollary 9.3.4, Lemma 9.3.2, Proposition 9.4.4, Lemma 9.3.6]{GaRo17}.
The statement \ref{relative_Cat^dual_is_category_of_modules} follows directly from \ref{action_strcont}, \ref{dualizable_relative_absolute}, \ref{lax_is_strict}.
\end{proof}
We refer to Remark \ref{rem:relative_and_absolute_duality} for a further discussion. 

Recall that by \cite[Theorem A]{Ram24a} for any presentable symmetric monoidal stable category $\cE,$ the category $\Cat_{\cE}^{\dual}$ is presentable. This in fact holds when $\cE$ is only $\bE_1$-monoidal, and the proof is basically the same.
We formulate the presentability statement in the case when $\cE$ is rigid, together with an explicit generator and a description of $\kappa$-compact objects. As in \cite[Section 4.4]{E24}, we denote by $\Shv_{\geq 0}(\R;\Sp)$ the category of sheaves of spectra on the real line with singular support in $\R\times\R_{\geq 0}\subset T^*\R$ (we refer to \cite{KS, RS} for the notion of singular support).

\begin{theo}\label{th:presentability_of_Cat_E^dual}
Let $\cE$ be a rigid $\bE_1$-monoidal category. Then the category $\Cat_{\cE}^{\dual}$ is $\omega_1$-presentable, and it is generated via colimits by the single $\omega_1$-compact object $\Shv_{\geq 0}(\R;\cE)\simeq\cE\otimes\Shv_{\geq 0}(\R;\Sp).$ Moreover, for an object $\cC\in\Cat_{\cE}^{\dual}$ and for an uncountable regular cardinal $\kappa$ the following are equivalent.
\begin{enumerate}[label=(\roman*),ref=(\roman*)]
\item $\cC$ is $\kappa$-compact in $\Cat_{\cE}^{\dual}.$ \label{kappa_compact_in_Cat_E^dual}
\item The functors $\ev_{\cC/\cE}:\cC\otimes\cC^{\vee}\to\cE$ and $\coev_{\cC/\cE}:\Sp\to\cC^{\vee}\tens{\cE}\cC$ preserve $\kappa$-compact objects. \label{ev_coev_kappa_strcont}
\end{enumerate}
\end{theo}

\begin{proof}
By \cite[Theorem D.1]{E24}, the category $\Cat_{\st}^{\dual}$ is generated via colimits by the $\omega_1$-compact object $\Shv_{\geq 0}(\R;\Sp).$ The equivalence $\Cat_{\cE}^{\dual}\simeq \Mod_{\cE}(\Cat_{\st}^{\dual})$ implies that the category $\Cat_{\cE}^{\dual}$ is generated via colimits by the $\omega_1$-compact object $\Shv_{\geq 0}(\R;\cE).$

The proof of the equivalence \Iff{kappa_compact_in_Cat_E^dual}{ev_coev_kappa_strcont} is analogous to the proof of Theorem \cite[Theorem C.6]{E24}, equivalence between (i) and (iii). 
\end{proof}
In fact, one can show that all the statements of \cite[Theorem C.6]{E24} hold in the relative context over a rigid $\bE_1$-monoidal base category. 

For a future reference, we prove the following statements about the category $\Cat_{\cE}^{\cg}.$

\begin{theo}\label{th:Cat^cg_compactly_assembled}
Let $\cE$ be a rigid $\bE_1$-monoidal category.
\begin{enumerate}[label=(\roman*),ref=(\roman*)]
\item The categories $\Alg_{\bE_1}(\cE)$ and $\Cat_{\cE}^{\cg}$ are compactly assembled. \label{compass_over_rigid}
\item The functor
\begin{equation*}
\Alg_{\bE_1}(\cC)\to \Cat_{\cE}^{\cg},\quad A\mapsto\Mod\hy A,
\end{equation*}
is strongly continuous, i.e. the following square commutes:
\begin{equation*}
\begin{CD}
\Alg_{\bE_1}(\cC)@>>> \Cat_{\cE}^{\cg}\\
@V\hat{\cY}VV @V\hat{\cY}VV\\
\Ind(\Alg_{\bE_1}(\cC)) @>>> \Ind(\Cat_{\cE}^{\cg}).
\end{CD}
\end{equation*} \label{strcont_from_E_1_alg_to_Cat^cg}
\item The right adjoint to the inclusion functor $\iota:\Cat_{\cE}^{\cg}\to \Cat_{\cE}^{\dual}$ commutes with filtered colimits. In particular, this inclusion functor preserves and reflects $\kappa$-compactness for any regular cardinal $\kappa.$ \label{strcont_Cat^cg_to_Cat^dual}
\end{enumerate}
\end{theo}

We need the following observation.

\begin{lemma}\label{lem:sufficient_condition_for_hat_Y}
Let $\cC$ be an accessible category with filtered colimits, and let $x_0\in\cC^{\omega}$ be a compact object. Suppose that an ind-object $\inddlim[i\in I](x_0\to y_i)\in\Ind(\cC_{x_0/})$ satisfies the universal property
\begin{equation*}
\Map_{\Ind(\cC_{x_0/})}(\inddlim[i]y_i,\inddlim[j]z_j)\xto{\sim}\Map_{\cC_{x_0/}}(\indlim[i]y_i,\indlim[j]z_j),\quad \inddlim[j\in J](x_0\to z_j)\in \Ind(\cC_{x_0/}). 
\end{equation*}
The the ind-object $\inddlim[i]y_i\in\Ind(\cC)$ satisfies the universal property
\begin{equation}\label{eq:univ_property_of_ind_system}
	\Map_{\Ind(\cC)}(\inddlim[i]y_i,\inddlim[j]z_j)\xto{\sim}\Map_{\cC}(\indlim[i]y_i,\indlim[j]z_j),\quad  \inddlim[j\in J]z_j\in \Ind(\cC).
\end{equation}
\end{lemma}

\begin{proof}
The assumption on the ind-system $(x_0\to y_j)_{i\in I}$ implies that the map \eqref{eq:univ_property_of_ind_system} induces a weak homotopy equivalence on the fibers over $\Map(x_0,\indlim[j]z_j)\simeq\indlim[j]\Map(x_0,z_j).$ It follows that the map \eqref{eq:univ_property_of_ind_system} itself is a weak homotopy equivalence.
\end{proof}

\begin{proof}[Proof of Theorem \ref{th:Cat^cg_compactly_assembled}]
We first prove \ref{strcont_Cat^cg_to_Cat^dual}. Denoting the right adjoint by $\iota^R:\Cat_{\cE}^{\dual}\to\Cat_{\cE}^{\cg},$ we observe that $\iota^R(\cC)\subset\cC$ is the cocomplete $\cE$-submodule, generated by the compact objects of $\cC.$ Given an ind-system $(\cC_i)_{i\in I}$ in $\Cat_{\cE}^{\dual},$ we deduce from \cite[Proposition 1.67]{E24} that the functor
\begin{equation}\label{eq:comparison_iota^R}
\indlim[i]^{\cont}\iota^R(\cC_i)\to\iota^R(\indlim[i]^{\cont}\cC_i)
\end{equation}
is fully faithful. By \cite[Proposition 1.71]{E24} any compact object in $\indlim[i]^{\cont}\cC_i$ is isomorphic to the image of a compact object in some $\cC_{i_0},$ hence the functor \eqref{eq:comparison_iota^R} is also essentially surjective.

Next, we prove that the category $\Alg_{\bE_1}(\cE)$ is compactly assembled. Note that the forgetful functor $\Alg_{\bE_1}(\cE)\to\cE$ is conservative and it commutes with filtered colimits. Hence, the image of its left adjoint $T_{\cE}(-):\cE\to\Alg_{\bE_1}(\cE)$ generates the target via colimits. It also follows that the (a priori partially defined) left adjoint to the colimit functor $\Ind(\Alg_{\bE_1}(\cE))\to \Alg_{\bE_1}(\cE)$ is defined on objects of the form $T_{\cE}(x),$ and it is given by
\begin{equation*}
\hat{\cY}_{\Alg_{\bE_1}(\cE)}(T_{\cE}(x))=\inddlim[i] T_{\cE}(x_i),\quad\text{where }\hat{\cY}_{\cE}(x)=\inddlim[i]x_i.
\end{equation*}
We conclude that this left adjoint is defined on the whole category $\Alg_{\bE_1}(\cE),$ which is therefore compactly assembled.

Now, for an $\bE_1$-algebra $A\in\Alg_{\bE_1}(\cE)$ we can consider the category $\Mod\hy A$ as an object of the category $(\Cat_{\cE}^{\cg})_{\cE/},$ where the functor $\cE\to\Mod\hy A$ sends $1$ to $A.$ Then the right adjoint to the functor $\Alg_{\bE_1}(\cE)\to (\Cat_{\cE}^{\cg})_{\cE/}$ is given by
\begin{equation*}
(\Cat_{\cE}^{\cg})_{\cE/}\to \Alg_{\bE_1}(\cE),\quad (\cE\xto{F}\cC)\mapsto \un{\End}_{\cC/\cE}(F(1)),
\end{equation*} 
hence it commutes with filtered colimits. It follows that the (a priori partially defined) left adjoint to the functor $\Ind((\Cat_{\cE}^{\cg})_{\cE/})\to (\Cat_{\cE}^{\cg})_{\cE/}$ is defined on objects of the form $\Mod\hy A,$ and it is given by
\begin{equation*}
\hat{\cY}_{(\Cat_{\cE}^{\cg})_{\cE/}}(\Mod\hy A)=\inddlim[i]\Mod\hy A_i,\quad\text{where }\hat{\cY}(A)=\inddlim[i]A_i.
\end{equation*}
By Lemma \ref{lem:sufficient_condition_for_hat_Y} we conclude that the same holds for the partially defined left adjoint to the colimit functor $\Ind(\Cat_{\cE}^{\cg})\to \Cat_{\cE}^{\cg}.$

It remains to observe that the category $\Cat_{\cE}^{\cg}$ is generated via colimits by the objects of the form $\Mod\hy A,$ $A\in\Alg_{\bE_1}^{\cg}(\cE).$ Indeed, take some category $\cC\in\Cat_{\cE}^{\cg},$ and for a finite collection $S$ of isomorphism classes of compact objects of $\cC$ denote by $\cC_S\subset\cC$ the cocomplete $\cE$-submodule generated by the objects in $S.$ Then $\cC\simeq\indlim[S]\cC_S,$ and we have
\begin{equation*}
\cC_S\simeq \Mod\hy \un{\End}_{\cC/\cE}(\biggplus[x\in S]x).
\end{equation*}
This proves \ref{compass_over_rigid} and \ref{strcont_from_E_1_alg_to_Cat^cg}.
\end{proof}

We recall the following adjunction.

\begin{prop}\label{prop:relative_nonstandard_adjunction} Let $\cE$ be a rigid $\bE_1$-monoidal category, and  choose an uncountable regular cardinal $\kappa.$ Then the inclusion functor $\Cat_{\cE}^{\dual}\to \Pr^L_{\cE,\kappa}$ has a right adjoint, given by $\cC\mapsto\Ind(\cC^{\kappa}).$\end{prop}

\begin{proof} The special case $\cE=\Sp$ is proved in \cite[Proposition 1.89]{E24}. The general case follows formally, since $\Cat_{\cE}^{\dual}\simeq \Mod_{\cE}(\Cat_{\st}^{\dual})$ (Proposition \ref{prop:dualizable_over_E_same_as_modules}) and $\Pr^L_{\cE,\kappa}\simeq \Mod_{\cE}(\Pr^L_{\st,\kappa}).$
	
Note that if $\cC$ is a $\kappa$-presentable stable left $\cE$-module, then the left $\cE$-module structure on $\Ind(\cC^{\kappa})$ is induced from the left $\Ind(\cE^{\kappa})$-module structure via the monoidal functor $\hat{\cY}:\cE\to\Ind(\cE^{\kappa}).$\end{proof}

We observe that the Calkin construction defined in \cite[Section 1.11]{E24} can be applied in the relative context.

\begin{prop}\label{prop:Calkin_over_rigid} Let $\cE$ be a rigid $\bE_1$-monoidal category, and let $\cC$ be a dualizable left $\cE$-module. Choose an uncountable regular cardinal $\kappa.$ Then the inclusion functor
\begin{equation*}
\hat{\cY}:\cC\to\Ind(\cC^{\kappa})
\end{equation*}
is naturally $\cE$-linear, and this gives the structure of a left $\cE$-module on the category
\begin{equation*}
\Ind(\Calk_{\kappa}^{\cont}(\cC))\simeq \Ind(\cC^{\kappa})/\hat{\cY}(\cC).
\end{equation*}
In particular, we obtain a functorial short exact sequence in $\Cat_{\cE}^{\dual}:$
\begin{equation*}
0\to\cC\xto{\hat{\cY}}\Ind(\cC^{\kappa})\to \Ind(\Calk_{\kappa}^{\cont}(\cC))\to 0.
\end{equation*}
\end{prop}

\begin{proof}
This is straightforward.
\end{proof}

We conclude this subsection with some remarks which might be helpful for understanding the relation between the absolute and relative duality. The statements in Remark \ref{rem:relative_and_absolute_duality} about dual categories can be obtained as a straightforward generalization of results from \cite[Section 9.4]{GaRo17} to the case of an $\bE_1$-monoidal rigid category (in loc. cit. the authors consider the case of a symmetric monoidal rigid category). 

\begin{remark}\label{rem:relative_and_absolute_duality}
Let $\cE$ be a rigid $\bE_1$-monoidal category as above.
\begin{enumerate}[label=(\roman*),ref=(\roman*)]
\item Under the identification $\cE^{\vee}\simeq \Fun^L(\cE,\Sp),$ the equivalence $\cE\simeq \cE^{\vee}$ from Proposition \ref{prop:rigid_categories_basics}  is given by
\begin{equation}\label{eq:functor_Phi_1}
\Phi_1:\cE\xto{\sim}\cE^{\vee},\quad \Phi_1(x)=\Hom(1,-\otimes x).
\end{equation}
Clearly, we have another equivalence
\begin{equation*}
	\Phi_2:\cE\xto{\sim}\cE^{\vee},\quad \Phi_2(x)=\Hom(1,x\otimes -).
\end{equation*}
The composition 
\begin{equation*}
\Phi_{\cE}=\Phi_2^{-1}\circ\Phi_1:\cE\xto{\sim}\cE
\end{equation*} 
is exactly the monoidal auto-equivalence from Proposition \ref{prop:autoequivalence_rigid_monoidal} \ref{autoequivalence}. By default, we use the functor \eqref{eq:functor_Phi_1} to identify $\cE$ and $\cE^{\vee}.$ Given a dualizable left $\cE$-module $\cC,$ we denote by $\Phi_{\cE,*}\cC$ the category $\cC$ with the  $\cE$-action given by the ``restriction of scalars'' via $\Phi_{\cE}.$ \label{functor_Phi_E}
\item It is easy to see that we have an isomorphism of monoidal functors $\Phi_{\cE^{mop}}\cong (\Phi_{\cE}^{-1})^{mop}.$
\item For a dualizable left $\cC$-module $\cE,$ the equivalence $\Fun^L(\cC,\Sp)\xto{\sim}\Fun_{\cE}^L(\cC,\cE)$ is given by
\begin{equation*}
(F:\cC\to\Sp)\mapsto (\cC\xto{\mu_{\cC}^R}\cE\otimes\cC\xto{\id\boxtimes F}\cE\otimes\Sp\simeq \cE).
\end{equation*}
Here the functor $\mu_{\cC}^R$ is the right adjoint to the $\cE$-action functor $\mu_{\cC}:\cE\otimes\cC\to\cC.$
The inverse functor $\Fun_{\cE}^L(\cC,\cE)\xto{\sim}\Fun^L(\cC,\Sp)$ is given by
\begin{equation*}
(G:\cC\to\cE)\mapsto (\cC\xto{G}\cE\xto{\Hom(1.-)}\Sp).
\end{equation*}
The induced right $\cE$-module structure on $\cC^{\vee}$ is given by the composition of monoidal functors
\begin{equation*}
\cE^{mop}\xto{(\Phi_{\cE})^{mop}}\cE^{mop}\to\Fun^L(\cC,\cC)^{mop}\simeq \Fun^L(\cC^{\vee},\cC^{\vee}).
\end{equation*}
Similarly, if $\cD$ is a dualizable right $\cE$-module, then the left $\cE$-module structure on $\cD^{\vee}$ is given by
\begin{equation*}
\cE\xto{\Phi^{-1}_{\cE}}\cE\to\Fun^L(\cD,\cD)^{mop}\simeq \Fun^L(\cD^{\vee},\cD^{\vee}).
\end{equation*}
\item Since the duality functor $(-)^{\vee}:\Cat_{\st}^{\dual}\to\Cat_{\st}^{\dual}$ is a covariant symmetric monoidal involution, the monoidal structure on $\cE$ induces a monoidal structure on $\cE^{\vee}.$ The equivalence \eqref{eq:functor_Phi_1} is naturally refined to a monoidal equivalence $\cE^{mop}\simeq\cE^{\vee}.$

Similarly, if $\cC$ is a dualizable left $\cE$-module, then we get a left $\cE^{\vee}$-module structure on $\cC^{\vee}.$ Under the above monoidal equivalence $\cE^{mop}\simeq\cE^{\vee},$ we obtain exactly the canonical right $\cE$-module structure on $\cC^{\vee}.$ On the other hand, for a dualizable right $\cE$-module $\cD$ this procedure gives the left $\cE$-module $\Phi_{\cE,*}(\cD^{\vee}).$ 
\item Given a presentable stable left $\cE$-module $\cC$ and a right $\cE$-module $\cD,$ the right adjoint to the natural functor $\cD\otimes\cC\to\cD\tens{\cE}\cC$ is given by the composition
\begin{equation*}
\cD\tens{\cE}\cC\simeq \cD\tens{\cE}\cE\tens{\cE}\cC\xto{\id\boxtimes\mult^R\boxtimes\id}\cD\tens{\cE}(\cE\otimes\cE)\tens{\cE}\cC\simeq\cD\otimes\cC.
\end{equation*}
\item Given a dualizable left $\cE$-module $\cC$ and a dualizable right $\cE$-module $\cD,$ we have an equivalence \begin{equation*}
(\cD\tens{\cE}\cC)^{\vee}\simeq\cC^{\vee}\tens{\cE}\Phi_{\cE,*}(\cD^{\vee}).\end{equation*}
The evaluation functor
\begin{equation*}
\ev:(\cD\tens{\cE}\cC)\otimes (\cC^{\vee}\tens{\cE}\Phi_{\cE,*}(\cD^{\vee}))\to\Sp
\end{equation*}
is given by
\begin{equation*}
\ev(d\boxtimes c,c'\boxtimes d')=\Hom(1,\ev_{\cD/\cE}(d',d)\otimes \ev_{\cC/\cE}(c,c')),
\end{equation*}
where $c\in\cC,$ $c'\in\cC^{\vee},$ $d\in\cD,$ $d'\in\cD^{\vee}.$ \label{non_trivial_duality_for_tensor_products}
\item More generally, we can consider the $(\infty,2)$-category in which the objects are $\bE_1$-monoidal rigid categories, the $1$-morphisms $\cE\to\cE'$ are dualizable $\cE'\hy\cE$-bimodules (the composition being the usual tensor product of bimodules), and the $2$-morphisms are continuous $\cE'\hy\cE$-linear functors. Then any $1$-morphism $\cC:\cE\to\cE'$ has both left and right adjoint. Namely, if we consider $\cC$ as a left module over $\cE^{mop}\otimes\cE',$ then $\cC^{\vee}$ has a canonical left $\cE\otimes\cE^{'mop}$-module structure as above, and the left resp. right adjoint to $\cC:\cE\to\cE'$ is given by $\Phi_{\cE^{'mop},*}\cC^{\vee}$ resp. $\Phi_{\cE,*}\cC^{\vee}.$  
\end{enumerate}
\end{remark}

\subsection{Localizing invariants of dualizable modules}

Let $\cE$ be a rigid $\bE_1$-monoidal category, and let $\cT$ be an accessible stable category. The notion of an accessible localizing invariant $\Cat_{\cE}^{\cg}\to \cT$ resp. $\Cat_{\cE}^{\dual}\to\cT$ is defined in the same way as in \cite{BGT}, see also \cite[Section 4]{E24}. If $\kappa$ is a regular cardinal, such that $\cT$ has $\kappa$-filtered colimits, then we denote by $\Fun_{\loc,\kappa}(\Cat_{\cE}^{\cg},\cT)$ resp. $\Fun_{\loc,\kappa}(\Cat_{\cE}^{\dual},\cT)$ the categories of localizing invariants which commute with $\kappa$-filtered colimits.

As in \cite[Definition 4.5]{E24}, for an accessible localizing invariant $F:\Cat_{\cE}^{\cg}\to\cT,$ we denote by $F^{\cont}:\Cat_{\cE}^{\dual}\to\cT$ the localizing invariant given by
\begin{equation*}
F^{\cont}(\cC)=\Omega F(\Ind(\Calk_{\omega_1}^{\cont}(\cC))),\quad \cC\in\Cat_{\cE}^{\dual}.
\end{equation*}

The following is a straightforward generalization of \cite[Theorem 4.10]{E24}.

\begin{theo}\label{th:continuous_loc_invar}
Let $\kappa$ be a regular cardinal and let $\cT$ be an accessible stable category with $\kappa$-filtered colimits. Let $\cE$ be a rigid $\bE_1$-monoidal category. Then the precomposition functor
\begin{equation*}
\Fun^{\acc}(\Cat_{\cE}^{\dual},\cT)\to \Fun^{\acc}(\Cat_{\cE}^{\cg},\cT)
\end{equation*}
induces an equivalence
\begin{equation*}
\Fun_{\loc,\kappa}(\Cat_{\cE}^{\dual},\cT)\xto{\sim} \Fun_{\loc,\kappa}(\Cat_{\cE}^{\cg},\cT).
\end{equation*}
The inverse equivalence is given by $F\mapsto F^{\cont}.$
\end{theo}

We recall the relative version of the universal localizing invariants (depending on the choice of a regular cardinal).

\begin{defi}
As above, let $\cE$ be a rigid $\bE_1$-monoidal category, and let $\kappa$ be a regular cardinal. We denote by
\begin{equation*}
\cU_{\loc,\kappa}:\Cat_{\cE}^{\cg}\to \Mot^{\loc}_{\cE,\kappa}
\end{equation*}
the localizing invariant such that the category $\Mot^{\loc}_{\cE,\kappa}$ is accessible with $\kappa$-filtered colimits, the functor $\cU_{\loc,\kappa}$ commutes with $\kappa$-filtered colimits, and for any other accessible stable category $\cT$ with $\kappa$-filtered colimits we have an equivalence
\begin{equation}\label{eq:universal_property_Mot^loc_kappa}
\Fun_{\loc,\kappa}(\Cat_{\cE}^{\cg},\cT)\simeq \Fun^{\kappa\hy\cont}(\Mot^{\loc}_{\cE,\kappa},\cT).
\end{equation}
For $\kappa=\omega$ we use the notation $\Mot^{\loc}_{\cE}=\Mot^{\loc}_{\cE,\omega}.$
\end{defi}

Here the target of \eqref{eq:universal_property_Mot^loc_kappa} is the category of exact functors which commute with $\kappa$-filtered colimits.

\begin{remark}
Let $\cE$ be a rigid $\bE_1$-monoidal category.
\begin{enumerate}[label=(\roman*),ref=(\roman*)]
\item If $\cE$ is $\bE_n$-monoidal for some $n\geq 1,$ then the categories $\Mot_{\cE,\kappa}^{\loc}$ are naturally $\bE_{n-1}$-monoidal. If $\cE$ is symmetric monoidal, then the categories  $\Mot_{\cE,\kappa}^{\loc}$ are also symmetric monoidal.
\item We will prove in \cite{E} that if $\cE$ is at least $\bE_2$-monoidal, then the category $\Mot_{\cE}^{\loc}$ is rigid monoidal. If $\cE$ is only $\bE_1$-monoidal, then the category $\Mot^{\loc}_{\cE}$ is dualizable and the object $\cU_{\loc}(\cE)\in \Mot^{\loc}_{\cE}$ is compact. Moreover, we have an equivalence $(\Mot^{\loc}_{\cE})^{\vee}\simeq \Mot^{\loc}_{\cE^{mop}},$ and the evaluation functor is given by
\begin{equation*}
\Mot^{\loc}_{\cE}\otimes \Mot^{\loc}_{\cE^{mop}}\to\Sp,\quad \cU_{\loc}(\cC)\otimes\cU_{\loc}(\cD)\mapsto K^{\cont}(\cD\tens{\cE}\cC).
\end{equation*}
\end{enumerate}
\end{remark} 

\subsection{Rigidification and nuclear objects}

In this subsection we will consider only symmetric monoidal categories. For an object $x$ in a symmetric monoidal category $\cC,$ we will use the notation $x^{\vee}=\un{\Hom}_{\cC}(x,1)$ and call $x^{\vee}$ the predual of $x.$ We will sometimes consider the space $\Map(1,x^{\vee}\otimes y)$ as the space of trace-class maps from $x$ to $y.$ If $\cC$ is stable, we will similarly say that $\Hom(1,x^{\vee}\otimes y)$ is the spectrum of trace-class maps.

\begin{remark} A trivial but useful observation is that a morphism $f:x\to y$ is trace-class in $\cC$ if and only if there exists an object $z\in\cC$ and morphisms $g:1\to z\otimes y,$ $h:x\otimes z\to 1$ such that $f$ is (homotopic to) the composition
\begin{equation*}
x\cong x\otimes 1\xto{\id\otimes g} x\otimes z\otimes y\xto{h\otimes\id} 1\otimes y\cong y.	
\end{equation*}	This characterization is helpful in practice, because the predual object $x^{\vee}$ can be very difficult to describe in general.\end{remark}

We recall the rigidification functor.

\begin{prop}\cite{Ram24b}\label{prop:rigidification}
Let $\kappa$ be an uncountable regular cardinal. The inclusion functor $\CAlg^{\rig}\hto \CAlg(\Pr^L_{\st,\kappa})$ has a right adjoint, $\cC\mapsto\cC^{\rig}.$ Explicitly, the category $\cC^{\rig}$ can be described as a full subcategory of $\Ind(\cC^{\kappa}),$ which is generated via colimits by the formal colimits $\inddlim[\Q_{\leq}](F:\Q_{\leq}\to\cC^{\kappa}),$ where for $a<b$ the map $F(a)\to F(b)$ is trace-class in $\cC.$ The adjunction counit is given by the composition
\begin{equation*}
\cC^{\rig}\to\Ind(\cC^{\kappa})\xto{\colim}\cC.
\end{equation*}
Moreover, the category $\cC^{\rig}$ does not depend on the choice of $\kappa.$ If in addition the unit object of $\cC$ is compact, then the functor $\cC^{\rig}\to\cC$ is fully faithful.
\end{prop}

\begin{proof}
The description of the rigidification functor and independence from the choice of $\kappa$ are proved in \cite[Construction 4.75, Theorem 4.77]{Ram24b}. The final assertion follows from the observation that the subcategory $\cC^{\rig}\subset\Ind(\cC^{\kappa})$ is contained in the image of the partially defined left adjoint to the colimit functor $\Ind(\cC^{\kappa})\to\cC$ (since compactness of the unit object implies that the trace-class maps are compact). 
\end{proof}

\begin{remark}\label{rem:rigidification}
\begin{enumerate}[label=(\roman*), ref=(\roman*)]
\item If $\cC\in \CAlg(\Pr^L_{\st})$ is compactly generated and the unit object $1_{\cC}$ is compact, then $\cC^{\rig}\subset \cC$ is generated via colimits by the objects $\indlim[\Q_{\leq}](F:\Q_{\leq}\to\cC^{\omega})\in\cC,$ where for $a<b$ the map $F(a)\to F(b)$ is trace-class in $\cC.$ This follows from the observation that any trace-class map in $\cC$ factors through a compact object (by Proposition \ref{prop:trace_class_iff_compact}). \label{C^rig_for_comp_gen}
\item Let $\cC\in\CAlg(\Pr^L_{\st,\kappa}),$ and take some object $x\in\cC^{\kappa}.$ The predual $\cY(x)^{\vee}\in\Ind(\cC^{\kappa})$ is isomorphic to the image of $x^{\vee}\in\cC$ under the functor $\cC\simeq\Ind_{\kappa}(\cC^{\kappa})\to \Ind(\cC^{\kappa}).$ Hence, a morphism $f:x\to y$ in $\cC^{\kappa}$ is trace-class in $\cC$ if and only if the morphism $\cY(x)\to \cY(x)$ is trace-class in $\Ind(\cC^{\kappa}).$  Therefore, the colimit functor $\colim:\Ind(\cC^{\kappa})\to\cC$ induces an equivalence $\Ind(\cC^{\kappa})^{\rig}\to\cC^{\rig}.$ 
\end{enumerate}
\end{remark}

We will be interested in the conditions which allow to replace colimits over $\Q$ with colimits over $\N$ in Proposition \ref{prop:rigidification}. We recall the general notion of nuclear objects, due to Clausen and Scholze.

\begin{defi}\cite[Definitions 13.10 and 13.12]{CS20}\label{def:nuclear_objects} Let $\cC$ be a presentable stable symmetric monoidal category, which is compactly generated, and suppose that the unit object $1_{\cC}$ is compact.
\begin{enumerate}[label=(\roman*), ref=(\roman*)]
\item An object $M\in C$ is called nuclear if for any compact object $P\in\cC^{\omega},$ any morphism $P\to M$ is trace-class. 
\item An object $M\in C$ is called basic nuclear if we have $M\cong\indlim[n\in\N]M_n,$ where each $M_n$ is compact and the maps $M_n\to M_{n+1}$ are trace-class.
\end{enumerate}
We denote by $\Nuc(\cC)\subset\cC$ the full subcategory of nuclear objects. 
\end{defi}

\begin{remark}
In fact the notion of a nuclear object naturally makes sense for not necessarily stable presentable symmetric monoidal categories which are compactly assembled with a compact unit object. We will use this more general notion in \cite{E}.
\end{remark}

We recall the basic facts about nuclear objects, which are straightforward generalizations of the results of \cite[Section 13]{CS20}.

\begin{prop}\label{prop:nuclear_is_omega_1_presentable}
	Let $\cC$ be as in Definition \ref{def:nuclear_objects}. The full subcategory $\Nuc(\cC)\subset\cC$ is closed under colimits. The category $\Nuc(\cC)$ is $\omega_1$-presentable and the $\omega_1$-compact objects of $\Nuc(\cC)$ are exactly the basic nuclear objects. The inclusion functor $\Nuc(\cC)\to \cC$ preserves and reflects $\omega_1$-compact objects. Considering $\cC^{\rig}$ as a full subcategory of $\cC,$ we have an inclusion $\cC^{\rig}\subset\Nuc(\cC).$ .
\end{prop}

\begin{proof}
The statements on the $\omega_1$-presentability and $\omega_1$-compact objects are obtained as a direct generalization of \cite[Proposition 13.13]{CS20}. The inclusion $\cC^{\rig}\subset\Nuc(\cC)$ follows from Remark \ref{rem:rigidification} \ref{C^rig_for_comp_gen}.
\end{proof}

\begin{prop}\label{prop:equiv_for_nuclear}
Let $\cC$ be as in Definition \ref{def:nuclear_objects}. For an object $M\in\cC,$ the following conditions are equivalent.
\begin{enumerate}[label=(\roman*), ref=(\roman*)]
\item $M$ is nuclear.
\item For any compact object $P\in\cC^{\omega},$ we have an equivalence of spaces
\begin{equation*}
\Map_{\cC}(1,P^{\vee}\otimes M)\xto{\sim} \Map_{\cC}(P,M).
\end{equation*}
\item For any compact object $P\in\cC^{\omega}$ we have an isomorphism
\begin{equation*}
P^{\vee}\otimes M\xto{\sim} \un{\Hom}(P,M)
\end{equation*}
in $\cC.$
\item We have an isomorphism $M\cong \indlim[i\in I]M_i,$ where $I$ is directed, each $M_i$ is compact, and for each $i\in I$ there exists $j\geq i$ such that the map $M_i\to M_j$ is trace-class. \label{nuclear_via_trace_class_approx}
\end{enumerate} 
\end{prop}

\begin{proof}
This is proved in the same way as \cite[Proposition 13.14]{CS20}. 
\end{proof}

The following observation describes a class of situations when the rigidification coincides with the subcategory of nuclear objects. 

\begin{prop}\label{prop:if_trace_class_can_be_decomposed}
Let $\cC$ be as in Definition \ref{def:nuclear_objects}. Suppose that every trace-class map between compact objects of $\cC$ is a composition of two trace-class maps. Then we have an equivalence $\cC^{\rig}\simeq\Nuc(\cC).$
\end{prop}

\begin{proof}
Indeed, take any functor $F:\N\to\cC^{\omega}$ such that each map $F(n)\to F(n+1)$ is trace-class in $\cC.$ Our assumption implies that $F$ can be extended to a functor $\wt{F}:\Q\to\cC^{\omega}$ such that for $a<b$ the map $\wt{F}(a)\to \wt{F}(b)$ is trace-class in $\cC.$ This shows that $\Nuc(\cC)=\cC^{\rig}$ as strictly full subcategories of $\cC.$ 
\end{proof}

The condition from Proposition \ref{prop:if_trace_class_can_be_decomposed} is difficult to verify directly. The following stronger version is conceptually better, and it is possible to verify it in some interesting situations.  

\begin{prop}\cite{And23, AM24} \label{prop:basic_conditions_good_rigidification}
Let $\cD$ be a small symmetric monoidal stable category. For an object $P\in\cD$ we denote by $P^{\vee}\in\Ind(\cD)$ the predual object. Under the identification $\Ind(\cD)\simeq \Fun^L(\cD^{op},\Sp),$ consider the functor
\begin{equation*}
\Ind(\cD)\to \Ind(\cD),\quad X\mapsto X^{\tr}=\Hom(1,(-)^{\vee}\otimes X).
\end{equation*}
The following are equivalent for an oobject $P\in\cD.$
\begin{enumerate}[label=(\roman*), ref=(\roman*)]
\item The object $P^{\tr}$ is nuclear in $\Ind(\cD).$ \label{P^tr_nuclear}
\item For any $Q\in\cD$ we have an isomorphism
\begin{equation}\label{eq:cond_via_tensor_product}
\Hom(1,(-)^{\vee}\otimes P)\tens{\cD}\Hom(1,Q^{\vee}\otimes -)\xto{\sim}\Hom(1,Q^{\vee}\otimes P). 
\end{equation} \label{cond_via_tensor_product}
If the conditions \ref{P^tr_nuclear}, \ref{cond_via_tensor_product} are satisfied for all $P\in\cD,$ then the inclusion functor $\Nuc(\Ind(\cD))\to\Ind(\cD)$ has a continuous right adjoint, which is given by $X\mapsto X^{\tr}.$ Moreover, we have an equivalence $\cC^{\rig}\simeq\Nuc(\cC).$
\end{enumerate}  
\end{prop}

\begin{proof}
The equivalence \Iff{P^tr_nuclear}{cond_via_tensor_product} is tautological: the target of \eqref{eq:cond_via_tensor_product} is isomorphic to $\Hom(Q,P^{\tr}),$ and the source of \eqref{eq:cond_via_tensor_product} is isomorphic to
\begin{equation*}
\indlim[\substack{Q'\in\cD,\\ Q'\to P^{\tr}}]\Hom(1,Q^{\vee}\otimes Q')\cong\Hom(1,Q^{\vee}\otimes P^{\tr}).
\end{equation*}
Thus, \eqref{eq:cond_via_tensor_product} is an isomorphism for all $Q\in\cD$ if and only if $P^{\tr}$ is nuclear.

Clearly, the functor $X\mapsto X^{\tr}$ commutes with colimits. If the conditions \ref{P^tr_nuclear}, \ref{cond_via_tensor_product} are satisfied, then for any $X\in\Ind(\cD)$ the object $X^{\tr}$ is nuclear. The assertion about the right adjoint functor is straightforward, see \cite[Proof of Lemma 3.32]{AM24}. 

Finally, the equivalence $\cC^{\rig}\simeq\Nuc(\cC)$ follows from Proposition \ref{prop:if_trace_class_can_be_decomposed}.
\end{proof}

The following observation relates the trace-class morphisms in $\Ind(\cD)$ and $\Ind(\cD^{op}).$

\begin{prop}\label{prop:trace_class_after_taking_opposite}
	Let $\cD$ be a small stable symmetric monoidal category. For an object $P\in\cD,$ we denote by $P^{op}$ the corresponding object of $\cD^{op}.$ We denote by $P^{\vee}$ resp. $P^{op,\vee}$ the predual object in $\Ind(\cD)$ resp. in $\Ind(\cD^{op}).$
	
	Then for $P,Q\in\cD$ we have a natural isomorphism $\Hom_{\Ind(\cD)}(1,P^{\vee}\otimes Q)\simeq\Hom_{\Ind(\cD^{op})}(1,Q^{op,\vee}\otimes P^{op}),$ such that the following square commutes:
	\begin{equation}\label{eq:comm_square_for_trace_class}
		\begin{CD}
			\Hom_{\Ind(\cD)}(1,P^{\vee}\otimes Q) @>{\sim}>> \Hom_{\Ind(\cD^{op})}(1,Q^{op,\vee}\otimes P^{op})\\
			@VVV @VVV\\
			\Hom_{\cD}(P,Q) @>{\sim}>> \Hom_{\cD^{op}}(Q^{op},P^{op})
		\end{CD}
	\end{equation}
\end{prop}

\begin{proof}
	We have the following isomorphisms, where $X$ runs through objects of $\cD$
	\begin{multline*}
		\Hom_{\Ind(\cD)}(1,P^{\vee}\otimes Q)\simeq \indlim[X\to P^{\vee}]\Hom_{\cD}(1,X\otimes Q)\simeq \indlim[P\otimes X\to 1]\Hom_{\cD}(1,X\otimes Q)\\
		\simeq \indlim[1\to X\otimes Q]\Hom_{\cD}(P\otimes X,1),
	\end{multline*}
	and a similar chain of equivalences identifies the latter spectrum with $\Hom_{\Ind(\cD^{op})}(1,Q^{op,\vee}\otimes P^{op}).$ It is clear that the square \eqref{eq:comm_square_for_trace_class} commutes.
\end{proof}

\begin{prop}\label{prop:from_Ind_D_to_Ind_D^op}
We keep the notation from Proposition \ref{prop:trace_class_after_taking_opposite}. The lax symmetric monoidal functor 
\begin{equation}\label{eq:explicit_functor_Ind_D_Ind_D^op}
\Ind(\cD)\to \Ind(\cD^{op}),\quad \inddlim[i] P_i\mapsto \indlim[i] P_i^{op,\vee}	
\end{equation}
induces an equivalence
\begin{equation*}
\Ind(\cD)^{\rig}\xto{\sim} \Ind(\cD^{op})^{\rig}.
\end{equation*}	
\end{prop}

\begin{proof}
Recall that $(-)^{\vee}:\Cat_{\st}^{\dual}\to \Cat_{\st}^{\dual}$ is a (covariant) symmetric monoidal involution. For a rigid symmetric monoidal category $\cE$ the natural equivalence $\cE\simeq\cE^{\vee}$ is in fact an equivalence of symmetric monoidal categories. This formally implies an equivalence $\Ind(\cD)^{\rig}\simeq \Ind(\cD^{op})^{\rig}.$

Next, we claim that the functor \eqref{eq:explicit_functor_Ind_D_Ind_D^op} takes $\Ind(\cD)^{\rig}$ to $\Ind(\cD^{op})^{\rig}.$ Indeed, if $P\to Q$ is a map in $\cD$ which is trace-class in $\Ind(\cD),$ then the map $P^{op,\vee}\to Q^{op,\vee}$ is trace-class in $\Ind(\cD^{op})$ by Proposition \ref{prop:trace_class_after_taking_opposite}. The assertion now follows from the description of rigidification from Proposition \ref{prop:rigidification}.  
	
It remains to show that the equivalence $\Ind(\cD)^{\rig}\simeq \Ind(\cD^{op})^{\rig}$ is induced by the functor \eqref{eq:explicit_functor_Ind_D_Ind_D^op}. We observe that the composition $F:\Ind(\cD)^{\rig}\simeq (\Ind(\cD)^{\rig})^{\vee}\to \Ind(\cD^{op})$ is the left adjoint to the dual of the inclusion $\iota:\Ind(\cD)^{\rig}\to \Ind(\cD).$ Denote by $\iota^R$ the right adjoint to $\iota.$ 
	
Now, for $\inddlim[i]M_i\in\Ind(\cD)^{\rig}$ and for $Q\in\cD$ we have
\begin{multline*}
\indlim[i]\Hom(Q,M_i)\cong \indlim[i]\Hom(1_{\Ind(\cD)},Q^{\vee}\otimes M_i)\cong \Hom(1_{\Ind(\cD)},Q^{\vee}\otimes (\inddlim[i]M_i))\\
\cong \Hom(1_{\Ind(\cD)^{\rig}},\iota^R(Q^{\vee})\otimes (\inddlim[i]M_i)).
\end{multline*}
This shows that the right adjoint $F^R:\Ind(\cD^{op})\to \Ind(\cD)^{\rig}$ is given on compact objects by $F^R(Q^{op})\cong \iota^R(Q^{\vee}).$ Finally, for $\inddlim[i]P_i\in\Ind(\cD)^{\rig}$ and for $Q\in\cD$ we observe the following isomorphisms
\begin{multline*}
\Hom(\indlim[i]P_i^{op,\vee},Q^{op})\cong \prolim[i]\Hom(P_i^{op,\vee},Q^{op})\cong\prolim[i]\Hom(1_{\cD^{op}},P_i^{op}\otimes Q^{op})\\
\cong \prolim[i]\Hom(P_i\otimes Q,1_{\cD})\cong \prolim[i]\Hom(P_i,Q^{\vee})\cong \Hom(\inddlim[i]P_i,Q^{\vee})\cong \Hom(\inddlim[i]P_i,\iota^R(Q^{\vee}))\\
\cong \Hom(\inddlim[i]P_i,F^R(Q^{op})).
\end{multline*}
This gives an isomorphism
\begin{equation*}
F(\inddlim[i]P_i)\cong\indlim[i] P_i^{op,\vee},
\end{equation*}
as required.
\end{proof}

We now recall the sufficient conditions which guarantee the assumption of Proposition \ref{prop:if_trace_class_can_be_decomposed}. The following two propositions can be considered as a more elaborate version of \cite[Lemma 3.32]{AM24}.

\begin{prop}\label{prop:conditions_for_good_rigidification}
We keep the notation of Proposition \ref{prop:trace_class_after_taking_opposite}. The following conditions are equivalent. 
\begin{enumerate}[label=(\roman*), ref=(\roman*)]
\item For $P,Q\in\cD$ we have an isomorphism
\begin{equation}\label{eq:tensor_product_of_presheaves}
\Hom(1,P\otimes(-)^{\vee})\tens{\cD}\Hom(1,-\otimes Q)\cong \Hom(1,P\otimes Q)
\end{equation}
\label{cond_via_tensor_product_of_presheaves}
\item For every $P\in\cD,$ the object $P^{op,\vee}$ is nuclear in $\Ind(\cD^{op}).$ \label{cond_via_P^v_is_nuclear}
\item For $P,Q\in\cD,$ we have
\begin{equation*}
P^{op,\vee}\otimes Q^{op,\vee}\cong (P\otimes Q)^{op,\vee}
\end{equation*}
\label{cond_via_tensor_product_of_duals}
\end{enumerate}
If these conditions hold, then we have equivalences
\begin{equation}\label{eq:nuclear_and_rigidification}
\Nuc(\Ind(\cD))\simeq\Ind(\cD)^{\rig}\simeq \Ind(\cD^{op})^{\rig}\simeq\Nuc(\Ind(\cD^{op}))
\end{equation}
\end{prop}

\begin{proof}
The equivalence \Iff{cond_via_P^v_is_nuclear}{cond_via_tensor_product_of_duals} follows from Proposition \ref{prop:equiv_for_nuclear}.
We prove \Iff{cond_via_tensor_product_of_presheaves}{cond_via_P^v_is_nuclear}. For $P,Q\in\cD$ we have the following isomorphisms, where $X$ and $Y$ run through objects of $\cD:$
\begin{multline*}
\Hom_{\Ind(\cD)}(1,P\otimes(-)^{\vee})\tens{\cD}\Hom_{\cD}(1,-\otimes Q)\cong \indlim[1\to P\otimes X^{\vee}]\Hom_{\cD}(1,X\otimes Q)\\
\cong \indlim[\substack{1\to P\otimes Y,\\ Y\to X^{\vee}}]\Hom_{\cD}(1,X\otimes Q)\cong\indlim[\substack{1\to P\otimes Y,\\ 1\to X\otimes Q}]\Hom_{\Ind(\cD)}(Y,X^{\vee})\cong \indlim[\substack{Y^{op}\to P^{op,\vee},\\ X^{op}\to Q^{op,\vee}}]\Map_{\cD^{op}}(1,X^{op}\otimes Y^{op})\\
\cong \Hom_{\Ind(\cD^{op})}(1,P^{op,\vee}\otimes Q^{op,\vee}).
\end{multline*}
We see that \eqref{eq:tensor_product_of_presheaves} exactly means that we have an isomorphism
\begin{equation*}
\Hom_{\Ind(\cD^{op})}(1,P^{op,\vee}\otimes Q^{op,\vee})\cong\Hom_{\Ind(\cD^{op})}(Q^{op},P^{op,\vee}).
\end{equation*}
This means that the object $P^{op,\vee}$ is nuclear in $\Ind(\cD^{op}),$ which proves the equivalence between \ref{cond_via_tensor_product_of_presheaves} and \ref{cond_via_P^v_is_nuclear}.

Now, if the conditions \ref{cond_via_tensor_product_of_presheaves}-\ref{cond_via_tensor_product_of_duals} hold, then the conditions of Proposition \ref{prop:basic_conditions_good_rigidification} also hold. Indeed, the isomorphisms \eqref{eq:tensor_product_of_presheaves} imply the isomorphisms \eqref{eq:cond_via_tensor_product}. 
By Proposition \ref{prop:trace_class_after_taking_opposite} the conditions of Proposition \ref{prop:basic_conditions_good_rigidification} also hold for $\cD^{op}.$ Using Propositions \ref{prop:basic_conditions_good_rigidification} and \ref{prop:from_Ind_D_to_Ind_D^op}, we obtain the equivalences \eqref{eq:nuclear_and_rigidification}.
\end{proof}

\begin{prop}\label{prop:right_adjoint_description_for_rigidification}
Suppose that a small stable symmetric monoidal category $\cD$ satisfies the conditions of Proposition \ref{prop:conditions_for_good_rigidification}. Then the inclusion functor $\iota:\Ind(\cD)^{\rig}\to\Ind(\cD)$ has a symmetric monoidal continuous right adjoint $\iota^R.$ Moreover, the composition
\begin{equation}\label{eq:rewriting_from_Ind_D_to_Ind_D^op}
\Ind(\cD)\xto{\iota^R}\Ind(\cD)^{\rig}\simeq \Ind(\cD^{op})^{\rig}\subset \Ind(\cD^{op})
\end{equation}
is given on compact objects by $P\mapsto P^{op,\vee}.$
\end{prop}

\begin{proof}
We know that the functor $\iota^R$ is given on compact objects by $\iota^R(P)=\Hom(1,(-)^{\vee}\otimes P),$ where we identify $\Ind(\cD)\simeq\Fun(\cD^{op},\Sp).$

Next, we identify $\Ind(\cD^{op})\simeq \Fun(\cD,\Sp).$ Applying Proposition \ref{prop:from_Ind_D_to_Ind_D^op}, we see that the composition \eqref{eq:rewriting_from_Ind_D_to_Ind_D^op} sends a compact object $P$ to the following object:
\begin{equation*}
\indlim[1\to Q^{\vee}\otimes P] Q^{op,\vee}\cong \indlim[1\to Q^{\vee}\otimes P] \Hom(1,Q\otimes -)\cong \Hom(1,P\otimes -)\cong P^{op,\vee}. 
\end{equation*}
Here the second isomorphism follows from the condition \ref{cond_via_tensor_product_of_presheaves} of Proposition \ref{prop:conditions_for_good_rigidification}.

Finally, we see that the functor $\iota^R$ is symmetric monoidal using the condition \ref{cond_via_tensor_product_of_duals} from Proposition \ref{prop:conditions_for_good_rigidification}.
\end{proof}

We conclude this subsection with a remark that the conditions of Proposition \ref{prop:conditions_for_good_rigidification} for $\cD$ do not imply the same conditions for $\cD^{op}.$

\begin{remark}
\begin{enumerate}[label=(\roman*), ref=(\roman*)]
\item The conditions of Proposition \ref{prop:conditions_for_good_rigidification} are trivially satisfied for the category $\cD=\cE^{\omega_1},$ where $\cE$ is a rigid symmetric monoidal category. Indeed, for $P\in\cE^{\omega_1}$ with $\hat{\cY}(P)=\inddlim[i]P_i,$ and for $Q\in\cE^{\omega_1}$ we have
\begin{multline*}
\Hom_{\Ind(\cD)}(1,P\otimes (-)^{\vee})\tens{\cD}\Hom(1,-\otimes Q)\cong (\indlim[i]\Hom(-,P_i))\tens{\cD}\Hom(1,-\otimes Q)\\
\cong \indlim[i]\Hom(1,P_i\otimes Q)\cong \Hom(1,P\otimes Q).
\end{multline*}
\item On the other hand, if $\mk$ is a field, then the category $\cD=D(\mk)^{\omega_1,op}$ does not satisfy these conditions. Indeed, the full subcategory of nuclear objects of the category $\Ind(\cD^{op})\simeq\Ind(D(\mk)^{\omega_1})$ is simply the essential image of $\hat{\cY}:D(\mk)\to \Ind(D(\mk)^{\omega_1}).$ But for $V=\biggplus[\N]\mk\in D(\mk)^{\omega_1}$ the object $V^{\vee}\in \Ind(D(\mk)^{\omega_1})$ is not in the image of $\hat{\cY}.$ 
\end{enumerate}
\end{remark}

\subsection{Locally rigid categories}

We recall the notion of a locally rigid category, which was originally defined in \cite{AGKRV20} under the name ``semi-rigid''. As in loc. cit., we consider only the case of symmetric monoidal categories.

\begin{defi}\cite[Definition C.1.1]{AGKRV20} A symmetric monoidal presentable stable category $\cE\in\CAlg(\Pr^L_{\st})$ is called locally rigid if the following conditions hold.
\begin{enumerate}[label=(\roman*),ref=(\roman*)]
\item $\cE$ is dualizable.

\item The multiplication functor $\mult:\cE\otimes\cE\to\cE$ is strongly continuous, and its right adjoint is $\cE\hy\cE$-linear.
\end{enumerate}\end{defi}

Clearly, any rigid symmetric monoidal category is also locally rigid. We recall the self-duality of a locally rigid category in terms of its symmetric monoidal structure. Consider the following functor
\begin{equation*}
\Gamma_!:\cE\to\Sp,\quad \Gamma_!(x)=\Hom_{\Ind(\cE)}(\cY(1),\hat{\cY}(x)).
\end{equation*}

\begin{prop}\cite[Section C.3]{AGKRV20} \cite[Proposition 4.21]{Ram24b} Let $\cE$ be a locally rigid $\bE_1$-monoidal category. Then we have an equivalence $\cE\simeq \cE^{\vee}$ such that the evaluation and coevaluation functors are given by
\begin{equation*}
\ev:\cE\otimes\cE^{\vee}\simeq\cE\otimes\cE\xto{\mult}\cE\xto{\Gamma_!}\Sp.	
\end{equation*}
\begin{equation*}
\coev:\Sp\xto{1}\cE\xto{\mult^R}\cE\otimes\cE\simeq\cE^{\vee}\otimes\cE.
\end{equation*}
\end{prop}

We give the following basic examples.

\begin{example}
\begin{enumerate}[label=(\roman*),ref=(\roman*)]
\item For a locally compact Hausdorff space $X,$ the category $\Shv(X;\Sp)$ is locally rigid.
\item Let $R$ be a (discrete) commutative ring, let $I\subset R$ be a finitely generated ideal. Then the full subcategory category $D_{I\hy\compl}(R)\subset D(R)$ of derived $I$-complete objects is locally rigid.
\item Let $G$ be a finite group, and let $\mk$ be an $\bE_{\infty}$-ring. Then the category $(\Mod_{\mk})^{B G}$ of representations of $G$ in $\Mod_{\mk}$ is locally rigid.
\item Suppose that $\cD$ is a rigid symmetric monoidal category and $\cE\subset\cD$ is a smashing ideal, i.e. the right adjoint to the inclusion commutes with colimits (it is automatically $\cD$-linear because of rigidity of $\cD$). Then $\cE,$ considered as a symmetric monoidal category, is locally rigid. \label{example_smashing_ideal}
\end{enumerate}
\end{example}

Next, we explain that the example \ref{example_smashing_ideal} is in fact exhaustive, i.e. any locally rigid category is naturally a smashing ideal in a rigid category. One way to see this is to observe that for any locally rigid category $\cE$ the functor $\cE^{\rig}\to \cE$ has a fully faithful left adjoint, making $\cE$ into a smashing ideal of $\cE^{\rig}.$ However, it makes sense to define the following much smaller rigid category which also contains $\cE$.

\begin{defi}\label{def:one_point_rigidification}
	Let $\cE$ be a locally rigid symmetric monoidal category. We denote by $\cE_+\subset\Ind(\cE)$ the full subcategory which is generated via colimits by the essential image of $\hat{\cY}:\cE\to\Ind(\cE)$ and $\cY(1_{\cE}).$ We consider $\cE_+$ as a symmetric monoidal category, and call it the one-point rigidification of $\cE.$
\end{defi}

Note that $\cE_+\subset\Ind(\cE)$ is closed under the tensor product, since the functor $\hat{\cY}$ commutes with tensor products and $\cY(1_{\cE})$ is the unit object of $\Ind(\cE).$ The following basic result justifies the terminology.

\begin{prop}
	Let $\cE$ be a locally rigid symmetric monoidal category.
	\begin{enumerate}[label=(\roman*),ref=(\roman*)]
		\item The category $\cE_+$ is rigid symmetric monoidal. \label{E_+_rigid}
		
		\item  The full subcategory $\hat{\cY}(\cE)\subset \cE_+$ is a smashing ideal in $\cE_+.$ The corresponding idempotent $\bE_{\infty}$-algebra in $\cE_+$ is given by $\Cone(\hat{\cY}(1_{\cE})\to\cY(1_{\cE})).$ \label{smashing_ideal}
	\end{enumerate}
\end{prop}

\begin{proof} To prove \ref{E_+_rigid} we first show that for any compact map $f:x\to y$ in $\cE,$ the map $\hat{\cY}(f):\hat{\cY}(x)\to\hat{\cY}(y)$ is trace-class in $\cE_+.$ Indeed, the map $f$ comes from a map
	\begin{equation*}
		\wt{f}:\cY(1)\to\hat{\cY}(\un{\Hom}_{\cE}(x,1)\otimes y)\cong \hat{\cY}(\un{\Hom}_{\cE}(x,1))\otimes\hat{\cY}(y).
	\end{equation*}
	Now, the natural map
	\begin{equation*}
		\hat{\cY}(\un{\Hom}_{\cE}(x,1))\otimes\hat{\cY}(x)\cong \hat{\cY}(\un{\Hom}_{\cE}(x,1)\otimes x)\to \hat{\cY}(1)\to \cY(1)
	\end{equation*}
	induces a map
	\begin{equation*}
		\hat{\cY}(\un{\Hom}_{\cE}(x,1))\to\un{\Hom}_{\cE_+}(\hat{\cY}(x),1_{\cE_+}).
	\end{equation*}
	Hence, the composition
	\begin{equation*}
		\cY(1)\xto{\wt{f}} \hat{\cY}(\un{\Hom}_{\cE}(x,1))\otimes\hat{\cY}(y)\to \un{\Hom}_{\cE_+}(\hat{\cY}(x),1_{\cE_+})\otimes\hat{\cY}(y)
	\end{equation*}
	is a trace-class witness of $\hat{\cY}(f)$ in $\cE_+.$
	
	Now, take some object $x\in\cE^{\omega_1},$ and choose a sequence $x_0\to x_1\to\dots$ in $\cE^{\omega_1}$ such that $\indlim[n] x_n\cong x$ and all the maps $x_n\to x_{n+1}$ are compact. Then $\hat{\cY}(x)\cong \indlim[n]\hat{\cY}(x_n),$ and by the above all the maps $\hat{\cY}(x_n)\to\hat{\cY}(x_{n+1})$ are trace-class in $\cE_+.$
	Since such objects $\hat{\cY}(x)$ together with $\cY(1)$ generate $\cE_+,$ we conclude that the category $\cE_+$ is rigid.
	
	To prove \ref{smashing_ideal}, we first observe that by construction the right adjoint to the inclusion $\hat{\cY}(\cE)\to\cE_+$ commutes with colimits. Hence, we only need to show that for $x\in\cE_+$ and $y\in\cE$ we have $x\otimes\hat{\cY}(y)\in\hat{\cY}(\cE).$ It suffices to check this when either $x=\cY(1)$ or $x\in\hat{\cY}(\cC).$ The first case is trivial, and in the second case if $x=\hat{\cY}(z),$ then
	\begin{equation*}
		x\otimes\hat{\cY}(y)\cong \hat{\cY}(z\otimes y)\in \hat{\cY}(\cE).
	\end{equation*} This proves the proposition.
\end{proof}

\subsection{Continuous approximation of an accessible functor}

We fix a rigid $\bE_1$-monoidal category $\cE.$ We will use the following notation.

{\noindent{\bf Notation.}} {\it Let $\cC$ and $\cD$ be presentable stable left $\cE$-modules. Denote by $\Fun^{\lax,\acc}_{\cE}(\cC,\cD)$ the category of accessible exact lax $\cE$-linear functors. We denote by 
\begin{equation*}
(-)^{\cont}:\Fun^{\lax,\acc}(\cC,\cD)\to\Fun^L_{\cE}(\cC,\cD)
\end{equation*}
the right adjoint to the inclusion functor. This right adjoint exists since the inclusion commutes with colimits and the category $\Fun^L_{\cE}(\cC,\cD)$ is presentable.
}

\begin{prop}\label{prop:properties_of_F^cont}
Let $F:\cC\to\cD$ be an accessible lax $\cE$-linear exact functor between presentable stable left $\cE$-modules.
\begin{enumerate}[label=(\roman*), ref=(\roman*)]
\item If $\cC$ is dualizable, then the functor $F^{\cont}$ is given by the composition
\begin{equation}\label{eq:description_of_F^cont}
F^{\cont}:\cC\to\Ind(\cC)\xto{\Ind(F)}\Ind(\cD)\xto{\colim}\cD. 
\end{equation} \label{description_of_F^cont}
\item Let $\cC'$ be another presentable stable left $\cE$-module, and let $G:\cC'\to\cC$ be a strongly continuous $\cE$-linear functor. Then we have an isomorphism \begin{equation*}
F^{\cont}\circ G\xto{\sim} (F\circ G)^{\cont}.
\end{equation*} \label{precomposing_with_strcont}
\item Let $\cD'$ be another presentable stable left $\cE$-module, and let $H:\cD\to\cD'$ be a continuous $\cE$-linear functor. If $\cC$ is dualizable, then we have an isomorphism
\begin{equation*}
H\circ F^{\cont}\xto{\sim} (H\circ F)^{\cont}.
\end{equation*} \label{composing_with_cont}
\end{enumerate}
\end{prop}

\begin{proof}
To prove \ref{description_of_F^cont}, choose an uncountable regular cardinal $\kappa$ such that $F$ is $\kappa$-accessible. Then in \eqref{eq:description_of_F^cont} we can replace $\Ind(\cC)$ resp. $\Ind(\cD)$ by $\Ind(\cC^{\kappa})$ resp. $\Ind(\cD^{\kappa}).$ Then for a continuous $\cE$-linear functor $G:\cC\to\cD,$ we have
\begin{multline*}
\Hom(G,\colim\circ\Ind(F)\circ\hat{\cY}_{\cC})\cong \Hom(G\circ\colim,\colim\circ\Ind(F))\\
\cong \Hom(\colim\circ\Ind(G),\colim\circ\Ind(F))\cong\Hom(G,F),
\end{multline*}
as required.

To prove \ref{precomposing_with_strcont}, note that for a continuous $\cE$-linear functor $\Phi:\cC'\to\cD$ we have
\begin{multline*}
\Hom(\Phi,F^{\cont}\circ G)\cong \Hom(\Phi\circ G^R,F^{\cont})\cong\Hom(\Phi\circ G^R,F)\\
\cong\Hom(\Phi,F\circ G)\cong \Hom(\Phi,(F\circ G)^{\cont}).
\end{multline*}

Finally, \ref{composing_with_cont} follows directly from \ref{description_of_F^cont}, since $H\circ\colim\cong \colim\circ\Ind(H).$ 
\end{proof}

\begin{example} Let $\cE=\Sp$ and let $\cC$ be a dualizable category. For an object $x\in\cC$ we can consider the continuous functor $\Hom(x,-)^{\cont}:\cC\to\Sp.$ The corresponding object of $\cC$ is denoted by $x^{\vee}$ in \cite{E24}. In the case $\cC=\Mod\hy A$ for an $\bE_1$-ring $A,$ for a right $A$-module $M$ we have $M^{\vee}=\Hom_A(M,A)\in A\hy\Mod.$
\end{example}

\begin{remark}\label{rem:F^cont_for_non_presentable_target}
If $F:\cC\to\cD$ is a functor from a dualizable category to a cocomplete (not necessarily presentable) stable category, we will use the notation $F^{\cont}:\cC\to\cD$ for the functor defined by \eqref{eq:description_of_F^cont}. If $H:\cD\to\cD'$ is a colimit-preserving functor to another cocomplete stable category, then we have an isomorphism $(H\circ F)^{\cont}\cong H\circ F^{\cont},$ as in Proposition \ref{prop:properties_of_F^cont}.
\end{remark}

We close this subsection with a non-standard criterion of strong continuity of functors between dualizable categories.

\begin{prop}\label{prop:criterion_strcont}
	\begin{enumerate}[label=(\roman*), ref=(\roman*)]
		\item Let $\cC$ be a dualizable category, and let $x\in\cC$ be an object with $\hat{\cY}(x)=\inddlim[i\in I]x_i.$ Then for any family of objects $(y_s\in\cC^{\vee})_{s\in S},$ the natural map 
		\begin{equation*}\label{eq:comparison_for_products}
			\ev_{\cC}(x,\prodd[s]y_s)\cong\indlim[i]\ev_{\cC}(x_i,\prodd[s]y_s)\to\indlim[i]\prodd[s]\ev_{\cC}(x_i,y_s).
		\end{equation*}
		is an isomorphism. \label{necessary_for_hat_Y}
		\item Let $F:\cC\to\cD$ be a continuous functor between dualizable categories. Then $F$ is strongly continuous if and only if $\cC$ is generated via colimits by a collection of objects $x\in\cC$ with $\hat{\cY}(x)=\inddlim[i\in I]x_i,$ such that for any family of objects $(y_s\in\cD^{\vee})_{s\in S},$ the natural map
		\begin{equation*}
			\ev_{\cD}(F(x),\prodd[s]y_s)\cong\indlim[i]\ev_{\cD}(F(x_i),\prodd[s]y_s)\to\indlim[i]\prodd[s]\ev_{\cD}(F(x_i),y_s).
		\end{equation*}
		is an isomorphism. \label{criterion_strcont}
	\end{enumerate}
\end{prop}

\begin{proof}
	The statement \ref{necessary_for_hat_Y} follows from Proposition \ref{prop:properties_of_F^cont} and from the isomorphism $$\ev_{\cC}(-,\prodd[s]y_s)\cong (\prodd[s]\ev_{\cC}(-,y_s))^{\cont}$$ of functors $\cC\to\Sp.$ 
	
	The ``only if'' direction in \ref{criterion_strcont} follows directly from \ref{necessary_for_hat_Y}. For the ``if'' direction we observe that the assumption on $F$ means that $F^{\vee}:\cD^{\vee}\to\cC^{\vee}$ commutes with infinite products, i.e. that $F^{\vee}$ has a left adjoint. This is equivalent to the strong continuity of $F,$ as required.
\end{proof}

\subsection{Right trace-class morphisms of functors}

Again, we fix a rigid $\bE_1$-monoidal category $\cE.$

\begin{defi}\label{def:right_trace_class} Let $F,G:\cC\to\cD$ be continuous $\cE$-linear functors between presentable stable left $\cE$-modules, and $F^R:\cD\to\cC$ the right adjoint to $F$. We say that a morphism $\varphi:F\to G$ is right trace-class if the associated morphism $\id_{\cC}\to F^R\circ G$ factors through $F^{R,\cont}\circ G.$ In other words, there exists a morphism $\wt{\varphi}:\id_{\cC}\to F^{R,\cont}\circ G$ such that $\varphi$ is given by the composition
\begin{equation*}
F\xto{F\wt{\varphi}}F\circ F^{R,\cont}\circ G\to F\circ F^R\circ G\to G. 	
\end{equation*}\end{defi}

\begin{example}
Let $\cE=\Sp,$ and let $\cC$ be a dualizable category. Identifying $\cC$ with $\Fun^L(\Sp,\cC),$ we see that a morphism $x\to y$ in $\cC$ is compact if and only if the associated morphism of functors is right trace-class.
\end{example}

We have the following basic properties of right trace-class maps.

\begin{prop}\label{prop:right_trace_class_basics}
Let $\cC_1,\cC_2,\cC_3$ be presentable stable left $\cE$-modules.
\begin{enumerate}[label=(\roman*), ref=(\roman*)]
\item For a continuous $\cE$-linear functor $F:\cC_1\to\cC_2,$ the identity map $\id:F\to F$ is right trace-class if and only if $F$ is strongly continuous. \label{identity_map_right_trace_class}
\item For a composable pair of maps $\varphi:F\to G,$ $\psi:G\to H$ in $\Fun^L_{\cE}(\cC_1,\cC_2),$ if $\varphi$ or $\psi$ is right trace-class, then the composition $\psi\circ \varphi$ is also right trace-class. \label{right_trace_class_2_sided_ideal}
\item Given right trace-class morphisms $\varphi:F_1\to G_1$ in $\Fun^L_{\cE}(\cC_1,\cC_2)$ and $\psi:F_2\to G_2$ in $\Fun^L_{\cE}(\cC_2,\cC_3),$ the map
\begin{equation*}
\psi\varphi:F_2\circ F_1\to G_2\circ G_1
\end{equation*}
is right trace-class in $\Fun^L_{\cE}(\cC_1,\cC_3).$ \label{right_trace_class_between_compositions}
\end{enumerate}
\end{prop}

\begin{proof}The statements \ref{right_trace_class_2_sided_ideal} and \ref{right_trace_class_between_compositions} are  formal consequences of Definition \ref{def:right_trace_class}. 
	
For example, to prove \ref{right_trace_class_between_compositions} we choose right trace-class witnesses $\wt{\varphi}:\id_{\cC_1}\to G_1^{R,\cont}\circ F_1,$ $\wt{\psi}:\id_{\cC_2}\to G_2^{R,\cont}\circ F_2,$ and take the composition
\begin{equation*}
\id_{\cC_1}\to G_1^{R,\cont}\circ F_1\to G_1^{R,\cont}\circ G_2^{R,\cont}\circ F_2\circ F_1\to (G_2\circ G_1)^{R,\cont} \circ (F_2\circ F_1),
\end{equation*} 
which is a trace-class witness for $\psi\varphi.$

The only statement which requires a proof is the ``only if'' direction in part \ref{identity_map_right_trace_class}. In this case the natural map $F^{R,\cont}\to F^R$ has a right inverse, given by the composition
\begin{equation*}
F^R\to (F^{R,\cont}\circ F)\circ F^R\cong F^{R,\cont}\circ (F\circ F^R)\to F^{R,\cont}.	
\end{equation*}
Thus, the functor $F^R$ is continuous, as required.\end{proof}

\subsection{Smoothness and properness}

Let $\cE$ be an $\bE_1$-monoidal rigid category. We recall the notions of smoothness and properness for dualizable left $\cE$-modules. These notions are straightforward generalizations of \cite[Definitions 8.2 and 8.8]{KonSob09}.

\begin{defi}
	Let $\cC$ be a dualizable left module over a rigid $\bE_1$-monoidal category $\cE.$
	\begin{enumerate}[label=(\roman*),ref=(\roman*)]
		\item $\cC$ is called smooth if the coevaluation functor $\coev:\Sp\to \cC^{\vee}\tens{\cE}\cC$ is strongly continuous, i.e. if the object $\coev(\bS)$ is compact in $\cC^{\vee}\tens{\cE}\cC.$ 
		\item $\cC$ is called proper if the evaluation functor $\cC\otimes\cC^{\vee}\to\cE$ is strongly continuous.     
	\end{enumerate}
\end{defi}

The following basic result relates the notions of smoothness and properness with the right trace-class maps of functors.

\begin{prop}\label{prop:criterions_smoothness_properness} Let $\cC$ be a dualizable left $\cE$-module. 
\begin{enumerate}[label=(\roman*), ref=(\roman*)]
\item The category $\cC$ is smooth over $\cE$ if and only if for any dualizable left $\cE$-module $\cD$ any right trace-class map $F\to G$ in the category $\Fun^L_{\cE}(\cC,\cD)$ is compact. \label{criterion_smoothness}

\item The category $\cC$ is proper if and only if for any dualizable left $\cE$-module $\cD$ any compact map $F\to G$ in the category $\Fun^L_{\cE}(\cC,\cD)$ is right trace-class. \label{criterion_properness}
\end{enumerate}
\end{prop}

\begin{proof}
We prove \ref{criterion_smoothness}. For the ``only if'' direction, suppose that $\cC$ is smooth. 
If $F\to G$ is a right trace-class map in $\Fun^L_{\cE}(\cC,\cD),$ then by Proposition \ref{prop:right_trace_class_basics} \ref{right_trace_class_between_compositions} the induced map between the compositions
\begin{equation*}
\Sp\xto{\coev_{\cC/\cE}}\cC^{\vee}\tens{\cE}\cC \xto{\id\boxtimes F} \cC^{\vee}\tens{\cE}\cD,\quad \Sp\xto{\coev_{\cC/\cE}}\cC^{\vee}\tens{\cE}\cC \xto{\id\boxtimes G} \cC^{\vee}\tens{\cE}\cD
\end{equation*}
is right trace-class, which exactly means that the morphism $F\to G$ is compact in $\Fun^L_{\cE}(\cC,\cD).$

For the ``if'' direction, it suffices to put $\cD=\cC$ and consider the identity map $\id_{\cC}\to\id_{\cC},$ which is right trace-class, hence compact.

The proof of \ref{criterion_properness} is similar. For the ``only if'' direction, suppose that $\cC$ is proper. Let $F\to G$ be a compact map in $\Fun^L_{\cE}(\cC,\cD).$ 
Applying Proposition \ref{prop:right_trace_class_basics} \ref{right_trace_class_between_compositions} again, we see that the induced map between the compositions
\begin{equation*}
F:\cC\xto{\id\boxtimes F}\cC\otimes\cC^{\vee}\tens{\cE}\cD\xto{\ev_{\cC/\cE}\boxtimes\id}\cD,\quad G:\cC\xto{\id\boxtimes G}\cC\otimes\cC^{\vee}\tens{\cE}\cD\xto{\ev_{\cC/\cE}\boxtimes\id}\cD
\end{equation*}
is right trace-class.

For the ``if'' direction, take $\cD=\cE.$ We need to show that for a compact map $x\to y$ in $\cC$ and a compact map $x'\to y'$ in $\cC^{\vee},$
the induced map $\ev_{\cC/\cE}(x\boxtimes x')\to \ev_{\cC/\cE}(y\boxtimes y')$ is compact in $\cE.$ By assumption, the map $x'\to y'$ is right trace-class in $\Fun_{\cE}^L(\cC,\cE).$ Therefore,
by Proposition \ref{prop:right_trace_class_basics} \ref{right_trace_class_between_compositions} the induced map between the compositions
\begin{equation*}
\cE\xto{x}\cC\xto{x'}\cE,\quad \cE\xto{y}\cC\xto{y'}\cE
\end{equation*}
is right trace-class. This exactly means that the map $\ev_{\cC/\cE}(x\boxtimes x')\to \ev_{\cC/\cE}(y\boxtimes y')$ is compact in $\cE,$ as required. 
\end{proof}

As a special case, we have the well-known relation between compactness and strong continuity of functors.

\begin{cor}
	Let $\cC$ be a dualizable left module over a rigid $\bE_1$-monoidal category $\cE.$
	\begin{enumerate}[label=(\roman*),ref=(\roman*)]
		\item If $\cC$ is smooth, then for any dualizable left $\cE$-module $\cD$ and for any strongly continuous $\cE$-linear functor $\cC\to\cD,$ the associated object of $\cC^{\vee}\tens{\cE}\cD$ is compact.
		\item If $\cC$ is proper, then for any dualizable left $\cE$-module $\cD$ and for any compact object of $\cC^{\vee}\tens{\cE}\cD,$ the associated functor $\cC\to\cD$ is strongly continuous.  
	\end{enumerate}
\end{cor}

\subsection{Algebra of objects in the tensor products}

In this subsection we explain some basic operations on objects of the relative tensor products of dualizable categories. They are straightforward generalizations of the usual operations on bimodules.

We fix a rigid $\bE_1$-monoidal category $\cE.$

\begin{defi}\label{def:operations_objects_tensor_products}
\begin{enumerate}[label=(\roman*),ref=(\roman*)]
\item Let $\cC$ be a dualizable right $\cE$-module, and let $\cC',\cD$ be dualizable left $\cE$-modules. We use the symbol $\ev_{\cD/\cE}$ for the following composition:
\begin{equation*}
\ev_{\cD/\cE}:(\cC\tens{\cE}\cD)\otimes (\cD^{\vee}\tens{\cE}\cC')\xto{\id\boxtimes\ev_{\cD/\cE}\otimes\id}\cC\tens{\cE}\cE\tens{\cE}\cC'\simeq \cC\tens{\cE}\cC'.
\end{equation*}
\label{ev_D/E}
\item Let $\cC$ and $\cC'$ be dualizable right $\cE$-modules, and let $\cD$ be a dualizable left $\cE$-module. For $M\in \cC\tens{\cE}\cD,$ $N\in\cC'\tens{\cE}\cD,$
we denote by $\Hom_{\cD}(M,N)$ the object of $\cC'\tens{\cE}\cC^{\vee}$ with the following universal property.
\begin{equation*}
\Hom_{\cC'\tens{\cE}\cC^{\vee}}(P,\Hom_{\cD}(M,N))\cong \Hom_{\cC'\tens{\cE}\cD}(\ev_{\cC^{\vee}/\cE}(P,M),N),\quad P\in\cC'\tens{\cE}\cC^{\vee}.
\end{equation*} \label{Hom_D}
\item Let $\cC$ be a dualizable right $\cE$-module and $\cD$ a dualizable left $\cE$-module. For $M\in \cC\tens{\cE}\cD,$ we put 
\begin{equation*}
M_{\cD}^{\vee} = \Hom_{\cD}(M,\coev_{\cD/\cE}(\bS))\in \cD^{\vee}\tens{\cE}\cC^{\vee}. 
\end{equation*}
\end{enumerate} \label{M_D^v}
\end{defi} 

The following example relates the operations from Definition \ref{def:operations_objects_tensor_products} with the familiar operations with bimodules.

\begin{example} Suppose for simplicity that $\cE = \Mod_{\mk},$ where $\mk$ is an $\bE_{\infty}$-ring. Let $\cC=\Mod\hy C,$ $\cC'=\Mod\hy C',$ $\cD=\Mod\hy D$ for some $\bE_1\hy\mk$-algebras $C,C',D.$
\begin{enumerate}[label=(\roman*),ref=(\roman*)]
\item We have
\begin{equation*}
\cC\tens{\mk}\cD\simeq \Mod\hy(C\tens{\mk}D),\quad \cD^{\vee}\tens{\mk}\cC'\simeq \Mod\hy(D^{op}\tens{\mk}C'),
\end{equation*}
and for $M\in \cC\tens{\mk}\cD,$ $N\in\cD^{\vee}\tens{\mk}\cC'$ we have
\begin{equation*}
\ev_{\cD/\cE}(M,N)\cong M\tens{D} N\in \Mod\hy(C\tens{\mk}C')\simeq \cC\tens{\mk}\cC'.
\end{equation*}
\item Similarly, for $M\in \cC\tens{\mk}\cD,$ $N\in \cC'\tens{\mk}\cD$ we have
\begin{equation*}
\Hom_{\cD}(M,N)\cong \Hom_D(M,N)\in \Mod\hy(C'\tens{\mk}C^{op})\simeq \cC'\tens{\mk}\cC^{\vee}.
\end{equation*} \label{Hom_of_bimodules_over_rings}
\item As a special case of \ref{Hom_of_bimodules_over_rings}, for $M\in \cC\tens{\mk}\cD$ we have
\begin{equation*}
M_{\cD}^{\vee}\cong \Hom_D(M,D)\in \Mod\hy(D^{op}\tens{\mk}C^{op})\simeq \cD^{\vee}\tens{\mk}\cC^{\vee}.
\end{equation*} 
\end{enumerate} 
\end{example}

The following is a tautological reformulation of the operations from Definition \ref{def:operations_objects_tensor_products} in terms of $\cE$-linear functors between dualizable left $\cE$-modules.

\begin{prop}
We keep the notation from Definition \ref{def:operations_objects_tensor_products}. The operation from \ref{ev_D/E} corresponds to the composition bifunctor
\begin{equation*}
\Fun_{\cE}^L(\cC^{\vee},\cD)\times\Fun_{\cE}^L(\cD,\cC')\to \Fun_{\cE}^L(\cC^{\vee},\cC').
\end{equation*}
The operation from \ref{Hom_D} corresponds to the bifunctor
\begin{equation*}
\Fun_{\cE}^L(\cC^{\vee},\cD)^{op}\times \Fun_{\cE}^L(\cC^{'\vee},\cD)\to \Fun_{\cE}^L(\cC^{'\vee},\cC^{\vee}),\quad (F,G)\mapsto (F^R\circ G)^{\cont}.
\end{equation*}
The operation from \ref{M_D^v} corresponds to the functor
\begin{equation*}
\Fun^L_{\cE}(\cC^{\vee},\cD)\to\Fun_{\cE}^L(\cD,\cC^{\vee}),\quad F\mapsto F^{R,\cont}.
\end{equation*}
\end{prop} 

\begin{proof}
This is straightforward.
\end{proof}

We will need the following properties of the above operations in the case when one of the factors in the tensor product is proper.

\begin{prop}\label{prop:general_diagonal_arrows}
Let $\cC,\cC',\cD$ be dualizable left $\cE$-modules, and suppose that $\cC$ is proper over $\cE.$ Suppose that $N_1\to N_2$ is a compact morphism in $\cC^{\vee}\tens{\cE}\cD.$ 
\begin{enumerate}[label=(\roman*), ref=(\roman*)]
\item There exists a map 
\begin{equation*}
\Hom_{\cD}(-,N_1)\to \ev_{\cD}(N_2,(-)_{\cD}^{\vee})\quad\text{in }\Fun((\cC^{'\vee}\tens{\cE}\cD)^{op},\cC^{\vee}\tens{\cE}\cC'),
\end{equation*}  making the following diagram commute:
\begin{equation*}
		\begin{tikzcd}
			\ev_{\cD}(N_1,(-)_{\cD}^{\vee})\ar[d]\ar[r] & \ev_{\cD}(N_2,(-)_{\cD}^{\vee})\ar[d]\\
			\Hom_{\cD}(-,N_1)\ar[r]\ar[ru] &  \Hom_{\cD}(-,N_2).
		\end{tikzcd}	
\end{equation*}\label{first_diagonal_arrow}
\item There exists a map $\Hom_{\cD^{\vee}}((N_1)_{\cD}^{\vee},-)\to \ev_{\cD}(N_2,-),$ in $\Fun(\cD^{\vee}\tens{\cE}\cC',\cC^{\vee}\tens{\cE}\cC),$ making the following diagram commute:
\begin{equation*}
\begin{tikzcd}
\ev_{\cD}(N_1,-)\ar[r]\ar[d] & \ev_{\cD}(N_2,-)\ar[d]\\
\Hom_{\cD^{\vee}}((N_1)_{\cD}^{\vee},-)\ar[r]\ar[ru] & \Hom_{\cD^{\vee}}((N_2)_{\cD}^{\vee},-)
\end{tikzcd}
\end{equation*}\label{second_diagonal_arrow}
\end{enumerate}
\end{prop}

\begin{proof} 
Denote by $F_1:\cC\to\cD$ resp. $F_2:\cC\to\cD$ the $\cE$-linear functor corresponding to $N_1$ resp. $N_2.$ By Proposition \ref{prop:criterions_smoothness_properness} the map $F_1\to F_2$ is right trace-class, i.e. it comes from a map $\id_{\cC}\to F_1^{R,\cont}\circ F_2$. 

To prove \ref{first_diagonal_arrow}, for an $\cE$-linear continuous functor $G:\cC'\to\cD$ we need to construct a map $(G^R\circ F_1)^{\cont}\to G^{R,\cont}\circ F_2,$ functorially in $G,$ making the following diagram commute:
\begin{equation}\label{eq:first_diagram_of_functors_diagonal_map}
\begin{tikzcd}
G^{R,\cont}\circ F_1\ar[d]\ar[r] & G^{R,\cont}\circ F_2 \ar[d]\\
(G^R\circ F_1)^{\cont}\ar[r]\ar[ru] & (G^R\circ F_2)^{\cont}.
\end{tikzcd}	
\end{equation}
We define this diagonal map to be the composition
\begin{equation*}
(G^R\circ F_1)^{\cont}\to (G^R\circ F_1)^{\cont}\circ F_1^{R,\cont}\circ F_2\to (G^R\circ F_1\circ F_1^R)^{\cont}\circ F_2\to G^{R,\cont}\circ F_2.
\end{equation*}
It is straightforward to check that the diagram \eqref{eq:first_diagram_of_functors_diagonal_map} naturally commutes.

Now we prove \ref{second_diagonal_arrow}. Let $G:\cC^{'\vee}\to\cD^{\vee}$ be a continuous $\cE$-linear functor. We need to construct a map $((F_1^{R,\cont})^{\vee,R}\circ G)^{\cont}\to F_2^{\vee}\circ G$ in $\Fun_{\cE}^L(\cC^{'\vee},\cC^{\vee}),$ making the following diagram commute: 
\begin{equation}\label{eq:second_diagram_of_functors_diagonal_map}
\begin{tikzcd}
	F_1^{\vee}\circ G\ar[d]\ar[r] & F_2^{\vee}\circ G \ar[d]\\
	((F_1^{R,\cont})^{\vee,R}\circ G)^{\cont}\ar[r]\ar[ru] & ((F_2^{R,\cont})^{\vee,R}\circ G)^{\cont}.
\end{tikzcd}	
\end{equation}
We define this diagonal map to be the composition
\begin{multline*}
((F_1^{R,\cont})^{\vee,R}\circ G)^{\cont}\to (F_1^{R,\cont}\circ F_2)^{\vee}\circ ((F_1^{R,\cont})^{\vee,R}\circ G)^{\cont}\\
\to F_2^{\vee}\circ ((F_1^{R,\cont})^{\vee}\circ (F_1^{R,\cont})^{\vee,R}\circ G)^{\cont}\to F_2^{\vee}\circ G.
\end{multline*}
Again, it is straightforward to check that the diagram \eqref{eq:second_diagram_of_functors_diagonal_map} naturally commutes. 
\end{proof}

The following corollary will be used in Section \ref{sec:dualizable_Hom} in the proof of Theorem \ref{th:internal_projectivity}.

\begin{cor}\label{cor:ind_isomorphisms_for_proper} Let $\cC,\cC',\cD$ be dualizable left $\cE$-modules, and suppose that $\cC$ is proper over $\cE.$ Consider an object $N\in\cC^{\vee}\tens{\cE}\cD,$ and let $\hat{\cY}(N)\cong\inddlim[i]N_i.$

\begin{enumerate}[label=(\roman*), ref=(\roman*)]
\item The natural map
\begin{equation*}
\inddlim[i]\ev_{\cD}(N_i,(-)_{\cD}^{\vee})\to \inddlim[i]\Hom_{\cD}(-,N_i)
\end{equation*}
is an isomorphism in $\Ind(\Fun((\cC^{'\vee}\tens{\cE}\cD)^{op},\cC^{\vee}\tens{\cE}\cC')).$ \label{first_ind_isom}
\item The natural map
\begin{equation*}
\inddlim[i]\ev_{\cD}(N_i,-)\to \inddlim[i]\Hom_{\cD^{\vee}}((N_i)_{\cD}^{\vee},-)
\end{equation*}
is an isomorphism in $\Ind(\Fun(\cD^{\vee}\tens{\cE}\cC',\cC^{\vee}\tens{\cE}\cC')).$ \label{second_ind_isom}
\end{enumerate}
\end{cor}

\subsection{A remark about homological epimorphisms}

We explain a point of view on homological epimorphisms which might be helpful for understanding the results of Sections \ref{sec:dualizable_Hom}, \ref{sec:rig_of_locally_rigid} and \ref{sec:Mittag-Leffler}.

First recall that for a small additive category $\cA$ there are notions of a left $\cA$-module, a right $\cA$-module, an  $\cA$-$\cA$-bimodule, the (non-derived) tensor product of such, as well as the notion of a two-sided ideal and a quasi-ideal (see for example \cite[Section 1.8]{E24}). To be precise, a quasi-ideal is a biadditive bifunctor $I:\cA^{op}\times \cA\to\Ab$ with a map $\alpha:I(-,-)\to\cA(-,-)$ such that for any $f\in I(x,y),$ $g\in I(y,z)$ we have $\alpha(g) f=g \alpha(f).$  The following statement is essentially a slight variation of \cite[Lemma 1.47]{E24}.

\begin{prop}\label{prop:idempotent_quasi_ideals}
Let $\cA$ be a small additive category, and let $I$ be a quasi-ideal of $\cA$ which is flat on the right, i.e. for each $x\in\cA$ the right $\cA$-module $I(-,x)$ is flat. Denote by $J$ the two-sided ideal of $\cA$ which is the image of $I.$ Then the following are equivalent.
\begin{enumerate}[label=(\roman*),ref=(\roman*)]
 \item The map $I\tens{\cA}I\to I$ is surjective. \label{surj_map}
 \item We have an isomorphism $I\tens{\cA}I\xto{\sim} I.$ \label{iso_map}
 \item We have $J^2=J,$ and the natural map of quasi-ideals $J\tens{\cA}J\to I$ is an isomorphism. \label{quasi_ideal_from_ideal}
\end{enumerate}
\end{prop}

\begin{proof} 
The implication \Implies{iso_map}{surj_map} is tautological, and the equivalence \Iff{iso_map}{quasi_ideal_from_ideal} follows directly from \cite[Lemma 1.47]{E24}.

It remains to prove the implication \Implies{surj_map}{quasi_ideal_from_ideal}. Clearly, we have $J^2=J.$ Since $I$ is flat on the right, the map
$I\tens{\cA}J\to I$ is injective. But it is also surjective by assumption, hence we have $I\tens{\cA}J\cong I.$ It follows that both maps $I\tens{\cA}J\to J\tens{\cA}J$ and $J\tens{\cA}J\to I$ are isomorphisms, which proves the implication.
\end{proof}

We deduce the equivalent conditions for an exact functor between small stable categories to be a homological epimorphism. 

\begin{prop}
Let $F:\cC\to\cD$ be an exact functor between small stable categories. Consider the category $\cA=\h\cC$ as an additive category. Denote by $J$ the ideal of $\cA$ given by
\begin{equation*}
J(x,y)=\ker(\pi_0\Hom_{\cC}(x,y)\to\pi_0\Hom_{\cD}(F(x),F(y))).
\end{equation*}
Denote by $I$ the quasi-ideal of $\cA,$ given by 
\begin{equation*}
I(x,y)=\pi_0\Fiber(\Hom_{\cC}(x,y)\to \Hom_{\cD}(F(x),F(y))).
\end{equation*}
Then the following conditions are equivalent.
\begin{enumerate}[label=(\roman*), ref=(\roman*)]
\item The map $I\tens{\cA}I\to I$ is surjective.\label{surj_map1}
\item We have an isomorphism $I\tens{\cA}I\xto{\sim} I.$ \label{iso_map1}
\item We have $J^2=J,$ and the natural map of quasi-ideals $J\tens{\cA}J\to I$ is an isomorphism. \label{quasi_ideal_from_ideal1}
\item The functor $F:\cC\to\cD$ is a homological epimorphism onto its image. \label{hom_epi_onto_image}
\end{enumerate}
\end{prop}

\begin{proof}The quasi-ideal $I$ is flat on the right (and on the left) since the flat right modules over $\h\cC$ are exactly the cohomological functors, see for example \cite[Proposition E.5]{E24}. Hence, the equivalences \Ifff{surj_map1}{iso_map1}{quasi_ideal_from_ideal1} follow from Proposition \ref{prop:idempotent_quasi_ideals}.

It remains to prove the equivalence \Iff{iso_map1}{hom_epi_onto_image}. Consider the bifunctor
\begin{equation*}
\cI:\cC^{op}\times\cC\to\Sp,\quad \cI(x,y)=\Fiber(\Hom_{\cC}(x,y)\to\Hom_{\cD}(F(x),F(y))),
\end{equation*}
so that $I(x,y)\cong \pi_0\cI(x,y).$ The condition \ref{hom_epi_onto_image} exactly means that for $x,y\in\cC$ we have an isomorphism
\begin{equation*}
\cI(-,y)\tens{\cC}\cI(x,-)\xto{\sim} \cI(x,y).
\end{equation*} 
It remains to observe that for exact functors $M:\cC^{op}\to\Sp,$ $N:\cC\to\Sp$ we have
\begin{equation}\label{eq:isom_tensor_product_of_modules}
\pi_0(M(-))\tens{\h\cC}\pi_0(N(-))\xto{\sim}\pi_0(M\tens{\cC}N).
\end{equation}
Indeed, the source and the target of \eqref{eq:isom_tensor_product_of_modules} commute with filtered colimits in $M$ and $N,$ hence we may assume that $M$ and $N$ are representable respectively by $x\in\cC,$ $y\in\cC.$ In this case the source and the target of \eqref{eq:isom_tensor_product_of_modules} are identified with $\pi_0\Hom_{\cC}(y,x).$ This proves the equivalence \Iff{iso_map1}{hom_epi_onto_image} and the proposition.
\end{proof}

\section{Statements about limits, colimits and filtered $\infty$-categories}
\label{sec:lim_colim_and_filtered}

\subsection{Combinations of limits and colimits}
In Sections \ref{sec:dualizable_Hom} and \ref{sec:Mittag-Leffler} we will need certain statements about various combinations of limits and colimits. First, we recall the following statement from \cite{E24}.

\begin{prop}\label{prop:seq_limits_of_filtered_colimits} \cite[Lemma D.4]{E24} Let $\cC$ be a presentable $\infty$-category such that filtered colimits in $\cC$ commute with finite limits (i.e. strong (AB5) holds in $\cC$) and (AB6) for countable products holds in $\cC$ (for example, $\cC$ can be any compactly assembled presentable category, such as $\cS$ or $\Sp$). Let $I$ be a directed poset, and consider $\N$ as a poset with the usual order. Let $F:\N^{op}\times I\to \cC$ be a functor. Then we have a natural isomorphism
$$\indlim[\varphi:\N\to I]\prolim[n\leq m]F(n,\varphi(m))\xto{\sim}\prolim[n]\indlim[i]F(n,i),$$
where $\varphi$ runs through order-preserving maps.
Here the set $\{(n,m)\mid n\leq m\}$ is considered as a full subposet of $\N^{op}\times\N.$\end{prop}

The next statement is much more complicated. We prove it in Appendix \ref{app:limits_colimits} as a corollary of a stronger statement (Theorem \ref{th:pullback_square_quadrofunctors}).

\begin{prop}\label{prop:quadrofunctors} Let $I$ and $J$ be directed posets. Let $\cC$ be a presentable $\infty$-category satisfying the assumptions of Proposition \ref{prop:limit_of_colimits_for_fibrations}. Let $F:\N^{op}\times I\times\N^{op}\times J\to\cC$ be a functor. Suppose that the following conditions hold:
	\begin{enumerate}[label=(\roman*),ref=(\roman*)]
		\item the map $$\prolim[k]\indlim[j]\prolim[n]\indlim[i]F(n,i,k,j)\to \prolim[k]\prolim[n]\indlim[j]\indlim[i]F(n,i,k,j)$$
		is an isomorphism in $\cC.$ \label{tricky_assump1}
		
		\item the map $$\prolim[n]\indlim[i]\prolim[k]\indlim[j]F(n,i,k,j)\to \prolim[n]\prolim[k]\indlim[i]\indlim[j]F(n,i,k,j)$$
		is an isomorphism in $\cC.$ \label{tricky_assump2}
	\end{enumerate}
	Then we have an isomorphism
	$$\indlim[\varphi:\N\to I]\indlim[\psi:\N\to J]\prolim[n\leq m]\prolim[k\leq l]F(n,\varphi(m),k,\psi(l))\xto{\sim}\prolim[n]\indlim[i]\indlim[j]F(n,i,n,j).$$ 
\end{prop}

\subsection{Lemmas about filtered $\infty$-categories}
\label{ssec:lemmas_about_filtered}

In Sections \ref{sec:loc_invar_inverse_limits} and \ref{sec:original_Nuc} we will need certain statements about lax limits of filtered $\infty$-categories. We consider the following setup. Let $I_n,$ $n\geq 0,$ and $J_n,$ $n\geq 1,$ be filtered $\infty$-categories. Suppose that we have sequences of functors $F_n:I_n\to J_{n+1},$ $n\geq 0,$ and $G_n:I_n\to J_n, $ $n\geq 1,$ such that each functor $G_n$ is cofinal. Denote by $f,g:\prodd[n\geq 0]I_n\to\prodd[n\geq 1]J_n$ the induced functors between the products. Consider the lax equalizer $\LEq(\prodd[n\geq 0]I_n\toto\prodd[n\geq 1]J_n)$ of $f$ and $g:$ its objects are given by $(i_n;\varphi_n)_{n\geq 0},$ where $i_n\in I_n$ and $\varphi_n:F_n(i_n)\to G_{n+1}(i_{n+1}).$

\begin{lemma}\label{lem:lax_equalizer_of_filtered_properties} Within the above notation, the following statements hold.
\begin{enumerate}[label=(\roman*),ref=(\roman*)]
\item The category $\LEq(\prodd[n\geq 0]I_n\toto\prodd[n\geq 1]J_n)$ is filtered. \label{filtered}

\item The functor $\LEq(\prodd[n\geq 0]I_n\toto\prodd[n\geq 1]J_n)\to\prodd[n\geq 0]I_n$ is cofinal. \label{cofinal}
\end{enumerate}
\end{lemma}

\begin{proof}[Proof that \Implies{filtered}{cofinal}]
Take some object $(i_n)_{n\geq 0}\in \prodd[n\geq 0]I_n.$ Then the comma category can be considered as a lax equalizer: we have
\begin{equation*}
	\LEq(\prodd[n\geq 0]I_n\toto\prodd[n\geq 1]J_n)_{(i_n)_{n\geq 0}/}\simeq \LEq(\prodd[n\geq 0](I_n)_{i_n/}\toto\prodd[n\geq 1]J_n).
\end{equation*} 
Since the compositions $(I_n)_{i_n/}\to I_n\xto{G_n}J_n$ are cofinal for $n\geq 1,$ it follows from \ref{filtered} that the comma category $\LEq(\prodd[n\geq 0]I_n\toto\prodd[n\geq 1]J_n)_{(i_n)_{n\geq 0}/}$ is filtered, hence weakly contractible. This proves the implication \Implies{filtered}{cofinal}.
\end{proof}

The proof of \ref{filtered} is given in Appendix \ref{app:proof_of_lemma_on_filtered}.

\begin{lemma}\label{lem:functor_between_oplax_equalizers}
Within the above notation, consider another pair of sequences of filtered $\infty$-categories $(I_n')_{n\geq 0},$ $(J_n')_{n\geq 1}$ and functors $F_n':I_n'\to J_{n+1}',$ $n\geq 0,$ and $G_n':I_n'\to J_n',$ $n\geq 1,$ such that each functor $G_n'$ is cofinal. Suppose that we have the cofinal functors $I_n'\to I_n,$ $n\geq 0,$ and $J_n'\to J_n,$ $n\geq 1,$ such that the following squares commute:
\begin{equation*}
\begin{tikzcd}
I_n'\ar[r, "F_n'"] \ar[d] & J_{n+1}'\ar[d] & & I_n'\ar[r, "G_n'"] \ar[d] & J_n'\ar[d] \\
I_n\ar[r, "F_n"] & J_{n+1} & & I_n\ar[r, "G_n"] & J_n.
\end{tikzcd}
\end{equation*}
Then the induced functor
\begin{equation*}
\LEq(\prodd[n\geq 0] I_n'\toto \prodd[n\geq 1] J_n')\to \LEq(\prodd[n\geq 0] I_n\toto \prodd[n\geq 1] J_n)
\end{equation*}
is cofinal.
\end{lemma}

\begin{proof}
Given an object $x=(x_n;\varphi_n)\in \LEq(\prodd[n\geq 0] I_n\toto \prodd[n\geq 1] J_n),$ it suffices to show that the comma category $\LEq(\prodd[n\geq 0] I_n'\toto \prodd[n\geq 1] J_n')_{x/}$ is filtered. But we have an equivalence
\begin{equation}\label{eq:comma_category_lax_equalizer}
\LEq(\prodd[n\geq 0] I_n'\toto \prodd[n\geq 1] J_n')_{x/}\xto{\sim} \LEq(\prodd[n\geq 0] (I_n')_{x_n/} \toto \prodd[n\geq 1] (J_n')_{F_{n-1}(x_{n-1})/}).
\end{equation} 
Here in the latter lax equalizer we take the natural functors $(I_n')_{x_n/}\to (J_{n+1}')_{F_n(x_n)/},$ $n\geq 0,$ and the compositions $(I_n')_{x_n/}\to (J_n')_{G_n(x_n)/}\to (J_n')_{F_{n-1}(x_{n-1})/},$ $n\geq 1.$ By Lemma \ref{lem:lax_equalizer_of_filtered_properties}, the target of \eqref{eq:comma_category_lax_equalizer} is a filtered category. This proves the lemma.
\end{proof}

\section{Dualizable internal Hom}
\label{sec:dualizable_Hom}

In this section we will study the relative internal $\Hom$ in the category $\Cat_{\cE}^{\dual}$ over $\Cat_{\st}^{\dual},$ where $\cE$ is a rigid $\bE_1$-monoidal category. The main result is the theorem on the internal projectivity of proper $\omega_1$-compact dualizable left $\cE$-modules (Theorem \ref{th:internal_projectivity}).

\subsection{Generalities on the dualizable internal Hom}

Let $\cE$ be a rigid $\bE_1$-monoidal category. We use the notation from Definition \ref{def:dualizable_over_E_1_monoidal}. In particular, we denote by $\Cat_{\cE}^{\dual}$ the category of dualizable left $\cE$-modules. For an uncountable regular cardinal $\kappa$ we denote by $\Pr^L_{\cE,\kappa}$ the category of left $\cE$-modules in $\Pr^L_{\st,\kappa}.$

We consider the category $\Cat_{\cE}^{\dual}$ as a module (in $\Pr^L_{\omega_1}$) over the symmetric monoidal category $\Cat_{\st}^{\dual}.$

\begin{defi}Let $\cE$ be a rigid $\bE_1$-monoidal presentable stable category. For $\cC,\cD\in\Cat_{\cE}^{\dual},$ we denote by $\un{\Hom}_{\cE}^{\dual}(\cC,\cD)\in\Cat_{\st}^{\dual}$ the relative internal $\Hom.$ In other words, $\un{\Hom}_{\cE}^{\dual}(\cC,\cD)$ is a dualizable category equipped with a strongly continuous $\cE$-linear functor
$$\un{\Hom}_{\cE}^{\dual}(\cC,\cD)\otimes \cC\to\cD,$$ such that for any dualizable category $\cA$ we have an equivalence
$$\Fun^{LL}(\cA,\un{\Hom}_{\cE}^{\dual}(\cC,\cD))\xto{\sim} \Fun^{LL}_{\cE}(\cA\otimes\cC,\cD).$$\end{defi}

We now describe the category $\un{\Hom}_{\cE}^{\dual}(\cC,\cD)$ more explicitly.

\begin{theo}\label{th:descr_Hom^dual}
	Let $\cE$ be a rigid $\bE_1$-monoidal presentable stable category, and let $\cC$ and $\cD$ be dualizable left $\cE$-modules. Let $\kappa$ be an uncountable regular cardinal such that $\cC$ is $\kappa$-compact in $\Cat_{\cE}^{\dual},$ or equivalently the relative evaluation and coevaluation functors preserve $\kappa$-compact objects. Then we have $$\un{\Hom}_{\cE}^{\dual}(\cC,\cD)\simeq \ker^{\dual}(\Ind(\Fun_{\cE}^L(\cC,\cD)^{\kappa})\to \Ind(\Fun_{\cE}^{LL}(\cC,\Ind(\Calk^{\cont}_{\omega_1}(\cD))))).$$	
\end{theo}

\begin{proof}
The choice of $\kappa$ implies that we have an equivalence
\begin{equation*}
\Fun^L_{\cE}(\cC,\cD)^{\kappa}\simeq \Fun_{\cE^{\kappa}}^{\kappa\hy\rex}(\cC^{\kappa},\cD^{\kappa})\simeq \Fun_{\cE}^{LL}(\cC,\Ind(\cD^{\kappa})),
\end{equation*}
where the second equivalence uses the adjunction from Proposition \ref{prop:relative_nonstandard_adjunction}. The proof is now analogous to the proof of \cite[Theorem 1.91]{E24}.
\end{proof}

We obtain the following description of the category $\un{\Hom}_{\cE}^{\dual}(\cC,\cD),$ similar to the description of a rigidification from Proposition \ref{prop:rigidification}. 

\begin{prop}
We keep the notation and assumption from Theorem \ref{th:descr_Hom^dual}. Then the essential image of $\un{\Hom}_{\cE}^{\dual}(\cC,\cD)$ in the category $\Ind(\Fun_{\cE}^L(\cC,\cD)^{\kappa})$ is generated via colimits by the formal colimits of the form $\inddlim[\Q_{\leq}](\Phi:\Q_{\leq}\to \Fun_{\cE}^L(\cC,\cD)^{\kappa}),$ where for $a<b$ the map $\Phi(a)\to \Phi(b)$ is right trace-class in $\Fun^L_{\cE}(\cC,\cD).$
\end{prop}

\begin{proof}
A map $F\to G$ in $\Fun^L_{\cE}(\cC,\cD)^{\kappa}$ is right trace-class in $\Fun^L_{\cE}(\cC,\cD)$ if and only if its image in the category $\Fun_{\cE}^{LL}(\cC,\Ind(\Calk^{\cont}_{\omega_1}(\cD)))$ is a zero map. The proposition now follows from Theorem \ref{th:descr_Hom^dual} and the general description of kernels in $\Cat_{\st}^{\dual}$ \cite[Proposition 1.84]{E24}.
\end{proof}

We mention the following relation between the relative internal $\Hom$ in $\Cat_{\cE}^{\dual}$ over $\Cat_{\st}^{\dual}$ and in $\Pr_{\cE}^L$ over $\Pr^L_{\st}.$  

\begin{prop}\label{prop:Hom^dual_from_proper_or_smooth}
Let $\cC,\cD\in\Cat_{\cE}^{\dual}$ be dualizable left $\cE$-modules. Consider the natural continuous functor
\begin{equation}\label{eq:from_Hom^dual_to_Fun}
\un{\Hom}_{\cE}^{\dual}(\cC,\cD)\to \Fun_{\cE}^L(\cC,\cD)\simeq\cC^{\vee}\tens{\cE}\cD.
\end{equation}
\begin{enumerate}[label=(\roman*),ref=(\roman*)]
\item If $\cC$ is proper over $\cE,$ then the functor \eqref{eq:from_Hom^dual_to_Fun} has a (strongly continuous) fully faithful left adjoint, given by the composition 
\begin{equation}\label{eq:inclusion_for_proper}
	\cC^{\vee}\tens{\cE}\cD\to \un{\Hom}_{\cE}^{\dual}(\cC,\cC\otimes\cC^{\vee}\tens{\cE}\cD)\to\un{\Hom}_{\cE}^{\dual}(\cC,\cD).
\end{equation}
Here the second functor is well-defined because the evaluation functor $\ev_{\cC/\cE}:\cC\otimes\cC^{\vee}\to\cE$ is strongly continuous.  \label{inclusion_for_proper}
\item  If $\cC$ is smooth over $\cE,$ then the functor \eqref{eq:from_Hom^dual_to_Fun} is strongly continuous and fully faithful. It is naturally isomorphic to the composition 
\begin{equation*}\label{eq:inclusion_for_smooth}
\un{\Hom}_{\cE}^{\dual}(\cC,\cD)\to \cC^{\vee}\tens{\cE}\cC\otimes \un{\Hom}_{\cE}^{\dual}(\cC,\cD)\to \cC^{\vee}\tens{\cE}\cD. .
\end{equation*} \label{inclusion_for_smooth}
\item If $\cC$ is both smooth and proper over $\cE,$ then the functor \eqref{eq:from_Hom^dual_to_Fun} is an equivalence. \label{equivalence_for_smooth_and_proper}
\end{enumerate}
\end{prop}

\begin{proof} First note that \ref{equivalence_for_smooth_and_proper} follows directly from the definitions of smoothness and properness: the functor $\cC^{\vee}\tens{\cE}-$ from $\Cat_{\cE}^{\dual}$ to $\Cat_{\st}^{\dual}$ is the right adjoint to the functor $\cC\otimes-.$
	
 We prove \ref{inclusion_for_proper}, and \ref{inclusion_for_smooth} is analogous. By Proposition \ref{prop:criterions_smoothness_properness}, any compact morphism in $\Fun^L_{\cE}(\cC,\cD)$ is right trace-class. Choosing an uncountable regular cardinal $\kappa$ such that $\cC$ is $\kappa$-compact, we see that the essential image $\hat{\cY}(\cC^{\vee}\tens{\cE}\cD)\subset \Ind((\cC^{\vee}\tens{\cE}\cD)^{\kappa})$ is contained in the essential image of $\un{\Hom}_{\cE}^{\dual}(\cC,\cD).$ The functor \eqref{eq:inclusion_for_proper} is identified with the inclusion of these full subcategories, hence it is fully faithful. This proves \ref{inclusion_for_proper}.
\end{proof}

\begin{remark}
We keep the notation and assumption from Theorem \ref{th:descr_Hom^dual}. Using again the fact that $(-)^{\vee}:\Cat_{\st}^{\dual}\to \Cat_{\st}^{\dual}$ is a covariant symmetric monoidal involution, we obtain a natural equivalence
\begin{equation*}
\un{\Hom}_{\cE}^{\dual}(\cC,\cD)^{\vee}\simeq \un{\Hom}_{\cE^{mop}}^{\dual}(\cC^{\vee},\cD^{\vee}).
\end{equation*}
Arguing as in Proposition \ref{prop:from_Ind_D_to_Ind_D^op}, one can show that the following diagram commutes
\begin{equation*}
\begin{CD}
\un{\Hom}_{\cE}^{\dual}(\cC,\cD)^{\vee} @>{\sim}>> \un{\Hom}_{\cE^{mop}}^{\dual}(\cC^{\vee},\cD^{\vee})\\
@VVV @VVV\\
\Ind((\Fun_{\cE}^L(\cC,\cD)^{\kappa})^{op}) @>>> \Ind(\Fun_{\cE^{mop}}^L(\cC^{\vee},\cD^{\vee})^{\kappa}),
\end{CD}
\end{equation*}
where the lower horizontal functor is continuous and it is given on compact objects by the composition
\begin{multline*}
\begin{CD}
(\Fun_{\cE}^L(\cC,\cD)^{\kappa})^{op}\xto{((-)^{R,\cont})^{\vee}}\Fun^L_{\cE^{mop}}(\cC^{\vee},\cD^{\vee})\simeq\Ind_{\kappa}(\Fun_{\cE^{mop}}^L(\cC^{\vee},\cD^{\vee})^{\kappa})\\
\to \Ind(\Fun_{\cE^{mop}}^L(\cC^{\vee},\cD^{\vee})^{\kappa}).
\end{CD}
\end{multline*}
\end{remark}

\subsection{Internal projectivity of proper $\omega_1$-compact categories}

We now formulate the main theorem of this section.

\begin{theo}\label{th:internal_projectivity} Let $\cE$ be a rigid $\bE_1$-monoidal category. Let $\cC$ be a dualizable left $\cE$-module which is proper and $\omega_1$-compact over $\cE.$ In other words, we require that the functor $$\ev_{\cC/\cE}:\cC\otimes\cC^{\vee}\to \cE$$ is strongly continuous and the relative diagonal object $\coev(\bS)\in \cC^{\vee}\tens{\cE}\cC$ is $\omega_1$-compact.
	
Then the category $\cC$ is relatively internally projective over $\Cat_{\st}^{\dual},$ i.e. the functor
$$\un{\Hom}_{\cE}^{\dual}(\cC,-):\Cat_{\cE}^{\dual}\to\Cat_{\st}^{\dual}$$ takes short exact sequences to short exact sequences.

In particular, if $\cE$ is symmetric monoidal, then $\cC$ is internally projective in $\Cat_{\cE}^{\dual}$ in the above sense.
\end{theo}

\begin{remark}We refer to Subsection \ref{ssec:analogy_Raynaud_Gruson} for the analogy between Theorem \ref{th:internal_projectivity} and the Raynaud-Gruson criterion of projectivity of modules over (discrete) associative rings. 
\end{remark}

Theorem \ref{th:internal_projectivity} is a highly non-trivial statement, which gives a deep connection between the algebra of dualizable categories and localizing invariants. We mention the following application to the morphisms in the category $\Mot_{\cE,\kappa}^{\loc},$ for an uncountable $\kappa.$

\begin{cor}\label{cor:internal_Hom_in_Mot^loc_kappa}
Let $\cE$ be a rigid $\bE_1$-monoidal category, and let $\cC$ be a dualizable left $\cE$-module, which is proper and $\omega_1$-compact. Let $\kappa$ be an uncountable regular cardinal, and consider the universal localizing invariant $\cU_{\loc,\kappa}:\Cat_{\cE}^{\dual}\to\Mot^{\loc}_{\cE,\kappa},$ commuting with $\kappa$-filtered colimits. Then for any dualizable left $\cE$-module $\cD$ we have
\begin{equation}\label{eq:morphisms_in_Mot^loc_kappa}
\Hom(\cU_{\loc,\kappa}(\cC),\cU_{\loc,\kappa}(\cD))\cong K^{\cont}(\un{\Hom}_{\cE}^{\dual}(\cC,\cD)).
\end{equation}
Moreover, we can describe the relative internal $\Hom$ in $\Mot^{\loc}_{\cE,\kappa}$ over $\Mot^{\loc}_{\kappa}:$
\begin{equation}\label{eq:rel_internal_Hom_in_Mot^loc_kappa}
\un{\Hom}_{\Mot^{\loc}_{\cE,\kappa}/\Mot^{\loc}_{\kappa}}(\cU_{\loc,\kappa}(\cC),\cU_{\loc,\kappa}(\cD))\cong \cU_{\loc,\kappa}(\un{\Hom}_{\cE}^{\dual}(\cC,\cD))
\end{equation}
If $\cE$ is symmetric monoidal, then we obtain a description of the internal $\Hom$ in $\Mot^{\loc}_{\cE,\kappa}:$
\begin{equation}\label{eq:internal_Hom_in_Mot^loc_kappa}
\un{\Hom}_{\Mot^{\loc}_{\cE,\kappa}}(\cU_{\loc,\kappa}(\cC),\cU_{\loc,\kappa}(\cD))\cong \cU_{\loc,\kappa}(\un{\Hom}_{\cE}^{\dual}(\cC,\cD)).
\end{equation}
\end{cor}

\begin{proof}
We prove \eqref{eq:rel_internal_Hom_in_Mot^loc_kappa}, which implies \eqref{eq:morphisms_in_Mot^loc_kappa}. The proof of \eqref{eq:internal_Hom_in_Mot^loc_kappa} is analogous.

Our assumption on $\cC$ implies that the functor
\begin{equation*}
\un{\Hom}_{\cE}^{\dual}(\cC,-):\Cat_{\cE}^{\dual}\to\Cat_{\st}^{\dual}
\end{equation*}
commutes with $\omega_1$-filtered colimits, in particular, with $\kappa$-filtered colimits. By Theorem \ref{th:internal_projectivity} this functor also takes short exact sequences to short exact sequences. Hence, we obtain the well-defined functor
\begin{equation*}
\Mot^{\loc}_{\cE,\kappa}\to\Mot^{\loc,\kappa},\quad \cU_{\loc,\kappa}(\cD)\mapsto \cU_{\loc,\kappa}(\un{\Hom}_{\cE}^{\dual}(\cC,\cD)).
\end{equation*}
It follows formally that this functor is the right adjoint to tensoring with $\cU_{\loc,\kappa}(\cC),$ which proves \eqref{eq:rel_internal_Hom_in_Mot^loc_kappa}.
\end{proof}

In fact, the statements of Corollary \ref{cor:internal_Hom_in_Mot^loc_kappa} hold for $\kappa=\omega,$ and we will prove this in \cite{E}. This requires more results, including Theorem \ref{th:local_invar_of_inverse_limits} below on the localizing invariants of dualizable inverse limits.

We mention the following example, showing that an analogue of Theorem \ref{th:internal_projectivity} for small stable categories fails, even if we are working over a field.

\begin{example}
Let $\mk$ be a field, and let $A=\Lambda_{\mk}(\xi)$ be the exterior algebra over $\mk$ on one generator of cohomological degree $1.$ Let $x$ be a variable of degree zero. Then the category $\Fun_{\mk}(\Perf(A),\Perf(\mk[x]))$ is generated by the single functor sending $A$ to $\mk[x].$ The endomorphism algebra of this functor is given by $\mk[x][[y]],$ hence we obtain an equivalence
\begin{equation*}
\Fun_{\mk}(\Perf(A),\Perf(\mk[x]))\simeq \Perf(\mk[x][[y]]).
\end{equation*}
Similarly, we have an equivalence
\begin{equation*}
\Fun_{\mk}(\Perf(A),\Perf(\mk[x^{\pm}]))\simeq \Perf(\mk[x^{\pm}][[y]]).
\end{equation*}
However, the functor $\Perf(\mk[x][[y]])\to \Perf(\mk[x^{\pm}][[y]])$ is not a quotient functor (it is not even a homological epimorphism). We conclude that the category $\Perf(A)$ is not internally projective in $\Cat_{\mk}^{\perf}.$
\end{example} 

To prove Theorem \ref{th:internal_projectivity} we will need to do some difficult computations. In what follows we use the notation from Definition \ref{def:operations_objects_tensor_products} for the operations on objects in the tensor products. Also, we put 
\begin{equation*}
	\THC(\cC/\cE;-):\cC^{\vee}\tens{\cE}\cC\to\Sp,\quad \THC(\cC/\cE;M)=\Hom(\coev_{\cC/\cE}(\bS),M).
\end{equation*}
Here $\THC$ stands for ``topological Hochschild cohomology'' (relative over $\cE$).

The following statement is crucial for the proof of Theorem \ref{th:internal_projectivity}. 

\begin{lemma}\label{lem:key_lemma_on_crazy_product}
Let $\cC,\cD\in\Cat_{\cE}^{\dual},$ where $\cC$ satisfies the assumptions from Theorem \ref{th:internal_projectivity}. Let $M\in\cD^{\vee}\tens{\cE}\cC$ and $N\in\cC^{\vee}\tens{\cE}\cD$ be arbitrary objects. Then the natural map
\begin{equation}\label{eq:crazy_tensor_product}
\THC(\cC/E;\ev_{\cD/\cE}(N,(-)_{\cD}^{\vee})\tens{(\cC^{\vee}\tens{\cE}\cD)^{\omega_1}} \THC(\cC/E;\ev_{\cD/\cE}(-,M))\to \THC(\cC/\cE;\ev_{\cD/\cE}(N,M))\end{equation}  is an isomorphism. 	
\end{lemma} 

\begin{proof} Let 
\begin{equation*}
\hat{\cY}(M)=\inddlim[i\in I] M_i\in \Ind((\cD^{\vee}\otimes\cC)^{\omega_1}),\quad \hat{\cY}(N)=\inddlim[j\in J] N_j\in \Ind((\cC^{\vee}\otimes\cD)^{\omega_1}),
\end{equation*} 
where $I$ and $J$ are directed posets. We also choose a sequence $P_0\to P_1\to\dots$ in $(\cC^{\vee}\tens{\cE}\cC)^{\omega_1}$ such that $$\hat{\cY}(\coev(\bS))\cong \inddlim[n] P_n\text{ in }\Ind(\cC^{\vee}\tens{\cE}\cC).$$ 

Now, we use properness of $\cC$ and Proposition \ref{prop:seq_limits_of_filtered_colimits} to find a presentation of the first functor in the LHS of \eqref{eq:crazy_tensor_product} as a directed colimit of representable functors (more precisely, functors represented by the objects of the bigger category $\cC^{\vee}\tens{\cE}\cD$). We obtain the following chain of isomorphisms
\begin{multline}\label{eq:presentation_as_colim_of_representables}
	\THC(\cC/E;\ev_{\cD/\cE}(N,(-)_{\cD}^{\vee})\cong \Hom_{\cC^{\vee}\tens{\cE}\cC}(\indlim[k] P_k,\indlim[j] \ev_{\cD/\cE}(N_j,(-)_{\cD}^{\vee}))\\
	\cong \prolim[k]\indlim[j] \Hom_{\cC^{\vee}\tens{\cE}\cC} (P_k,\ev_{\cD/\cE}(N_j,(-)_{\cD}^{\vee}))
	\cong \prolim[k]\indlim[j] \Hom_{\cC^{\vee}\tens{\cE}\cC} (P_k,\Hom_{\cD}(-,N_j))\\
	\cong \prolim[k]\indlim[j] \Hom_{\cC^{\vee}\tens{\cE}\cD}(-,\Hom_{\cC^{\vee}}(P_k,N_j))\cong \indlim[\psi:\N\to J] \prolim[k\leq l] \Hom_{\cC^{\vee}\tens{\cE}\cD}(-,\Hom_{\cC^{\vee}}(P_k,N_{\psi(l)}))\\
	\cong \indlim[\psi:\N\to J] \Hom_{\cC^{\vee}\tens{\cE}\cD}(-,\prolim[k\leq l] \Hom_{\cC^{\vee}}(P_k,N_{\psi(l)})).
\end{multline}
Here the second isomorphism uses the identification $\hat{\cY}(\coev(\bS))\cong\inddlim[n] P_n.$ The third isomorphism uses Corollary \ref{cor:ind_isomorphisms_for_proper} \ref{first_ind_isom}. The fifth isomorphism uses Proposition \ref{prop:seq_limits_of_filtered_colimits}.

Using \eqref{eq:presentation_as_colim_of_representables} we are now able to obtain the following expression for the source of \eqref{eq:crazy_tensor_product}:
\begin{multline}\label{eq:intermediate_expr_for_tensor_product}
\THC(\cC/E;\ev_{\cD/\cE}(N,(-)_{\cD}^{\vee})\tens{(\cC^{\vee}\tens{\cE}\cD)^{\omega_1}} \THC(\cC/E;\ev_{\cD/\cE}(-,M))\\
\cong \indlim[\psi:\N\to J] \THC(\cC/E;\ev_{\cD/\cE}(\prolim[k\leq l] \Hom_{\cC^{\vee}}(P_k,N_{\psi(l)}),M)\\
\cong \indlim[\psi:\N\to J] \Hom_{\cC^{\vee}\tens{\cE}\cC}(\indlim[n] P_n,\indlim[i] \ev_{\cD/\cE}(\prolim[k\leq l] \Hom_{\cC^{\vee}}(P_k,N_{\psi(l)}),M_i))\\
\cong \indlim[\psi:\N\to J] \Hom_{\cC^{\vee}\tens{\cE}\cC}(\indlim[n] P_n,\indlim[i] \prolim[k\leq l] \ev_{\cD/\cE}(\Hom_{\cC^{\vee}}(P_k,N_{\psi(l)}),M_i))\\
\cong \indlim[\psi:\N\to J] \Hom_{\cC^{\vee}\tens{\cE}\cC}(\indlim[n] P_n,\indlim[i] \prolim[k\leq l] \Hom_{\cC^{\vee}}(P_k,\ev_{\cD/\cE}(N_{\psi(l)},M_i)))\\
\cong \indlim[\psi:\N\to J] \prolim[n] \indlim[i] \Hom_{\cC^{\vee}\tens{\cE}\cC}(P_n,\prolim[k\leq l] \Hom_{\cC^{\vee}}(P_k,\ev_{\cD/\cE}(N_{\psi(l)},M_i)))\\
\cong \indlim[\psi:\N\to J] \prolim[n] \indlim[i] \prolim[k\leq l] \Hom_{\cC^{\vee}\tens{\cE}\cC}(P_n,\Hom_{\cC^{\vee}}(P_k,\ev_{\cD/\cE}(N_{\psi(l)},M_i)))\\
\cong \indlim[\psi:\N\to J] \prolim[n] \indlim[i] \prolim[k\leq l] \Hom_{\cC^{\vee}\tens{\cE}\cC}(\ev_{\cC/\cE}(P_k,P_n),\ev_{\cD/\cE}(N_{\psi(l)},M_i))) \\
\cong \indlim[\varphi:\N\to I]\indlim[\psi:\N\to J] \prolim[n\leq m]\prolim[k\leq l]\Hom_{\cC^{\vee}\tens{\cE}\cC}(\ev_{\cC/\cE}(P_k,P_n),\ev_{\cD/\cE}(N_{\psi(l)},M_{\varphi(m)})).
\end{multline}

Here the first isomorphism uses \eqref{eq:presentation_as_colim_of_representables} and the fact that the functor
$$\THC(\cC/\cE;\ev_{\cD/\cE}(-,M)):\cC^{\vee}\tens{\cE}\cD\to\Sp$$ is left Kan-extended from $(\cC^{\vee}\tens{\cE}\cC)^{\omega_1}$ (since by assumption on $\cC$ the functor $\THC(\cC/\cE;-)$ commutes with $\omega_1$-filtered colimits). The third and fourth isomorphisms use Corollary \ref{cor:ind_isomorphisms_for_proper} \ref{second_ind_isom}. The fifth isomorphism uses the identification $\coev_{\cC/\cE}(\bS)\cong \inddlim[n]P_n.$ The eighth isomorphism uses Proposition \ref{prop:seq_limits_of_filtered_colimits}.

{\noindent{\bf Claim.}} {\it The functor
\begin{equation}\label{eq:key_quadrofunctor_for_Hom^dual}
	F:\N^{op}\times I\times \N^{op}\times J\to\Sp,\quad F(n,i,k,j)=\Hom_{\cC^{\vee}\tens{\cE}\cC}(\ev_{\cC/\cE}(P_k,P_n),\ev_{\cD/\cE}(N_j,M_i)),
	\end{equation}
satisfies the following refinements of conditions \ref{tricky_assump1} and \ref{tricky_assump2} from Proposition \ref{prop:quadrofunctors}:
\begin{enumerate}[label=(\roman*),ref=(\roman*)]
	\item the map \begin{equation}\label{eq:maps_of_proobjects}
		\proolim[k](\indlim[j]\prolim[n]\indlim[i]F(n,i,k,j))\to \proolim[k](\prolim[n]\indlim[i]\indlim[j]F(n,i,k,j))\end{equation}
	is an isomorphism in $\Pro(\Sp).$ \label{refined_assump1}
	
	\item the map  \begin{equation*}\proolim[n](\indlim[i]\prolim[k]\indlim[j]F(n,i,k,j))\to \proolim[n](\prolim[k]\indlim[i]\indlim[j]F(n,i,k,j))\end{equation*}
is an isomorphism in $\Pro(\Sp).$ \label{refined_assump2}
\end{enumerate} 
}

Assuming the claim, we can apply Proposition \ref{prop:quadrofunctors} to the quadrofunctor \eqref{eq:key_quadrofunctor_for_Hom^dual}. Combining this with \eqref{eq:intermediate_expr_for_tensor_product}, we obtain
\begin{multline*}
\THC(\cC/\cE;\ev_{\cD/\cE}(N,(-)_{\cD}^{\vee})\tens{(\cC^{\vee}\tens{\cE}\cD)^{\omega_1}} \THC(\cC/\cE;\ev_{\cD/\cE}(-,M))\\
\cong \indlim[\varphi:\N\to I]\indlim[\psi:\N\to J]\prolim[n\leq m] \prolim[k\leq l] \Hom_{\cC^{\vee}\tens{\cE}\cC}(\ev_{\cC/\cE}(P_k,P_n),\ev_{\cD/\cE}(N_{\psi(l)},M_{\varphi(m)}))\\
\cong \prolim[n]\indlim[i]\indlim[j] \Hom_{\cC^{\vee}\tens{\cE}\cC}(\ev_{\cC/\cE}(P_n,P_n),ev_{\cD}(N_j,M_i))\\
\cong \Hom_{\cC^{\vee}\tens{\cE}\cC}(\indlim[n]\ev_{\cC/\cE}(P_n,P_n),\indlim[i]\indlim[j]\ev_{\cD/\cE}(N_j,M_i))\cong \THC(\cC/\cE;\ev_{\cD/\cE}(N,M)),
\end{multline*}
as required. Here we used the isomorphism $\hat{\cY}(\coev_{\cC/\cE}(\bS))\cong\inddlim[n]\ev_{\cC/\cE}(P_n,P_n),$ which follows from the properness of $\cC.$ 
\begin{proof}[Proof of Claim] We prove the condition \ref{refined_assump1}, and the proof of \ref{refined_assump2} is analogous. We may assume that for each $k\geq 0$ the map $f_k:P_k\to P_{k+1}$ is compact in $\cC^{\vee}\tens{\cE}\cC,$ and we choose a map $\wt{f_k}:\cY(P_k)\to\hat{\cY}(P_{k+1}),$ witnessing the compactness. By properness of $\cC$ over $\cE,$ for each $k\geq 0$ we obtain a factorization \begin{equation}\label{eq:factorization} \inddlim[n]\ev_{\cC/\cE}(P_k,P_n)\to \cY(P_k)\xto{\wt{f_k}} \hat{\cY}(P_{k+1})\to \inddlim[n]\ev_{\cC/\cE}(P_{k+1},P_n)\end{equation}
in $\Ind(\cC^{\vee}\tens{\cE}\cC).$ 
	
Now, for each $k\geq 0$  the map $\wt{f_k}$ allows to construct the following composition map:
\begin{multline}\label{eq:constructing_the_composition_comb_of_lim_colim}
\prolim[n]\indlim[i]\indlim[j] F(n,i,k+1,j)\cong \Hom(\inddlim[n]\ev_{\cC/\cE}(P_{k+1},P_n),\inddlim[(i,j)]\ev_{\cD/\cE}(N_j,M_i))\to \\
\Hom(\hat{\cY}(P_{k+1}),\inddlim[(i,j)]\ev_{\cD/\cE}(N_j,M_i))\to \Hom(\cY(P_k),\inddlim[(i,j)]\ev_{\cD/\cE}(N_j,M_i))\\
\cong \indlim[i]\indlim[j]\prolim[n]F(n,i,k,j)\to \indlim[j]\prolim[n]\indlim[i]F(n,i,k,j). 
\end{multline} 
The compositions \eqref{eq:constructing_the_composition_comb_of_lim_colim} define a map 
\begin{equation}\label{eq:inverse_map_lim_colim}
\proolim[k](\prolim[n]\indlim[i]\indlim[j]F(n,i,k,j))\to \proolim[k](\indlim[j]\prolim[n]\indlim[i]F(n,i,k,j)).\end{equation}
Using the factorizations \eqref{eq:factorization} we see that the map \eqref{eq:inverse_map_lim_colim} is the inverse to the map \eqref{eq:maps_of_proobjects}. This proves \ref{refined_assump1}.
\end{proof}
This proves the lemma.
\end{proof}

\begin{cor}\label{cor:hom_epi_onto_the_image} Let $\cC$ and $\cD$ be as in Lemma \ref{lem:key_lemma_on_crazy_product}. The functor \begin{equation*}
F:(\cC^{\vee}\tens{\cE}\cD)^{\omega_1}\simeq \Fun_{\cE}^{LL}(\cC,\Ind(\cD^{\omega_1}))\to \Fun_{\cE}^{LL}(\cC,\Ind(\Calk^{\cont}_{\omega_1}(\cD)))	
\end{equation*}
is a homological epimorphism onto its image.
\end{cor}

\begin{proof}
Given two objects $M,N\in (\cC^{\vee}\tens{\cE}\cD)^{\omega_1},$ we have
$$\Fiber(\Hom(M,N)\to \Hom(F(M),F(N)))\cong \THC(\cC/\cE;\ev_{\cD/\cE}(N,M_{\cD}^{\vee})).$$
Hence, we need to prove that for any $M,N\in (\cC^{\vee}\tens{\cE}\cD)^{\omega_1}$ the natural map
\begin{multline*}
\THC(\cC/\cE;\ev_{\cD/\cE}(N,(-)_{\cD}^{\vee}))\tens{(\cC^{\vee}\tens{\cE}\cD)^{\omega_1}} \THC(\cC/\cE;\ev_{\cD/\cE}(-,M_{\cD}^{\vee}))\\
\to \THC(\cC/\cE;\ev_{\cD/\cE}(N,M_{\cD}^{\vee}))
\end{multline*}
is an isomorphism. This follows from Lemma \ref{lem:key_lemma_on_crazy_product}.
\end{proof}

\begin{prop}\label{prop:dual_of_internal_Hom}Let $\cC$ and $\cD$ be as in Lemma \ref{lem:key_lemma_on_crazy_product}. Then the essential image of the natural fully faithful functor
\begin{equation*}
\Phi:\un{\Hom}_{\cE}^{\dual}(\cC,\cD)^{\vee}\to \Ind(((\cC^{\vee}\tens{\cE}\cD)^{\omega_1})^{op})\simeq \Fun((\cC^{\vee}\tens{\cE}\cD)^{\omega_1},\Sp)  
\end{equation*}
contains all objects of the form $\THC(\cC/\cE;\ev_{\cD/\cE}(-,M)),$ where $M\in\cD^{\vee}\tens{\cE}\cC.$ Moreover, the image of $\Phi$ is generated via colimits by such objects with $M\in (\cD^{\vee}\tens{\cE}\cC)^{\omega_1}.$
\end{prop}

\begin{proof} By the proof of Corollary \ref{cor:hom_epi_onto_the_image}, the essential image of $\Phi$ consists of (exact) functors $G:(\cC^{\vee}\tens{\cE}\cD)^{\omega_1}\to \Sp,$ such that for all $N\in (\cC^{\vee}\tens{\cE}\cD)^{\omega_1},$ the natural map
\begin{equation*}
\THC(\cC/\cE;\ev_{\cD/\cE}(N,(-)_{\cD}^{\vee}))\tens{(\cC^{\vee}\tens{\cE}\cD)^{\omega_1}} G\to G(N)
\end{equation*} 
is an isomorphism. By Lemma \ref{lem:key_lemma_on_crazy_product}, this condition holds for $G$ of the form $\THC(\cC/\cE;\ev_{\cD/\cE}(-,M))$ for $M\in\cD^{\vee}\tens{\cE}\cC.$ Therefore, such functors are in the essential image of $\Phi.$

Furthermore, by the proof of Corollary \ref{cor:hom_epi_onto_the_image}, the essential image of $\Phi$ is generated by the objects of the form $\THC(\cC/\cE;\ev_{\cD/\cE}(-,P_{\cD}^{\vee})),$ where $P\in (\cC^{\vee}\tens{\cE}\cD)^{\omega_1}.$ For such $P$ we can choose an $\omega_1$-filtered system $(M_i\in (\cD^{\vee}\tens{\cE}\cC)^{\omega_1})_i$ such that $P_{\cD}^{\vee}\cong\indlim[i] M_i.$ Since $\coev(\bS)\in\cC^{\vee}\tens{\cE}\cC$ is $\omega_1$-compact, we see that
\begin{equation*}
\THC(\cC/\cE;\ev_{\cD/\cE}(-,P_{\cD}^{\vee}))\cong \indlim[i] \THC(\cC/\cE;\ev_{\cD/\cE}(-,M_i)).
\end{equation*} 
This shows that the essential image of $\Phi$ is generated by the objects $\THC(\cC/\cE;\ev_{\cD/\cE}(-,M))$ with $M\in (\cC^{\vee}\tens{\cE}\cD)^{\omega_1}.$
\end{proof}

\begin{prop}\label{prop:comm_square_int_Homs} Let $\cC$ and $\cD$ be as in Lemma \ref{lem:key_lemma_on_crazy_product}. Let $\cD'$ be another dualizable left $\cE$-module, and let $F:\cD\to\cD'$ be a strongly continuous $\cE$-linear functor. Then we have a commutative square
\begin{equation}\label{eq:comm_square_int_Homs}
\begin{CD}
\un{\Hom}_{\cE}^{\dual}(\cC,\cD')^{\vee} @>{\un{\Hom}_{\cE}^{\dual}(\cC,F)^{\vee}}>> \un{\Hom}_{\cE}^{\dual}(\cC,\cD)^{\vee}\\
@V{\Phi'}VV @V{\Phi}VV\\
\Fun((\cC^{\vee}\tens{\cE}\cD')^{\omega_1},\Sp) @>>> \Fun((\cC^{\vee}\tens{\cE}\cD)^{\omega_1},\Sp).
\end{CD}	
\end{equation} 
Here the vertical (fully faithful) functors are as in Proposition \ref{prop:dual_of_internal_Hom}, and the lower horizontal functor is the precomposition with $(\id_{\cC^{\vee}}\boxtimes F)^{\omega_1}.$
\end{prop}

\begin{proof}We need to check that the lower horizontal functor takes $\im(\Phi')$ to $\im(\Phi).$ By Proposition \ref{prop:dual_of_internal_Hom}, it suffices to show that for $M\in (\cD^{'\vee}\tens{\cE}\cC)^{\omega_1},$ the image of 
$\THC(\cC/\cE;\ev_{\cD'/\cE}(-,M))\in\im(\Phi')$ is contained in $\im(\Phi).$ But we have
\begin{equation*}
\THC(\cC/\cE;\ev_{\cD'/\cE}((\id_{\cC^{\vee}}\boxtimes F)(-),M))\cong \THC(\cC/\cE;\ev_{\cD/\cE}(-,(\id_{\cC^{\vee}}\boxtimes F^{\vee})(M))).
\end{equation*}
By Proposition \ref{prop:dual_of_internal_Hom}, the latter object is contained in $\im(\Phi),$ as required.
\end{proof}

\begin{proof}[Proof of Theorem \ref{th:internal_projectivity}] We only need to prove the following statement: if $F:\cD\to\cD'$ is an $\cE$-linear strongly continuous quotient functor between dualizable left $\cE$-modules, then the functor \begin{equation*}
		\un{\Hom}_{\cE}^{\dual}(\cC,F):\un{\Hom}_{\cE}^{\dual}(\cC,\cD)\to \un{\Hom}_{\cE}^{\dual}(\cC,\cD')
	\end{equation*}
is a quotient functor. Equivalently, we need to prove that the dual functor
\begin{equation*}
	\un{\Hom}_{\cE}^{\dual}(\cC,F)^{\vee}:\un{\Hom}_{\cE}^{\dual}(\cC,\cD')^{\vee}\to \un{\Hom}_{\cE}^{\dual}(\cC,\cD)^{\vee}
\end{equation*}
is fully faithful. Applying Proposition \ref{prop:comm_square_int_Homs}, we see that in the associated commutative square \eqref{eq:comm_square_int_Homs}
both vertical functors and the lower horizontal functor are fully faithful. Hence, the upper horizontal functor is also fully faithful. This proves the theorem.
\end{proof}

We obtain the following surprising result, which is a strengthening of Corollary \ref{cor:hom_epi_onto_the_image}.

\begin{cor}\label{cor:hom_epi_internal_hom}
Let $\cE$ be a rigid $\bE_1$-monoidal category. Let $\cC,\cD\in\Cat_{\cE}^{\dual}$ be dualizable left $\cE$-modules, and suppose that $\cC$ is proper and $\omega_1$-compact over $\cE.$ Let $\kappa$ be an uncountable regular cardinal. Then the functor
\begin{equation*}
(\cC^{\vee}\tens{\cE}\cD)^{\kappa}\simeq \Fun_{\cE}^{LL}(\cC,\Ind(\cD^{\kappa}))\to \Fun_{\cE}^{LL}(\cC,\Ind(\Calk_{\kappa}^{\cont}(\cD)))
\end{equation*}
is a homological epimorphism. Equivalently, we have a short exact sequence
\begin{equation*}
0\to\un{\Hom}_{\cE}^{\dual}(\cC,\cD)\to \Ind((\cC^{\vee}\tens{\cE}\cD)^{\kappa})\to \Ind(\Fun_{\cE}^{LL}(\cC,\Ind(\Calk_{\kappa}^{\cont}(\cD))))\to 0.
\end{equation*}
\end{cor}

\begin{proof}
We have the following commutative square
\begin{equation*}
\begin{CD}
\Ind(\Fun_{\cE}^{LL}(\cC,\Ind(\cD^{\kappa}))) @>>> \Ind(\Fun_{\cE}^{LL}(\cC,\Ind(\Calk_{\kappa}^{\cont}(\cD))))\\
@V{\sim}VV @VVV\\
\un{\Hom}_{\cE}^{\dual}(\cC,\Ind(\cD^{\kappa})) @>>> \un{\Hom}_{\cE}^{\dual}(\cC,\Ind(\Calk_{\kappa}^{\cont}(\cD)))
\end{CD}
\end{equation*}
Here the left vertical functor is an equivalence, and the right vertical functor is fully faithful. By Theorem \ref{th:internal_projectivity}, the lower horizontal functor is a quotient functor. We conclude that the right vertical functor is an equivalence, and the upper horizontal functor is a quotient functor. This proves the corollary. 
\end{proof}

We mention a few statements on the dualizable internal $\Hom,$ which can be deduced from the results above.

\begin{remark}
Let $\cC$ and $\cD$ be dualizable left $\cE$-modules, and suppose that $\cC$ is proper and $\omega_1$-compact.
\begin{enumerate}
\item Consider the functor $\Lambda_{\cC,\cD}:\cC^{\vee}\tens{\cE}\cD\to \un{\Hom}_{\cE}^{\dual}(\cC,\cD),$ which is the right adjoint to the natural functor discussed in Proposition \ref{prop:Hom^dual_from_proper_or_smooth}. Then the functor $\Lambda_{\cC,\cD}$ is fully faithful, commutes with $\omega_1$-filtered colimits and the image $\Lambda_{\cC,\cD}((\cC^{\vee}\tens{\cE}\cD)^{\omega_1})$ generates the category $\un{\Hom}_{\cE}^{\dual}(\cC,\cD)$ via colimits.
\item The composition
\begin{equation*}
\cC^{\vee}\tens{\cE}\cD\xto{\Lambda_{\cC,\cD}} \un{\Hom}_{\cE}^{\dual}(\cC,\cD)\to \Ind((\cC^{\vee}\tens{\cE}\cD)^{\omega_1})\simeq\Fun(((\cC^{\vee}\tens{\cE}\cD)^{\omega_1})^{op},\Sp)
\end{equation*}
is given by $M\mapsto \THC(\cC/\cE;\ev_{\cD/\cE}(M,(-)_{\cD}^{\vee})).$
\item The composition
\begin{multline*}
\cC^{\vee}\tens{\cE}\cD\xto{\Lambda_{\cC,\cD}} \un{\Hom}_{\cE}^{\dual}(\cC,\cD)\simeq\un{\Hom}_{\cE^{mop}}^{\dual}(\cC^{\vee},\cD^{\vee})^{\vee}\\
\to \Ind(((\cD^{\vee}\tens{\cE}\cC)^{\omega_1})^{op})\simeq \Fun((\cD^{\vee}\tens{\cE}\cC)^{\omega_1},\Sp)
\end{multline*}
is given by $M\mapsto \THC(\cC/\cE;\ev_{\cD/\cE}(M,-)).$
\item The composition
\begin{multline*}
(\cC^{\vee}\tens{\cE}\cD)\times (\cD^{\vee}\tens{\cE}\cC)\xto{(\Lambda_{\cC,\cD},\Lambda_{\cC^{\vee},\cD^{\vee}})} \un{\Hom}_{\cE}^{\dual}(\cC,\cD)\times \un{\Hom}_{\cE^{mop}}^{\dual}(\cC^{\vee},\cD^{\vee})\\
\simeq \un{\Hom}_{\cE}^{\dual}(\cC,\cD)\times \un{\Hom}_{\cE}^{\dual}(\cC,\cD)^{\vee}\xto{\ev}\Sp
\end{multline*}
is given by $(M,N)\mapsto \THC(\cC/\cE;\ev_{\cD/\cE}(M,N)).$
\end{enumerate}
\end{remark}

\subsection{Analogy with the Raynaud-Gruson criterion of projectivity}
\label{ssec:analogy_Raynaud_Gruson}

We recall the following important characterization of projective modules over a (usual, discrete) associative unital ring due to Raynaud and Gruson \cite{RayGru}, partially based on Kaplansky's work \cite{Ka58}. In the following theorem all the modules are discrete (concentrated in degree zero), and the tensor products are non-derived.

\begin{theo}\cite{Ka58, RayGru}\cite[\href{https://stacks.math.columbia.edu/tag/059Z}{Tag 059Z}]{Stacks}\label{th:Raynaud_Gruson}
Let $R$ be an associative unital ring, and let $M$ be a left $R$-module. Then $M$ is projective if and only if the following conditions are satisfied.
\begin{enumerate}[label=(\roman*),ref=(\roman*)]
\item $M$ is flat over $R.$ \label{cond_flat_module}
\item $M$ is Mittag-Leffler, i.e. for any collection of right $R$-modules $(N_s)_{s\in S}$ the natural map $(\prodd[s]N_s)\tens{R}M\to \prodd[s](N_s\tens{R}M)$ is a monomorphism. \label{cond_ML_module}
\item $M$ is isomorphic to a (possibly uncountable) direct sum of countably presented left $R$-modules. \label{cond_coproduct_of_countably presented}
\end{enumerate} 
\end{theo}

\begin{remark} 1) Usually the notion of a Mittag-Leffler module is defined in a different, but equivalent, way. Namely, let $M\cong \indlim[i\in I]M_i,$ where $I$ is directed and each $M_i$ is a finitely presented left $R$-module. Then $M$ is Mittag-Leffler if for any left $R$-module $N$ the pro-system $(\Hom_R(M_i,N))_i$ is Mittag-Leffler. Equivalently, this means that for any $i\in I$ there exists $j\geq i$ such that for any finitely presented right $R$-module $P$
we have
\begin{equation*}
\ker(P\tens{R}M_i\to P\tens{R}M_j)=\ker(P\tens{R}M_i\to P\tens{R}M).
\end{equation*}

If in addition $M$ is flat, then $M$ is Mittag-Leffler if and only if the pro-system $(\Hom_R(M_i,R))_i$ is Mittag-Leffler (note that by Lazard's theorem we may assume that each $M_i$ is finitely generated projective). We refer to \cite[\href{https://stacks.math.columbia.edu/tag/0599}{Tag 0599}, \href{https://stacks.math.columbia.edu/tag/059M}{Tag 059M}]{Stacks} for details.

2) If $M$ is Mittag-Leffler, then $M$ is countably presented if and only if it is countably generated \cite[\href{https://stacks.math.columbia.edu/tag/059W}{Tag 059W}]{Stacks}. 
\end{remark}

We refer to \cite{Dri06} for a discussion and deep applications of the Raynaud-Gruson criterion. 

We claim that Theorem \ref{th:internal_projectivity} is an analogue of the ``if'' direction in Theorem \ref{th:Raynaud_Gruson}. As above, let $\cE$ be a rigid $\bE_1$-monoidal category. Then any dualizable left $\cE$-module $\cC$ is flat, i.e. the functor $-\tens{\cE}\cC$ takes short exact sequences of dualizable right $\cE$-modules to short exact sequences of dualizable categories. Hence, an analogue of the condition \ref{cond_flat_module} is satisfied automatically. Next, our assumption that $\cC$ is $\omega_1$-compact is analogous to the condition \ref{cond_coproduct_of_countably presented} (note that for the ``if'' direction in Theorem \ref{th:Raynaud_Gruson} we may assume that $M$ is a single countably presented module). Finally, the following proposition shows that properness of $\cC$ over $\cE$ is an analogue of the condition \ref{cond_ML_module}.

\begin{prop}\label{prop:proper_same_as_Mittag_Leffler}
	Let $\cC$ be a dualizable left module over a rigid $\bE_1$-monoidal category $\cE.$ The following are equivalent.
	\begin{enumerate}[label=(\roman*),ref=(\roman*)]
		\item $\cC$ is proper over $\cE.$ \label{cond_properness}
		\item For any family of dualizable right $\cE$-modules $(\cD_s)_{s\in S},$ the natural functor
		\begin{equation*}
			(\prodd[s]^{\dual}\cD_s)\tens{\cE}\cC\to\prodd[s]^{\dual}(\cD_s\tens{\cE}\cC)	
		\end{equation*}
		is fully faithful. \label{cond_via_tensoring_with_products}
	\end{enumerate}
\end{prop}

\begin{proof}
\Implies{cond_properness}{cond_via_tensoring_with_products}. If $\cC$ is proper, then by Proposition \ref{prop:Hom^dual_from_proper_or_smooth} for any dualizable right $\cE$-module $\cD$ we have a natural fully faithful functor
\begin{equation*}
\cD\tens{\cE}\cC\to \un{\Hom}^{\dual}_{\cE^{mop}}(\cC^{\vee},\cD).
\end{equation*}
induced by the (strongly continuous) evaluation functor $\cC\otimes\cC^{\vee}\to\cE.$ Therefore, we have a commutative square
\begin{equation*}
\begin{CD}
(\prodd[s]^{\dual}\cD_s)\tens{\cE}\cC @>>> \prodd[s]^{\dual}(\cD_s\tens{\cE}\cC)\\
@VVV @VVV\\
 \un{\Hom}^{\dual}_{\cE^{mop}}(\cC^{\vee},\prodd[s]^{\dual}\cD_s) @>{\sim}>> \prodd[s]^{\dual}\un{\Hom}^{\dual}_{\cE^{mop}}(\cC^{\vee},\cD_s), 
\end{CD}
\end{equation*}
in which the vertical functors are fully faithful and the lower horizontal functor is an equivalence. It follows that the upper horizontal functor is fully faithful, as required.

\Implies{cond_via_tensoring_with_products}{cond_properness}. For a general discussion of products in $\Cat_{\st}^{\dual}$ we refer to \cite[Section 1.17]{E24}. We recall some notation. Given a family of dualizable categories $(\cT_s)_{s\in S},$ consider the right adjoint to the natural functor $\prodd[s]^{\dual}\cT_s\to\prodd[s]\cT_s.$ We denote this right adjoint by
\begin{equation*}
\prodd[s]:\prodd[s]\cT_s\to \prodd[s]^{\dual}\cT_s,\quad (x_s)_s\mapsto \prodd[s] x_s. 
\end{equation*}   
This notation is justified: the object $\prodd[s]x_s\in\prodd[s]^{\dual}\cT_s$ is indeed the product of the images of $x_s$ under the natural functors $\cT_s\to\prodd[s]^{\dual}\cT_s.$ Moreover, under the identification $(\prodd[s]^{\dual}\cT_s)^{\vee}\simeq \prodd[s]^{\dual}\cT_s^{\vee},$ we have an isomorphism
\begin{equation*}
\ev_{\prodd[s]^{\dual}\cT_s}(\prodd[s]x_s,\prodd[s]y_s)\cong \prodd[s]\ev_{\cT_s}(x_s,y_s),\quad x_s\in\cT_s,\,y_s\in\cT_s^{\vee},
\end{equation*}
and similarly for the relative duality over $\cE.$

Now, the condition \ref{cond_via_tensoring_with_products} in particular implies that for any set $S$ the functor
\begin{equation}\label{eq:fully_faithful_comparison}
(\prodd[S]^{\dual}\cE)\tens{\cE}\cC\to \prodd[S]^{\dual}\cC
\end{equation}
is fully faithful. By Remark \ref{rem:relative_and_absolute_duality} \ref{non_trivial_duality_for_tensor_products}, we have a natural equivalence
\begin{equation*}
((\prodd[S]^{\dual}\cE)\tens{\cE}\cC)^{\vee}\simeq \cC^{\vee}\tens{\cE}\Phi_{\cE,*}(\prodd[S]^{\dual}\cE),
\end{equation*}
where the notation is taken from Remark \ref{rem:relative_and_absolute_duality} \ref{functor_Phi_E}.
The corresponding evaluation functor
\begin{equation*}
\ev:((\prodd[S]^{\dual}\cE)\tens{\cE}\cC)\otimes (\cC^{\vee}\tens{\cE}\Phi_{\cE,*}(\prodd[S]^{\dual}\cE))\to\Sp
\end{equation*}
 is given by
\begin{equation}\label{eq:first_evaluation}
\ev((\prodd[s]x_s)\boxtimes c,c'\boxtimes(\prodd[s]y_s)) = \Hom(1,(\prodd[s](y_s\otimes x_s))\otimes \ev_{\cC/\cE}(c,c')),
\end{equation}
for $x_s,y_s\in\cE,$ $s\in S,$ and $c\in\cC,$ $c'\in\cC^{\vee}.$ 

Now fix the objects $x_s,y_s, c,c'$ as above, and let $\hat{\cY}(c)=\inddlim[i]c_i,$ $\hat{\cY}(c')=\inddlim[j]c'_j.$ Then the functor \eqref{eq:fully_faithful_comparison} sends the object $(\prodd[s]x_s)\boxtimes c\in (\prodd[S]^{\dual}\cE)\tens{\cE}\cC$ to the object
\begin{equation*}
\indlim[i]\prodd[s](x_s\otimes c_i)\in \prodd[S]^{\dual}\cC.
\end{equation*}
The left adjoint of the dual of \eqref{eq:fully_faithful_comparison} sends an object $c'\boxtimes (\prodd[s]y_s)\in \cC^{\vee}\tens{\cE}\Phi_{\cE,*}(\prodd[s]^{\dual}\cE)$ to
\begin{equation*}
\indlim[j]\prodd[s](c_j'\otimes \Phi^{-1}(y_s))\in\prodd[S]^{\dual}\cC^{\vee}\simeq (\prodd[S]^{\dual}\cC)^{\vee}.
\end{equation*}
We have
\begin{multline}\label{eq:second_evaluation}
\ev(\indlim[i]\prodd[s](x_s\otimes c_i),\indlim[j]\prodd[s](c_j'\otimes \Phi^{-1}(y_s)))\\
\cong \indlim[i]\indlim[j]\prodd[s]\Hom(1,\ev_{\cC/\cE}(x_s\otimes c_i,c_j'\otimes\Phi_{\cE}^{-1}(y_s)))\\
\cong \indlim[i]\indlim[j]\prodd[s]\Hom(1,y_s\otimes x_s\otimes \ev_{\cC/\cE}(c_i,c_j')).
\end{multline}
Now, fully faithfulness of \eqref{eq:fully_faithful_comparison} implies that the objects in \eqref{eq:first_evaluation} and \eqref{eq:second_evaluation} are isomorphic. Using the self-duality of $\cE,$ we conclude that for any collection $(z_s\in\cE^{\vee})_{s\in S}$ we have an isomorphism
\begin{equation*}
\ev_{\cE}(\ev_{\cC/\cE}(c,c'),\prodd[s]z_s)\cong \indlim[i]\indlim[j]\prodd[s]\ev_{\cE}(\ev_{\cC/\cE}(c_i,c_j'),z_s).
\end{equation*}
By Proposition \ref{prop:criterion_strcont} this means that the functor $\ev_{\cC/\cE}$ is strongly continuous, i.e. $\cC$ is proper over $\cE.$ This proves the proposition. 
\end{proof}

It is not clear how to describe the class of all relatively internally projective dualizable left 
$\cE$-modules $\cC\in\Cat_{\cE}^{\dual},$ not necessarily $\omega_1$-compact. It seems plausible that at least the relative internal projectivity implies properness. The following example shows that a naive attempt to construct a counterexample to this fails.  

\begin{prop}
Suppose that $\cE=D(\mk),$ where $\mk$ is a field. Consider the $\mk$-linear category $\cC$ of triples $(V,W,\varphi),$ where $V,W\in D(\mk)$ and $\varphi:V\to\biggplus[\N]W.$ Then $\cC$ is not internally projective in $\Cat_{\mk}^{\dual}.$ In particular, the class of internally projective dualizable $\mk$-linear categories is not closed under semiorthogonal gluings (split extensions) in $\Cat_{\mk}^{\dual}.$  
\end{prop}

\begin{proof}
Given dualizable $\mk$-linear categories $\cB$ and $\cD,$ the strongly continuous $\mk$-linear functors $\cB\tens{\mk}\cC\to\cD$ are given by triples $(F,G,\psi),$ where $F,G:\cB\to\cD$ are strongly continuous $\mk$-linear functors and $\psi:\biggplus[\N]F\to G.$ It follows that we have
\begin{equation*}
\un{\Hom}^{\dual}(\cC,\cD)\simeq \cD\oright_{\Phi}\cD,\quad \Phi=(\prodd[\N]\id_{\cD})^{\cont}\in\Fun^L(\cD,\cD).
\end{equation*}
It follows that the functor
\begin{equation*}
\un{\Hom}^{\dual}(\cC,D(\mk[x]))\to \un{\Hom}^{\dual}(\cC,D(\mk[x^{\pm}]))
\end{equation*}
is not a quotient functor. Indeed, this would imply that the map
\begin{equation*}
\mk[x^{\pm}]\tens{\mk[x]}(\prodd[\N]\mk[x])\tens{\mk[x]}\mk[x^{\pm}]\to \prodd[\N]\mk[x^{\pm}]
\end{equation*}
is an isomorphism, which is not the case. We conclude that $\cC$ is not internally projective in $\Cat_{\mk}^{\dual}.$
\end{proof}

\subsection{Some examples}

We give some examples of the dualizable internal $\Hom$ in the special case $\cE=\Mod_{\mk},$ where $\mk$ is an $\bE_{\infty}$-ring.

\begin{prop}\label{prop:when_dualizable_Hom_is_naive}
Let $A$ be a proper $\bE_1$-$\mk$-algebra. Suppose that the category $(A\hy Mod)^{\omega_1}$ of $\omega_1$-compact left $A$-modules is generated as an idempotent-complete stable subcategory by the objects of the form $A\tens{\mk} V,$ where $V\in(\Mod_{\mk})^{\omega_1}.$ Then we have
an equivalence
\begin{equation*}
\un{\Hom}_{\mk}^{\dual}(\Mod\hy A,\Mod_{\mk})\simeq A\hy Mod.
\end{equation*}
\end{prop}

\begin{proof}
For a left $A$-module $M,$ we put
\begin{equation*}
M^*=\Hom_{\mk}(M,\mk)\in\Mod\hy A,\quad M^{\vee}=\Hom_A(M,A)\in\Mod\hy A.
\end{equation*}
 It suffices to prove that for $M,N\in (A\hy\Mod)^{\omega_1}$ the natural map
 \begin{equation*}
 M^{\vee}\tens{A}N\cong \THC(A/\mk;A\tens{\mk}M^*)\tens{A}N\to \THC(A/\mk;N\tens{\mk}M^*)
 \end{equation*}
 is an isomorphism. Equivalently, we want to show that for $M\in (A\hy\Mod)^{\omega_1}$ the functor
\begin{equation}\label{eq:functor_of_N}
A\hy\Mod\to \Mod_{\mk},\quad N\mapsto \THC(A/\mk;N\tens{\mk}M^*),
\end{equation}
commutes with colimits. By the assumption of the proposition, we may and will assume that $M=A\tens{\mk}V$ for some $V\in(\Mod_{\mk})^{\omega_1}.$ Then for $N\in A\hy\Mod$ we have
\begin{multline*}
\THC(A/\mk;N\tens{\mk}M^*)\cong \THC(A/\mk;N\tens{\mk}V^*\tens{\mk}A^*)\cong \THC(A/\mk;\Hom_{\mk}(A,N\tens{\mk}V^*))\\
\cong N\tens{\mk}V^*.
\end{multline*}
This proves that the functor \eqref{eq:functor_of_N} commutes with colimits, as required.
\end{proof}

We obtain the following corollary.

\begin{cor}\label{cor:naive_dualizable_Hom_for_schemes}
Suppose that $\mk$ is a usual (discrete) commutative ring, which is noetherian. Let $X$ be a proper scheme over $\mk,$ such that $X$ is of finite $\Tor$-dimension over $\mk.$ If $X$ is regular of finite Krull dimension, then we have
\begin{equation}\label{eq:naive_dualizable_Hom_for_schemes}
\un{\Hom}_{\mk}^{\dual}(D(X),D(\mk))\simeq D(X).
\end{equation}
\end{cor}

\begin{proof}
This is a special case of Proposition \ref{prop:when_dualizable_Hom_is_naive}. Namely, let $\cG\in\Perf(X)$ be a generator. Then $A=\End(\cG)$ is a proper $\bE_1$-$\mk$-algebra, and we have $D(X)\simeq \Mod\hy A.$ By the self-duality of $D(X),$ we also have an equivalence $D(X)\simeq A\hy\Mod,$ sending $\cG^{\vee}$ to $A.$ By \cite[Theorem 2.1]{Nee21} (more precisely, the countable version of loc. cit. with the same proof), the category $D(X)^{\omega_1}$ is generated as an idempotent-complete stable subcategory by countable direct sums of shifts of $\cG^{\vee}.$ Applying Proposition \ref{prop:when_dualizable_Hom_is_naive}, we obtain the equivalence \eqref{eq:naive_dualizable_Hom_for_schemes}. 
\end{proof}

Next, we observe the following relation between the dualizable internal $\Hom$ and the category $\IndCoh(X)$ for a noetherian scheme $X.$

\begin{prop}\label{prop:IndCoh_and_Hom^dual}
Suppose that $\mk$ is a usual commutative ring, which is regular noetherian. Let $X$ be a proper scheme over $\mk.$ Then we have an equivalence
\begin{equation*}
\un{\Hom}_{\mk}^{\dual}(D(X),D(\mk))^{\omega}\simeq D^b_{\coh}(X).
\end{equation*}
In particular, we have a fully faithful strongly continuous functor
\begin{equation}\label{eq:IndCoh_to_Hom^dual}
\Ind(D^b_{\coh}(X))\simeq \IndCoh(X)\to \un{\Hom}_{\mk}^{\dual}(D(X),D(\mk)).
\end{equation}
\end{prop}  





\begin{proof}
Indeed, the category $\un{\Hom}_{\mk}^{\dual}(D(X),D(\mk))^{\omega}$ is simply the category of $\mk$-linear functors 
\begin{equation*}
\Fun_{\mk}(\Perf(X),\Perf(\mk))\simeq \Fun_{\mk}(\Perf(X)^{op},\Perf(\mk))\subset \Fun_{\mk}(\Perf(X)^{op},D(\mk))\simeq D(X).
\end{equation*} Thus, we need to show that an object $\cF\in D(X)$ is contained in $D^b_{\coh}(X)$ if and only if for any perfect complex $\cG\in\Perf(X)$ the complex of $\mk$-modules $\bR\Hom(\cG,\cF)$ is perfect. The ``only if'' direction is clear. The ``if'' direction follows from \cite[Corollary 0.5]{Nee18}.
\end{proof}

\begin{remark}
Using the results below, namely Theorems \ref{th:comparison_Hom^dual_with_ML} and \ref{th:local_invar_of_inverse_limits}, we will prove in \cite{E} that the functor \eqref{eq:IndCoh_to_Hom^dual} induces an isomorphism on continuous $K$-theory, i.e. we have
\begin{equation*}
K^{\cont}(\un{\Hom}_{\mk}^{\dual}(D(X),D(\mk)))\cong G(X) = K(D^b_{\coh}(X)).
\end{equation*}
\end{remark}

\begin{remark}
More generally, let $X$ be a separated scheme of finite type over a noetherian ring $\mk,$ of finite $\Tor$-dimension over $\mk.$ Then one can define the ``correct'' $\mk$-linear dualizable category $\cC(X/\mk),$ such that its full subcategory of compact objects is identified with the category of relatively perfect complexes. If $X$ is proper, then we have $$\cC(X/\mk)\simeq\un{\Hom}_{\mk}^{\dual}(D(X),D(\mk)).$$ If on the contrary $X$ is affine, then choosing a surjection $\mk[x_1,\dots,x_n]\to \cO(X)$ we have $$\cC(X/\mk)\simeq\un{\Hom}_{\mk[x_1,\dots,x_n]}^{\dual}(D(X),D(\mk[x_1,\dots,x_n])).$$

The assignment $U\mapsto \cC(U/\mk)$ is in fact a sheaf on $X$ (in the Zariski topology) with values in $\Cat_{\mk}^{\dual}.$ Moreover, for open subsets $V\subset U\subset X,$ the functor $\cC(U/\mk)\to \cC(V/\mk)$ is a (strongly continuous) quotient functor. The spectrum $K^{\cont}(\cC(X/\mk))$ can be considered as the ``correct'' notion of relative $G$-theory. We will provide the details in \cite{E}.
\end{remark}

The following example shows that even for very basic non-regular schemes the functor \eqref{eq:IndCoh_to_Hom^dual} is not an equivalence.

\begin{prop}
Suppose that $\mk$ is a field, and consider the algebra of dual numbers $\mk[\veps].$ Then the category $\un{\Hom}_{\mk}^{\dual}(D(\mk[\veps]),D(\mk))$ is not compactly generated.
\end{prop}

\begin{proof}
We directly construct an $\omega_1$-compact object of the category $\un{\Hom}_{\mk}^{\dual}(D(\mk[\veps]),D(\mk))$ which is not in the image of $\IndCoh(\mk[\veps]).$ 

Recall that we have $\bR\End_{\mk[\veps]}(\mk)\cong\mk[y],$ where $y$ is of cohomological degree $1.$ For $n\geq 0,$ we put $$X_n=\biggplus[k\geq 0]\mk[nk]\in D(\mk[\veps]).$$ We define the map $X_n\to X_{n+1}$ to be the direct sum of maps $y^k:\mk[nk]\to \mk[(n+1)k]$ over $k\geq 0.$ Then the object $$X=\inddlim[n]X_n\in\Ind(D(\mk[\veps]))$$ is contained in the image of $\un{\Hom}_{\mk}^{\dual}(D(\mk[\veps]),D(\mk)).$ Indeed, each map $X_n\to X_{n+1},$ considered as a class in $H_0(\bR\Hom_{\mk[\veps]}(X_n,X_{n+1})),$ is contained in the image of the map
\begin{equation*}
H_0(\THC(\mk[\veps]/\mk;X_{n+1}\tens{\mk}X_n^*))\to H_0(\bR\Hom_{\mk[\veps]}(X_n,X_{n+1})),
\end{equation*}
which can be seen by a direct computation of Hochschild cohomology.

To prove that $X$ is not in the image of $\IndCoh(\mk[\veps]),$ it suffices to observe that for any $n\geq 1$ the map $X_1\to X_n$ does not factor through an object of $D^b_{\coh}(\mk[\veps]).$ Indeed, suppose that we have a factorization
\begin{equation*}
X_1\to Y\to X_n,\quad Y\in D^b_{\coh}(\mk[\veps]).
\end{equation*}
Choose $l\geq 0$ such that $Y\in D_{\leq l}(\mk[\veps]),$ i.e. $H_k(Y)=0$ for $k\geq l+1.$ Then the composition
\begin{equation*}
\biggplus[k\geq l+1] \mk[k]\to X_1\to Y\to X_n
\end{equation*}
is zero, which contradicts the definition of the map $X_1\to X_n.$ This proves the proposition. 
\end{proof}

\begin{remark}
A similar argument shows that for a prime $p$ the category $\un{\Hom}_{\Z}^{\dual}(D(\Z/p^2),D(\Z))$ is not compactly generated.
\end{remark}

We mention an interesting example coming from topology, which is studied in \cite{KNP}. Recall from \cite[Definition A.1.5, Proposition A.1.8]{Lur17} that a topological space $X$ is locally of constant shape if the functor $p^*:\cS\simeq\Shv(\pt;\cS)\to\Shv(X;\cS)$ has a left adjoint, where $p:X\to\pt.$ In this case the same is true for sheaves of spectra.

\begin{prop}Let $X$ be a locally compact Hausdorff space, and let $\cC$ be a dualizable category. Consider the dualizable category
\begin{equation*}
\hhat{\Cosh}(X;\cC)=\un{\Hom}_{\Sp}^{\dual}(\Shv(X;\Sp),\cC).
\end{equation*}
Suppose that $X$ is locally of constant shape and second-countable (e.g. a topological manifold with boundary, which is countable at infinity). Then the functor
\begin{equation*}
\hhat{\Cosh}(X;-):\Cat_{\st}^{\dual}\to\Cat_{\st}^{\dual}	
\end{equation*}
takes short exact sequences to short exact sequences.\end{prop}

\begin{proof}
By Theorem \ref{th:internal_projectivity}, it suffices to prove that the category $\Shv(X;\Sp)$ is proper (over $\Sp$)  and $\omega_1$-compact in $\Cat_{\st}^{\dual}.$
To show properness, recall that under the self-duality of $\Shv(X;\Sp)$ the evaluation functor is given by the composition
\begin{equation*}
\Shv(X;\Sp)\otimes\Shv(X;\Sp)\simeq \Shv(X\times X;\Sp)\xto{\Delta_X^*}\Shv(X;\Sp)\xto{\Gamma_c(X;-)}\Sp.
\end{equation*}
The functor $\Delta_X^*$ is clearly strongly continuous. The dual of the functor $\Gamma_c(X;-)$ is given by $p^*:\Sp\to\Shv(X;\Sp),$ where $p:X\to\pt.$ Our assumption on $X$ implies that the functor $p^*$ has a left adjoint, which exactly means that the functor $\Gamma_c(X;-)$ is strongly continuous. This proves the properness of $\Shv(X;\Sp).$

Now the $\omega_1$-compactness of $\Shv(X;\Sp)$ means that the object $\Delta_{X,*}(\bS_X)$ is $\omega_1$-compact in $\Shv(X\times X;\Sp).$ This follows from the second-countability of $X.$
\end{proof}

\subsection{The category of nuclear modules as a dualizable internal $\Hom$}
\label{ssec:Nuc_as_Hom^dual}

Now let $R$ be a (discrete) commutative ring, and let $I=(a_1,\dots,a_m)\subset R$ be a finitely generated ideal. We will work over $R,$ or equivalently over the rigid symmetric monoidal category $D(R).$ 

We denote by $R^{\wedge}_I$ the derived $I$-completion of $R,$ considered as an $\bE_{\infty}$-ring.  We give the following definition of the category $\Nuc(R^{\wedge}_{I})$ of nuclear modules on the affine formal scheme $\Spf(R^{\wedge}_I).$

\begin{defi}\label{def:nuclear_via_internal_Hom}
We define 
\begin{equation}\label{eq:nuclear_via_internal_Hom}
\Nuc(R^{\wedge}_I)=\un{\Hom}_R^{\dual}(D_{I\hy\tors}(R),D(R)).
\end{equation}
Here $D_{I\hy\tors}(R)\subset D(R)$ is the full subcategory of objects with locally $I$-torsion homology.
\end{defi}

This category is different from the category of nuclear solid modules defined by Clausen and Scholze \cite{CS20}, see Section \ref{sec:original_Nuc} for details.

Recall that the category $D_{I\hy\tors}(R)$ is generated by the single compact object, given by the Koszul complex $\Kos(R;a_1,\dots,a_n).$ In particular, the category $D_{I\hy\tors}(R)$ is proper over $R,$ and it is also $\omega_1$-compact. Therefore, by Theorem \ref{th:internal_projectivity} the category $D_{I\hy\tors}(R)$ is internally projective in $\Cat_R^{\dual}.$ This shows that the internal $\Hom$ in \eqref{eq:nuclear_via_internal_Hom} is well-behaved. This will eventually allow us to compute the $K$-theory and more general localizing invariants of the category $\Nuc(R^{\wedge}_I)$ in Section \ref{sec:loc_invar_inverse_limits}. 

The following basic statement describes the compact objects in the category $\Nuc(R^{\wedge}_I).$

\begin{prop}
We have $\Nuc(R^{\wedge}_I)^{\omega}\simeq \Perf(R^{\wedge}_I).$
\end{prop}

\begin{proof}
By definition, the category $\Nuc(R^{\wedge}_I)^{\omega}$ is identified with the category of $R$-linear functors $\Perf_{I\hy\tors}(R)\to\Perf(R).$ Equivalently, this is the category
\begin{equation*}
\{M\in D_{I\hy\compl}(R)\mid M\tens{R}\Kos(R;a_1,\dots,a_m)\in\Perf(R)\}\simeq \Perf(R^{\wedge}_I).
\end{equation*}
This proves the proposition.
\end{proof}

\subsection{A remark on dualizable categories with a finite group action}

Let $G$ be a finite group, and consider the category $(\Cat_{\st}^{\dual})^{BG}$ of dualizable categories with a $G$-action. Taking the dualizable limit over $BG,$ we obtain the functor
\begin{equation}\label{eq:dualizable_invariants}
(-)^{\h G,\dual}:(\Cat_{\st}^{\dual})^{BG}\to \Cat_{\st}^{\dual}.
\end{equation} 

The following statement is in fact a special case of Theorem \ref{th:internal_projectivity}.

\begin{cor}\label{cor:dualizable_invariants}
The functor \eqref{eq:dualizable_invariants} takes short exact sequences to short exact sequences.
\end{cor}

We will discuss the details elsewhere. Note that for a dualizable category $\cC$ with a trivial $G$-action we have an equivalence
\begin{equation*}
\cC^{\h G,\dual}\simeq \un{\Hom}_{\Sp}^{\dual}(\Sp^{BG},\cC),
\end{equation*} 
and the category $\Sp^{BG}$ is proper over $\Sp$ and $\omega_1$-compact in $\Cat_{\st}^{\dual},$ hence internally projective.

Moreover, in Corollary \ref{cor:dualizable_invariants} we can replace $BG$ with a space ($\infty$-groupoid) $X,$ such that each based loop space of $X$ is finitely dominated (i.e. compact in $\cS$). 

\subsection{A remark about internal injectivity} In \cite{E} we will prove the following statement. For a dualizable left module $\cC$ over a rigid $\bE_1$-monoidal category $\cE$ and for any uncountable regular cardinal $\kappa$ the left $\cE$-module $\Ind(\Calk_{\kappa}^{\cont}(\cC))$ is relatively internally $\omega_1$-injective over $\Cat_{\st}^{\dual}.$ This means that the functor
\begin{equation*}
\un{\Hom}_{\cE}^{\dual}(-,\Ind(\Calk_{\kappa}^{\cont}(\cC))):\Cat_{\cE}^{\dual}\to\Cat_{\st}^{\dual}
\end{equation*}
takes short exact sequences of $\omega_1$-compact dualizable left $\cE$-modules to short exact sequences of dualizable categories. In the case when $\cE$ is compactly generated, this can be deduced from Theorem \ref{th:internal_projectivity} by ``soft'' methods. In general, this follows from the following surprising statement, which we will prove in \cite{E} using Theorem \ref{th:pullback_square_quadrofunctors} (more precisely, Corollary \ref{cor:pullback_square_products_and_colimits}).

\begin{theo}
Let $\cC$ be a dualizable category and let $\kappa$ be an uncountable regular cardinal. Then the category $\cD=\Ind(\Calk_{\kappa}^{\cont}(\cC))$ satisfies the following property: the functor $$\hat{\cY}:\cD\to\Ind(\cD^{\omega_1})$$ commutes with countable limits.
\end{theo}

\section{Rigidification of locally rigid categories}
\label{sec:rig_of_locally_rigid}

\subsection{Properties of the rigidification}

The goal of this section is to prove that rigidification of a locally rigid category behaves nicely in a certain sense under a countability condition. In particular, we prove that if the unit object is $\omega_1$-compact, then any trace-class map is a composition of two trace-class maps. A more precise statement is Theorem \ref{th:rigidification_of_locally_rigid_countable} below. 

We first explain the relation between the rigidification of locally rigid categories and the dualizable internal $\Hom$ from the previous section.

\begin{prop}\label{prop:rigidification_as_Hom^dual}
Let $\cE$ be a locally rigid symmetric monoidal category. We choose an inclusion $\cE\hto\cD,$ where $\cD$ is rigid, and $\cE$ is a smashing ideal in $\cD$ (for example, we can take $\cD=\cE_+$ as in Definition \ref{def:one_point_rigidification}). Then we have a natural symmetric monoidal equivalence
\begin{equation}
\cE^{\rig}\simeq \Hom_{\cD}^{\dual}(\cE,\cD).
\end{equation}
Here the symmetric monoidal structure on the internal $\Hom$ comes from the (coidempotent) $\bE_{\infty}$-coalgebra structure on $\cE$ as an object of $\Cat_{\cD}^{\dual}.$
\end{prop} 

\begin{proof}
Recall that the category $\cE^{\rig}$ is naturally a full subcategory of $\Ind(\cE),$ with symmetric monoidal structure induced from $\Ind(\cE).$ Now, we also have a natural inclusion
\begin{equation*}
\Hom_{\cD}^{\dual}(\cE,\cD)\subset \Ind(\Fun_{\cD}^L(\cE,\cD))\simeq \Ind(\cE),
\end{equation*} 
again with symmetric monoidal structure induced from $\Ind(\cE).$ Hence, it suffices to prove that these two full subcategories of $\Ind(\cE)$ coincide.

Using the notation from Definition \ref{def:operations_objects_tensor_products}, we observe that the composition
\begin{equation*}
\cE\simeq \cE^{\vee}\tens{\cD}\cD\xto{(-)_{\cD}^{\vee}}\cD^{\vee}\tens{\cD}\cE\simeq \cE
\end{equation*}
is simply the functor $\un{\Hom}(-,1).$ Next, the functor
$$\THC(\cE/\cD;-):\cE^{\vee}\tens{\cD}\cE\simeq\cE\to\Sp$$ is identified with the functor $\Hom(1,-):\cE\to\Sp.$ Finally, the functor
$$\ev_{\cD/\cD}:(\cE^{\vee}\tens{\cD}\cD)\otimes(\cD^{\vee}\tens{\cD}\cE)\to\cE^{\vee}\tens{\cD}\cE\simeq \cE$$
is identified with the multiplication functor $\mult:\cE\otimes\cE\to\cE.$ We conclude that the full subcategory $\un{\Hom}^{\dual}(\cE,\cD)\subset\Ind(\cE)$ is generated via colimits by the objects of the form $\inddlim[\Q_{\leq}](F:\Q_{\leq}\to \cE),$ where each map $F(a)\to F(b)$ for $a<b$ is trace-class in $\cE.$ Therefore, this subcategory is exactly the category $\cE^{\rig}.$
\end{proof}

As an application of Theorem \ref{th:internal_projectivity}, we obtain the following result in the case when the unit object is $\omega_1$-compact.

\begin{theo}\label{th:rigidification_of_locally_rigid_countable} Let $\cE$ be a locally rigid symmetric monoidal category, and suppose that the unit object $1_{\cE}$ is $\omega_1$-compact. 
\begin{enumerate}[label=(\roman*),ref=(\roman*)]
\item For any trace-class map $x\to y$ in $\cE,$ there exists a factorization $x\xto{f} z\xto{g} y,$ where $z\in\cE^{\omega_1}$ and the maps $f$ and $g$ are trace-class in $\cE.$ \label{trace_class_decomposes_in_locally_rigid}
\item We have equivalences
\begin{equation*}
\cE^{\rig}\simeq \Nuc(\Ind(\cE^{\omega_1}))\simeq \Nuc(\Ind((\cE^{\omega_1})^{op}))
\end{equation*} \label{rigidification_as_nuclear}
\item The natural functor $\cE^{\rig}\to\cE$ has a symmetric monoidal right adjoint $\Lambda_{\cE}:\cE\to\cE^{\rig}.$ The functor $\Lambda_{\cE}$ is fully faithful, commutes with $\omega_1$-filtered colimits, and the image $\Lambda_{\cE}(\cE^{\omega_1})$ generates $\cE^{\rig}$ via colimits. The composition
\begin{equation}\label{eq:composition_from_E_to_Ind_E_op}
\cE^{\omega_1}\to\cE\xto{\Lambda_{\cE}}\cE^{\rig}\simeq \Nuc(\Ind((\cE^{\omega_1})^{op}))\hto \Ind((\cE^{\omega_1})^{op})
\end{equation}
is given by $P\mapsto P^{op,\vee}.$ Here $P^{op}$ is the associated compact object of $\Ind((\cE^{\omega_1})^{op})$ and $P^{op,\vee}$ is its predual object. \label{right_adjoint_to_E^rig_to_E}
\end{enumerate}
\end{theo}

\begin{proof} Consider the inclusion $\cE\hto\cE_+=\cD,$ where $\cE_+$ is the one-point rigidification from Definition \ref{def:one_point_rigidification}. Our assumption implies that $\cE$ is proper and $\omega_1$-compact over $\cD.$ We work over $\cD,$ and we apply Lemma \ref{lem:key_lemma_on_crazy_product} to a pair of objects $M\in\cE\simeq \cD^{\vee}\tens{\cD}\cE,$ $N\in \cE\simeq\cE^{\vee}\tens{\cD}\cD.$ We obtain an isomorphism
\begin{equation}\label{eq:key_tensor_product_for_locally_rigid}
\Hom_{\cE}(1,N\otimes(-)^{\vee})\tens{\cE^{\omega_1}}\Hom_{\cE}(1,-\otimes M)\cong \Hom(1,N\otimes M).
\end{equation}
In particular, this implies \ref{trace_class_decomposes_in_locally_rigid}. Moreover, the isomorphism \eqref{eq:key_tensor_product_for_locally_rigid} for $M,N\in\cE^{\omega_1}$ implies that the conditions from Proposition \ref{prop:conditions_for_good_rigidification} are satisfied for the symmetric monoidal category $\cE^{\omega_1}.$ This implies \ref{rigidification_as_nuclear}.

Next, note that the restriction $(\Lambda_{\cE})_{\mid \cE^{\omega_1}}$ is isomorphic to the composition
\begin{equation*}
	\cE^{\omega_1}\xto{\cY}\Ind(\cE^{\omega_1})\to \Ind(\cE^{\omega_1})^{\rig}\simeq\cE^{\rig},
\end{equation*}
Here the second functor is the right adjoint to the inclusion. Applying Proposition \ref{prop:right_adjoint_description_for_rigidification}, we obtain the monoidality of $\Lambda_{\cE}$ and the desired description of the composition \eqref{eq:composition_from_E_to_Ind_E_op}. We also see that the category $\cE^{\rig}$ is generated via colimits by $\Lambda_{\cE}(\cE^{\omega_1}).$

To show that the functor $\Lambda_{\cE}$ commutes with $\omega_1$-filtered colimits, we need to check that the functor $\cE^{\rig}\simeq\Nuc(\Ind(\cE^{\omega_1}))\to 
\cE$ preserves $\omega_1$-compact objects. This is clear: the $\omega_1$-compact objects of $\Nuc(\Ind(\cE^{\omega_1}))$ are the basic nuclear objects by Proposition \ref{prop:nuclear_is_omega_1_presentable}. 

Finally, to see that the functor $\Lambda_{\cE}$ is fully faithful, we observe that the functor $\cE^{\rig}\to\cE$ has a fully faithful left adjoint. Namely, the composition of this left adjoint with the inclusion $\cE^{\rig}\to\Ind(\cE^{\omega_1})$ is simply the functor $\hat{\cY}.$
\end{proof}
	

\subsection{The category of nuclear modules as a rigidification}

We use the notation of Subsection \ref{ssec:Nuc_as_Hom^dual}, i.e. let $R$ be a discrete commutative ring, $I\subset R$ a finitely generated ideal, and $R^{\wedge}_I$ the derived $I$-completion of $R,$ considered as an $\bE_{\infty}$-$R$-algebra. We obtain the following description of the category $\Nuc(R^{\wedge}_I)$ from Definition \ref{def:nuclear_via_internal_Hom}.

\begin{prop}\label{prop:Nuc_via_rigidification} We have a natural equivalence
\begin{equation*}
\Nuc(R^{\wedge}_{I})\simeq (D_{I\hy \tors}(R))^{\rig}\simeq (D_{I\hy\compl}(R))^{\rig}.
\end{equation*}\end{prop} 

\begin{proof}Indeed, this follows directly from Proposition \ref{prop:rigidification_as_Hom^dual}, applied to the inclusion of $D_{I\hy\tors}(R)$ into $D(R)$ as a smashing ideal.\end{proof}

All the statements from Theorem \ref{th:rigidification_of_locally_rigid_countable} can be applied to the category $\Nuc(R^{\wedge}_I).$ In particular, we obtain the following description of this category.

\begin{cor}\label{cor:nuclear_via_colimits_of_trace-class}
We have equivalences
\begin{equation*}
\Nuc(R^{\wedge}_{I})\simeq \Nuc(\Ind(D_{I\hy\compl}(R)^{\omega_1}))\simeq \Nuc(\Ind(D_{I\hy\compl}(R)^{\omega_1,op})).
\end{equation*}
The functor $\Nuc(R^{\wedge}_{I})\to D_{I\hy\compl}(R)$ has a symmetric monoidal right adjoint $\Lambda_I:D_{I\hy\compl}(R)\to \Nuc(R^{\wedge}_{I}),$ which is fully faithful, commutes with $
\omega_1$-filtered colimits, and the image $\Lambda_I(D_{I\hy\compl}(R)^{\omega_1})$ generates the category $\Nuc(R^{\wedge}_I)$ via colimits.
\end{cor}

\begin{remark}
Denote by $\iota:D_{I\hy\compl}(R)\to\Ind(D_{I\hy\compl}(R)^{\omega_1,op})$ the dual of the colimit functor. Since the category $\Ind(D_{I\hy\compl}(R)^{\omega_1,op})$ is $R$-linear, the notion of $I$-complete objects is defined, as well as the functor of $I$-completion. Then the composition  $$D_{I\hy\compl}(R)\xto{\Lambda_I}\Nuc(R^{\wedge}_I)\hto \Ind(D_{I\hy\compl}(R)^{\omega_1,op})$$
is given by $M\mapsto \iota(M)^{\wedge}_I.$
\end{remark}

\section{Mittag-Leffler inverse sequences of dualizable categories}
\label{sec:Mittag-Leffler}

Let $\cC_0\leftto \cC_1\leftto\dots$ be an inverse sequence in $\Cat_{\st}^{\dual}.$ By \cite[Theorem 1.91]{E24} the inverse limit $\prolim[n]^{\dual}\cC_n$ in $\Cat_{\st}^{\dual}$ is described as follows:
\begin{equation*}
\prolim[n]^{\dual}\cC_n\simeq\ker^{\dual}(\Ind(\prolim[n]\cC_n^{\omega_1})\to\Ind(\prolim[n]\Calk_{\omega_1}^{\cont}(\cC_n))).
\end{equation*} 
To be able to compute $K$-theory or more general localizing invariants of this inverse limit in interesting situations, we want the functor
\begin{equation*}
\prolim[n]\cC_n^{\omega_1}\to \prolim[n]\Calk_{\omega_1}^{\cont}(\cC_n)
\end{equation*}
to be a homological epimorphism. The goal of this section is to obtain a sufficient condition for this.

\subsection{Mittag-Leffler condition}

\begin{defi}\label{def:strong_ML} Let $\cE$ be a rigid $\bE_1$-monoidal category. Let $\cC_0\leftto\cC_1\leftto\dots$ be an inverse sequence in $\Cat_{\cE}^{\dual}.$ We denote by $F_{mn}:\cC_m\to\cC_n$ the transition functors, $m\geq n.$ We say that $(\cC_n)$ is a strongly Mittag-Leffler sequence over $\cE$ if the following conditions hold:
	\begin{enumerate}[label=(\roman*),ref=(\roman*)]
		\item for each $n\geq 0,$ the inverse system $(F_{mn}F_{mn}^R)_{m\geq n}$ is essentially constant in $\Fun_{\cE}^L(\cC_n,\cC_n).$ \label{ML1}
		
	\item for any $n,k\geq 0,$ the functor $\Phi_{nk}=(\prolim[m\geq n,k]F_{mk}F_{mn}^R):\cC_n\to\cC_k$ is strongly continuous and has a left adjoint.\label{ML2} 
	\end{enumerate}
\end{defi}

\begin{remark}\label{rem:remarks_on_ML}
	\begin{enumerate}[label=(\roman*),ref=(\roman*)]
\item The obvious drawback of Definition \ref{def:strong_ML} is that condition \ref{ML2} depends not only on the pro-object $\proolim[n]\cC_n\in \Pro(\Cat_{\cE}^{\dual}),$ but on the actual choice of an inverse sequence. This can be fixed by introducing the so-called weakly Mittag-Leffler sequences, for which Theorem \ref{th:hom_epi_for_ML} below still holds. However, to restrict the technical complexity of the paper we choose not to do so.
\item In the formulation of condition \ref{ML1} it suffices to require that the inverse system $(F_{mn}F_{mn}^R)_{m\geq n}$ is essentially constant in the category of accessible lax $\cE$-linear endofunctors of $\cC_n.$ Therefore, this condition depends only on the pro-object $\proolim[n]\cC_n\in\Pro(\Pr^L_{\cE}).$ \label{ML1_depends_only_on_pro_object_of_Pr^L}
\end{enumerate}
\end{remark}

We observe the following basic properties of strongly Mittag-Leffler sequences.

\begin{prop}\label{prop:basic_properties_of_ML}
Let $\cE$ and $\cE'$ be rigid $\bE_1$-monoidal categories, and let $(\cC_n)_{n\geq 0}$ be an inverse sequence in $\Cat_{\cE}^{\dual}$ which is strongly Mittag-Leffler over $\cE.$
\begin{enumerate}[label=(\roman*),ref=(\roman*)]
\item The inverse sequence $(\cC_n^{\vee})_{n\geq 0}$ is strongly Mittag-Leffler over $\cE^{mop}.$ \label{dual_of_ML}
\item The sequence $(\cC_n)_{n\geq 0}$ is strongly Mittag-Leffler over $\Sp.$ \label{relative_ML_implies_absolute}
\item Let $\cD\in\Cat^{\dual}_{\cE^{mop}\otimes\cE'}$ be a dualizable $\cE'\hy\cE$-bimodule. Then the sequence $(\cD\tens{\cE}\cC_n)_{n\geq 0}$ is strongly Mittag-Leffler over $\cE'.$ 
\label{tensoring_ML_with_bimodule}
\item Let $(\cC_n')_{n\geq 0}$ be an inverse sequence in $\Cat_{\cE'}^{\dual}$ which is strongly Mittag-Leffler over $\cE'.$ The the sequence $(\cC_n\otimes\cC_n')_{n\geq 0}$ is strongly Mittag-Leffler over $\cE\otimes\cE'.$ \label{tensor_product_of_ML}
\end{enumerate}
\end{prop}

\begin{proof}
The statements \ref{relative_ML_implies_absolute}, \ref{tensoring_ML_with_bimodule} and \ref{tensor_product_of_ML} follow directly from Definition \ref{def:strong_ML}.

We prove \ref{dual_of_ML}.  Using the notation from Definition \ref{def:strong_ML}, we note that for any $n\geq 0$ the inverse sequence $(F_{mn}^{\vee,L}F_{mn}^{\vee})_{m\geq n}$ in $\Fun_{\cE^{mop}}^L(\cC_n^{\vee},\cC_n^{\vee})$ is identified with the inverse sequence $(F_{mn}F_{mn}^R)_{m\geq n}$ under the equivalence
$$\Fun_{\cE^{mop}}^L(\cC_n^{\vee},\cC_n^{\vee})\simeq \cC_n^{\vee}\tens{\cE}\cC_n\simeq \Fun^L_{\cE}(\cC_n,\cC_n).$$ Hence, the sequence $(\cC_n^{\vee})$ satisfies the condition \ref{ML1}. Similarly, we observe that for $n,k\geq 0$ the functor $(\prolim[m\geq n,k]F_{mk}^{\vee,L}F_{mn}^{\vee}):\cC_n^{\vee}\to\cC_k^{\vee}$ is dual to the functor $(\prolim[m\geq n,k]F_{mn}F_{mk}^R):\cC_k\to\cC_n.$ This proves \ref{dual_of_ML}.
\end{proof}

The important class of examples of strongly Mittag-Leffler sequences comes from formal schemes, see Proposition \ref{prop:Nuc_as_limit}.

\subsection{The theorem about a homological epimorphism}

In this subsection we consider only strongly Mittag-Leffler sequences over $\Sp.$ The goal of this section is to prove the following result.

\begin{theo}\label{th:hom_epi_for_ML}
	Let $(\cC_n)_{n\geq 0}$ be a strongly Mittag-Leffler inverse sequence in $\Cat_{\st}^{\dual},$ and let $\kappa$ be an uncountable regular cardinal. Then the functor $\prolim[n]\cC_n^{\kappa}\to \prolim[n]\Calk^{\cont}_{\kappa}(\cC_n)$ is a homological epimorphism. Equivalently, we have a short exact sequence
	$$0\to \prolim[n]^{\dual}\cC_n\to\Ind(\prolim[n]\cC_n^{\kappa})\to\Ind(\prolim[n]\Calk^{\cont}_{\kappa}(\cC_n))\to 0.$$
\end{theo}

Throughout this section, we keep the notation from Definition \ref{def:strong_ML}. In particular, $F_{nk}:\cC_n\to\cC_k$ are the transition functors, and $\Phi_{nk}:\cC_n\to\cC_k$ are the functors from the condition \ref{ML2}.

We first analyze the usual limit of $\cC_n$ (taken in $\Pr^L_{\st}$).

\begin{prop}\label{prop:naive_limit_dualizable} Let $(\cC_n)_{n\geq 0}$ be a strongly Mittag-Leffler sequence. 

\begin{enumerate}[label=(\roman*),ref=(\roman*)]
\item For any $k\geq 0$ the projection functor $\pi_k:\prolim[n]\cC_n\to \cC_k$ is strongly continuous and it has a left adjoint.\label{left_right_adj_to_pi_k}

\item The category $\prolim[n]\cC_n$ is dualizable. \label{naive_limit_dualizable}

\item An object $x\in \prolim[n]\cC_n$ is $\kappa$-compact if and only if each $\pi_k(x)$ is $\kappa$-compact in $\cC_k.$ In other words, we have
$$(\prolim[n]\cC_n)^{\kappa} = \prolim[n]\cC_n^{\kappa}.$$ \label{kappa_compact_in_limit} 
\end{enumerate}
\end{prop}

\begin{proof}We first prove \ref{left_right_adj_to_pi_k}. 
It follows from the condition \ref{ML1} of Definition \ref{def:strong_ML} that the natural transformations $F_{n+1,n}\Phi_{k,n+1}\to\Phi_{kn}$ are isomorphisms. Hence, for each $k\geq 0,$ the sequence of functors $(\Phi_{kn})_{n\geq 0}$ defines a continuous functor $\cC_k\to \prolim[n]\cC_n.$ We claim that this functor is the right adjoint to $\pi_k.$ Indeed, for $x\in\prolim[n]\cC_n,$ $y\in\cC_k,$ we have
\begin{multline*}\Hom(x,(\Phi_{kn}(y))_n)\cong \prolim[n]\Hom(\pi_n(x),\Phi_{kn}(y))\cong \prolim[m\geq n\geq k]\Hom(\pi_n(x),F_{mn}F_{mk}^R(y))\\
	\cong \prolim[m=n\geq k]\Hom(\pi_n(x),F_{nk}^R(y))\cong\prolim[n\geq k]\Hom(F_{nk}(\pi_n(x)),y)\cong \Hom(\pi_k(x),y).
\end{multline*}
This proves the strong continuity of $\pi_k.$

Next, we have natural isomorphisms $\Phi_{nk}\cong \Phi_{n+1,k}F_{n+1,n}^R.$ Passing to the left adjoints, we obtain the isomorphisms $F_{n+1,n}\Phi_{n+1,k}^L\cong \Phi_{nk}^L.$ Thus, for each $k\geq 0,$ the sequence $(\Phi_{nk}^L)_{n\geq 0}$ defines a functor $\cC_k\to \prolim[n]\cC_n.$ We claim that this functor is the left adjoint to $\pi_k.$ Indeed, for $x\in\cC_k,$ $y\in\prolim[n]\cC_n,$ we have
\begin{multline*}
\Hom((\Phi_{nk}^L(x))_n,y)\cong\prolim[n]\Hom(\Phi_{nk}^L(x),\pi_n(y))\cong\prolim[n]\Hom(x,\Phi_{nk}(\pi_n(y)))\\
\cong\prolim[m\geq n\geq k]\Hom(x,F_{mk}(F_{mn}^R(\pi_n(y)))) \cong \prolim[m=n\geq k]\Hom(x,F_{nk}(\pi_n(y)))\cong\Hom(x,y). 
\end{multline*}
Therefore, the functor $\pi_n$ has a left adjoint, as required. This proves \ref{left_right_adj_to_pi_k}.

Next, we prove \ref{naive_limit_dualizable}. Denote by $\pi_k^L:\cC_k\to\prolim[n]\cC_n$ the left adjoint to $\pi_k.$ Then the functors $\pi_k^L$ are strongly continuous, and their images generate the target (since their right adjoints form a conservative family). Since the categories $\cC_n$ are dualizable, we conclude that the category $\prolim[n]\cC_n$ is also dualizable by \cite[Proposition 1.55]{E24}.

Finally, we prove \ref{kappa_compact_in_limit}. Let $x$ be a $\kappa$-compact object in $\prolim[n]\cC_n.$ Since the functors $\pi_k$ are strongly continuous, each object $\pi_k(x)$ is $\kappa$-compact in $\cC_k.$

Conversely, suppose that $\pi_k(x)\in\cC_k^{\kappa}$ for all $k\geq 0.$ Then we have $\pi_k^L(\pi_k(x))\in (\prolim[n]\cC_n)^{\kappa}.$ Therefore,
\begin{equation*}x\cong \indlim[k]\pi_k^L(\pi_k(x))\in (\prolim[n]\cC_n)^{\kappa}.\qedhere\end{equation*}\end{proof}

\begin{prop}\label{prop:dual_pointwise_ML}
Let $(\cC_n)_{n\geq 0}$ be a strongly Mittag-Leffler sequence. Then for any $x\in\prolim[n]\cC_n$ and for any $k\geq 0$ we have an isomorphism
$$F_{k+1,k}^{\vee,L}(\pi_{k+1}(x)^{\vee})\xto{\sim}\pi_k(x)^{\vee}.$$
\end{prop}

\begin{proof}We put $\cD=\prolim[n]\cC_n.$ By Proposition \ref{prop:naive_limit_dualizable}, the category $\cD$ is dualizable and the functors $\pi_n:\cD\to\cC_n$ are strongly continuous. Consider the composition
$$\Psi:\cD\xto{(-)^{\vee}}\cD^{\vee,op}\xto{((\pi_n^{\vee,L})^{op})_n}\prolim[n]\cC_n^{\vee,op}.$$ Applying Proposition \ref{prop:naive_limit_dualizable} to the inverse sequence $(\cC_n^{\vee}),$ we see that for each $k\geq 0$ the projection functor $\pi_k^\prime:\prolim[n]\cC_n^{\vee}\to\cC_k^{\vee}$ commutes with limits, hence the functor $\pi_k^{\prime op}:\prolim[n]\cC_n^{\vee,op}\to\cC_k^{\vee,op}$ commutes with colimits. Since the category $\cD$ is dualizable, by Remark \ref{rem:F^cont_for_non_presentable_target} we have the isomorphisms $$\pi_k^{\prime op}\circ \Psi^{\cont}\xto{\sim} (\pi_k^{\prime op}\circ \Psi)^{\cont}\cong ((\pi_k^{\vee,L})^{op}\circ (-)^{\vee})^{\cont}\cong (-)^{\vee}\circ \pi_k$$
of functors $\cD\to\cC_k^{\vee,op}.$ Therefore, for any $x\in\cD,$ we have the isomorphisms
$$F_{k+1,k}^{\vee,L}(\pi_{k+1}(x)^{\vee})\cong F_{k+1,k}^{\vee,L}(\pi_{k+1}^\prime(\Psi^{\cont}(x)))\cong \pi_k^\prime(\Psi^{\cont}(x))\cong \pi_k(x)^{\vee}.$$
This proves the proposition.\end{proof}

The following lemma is crucial for proving Theorem \ref{th:hom_epi_for_ML}.

\begin{lemma}\label{lem:key_lemma_ML} Let $(\cC_n)$ be a strongly Mittag-Leffler sequence. As above, we denote by $\pi_n:\prolim[k]\cC_k\to\cC_n,$ $\pi_n':\prolim[k]\cC_k^{\vee}\to\cC_n^{\vee}$ the projection functors. Take any pair of objects $x\in\prolim[k]\cC_k^{\vee},$ $y\in\prolim[k]\cC_k.$ Then the natural map
\begin{equation}\label{eq:tricky_tensor_product_for_ML}(\prolim[n]\ev_{\cC_n}(\pi_n(y),\pi_n(-)^{\vee}))\tens{\prolim[n]\cC_k^{\kappa}} (\prolim[n]\ev_{\cC_n}(\pi_n(-),\pi_n^\prime(x)))\to\prolim[n]\ev_{\cC_n}(\pi_n(y),\pi_n^\prime(x))\end{equation}
is an isomorphism.\end{lemma}

\begin{proof}Recall from Propositions \ref{prop:basic_properties_of_ML} and \ref{prop:naive_limit_dualizable} that both categories $\prolim[n]\cC_n^{\vee}$ and $\prolim[n]\cC_n$ are dualizable, and we have
$$(\prolim[n]\cC_n^{\vee})^{\kappa}\simeq \prolim[n](\cC_n^{\vee})^{\kappa},\quad (\prolim[n]\cC_n)^{\kappa}\simeq \prolim[n]\cC_n^{\kappa}$$
(recall that $\kappa$ is uncountable). Let $$\hat{\cY}(x)=\inddlim[i\in I] x_i,\quad \hat{\cY}(y)=\inddlim[j\in J] y_j,$$
where $I$ and $J$ are directed posets and $x_i\in\prolim[n](\cC_n^{\vee})^{\kappa},$ $y_j\in\prolim[n]\cC_n^{\kappa}.$ 

We have
\begin{multline}\label{eq:one_of_the_factors_as_colim}
\prolim[k]\ev_{\cC_k}(\pi_k(y),\pi_k(-)^{\vee})\cong\prolim[k]\indlim[j]\ev_{\cC_k}(\pi_k(y_j),\pi_k(-)^{\vee})\\
\cong \prolim[k]\indlim[j]\Hom_{\cC_k}(\pi_k(-),\pi_k(y_j))\cong \prolim[k]\indlim[j]\Hom(-,\pi_k^R(\pi_k(y_j)))\\
\cong\indlim[\psi:\N\to J]\prolim[k\leq l]\Hom(-,\pi_k^R(\pi_k(y_{\psi(l)})))\cong \indlim[\psi:\N\to J]\Hom(-,\prolim[k\leq l]\pi_k^R(\pi_k(y_{\psi(l)})).
\end{multline}
Here the second isomorphism uses the strong continuity of $\pi_k$ (by Proposition \ref{prop:naive_limit_dualizable}). The fourth isomorphism uses Proposition \ref{prop:seq_limits_of_filtered_colimits}. 

Using \eqref{eq:one_of_the_factors_as_colim}, we obtain the following expression for the source of \eqref{eq:tricky_tensor_product_for_ML}.

\begin{multline}\label{eq:expr_for_tricky_tensor_product_ML}
(\prolim[n]\ev_{\cC_n}(\pi_n(y),\pi_n(-)^{\vee}))\tens{\prolim[k]\cC_k^{\kappa}} (\prolim[n]\ev_{\cC_n}(\pi_n(-),\pi_n^\prime(x)))\\
\cong \indlim[\psi:\N\to J] \prolim[n]\ev_{\cC_n}(\pi_n(\prolim[k\leq l]\pi_k^R(\pi_k(y_{\psi(l)}))),\pi_n^\prime(x))\\
\cong \indlim[\psi:\N\to J] \prolim[n]\ev_{\cC_n}(\prolim[k\leq l]\pi_n(\pi_k^R(\pi_k(y_{\psi(l)}))),\pi_n^\prime(x))\\
\cong \indlim[\psi:\N\to J] \prolim[n]\ev_{\cC_n}(\prolim[k\leq l]\Phi_{kn}(\pi_k(y_{\psi(l)})),\pi_n^\prime(x))\\
\cong \indlim[\psi:\N\to J] \prolim[n]\indlim[i]\ev_{\cC_n}(\prolim[k\leq l]\Phi_{kn}(\pi_k(y_{\psi(l)})),\pi_n^\prime(x_i))\\
\cong \indlim[\psi:\N\to J] \prolim[n]\indlim[i]\prolim[k\leq l] \ev_{\cC_n}(\Phi_{kn}(\pi_k(y_{\psi(l)})),\pi_n^\prime(x_i))\\
\cong \indlim[\varphi:\N\to I]\indlim[\psi:\N\to J]\prolim[n\leq m]\prolim[k\leq l] \ev_{\cC_n}(\Phi_{kn}(\pi_k(y_{\psi(l)})),\pi_n^\prime(x_{\varphi(m)}))
\end{multline}
Here the first isomorphism uses the fact that the functor
\begin{equation*}
\prolim[n]\ev_{\cC_n}(\pi_n(-),\pi_n'(x)):\prolim[k]\cC_k\to\Sp
\end{equation*}
commutes with $\omega_1$-filtered colimits, hence it is left Kan-extended from $\prolim[k]\cC_k^{\kappa}.$ The second isomorphism uses the fact that $\pi_n$ commutes with limits (since it has a left adjoint by Proposition \ref{prop:naive_limit_dualizable}). The fourth isomorphism uses the description of the functors $\pi_n^R$ from the proof of Proposition \ref{prop:naive_limit_dualizable}. The fifth isomorphism uses the strong continuity of $\pi_n',$ by Propositions \ref{prop:basic_properties_of_ML} and \ref{prop:naive_limit_dualizable}.

Now, consider the functor
\begin{equation}\label{eq:quadrofunctor_for_ML}
F:\N^{op}\times I\times\N^{op}\times J\to\Sp,\quad F(n,i,k,j)=\ev_{\cC_n}(\Phi_{kn}(\pi_k(y_j)),\pi_n^\prime(x_i)).\end{equation}
To conclude the proof of the lemma, it suffices to prove the following statement.

{\noindent{\bf Claim.}} {\it The functor $F$ defined in \eqref{eq:quadrofunctor_for_ML} satisfies the following conditions:
\begin{enumerate}[label=(\roman*),ref=(\roman*)]
\item For each $k\geq 0,$ we have an isomorphism $$\indlim[j]\prolim[n]\indlim[i]F(n,k,i,j)\xto{\sim} \prolim[n]\indlim[i]\indlim[j]F(n,k,i,j).$$ \label{assump1}
\item For each $n\geq 0,$ we have an isomorphism $$\indlim[i]\prolim[k]\indlim[j]F(n,k,i,j)\xto{\sim} \prolim[k]\indlim[i]\indlim[j]F(n,k,i,j).$$ \label{assump2}
\end{enumerate}
}

Assuming the claim, we can apply Proposition \ref{prop:quadrofunctors} to the quadrofunctor \eqref{eq:quadrofunctor_for_ML}. Using \eqref{eq:expr_for_tricky_tensor_product_ML}, we obtain
\begin{multline*}
(\prolim[n]\ev_{\cC_n}(\pi_n(y),\pi_n(-)^{\vee}))\tens{\prolim[k]\cC_k^{\kappa}} (\prolim[n]\ev_{\cC_n}(\pi_n(-),\pi_n^\prime(x)))\\
\cong \indlim[\varphi:\N\to I]\indlim[\psi:\N\to J]\prolim[n\leq m]\prolim[k\leq l] \ev_{\cC_n}(\Phi_{kn}(\pi_k(y_{\psi(l)})),\pi_n^\prime(x_{\varphi(m)}))\\
\cong\prolim[n]\indlim[i]\indlim[j]\ev_{\cC_n}(\Phi_{nn}(\pi_n(y_j)),\pi_n^\prime(x_i))\cong \prolim[n]\ev_{\cC_n}(\Phi_{nn}(\pi_n(y)),\pi_n^\prime(x))\\
\cong \prolim[n]\prolim[m\geq n]\ev_{\cC_n}(F_{mn}(F_{mn}^R(\pi_n(y))),\pi_n^\prime(x))\cong\prolim[m=n]\ev_{\cC_n}(\pi_n(y),\pi_n^\prime(x)),
\end{multline*}
as required.

\begin{proof}[Proof of Claim.]
Note that for $z\in\prolim[n]\cC_n,$ $w\in\prolim[n]\cC_n^{\vee},$ we have a natural isomorphism
$$\ev_{\cC_n}(\Phi_{kn}(\pi_k(z)),\pi_n^\prime(w))\cong \ev_{\cC_k}(\pi_k(z),\Phi_{kn}^{\vee}(\pi_n^\prime(w))).$$
Hence, \ref{assump2} is obtained from \ref{assump1} by interchanging the roles of $(\cC_n)$ and $(\cC_n^{\vee}),$ $x$ and $y,$ $I$ and $J,$ $x_i$ and $y_j.$

We now prove \ref{assump1}. For any $k\geq 0,$ consider the following bifunctor:
$$G_k:(\prolim[n]\cC_n^{\vee})\times \cC_k\to\Sp,\quad G_k(z,w)=\prolim[n]\ev_{\cC_n}(\Phi_{kn}(w),\pi_n'(z)).$$ We claim that the bifunctor $G_k$ commutes with colimits in each argument. Indeed, since $\Phi_{kn}$ is the limit of the essentially constant system $(F_{mn}F_{mk}^R)_{m\geq n,k},$ we have the following isomorphisms:
\begin{multline*}
\prolim[n]\ev_{\cC_n}(\Phi_{kn}(w),\pi_n^\prime(z))\cong\prolim[m\geq n\geq k]\ev_{\cC_n}(F_{mn}(F_{mk}^R(w)),\pi_n^\prime(z))\\
\cong\prolim[m=n\geq k]\ev_{\cC_n}(F_{nk}^R(w),\pi_n^\prime(z))\cong \prolim[n\geq k]\ev_{\cC_k}(w,F_{nk}^{\vee,L}(\pi_n^\prime(z)))\cong\ev_{\cC_k}(w,\pi_k^\prime(z)).
\end{multline*} 
The latter bifunctor is bicontinuous since the functor $\pi_k^\prime$ is continuous.

Therefore, for $k\geq 0$ we obtain \begin{multline*}\indlim[j]\prolim[n]\indlim[i]\ev_{\cC_n}(\Phi_{kn}(\pi_k(y_j),\pi_n^\prime(x_i))\cong  \indlim[j]G_k(x,\pi_k(y_j))\cong G_k(x,\pi_k(y))\\
\cong \prolim[n]\indlim[i]\indlim[j]\ev_{\cC_n}(\Phi_{kn}(\pi_k(y_j)),\pi_n^\prime(x_i)).
\end{multline*}
This proves \ref{assump1} and the claim.
\end{proof}
This proves the lemma.
\end{proof}

\begin{cor}\label{cor:hom_epi_onto_the_image_ML}
Let $(\cC_n)_{n\geq 0}$ be a strongly Mittag-Leffler sequence.
\begin{enumerate}[label=(\roman*),ref=(\roman*)]
\item The functor $\prolim[n]\cC_n^{\kappa}\to\prolim[n]\Calk_{\kappa}^{\cont}(\cC_n)$ is a homological epimorphism onto its image. \label{hom_epi_onto_image_ML}

\item Consider the natural fully faithful functor
\begin{equation}\label{eq:incl_of_dual_of_lim_dual}
\Psi:(\prolim[n]^{\dual}\cC_n)^{\vee}\to \Ind((\prolim[n]\cC_n^{\kappa})^{op})\simeq \Fun(\prolim[n]\cC_n^{\kappa},\Sp).
\end{equation}
Its essential image contains all functors of the form $\prolim[n]\ev_{\cC_n}(\pi_n(-),\pi_n^\prime(x)),$ where $x\in\prolim[n]\cC_n^{\vee}.$ Moreover, it is generated by such functors with $x\in\prolim[n](\cC_n^{\vee})^{\kappa}.$ \label{dual_of_dualizable_limit}
\end{enumerate}
\end{cor}

\begin{proof}
We prove \ref{hom_epi_onto_image_ML}. Given $x,y\in\prolim[n]\cC_n^{\kappa},$ we have
\begin{equation*}
\Fiber(\Hom_{\prolim[n]\cC_n^{\kappa}}(x,y)\to \Hom_{\prolim[n]\Calk_{\kappa}^{\cont}(\cC_n)}(x,y))\cong \prolim[n]\ev_{\cC_n}(\pi_n(y),\pi_n(x)^{\vee}).
\end{equation*}
Hence, we need to show that
\begin{equation}
\label{eq:idemp_quasi_ideal_ML}(\prolim[n]\ev_{\cC_n}(\pi_n(y),\pi_n(-)^{\vee}))\tens{\prolim[k]\cC_k^{\kappa}}(\prolim[n]\ev_{\cC_n}(\pi_n(-),\pi_n(x)^{\vee}))\to \prolim[n]\ev_{\cC_n}(\pi_n(y),\pi_n(x)^{\vee})
\end{equation}
is an isomorphism. But the sequence $(\pi_n(x)^{\vee})_{n\geq 0}$ is in fact an object of $\prolim[n]\cC_n^{\vee}$ by Proposition \ref{prop:dual_pointwise_ML}. Hence, by Lemma \ref{lem:key_lemma_ML} the map \eqref{eq:idemp_quasi_ideal_ML} is an isomorphism. This proves \ref{hom_epi_onto_image_ML}.

It follows from the proof of \ref{hom_epi_onto_image_ML} that the essential image of the functor $\Psi$ from \eqref{eq:incl_of_dual_of_lim_dual} consists of (exact) functors $G:\prolim[n]\cC_n^{\kappa}\to\Sp,$ such that for any $y\in\prolim[n]\cC_n^{\kappa}$ we have an isomorphism
\begin{equation*}
(\prolim[n]\ev_{\cC_n}(\pi_n(y),\pi_n(-)^{\vee}))\tens{\prolim[k]\cC_k^{\kappa}}G\xto{\sim}G(y).
\end{equation*} 
Thus, by Lemma \ref{lem:key_lemma_ML} the functors of the form $\prolim[n]\ev_{\cC_n}(\pi_n(-),\pi_n(x))$ are contained in the essential image of $\Psi,$ where $x\in\prolim[n]\cC_n^{\vee}.$ Moreover, by the proof of \ref{hom_epi_onto_image_ML}, this essential image is generated by the objects of the form $\prolim[n]\ev_{\cC_n}(\pi_n(-),\pi_n(z)^{\vee}),$ where $z\in \prolim[n]\cC_n^{\kappa}.$ For such $z$ we can choose a $\kappa$-filtered system $(x_i)_i$ in $\prolim[n](\cC_n^{\vee})^{\kappa}\simeq (\prolim[n]\cC_n^{\vee})^{\kappa}$ such that $(\pi_n(z)^{\vee})_{n\geq 0}\cong \indlim[i]x_i.$ Since the bifunctor $\prolim[n]\ev_{\cC_n}(\pi_n(-),\pi_n^\prime(-))$ commutes with $\omega_1$-filtered colimits in each argument, we have
\begin{equation*}
\prolim[n]\ev_{\cC_n}(\pi_n(-),\pi_n(z)^{\vee})\cong \indlim[i]\prolim[n]\ev_{\cC_n}(\pi_n(-),\pi_n^\prime(x_i)).
\end{equation*}
This shows that the essential image of $\Psi$ is generated by the objects of the form $\prolim[n]\ev_{\cC_n}(\pi_n(-),\pi_n^\prime(x)),$ where $x\in\prolim[n](\cC_n^{\vee})^{\kappa}.$ This proves the corollary.
\end{proof}

To prove the essential surjectivity of the functor $\prolim[n]\cC_n^{\kappa}\to\prolim[n]\Calk_{\kappa}^{\cont}(\cC_n)$ (up to retracts) we will need the following notion.

\begin{defi}\label{def:special_map_of_ML} Let $(T_n)_{n\geq 0}:(\cC_n)_{n\geq 0}\to(\cD_n)_{n\geq 0}$ be a morphism of strongly Mittag-Leffler sequences in $\Cat_{\st}^{\dual}.$ We denote by $F_{nk}:\cC_n\to\cC_k,$ $G_{nk}:\cD_n\to\cD_k$ the transition functors for $n\geq k.$ We use the notation \begin{equation*}
\Phi_{nk}=(\prolim[m\geq n,k]F_{mk}F_{mn}^R):\cC_n\to\cC_k,\quad \Psi_{nk}=(\prolim[m\geq n,k]G_{mk}G_{mn}^R):\cD_n\to\cD_k,\quad n,k\geq 0.	
\end{equation*} We say that the morphism $(T_n)_{n\geq 0}$ is special if for $n,k\geq 0$ the natural maps of functors
\begin{equation*}
T_k\Phi_{nk}\to\Psi_{nk}T_n,\quad \Phi_{nk}T_n^R\to T_k^R\Psi_{nk}	
\end{equation*}
are isomorphisms.\end{defi}

\begin{prop}\label{prop:basic_descr_of_special_maps} We keep the notation from Definition \ref{def:special_map_of_ML}.
\begin{enumerate}[label=(\roman*),ref=(\roman*)]
\item If the morphism $(T_n):(\cC_n)\to (\cD_n)$ is special, then so is the morphism $(T_n^{\vee,L}):(\cC_n^{\vee})\to (\cD_n^{\vee}).$ \label{speciality_is_self_dual}

\item More precisely, the morphism $(T_n):(\cC_n)\to (\cD_n)$ is special if and only if the following squares commute for each $k\geq 0:$
\begin{equation}
\label{eq:right_adj_naive1}
\begin{CD}
\prolim[n]\cD_n @>{(\prolim[n]T_n)^R}>> \prolim[n]\cC_n\\
@VVV @VVV\\
\cD_k @>{T_k^R}>> \cC_k,
\end{CD}
\end{equation}
\begin{equation}
	\label{eq:right_adj_naive2}
	\begin{CD}
		\prolim[n]\cD_n^{\vee} @>{(\prolim[n]T_n^{\vee,L})^R}>> \prolim[n]\cC_n^{\vee}\\
		@VVV @VVV\\
		\cD_k^{\vee} @>{T_k^{\vee}}>> \cC_k^{\vee}.
	\end{CD}
\end{equation}
\label{right_adjoints_termwise}
\end{enumerate}
\end{prop}

\begin{proof}To prove \ref{speciality_is_self_dual}, we observe that for $n,k\geq 0$ the morphism $T_k^{\vee,L}\Phi_{kn}^{\vee}\to \Psi_{kn}^{\vee}T_n^{\vee,L}$ corresponds to the morphism $\Phi_{kn}T_k^R\to T_n^R\Psi_{kn}$ under the equivalence $\Fun^L(\cC_n^{\vee},\cD_k^{\vee})\simeq \Fun^L(\cD_k,\cC_n).$ Similarly, the morphism $\Phi_{kn}^{\vee}T_n^{\vee}\to T_k^{\vee}\Psi_{kn}^{\vee}$ corresponds to the morphism $T_n\Phi_{kn}\to\Psi_{kn}T_k$ under the equivalence $\Fun^L(\cD_n^{\vee},\cC_k^{\vee})\simeq \Fun^L(\cC_k,\cD_n).$ Hence, if $(T_n)$ is a special morphism, then so is $(T_n^{\vee,L}).$

To prove \ref{right_adjoints_termwise}, we observe by passing to the left adjoints in \eqref{eq:right_adj_naive1} that the commutativity of \eqref{eq:right_adj_naive1} is equivalent to the condition that for any $n\geq 0$ the morphism $\Psi_{nk}^L T_k\to T_n\Phi_{nk}^L$ is an isomorphism. Equivalently (by passing to the right adjoints), this means that the morphism $\Phi_{nk}T_n^R\to T_k^R\Psi_{nk}$ is an isomorphism. Similarly, the commutativity of \eqref{eq:right_adj_naive2} is equivalent to the condition that for any $n\geq 0$ the morphism $T_n\Phi_{kn}\to\Psi_{kn}T_k$ is an isomorphism. This proves \ref{right_adjoints_termwise}.\end{proof}

\begin{prop}\label{prop:limit_of_epi_is_epi} We keep the notation from Definition \ref{def:special_map_of_ML}. Suppose that the morphism $(T_n)$ is special. If each functor $T_n:\cC_n\to\cD_n$ is a quotient functor, then the functor $T=(\prolim[n] T_n):\prolim[n]\cC_n\to\prolim[n]\cD_n$ is also a quotient functor.\end{prop}

\begin{proof}
	Indeed, the commutativity of the square \eqref{eq:right_adj_naive1} from Proposition \ref{prop:basic_descr_of_special_maps} implies that the adjunction counit $T\circ T^R\to\id$ is an isomorphism.
\end{proof}

\begin{cor}\label{cor:comm_square_for_duals_of_limits}
We keep the notation of Definition \ref{def:special_map_of_ML}. Suppose that the morphism $(T_n)$ is special.

\begin{enumerate}[label=(\roman*),ref=(\roman*)]
\item The following square commutes:
\begin{equation}
	\label{eq:key_comm_square_ML}
\begin{CD}
(\prolim[n]^{\dual}\cD_n)^{\vee} @>>> (\prolim[n]^{\dual}\cC_n)^{\vee}\\
@VVV @VVV\\
\Fun(\prolim[n]\cD_n^{\kappa},\Sp) @>>> \Fun(\prolim[n]\cC_n^{\kappa},\Sp)
\end{CD}
\end{equation} \label{comm_sqaure_for_duals_of_limits}

\item If each functor $T_n:\cC_n\to\cD_n$ is a quotient functor, then the functor $\prolim[n]^{\dual}\cC_n\to\prolim[n]^{\dual}\cD_n$ is also a quotient functor. \label{dualizable_limit_or_epi_is_epi}
\end{enumerate}
\end{cor}

\begin{proof}
We put $T=(\prolim[n]T_n):\prolim[n]\cC_n\to\prolim[n]\cD_n,$ and $T^\prime=(\prolim[n]T_n^{\vee,L}):\prolim[n]\cC_n^{\vee}\to\prolim[n]\cD_n^{\vee}.$

We prove \ref{comm_sqaure_for_duals_of_limits}. By Corollary \ref{cor:hom_epi_onto_the_image_ML}, it suffices to prove that for any $x\in\prolim[n](\cD_n^{\vee})^{\kappa}$ the object
\begin{equation*}
\prolim[n]\ev_{\cD_n}(\pi_n(T(-)),\pi_n^\prime(x))\in \Fun(\prolim[n]\cC_n^{\kappa},\Sp)
\end{equation*}
is in the essential image of the right vertical arrow of \eqref{eq:key_comm_square_ML}. But this follows immediately from Proposition \ref{prop:basic_descr_of_special_maps} and Corollary \ref{cor:hom_epi_onto_the_image_ML}: we have
\begin{multline*}
\prolim[n]\ev_{\cD_n}(\pi_n(T(-)),\pi_n^\prime(x))\cong \prolim[n]\ev_{\cD_n}(T_n(\pi_n(-)),\pi_n^\prime(x))\\ \cong\prolim[n]\ev_{\cC_n}(\pi_n(-),T_n^{\vee}(\pi_n^\prime(-)))\cong\prolim[n]\ev_{\cC_n}(\pi_n(-),\pi_n^\prime(T^{\prime R}(x))).
\end{multline*}
Here the last isomorphism follows from Proposition \ref{prop:basic_descr_of_special_maps}. Therefore, by Corollary \ref{cor:hom_epi_onto_the_image_ML} the right hand side of this isomorphism is in the essential image of the right vertical arrow of \eqref{eq:key_comm_square_ML}. This proves \ref{comm_sqaure_for_duals_of_limits}.

Now, if each $T_n$ is a quotient functor, then by Proposition \ref{prop:limit_of_epi_is_epi} the lower horizontal arrow in \eqref{eq:key_comm_square_ML} is fully faithful. Since both vertical arrows are also fully faithful, it follows that the upper horizontal arrow is fully faithful. This exactly means that the functor $\prolim[n]^{\dual}\cC_n\to\prolim[n]^{\dual}\cD_n$ is a quotient functor. This proves \ref{dualizable_limit_or_epi_is_epi}.
\end{proof}

\begin{prop}\label{prop:cokernel_of_special_mono}
Again, we keep the notation of Definition \ref{def:special_map_of_ML} and suppose that the morphism $(T_n)$ is special. Suppose that each functor $T_n:\cC_n\to\cD_n$ is fully faithful. Then the sequence $(\cD_n/\cC_n)_{n\geq 0}$ is strongly Mittag-Leffler and the morphism $(\cD_n)\to (\cD_n/\cC_n)$ is special.
\end{prop}

\begin{proof}
Denote by $Q_n:\cD_n\to\cD_n/\cC_n$ the quotient functors, $n\geq 0.$ Also, denote by $H_{nk}:\cD_n/\cC_n\to\cD_k/\cC_k$ the transition functors, $n\geq k.$  
Then for any $n\geq k$ we have an isomorphism $Q_k G_{nk}Q_n^R\cong H_{nk}.$ Hence for any $n,k\geq 0$ and $m\geq \max(n,k),$ we have
\begin{equation*}
H_{mk}H_{mn}^R\cong Q_k G_{mk} Q_m^R H_{mn}^R\cong Q_k G_{mk} G_{mn}^R Q_n^R.
\end{equation*}
It follows that for any $n,k\geq 0$ the sequence $(H_{mk}H_{mn}^R)_{m\geq n,k}$ is the image of the sequence $(G_{mk}G_{mn}^R)_{m\geq n,k}$ under the functor
\begin{equation*}
\Fun^L(\cD_n,\cD_k)\to \Fun^L(\cD_n/\cC_n,\cD_k/\cC_k),\quad F\mapsto Q_k F Q_n^R.
\end{equation*}
Therefore, the sequence $(H_{mk}H_{mn}^R)_{m\geq n,k}$ is essentially constant in $\Fun^L(\cD_n/\cC_n,\cD_k/\cC_k).$ We put $\Theta_{nk}=\prolim[m\geq n,k]H_{mk}H_{mn}^R.$

It remains to prove that the maps
\begin{equation*}
\alpha_{nk}:Q_k \Psi_{nk}\to \Theta_{nk} Q_n,\quad \beta_{nk}:\Psi_{nk}Q_n^R\to Q_k^R \Theta_{nk}
\end{equation*}
are isomorphisms. Indeed, the isomorphism $\alpha_{nk}$ would imply that the functor $\Theta_{nk}$ is strongly continuous, since $Q_n$ is a strongly continuous quotient functor. By duality, the isomorphism $\beta_{nk}$ would imply that the functor $\Theta_{nk}^{\vee}$ is strongly continuous, which equivalently means that $\Theta_{nk}$ has a left adjoint.

To see that $\alpha_{nk}$ is an isomorphism, note that we have
\begin{equation*}
\Theta_{nk}Q_n\cong Q_k \Psi_{nk} Q_n^R Q_n. 
\end{equation*}
Therefore, we have
\begin{multline*}
\Fiber(Q_k\Psi_{nk}\xto{\alpha_{nk}} \Theta_{nk} Q_n)\cong Q_k\Psi_{nk}\circ \Fiber(\id_{\cD_n}\to Q_n^R Q_n)\\
\cong Q_k\Psi_{nk} T_n T_n^R\cong Q_k T_k \Phi_{nk} T_n^R = 0.
\end{multline*}
Here the second isomorphism follows from the assumption that the morphism $(T_n)$ is special. 

The proof that $\beta_{nk}$ is an isomorphism is similar. This proves the proposition.\end{proof}

\begin{prop}\label{prop:Ind_and_Calk_of_ML} Let $(\cC_n)$ be a strongly Mittag-Leffler sequence. Then the sequences $(\Ind(\cC_n^{\kappa}))$ and $(\Ind(\Calk_{\kappa}^{\cont}(\cC_n)))$ are also strongly Mittag-Leffler sequences, and the natural morphism $(\Ind(\cC_n^{\kappa}))\to (\Ind(\Calk_{\kappa}^{\cont}(\cC_n)))$ is special.\end{prop}

\begin{proof}We first prove that the sequence $(\Ind(\cC_n^{\kappa}))$ is strongly Mittag-Leffler. Denote by $G_{nk}$ the functor $\Ind(F_{nk}^{\kappa}):\Ind(\cC_n^{\kappa})\to\Ind(\cC_k^{\kappa})$ for $n\geq k.$ For any $n.k\geq 0$ the inverse sequence $(G_{mk}G_{mn}^R)_{m\geq n.k}$ is the image of $(F_{mk}F_{mn}^R)_{m\geq n.k}$ under the functor
\begin{equation*}\Psi:\Fun^L(\cC_n,\cC_k)\to\Fun^L(\Ind(\cC_n^{\kappa}),\Ind(\cC_k^{\kappa})),\end{equation*}
sending a continuous functor $F:\cC_n\to\cC_k$ to the composition
\begin{equation*}
	\Ind(\cC_n^{\kappa})\xto{\Ind(F_{\mid \cC_n^{\kappa}})}\Ind(\cC_k)\to\Ind(\cC_k^{\kappa}),	
\end{equation*}
where the latter functor is the right adjoint to the inclusion.
	
It follows that the sequence $(G_{mk}G_{mn}^R)_{m\geq n.k}$ is essentially constant, and its limit is given by $\Psi(\Phi_{nk})=\Ind(\Phi_{nk}^{\kappa})$ (since the functor $\Phi_{nk}$ is strongly continuous, it takes $\cC_n^{\kappa}$ to $\cC_k^{\kappa}$). The functor $\Ind(\Phi_{nk}^{\kappa})$ is clearly strongly continuous, and its left adjoint is given by $\Ind((\Phi_{nk}^L)^{\kappa}).$ This proves that the sequence $\Ind(\cC_n^{\kappa})$ is strongly Mittag-Leffler.
	
Next, we claim that the morphism $(\hat{\cY}_{\cC_n}):(\cC_n)\to(\Ind(\cC_n^{\kappa}))$ is special. Indeed, this exactly means that the functors $\Phi_{nk}:\cC_n\to\cC_k$ are strongly continuous. Finally, it follows from Proposition \ref{prop:cokernel_of_special_mono} that the sequence $\Ind(\Calk_{\kappa}^{\cont}(\cC_n))$ is strongly Mittag-Leffler and the morphism $(\Ind(\cC_n^{\kappa}))\to (\Ind(\Calk_{\kappa}^{\cont}(\cC_n)))$ is special.\end{proof}

\begin{proof}[Proof of Theorem \ref{th:hom_epi_for_ML}] Consider the commutative square
\begin{equation*}
\begin{CD}
\Ind(\prolim[n]\cC_n^{\kappa}) @>>> \Ind(\prolim[n]\Calk_{\kappa}^{\cont}(\cC_n))\\
@V{\sim}VV @VVV\\
\prolim[n]^{\dual}\Ind(\cC_n^{\kappa}) @>>> \prolim[n]^{\dual}\Ind(\Calk_{\kappa}^{\cont}(\cC_n))
\end{CD}		
\end{equation*}

Since the functor $\Ind((-)^{\kappa}):\Pr^L_{\st,\kappa}\to\Cat_{\st}^{\dual}$ is the right adjoint to the inclusion, we see that the left vertical functor is an equivalence. The right vertical functor is clearly fully faithful. It follows from Proposition \ref{prop:Ind_and_Calk_of_ML} and Corollary \ref{cor:comm_square_for_duals_of_limits} \ref{dualizable_limit_or_epi_is_epi} that the lower horizontal functor is a quotient functor. Therefore, the right vertical functor is an equivalence and the upper horizontal functor is a quotient functor. This proves the theorem.
\end{proof}

We mention a few statements about the dualizable limit of a strongly Mittag-Leffler sequence, which can be deduced from the results above.

\begin{remark}
Let $(\cC_n)_{n\geq 0}$ be an inverse sequence in $\Cat_{\st}^{\dual},$ which is strongly Mittag-Leffler (over $\Sp$). Consider the natural continuous functor
\begin{equation}\label{eq:from_lim_dual_to_lim}
\prolim[n]^{\dual}\cC_n\to \prolim[n]\cC_n.
\end{equation}
\begin{enumerate}[label=(\roman*),ref=(\roman*)]
	\item Recall from Proposition \ref{prop:naive_limit_dualizable} that the projection functors $\pi_k:\prolim[n]\cC_n\to\cC_k$ are strongly continuous. The sequence $(\pi_k)_{k\geq 0}$ gives a fully faithful strongly continuous functor $\iota:\prolim[n]\cC_n\to \prolim[n]^{\dual}\cC_n,$ which is the left adjoint to the functor \eqref{eq:from_lim_dual_to_lim}.
	\item Denote by $\Lambda:\prolim[n]\cC_n\to \prolim[n]^{\dual}\cC_n$ the right adjoint to \eqref{eq:from_lim_dual_to_lim}. Then the functor $\Lambda$ is fully faithful, commutes with $\omega_1$-filtered colimits, and the essential image $\Lambda(\prolim[n]\cC_n^{\omega_1})$ generates the category $\prolim[n]^{\dual}\cC_n$ via colimits.
	\item The composition 
	\begin{equation*}
	\prolim[n]\cC_n\xto{\Lambda}\prolim[n]^{\dual}\cC_n\to \Ind(\prolim[n]\cC_n^{\omega_1})\simeq \Fun((\prolim[n]\cC_n^{\omega_1})^{op},\Sp)	
	\end{equation*} 
	is given by $x\mapsto \prolim[n]\ev_{\cC_n}(\pi_n(x),\pi_n(-)^{\vee}).$
	\item The composition \begin{equation*}
		\prolim[n]\cC_n\xto{\Lambda}\prolim[n]^{\dual}\cC_n\simeq (\prolim[n]^{\dual}\cC_n^{\vee})^{\vee}
		\to \Ind((\prolim[n]\cC_n^{\vee,\omega_1})^{op})\simeq \Fun(\prolim[n]\cC_n^{\vee,\omega_1},\Sp)	
	\end{equation*}
	is given by $x\mapsto \prolim[n]\ev_{\cC_n}(\pi_n(x),\pi_n'(-)).$ Here as before we denote by $\pi_k':\prolim[n]\cC_n^{\vee}\to\cC_k^{\vee}$ the projection functors.
	\item Denote by $\Lambda':\prolim[n]\cC_n^{\vee}\to\prolim[n]^{\dual}\cC_n^{\vee}$ the similar functor for the sequence $(\cC_n^{\vee})_{n\geq 0}.$ Then the composition
	\begin{equation*}
	    (\prolim[n]\cC_n)\times(\prolim[n]\cC_n^{\vee})\xto{(\Lambda,\Lambda')} (\prolim[n]^{\dual}\cC_n)\times(\prolim[n]^{\dual}\cC_n^{\vee})
	    \simeq (\prolim[n]^{\dual}\cC_n)\times(\prolim[n]^{\dual}\cC_n)^{\vee}\xto{\ev}\Sp
	\end{equation*}
	is given by $(x,y)\mapsto \prolim[n]\ev_{\cC_n}(\pi_n(x),\pi_n'(y)).$
\end{enumerate}
\end{remark}

\subsection{Relation with the dualizable internal $\Hom$}

We now explain how the dualizable internal $\Hom$ from Section \ref{sec:dualizable_Hom} in certain situations can be described as a dualizable inverse limit of a strongly Mittag-Leffler sequence of dualizable categories. For simplicity we work over an $\bE_{\infty}$-ring $\mk,$ but all the results in this subsection can be generalized to the case when the base category is compactly generated rigid symmetric monoidal (and with some modifications we can assume it to be only $\bE_1$-monoidal).

We are interested in the categories of the form $\un{\Hom}_{\mk}^{\dual}(\Ind(\cA),\cD),$ where $\cD$ is dualizable over $\mk,$ and $\cA$ is a proper $\omega_1$-compact small $\mk$-linear idempotent-complete stable category. Recall that this condition on $\cA$ means that the triangulated category $\h\cA$ is generated by countably many objects (as an idempotent-complete triangulated subcategory), and the $\mk$-modules $\cA(x,y)$ are perfect for $x,y\in\cA.$

The goal of this subsection is to prove the following result.

\begin{theo}\label{th:comparison_Hom^dual_with_ML}Let $\cA\in\Cat_{\mk}^{\perf}$ be a proper $\omega_1$-compact idempotent-complete small stable category over $\mk.$ Choose a sequence of finitely presented categories $\cA_0\to\cA_1\to\dots$ in $\Cat_{\mk}^{\perf}$ such that $\cA\simeq\indlim[n]\cA_n.$ Let $\cD$ be any dualizable $\mk$-linear category.
\begin{enumerate}[label=(\roman*), ref=(\roman*)]
\item For $n\geq 0$ denote by $\cC_n$ the category $\Ind(\Fun_{\mk}(\cA_n,\Perf(\mk))).$ Then the inverse sequence $(\cC_n)_{n\geq 0}$ is strongly Mittag-Leffler over $\mk.$ Hence, the sequence $(\cC_n\tens{\mk}\cD)_{n\geq 0}$ is also strongly Mittag-Leffler over $\mk.$ \label{Mittag-Leffler}

\item We have a natural equivalence of short exact sequences:
\begin{equation}
\label{eq:isom_of_ses}
\begin{tikzcd}
\un{\Hom}_{\mk}^{\dual}(\Ind(\cA),\cD)\ar[r]\ar[d, "\sim"] & \Ind(\Fun_{\mk}(\cA,\cD^{\omega_1}))\ar[r]\ar[d, "\sim"] & \Ind(\Fun_{\mk}(\cA,\Calk_{\omega_1}^{\cont}(\cD)))\ar[d, "\sim"]\\
\prolim[n]^{\dual}(\cC_n\tens{\mk}\cD)\ar[r] & \Ind(\prolim[n] (\cC_n\tens{\mk}\cD)^{\omega_1})\ar[r] & \Ind(\prolim[n]\Calk_{\omega_1}^{\cont}(\cC_n\tens{\mk}\cD))
\end{tikzcd}
\end{equation}\label{isom_of_ses}
\end{enumerate}\end{theo}

 To prove Theorem \ref{th:comparison_Hom^dual_with_ML}, we need several auxiliary statements. First of all, we recall that a finitely presented small $\mk$-linear stable category $(\cB\in\Cat_{\mk}^{\perf})^{\omega}$ is smooth by \cite[Proposition 2.14]{TV07} (more precisely, the result in loc. cit. is over a discrete commutative ring, but the general case is proved in the same way). 
 
 For a small $\mk$-linear stable category $\cB\in\Cat^{\perf}_{\mk}$ we denote by $\Delta_{\cB}\in\Ind(\cB\tens{\mk}\cB^{op})$ the diagonal bimodule, i.e. it corresponds to the bifunctor $\cB(-,-):\cB^{op}\times\cB\to\Mod_{\mk}.$ Given a functor $F:\cB'\to \cB,$ we denote by $(\Delta_{\cB})_{\mid \cB'\tens{\mk}\cB^{'op}}\in\Ind(\cB'\tens{\mk}\cB^{'op})$ the ``restriction of scalars'' of $\Delta_{\cB},$ i.e. it corresponds to the bifunctor $\cB(F(-),F(-)):\cB^{'op}\times\cB'\to\Mod_{\mk}.$

We will also need the following construction, which seems to be not widely known.

\begin{defi}\label{def:deformed_tensor_algebra}
\begin{enumerate}[label=(\roman*),ref=(\roman*)]
\item Let $\cE\in\Alg_{\bE_1}(\Pr^L_{\st})$ be an $\bE_1$-monoidal presentable stable category (not necessarily rigid). We define the deformed tensor algebra functor	\begin{equation*}
T^{\deff}(-):\cE_{/1_{\cE}[1]}\to \Alg_{\bE_1}(\cE)
\end{equation*}
to be the left adjoint to the composition functor
\begin{equation*}
	\Alg_{\bE_1}(\cE)\to \Alg_{\bE_0}(\cE)\simeq \cE_{1_{\cE}/}\simeq \cE_{/1_{\cE}[1]},
\end{equation*}
which is simply the functor
\begin{equation*}
 A\mapsto (\Cone(1_{\cE}\to A)\to 1_{\cE}[1]).
\end{equation*}
We will write $T^{\deff}(x),$ assuming the choice of the deformation parameter $x\to 1_{\cE}[1].$ 
\item We have a natural (exhaustive) multiplicative non-negative increasing filtration on the deformed tensor algebra, i.e. the functor
\begin{equation*}
\Fil_{\bullet}T^{\deff}(-):\cE_{/1_{\cE}[1]}\to \Alg_{\bE_1}(\Fun(\N,\cE)).
\end{equation*}
It is the left adjoint to the functor
\begin{equation*}
\Fil_{\bullet}A\mapsto (\Cone(1_{\cE}\to \Fil_1 A)\to 1_{\cE}[1]). 
\end{equation*}
The associated graded of $\Fil_{\bullet}T^{\deff}(x)$ is simply the usual graded tensor algebra $T(x).$
\item 
Let $B$ be an $\bE_1$-$\mk$-algebra, where $\mk$ is the base $\bE_{\infty}$-ring. Then we denote by $T_B^{\deff}(-)$ the composition
\begin{multline*}
T_B^{\deff}(-):(B\hy\Mod\hy B)_{/B[1]}\simeq \Fun^L_{\mk}(\Mod\hy B,\Mod\hy B)_{/\id[1]}\\
\xto{T^{\deff}(-)}\Alg_{\bE_1}(\Fun^L_{\mk}(\Mod\hy B,\Mod\hy B))\simeq \Alg_{\bE_1}(\Mod_{\mk})_{B/},
\end{multline*} 
and similarly for $\Fil_{\bullet}T_{B}^{\deff}(-).$
\end{enumerate}
\end{defi}

\begin{prop}\label{prop:diag_bimodules_ess_const} We keep the notation from Theorem \ref{th:comparison_Hom^dual_with_ML}. For each $k\geq 0$ the sequence $((\Delta_{\cA_n})_{\mid\cA_k\tens{\mk}\cA_k^{op}})_{n\geq k}$ is essentially constant in $\Ind(\cA_k\tens{\mk}\cA_k^{op}),$ and its colimit is naturally isomorphic to $(\Delta_{\cA})_{\mid \cA_k\tens{\mk}\cA_k^{op}}.$\end{prop}

\begin{proof}
Note that the required property of the sequence $(\cA_n)$ depends only on the ind-object $\inddlim[n]\cA_n\in\Ind(\Cat_{\mk}^{\perf}).$ Hence, it suffices to prove the proposition for some convenient choice of an approximation of $\cA$ by finitely presented $\mk$-linear categories. We will give a direct inductive construction of such an approximation.

We first consider the special case when $\cA$ is generated by a single object, and then explain the modifications for the general case. Suppose that $\cA\simeq\Perf(A)$ for some (proper) $\bE_1$-$\mk$-algebra $A.$ We define the sequence of $\bE_1$-$\mk$-algebras $(A_n)_{n\geq 0}$ by the following inductive construction. We put $A_0=\mk,$ and $A_{n+1}=T_{A_n}^{\deff}(\Cone(A_n\to A)).$ Here the deformation parameter is given by the natural map $\Cone(A_n\to A)\to A_n[1].$ We define the map of $\bE_1$-algebras $A_{n+1}\to A$ to be the adjunction counit, i.e. it corresponds to the identity map of the object $\Cone(A_n\to A)$ in $(A_n\hy\Mod\hy A_n)_{/A_n[1]}.$

Inductively, we see that $A_n$ is finitely presented. Indeed, this is clear for $n=0.$ Assume that some $A_n$ is finitely presented. Then $A_n$ is smooth, and so is $A_n\tens{\mk}A_n^{op}.$ By assumption, $A$ is perfect over $\mk,$ hence $A$ is also perfect over $A_n\tens{\mk}A_n^{op}.$ Therefore, $\Cone(A_n\to A)$ is also perfect over $A_n\tens{\mk}A_n^{op},$ which implies that $A_{n+1}$ is finitely presented.

By construction, we have a factorization in the category of $A_n\hy A_n$-bimodules:
\begin{equation*}
A_n\to A\to A_{n+1}
\end{equation*}
In particular, for each $k\geq 0$ the sequence $(A_n)_{n\geq k}$ is essentially constant in $A_k\hy\Mod\hy A_k,$ and its colimit is isomorphic to $A.$ This proves the proposition in the case when $\cA$ is generated by a single object.

In the general case, we choose a generating sequence $(x_n)_{n\geq 0}$ of objects of $\cA.$ We construct the sequence $(\cA_n)_{n\geq 0}$ in $\Cat^{\perf}_{\mk}$ with a compatible sequence of strongly continuous $\mk$-linear functors $F_n:\Ind(\cA_n)\to\Ind(\cA)$ by induction. 

We put $\cA_0=\Perf(\mk),$ and the functor $F_0:\Ind(\cA_0)\to\Ind(\cA)$ sends $\mk$ to $x_0.$ For $n$ even, the category $\cA_{n+1}$ is defined as follows. We first consider the $\bE_1$-algebra in the $\bE_1$-monoidal category $\Fun_{\mk}^L(\Ind(\cA_n),\Ind(\cA_n)),$ given by the deformed tensor algebra $B_{n+1}=T^{\deff}(\Cone(\id\to F_n^R F_n)),$ where the deformation parameter is given by the natural map $\Cone(\id\to F_n^R F_n)\to \id[1].$ We consider $B_{n+1}$ as a monad on $\Ind(\cA_n),$ and put $\cA_{n+1}=\Mod_{B_{n+1}}(\Ind(\cA_n))^{\omega}.$ The functor $\Ind(\cA_{n+1})\to \Ind(\cA),$ considered as a morphism in $(\Cat_{\mk}^{\cg})_{\Ind(\cA_n)/},$ corresponds again to the identity map of the object $\Cone(\id\to F_n^R F_n)$ in $\Fun_{\mk}^L(\Ind(\cA_n),\Ind(\cA_n))_{/\id[1]}.$

For $n=2m+1$ odd, we put $\cA_{n+1}=\cA_n\times\Perf(\mk),$ and the functor 
\begin{equation*}
F_{n+1}:\Ind(\cA_{n+1})\simeq\Ind(\cA_n)\times \Mod_{\mk}\to\Ind(\cA)
\end{equation*} is given by the pair $(F_n,x_{m+1}).$

Arguing inductively as above, we see that each $\cA_n$ is finitely presented. For even $n$ we have a factorization
\begin{equation*}
\Delta_{\cA_n}\to (\Delta_{\cA})_{\mid \cA_n\tens{\mk}\cA_n^{op}}\to (\Delta_{\cA_{n+1}})_{\mid \cA_n\tens{\mk}\cA_n^{op}}
\end{equation*}
in the category $\Ind(\cA_n\tens{\mk}\cA_n^{op}).$ This proves that for $k\geq 0$ the sequence $((\Delta_{\cA_n})_{\mid\cA_k\tens{\mk}\cA_k^{op}})_{n\geq k}$ is essentially constant in $\Ind(\cA_k\tens{\mk}\cA_k^{op}),$ and its colimit is isomorphic to $(\Delta_{\cA})_{\mid \cA_k\tens{\mk}\cA_k^{op}}.$ This also proves that the functor $\indlim[n]\cA_n\to \cA$ is fully faithful. It is in fact an equivalence, because its image contains all the objects $x_n,$ $n\geq 0,$ by construction. This proves the proposition. 
\end{proof}

Next, we need the following statement about trace-class morphisms in $\Cat_{\mk}^{\perf}.$

\begin{lemma}\label{lem:trace_class_functors_from_smooth}
Let $A\to B$ be a map of $\bE_1$-$\mk$-algebras, and suppose that $A$ is smooth over $\mk.$ The following are equivalent.
\begin{enumerate}[label=(\roman*),ref=(\roman*)]
\item The functor $\Perf(A)\to\Perf(B)$ is trace-class as a morphism in the symmetric monoidal category $\Cat_{\mk}^{\perf}.$ \label{trace_class_in_Cat^perf}
\item As a left $A$-module, $B$ is contained in the category $\Ind(\Rep_{\mk}(A,\Perf(\mk)))\subset A\hy\Mod.$ \label{B_ind_finite_dimensional}
\end{enumerate}
\end{lemma}

\begin{proof}
First, recall that smoothness of $A$ implies that we have an inclusion $\Rep_{\mk}(A,\Perf(\mk))\subset\Perf(A^{op}),$ which induces a fully faithful functor on the ind-completions.

\Implies{trace_class_in_Cat^perf}{B_ind_finite_dimensional}. The condition \ref{trace_class_in_Cat^perf} implies that the $A\hy B$-bimodule $B$ is contained in the idempotent-complete stable subcategory generated by the objects $M\otimes B,$ where $M\in \Rep_{\mk}(A,\Perf(\mk)).$ This implies \ref{B_ind_finite_dimensional}.

\Implies{B_ind_finite_dimensional}{trace_class_in_Cat^perf}. Smoothness of $A$ implies that we have inclusions 
\begin{equation*}
\Rep_{\mk}(A,\Perf(\mk))\tens{\mk} \Perf(B)\subset \Fun_{\mk}(\Perf(A),\Perf(B))\subset \Perf(A^{op}\tens{\mk}B).
\end{equation*}
Therefore, it suffices to show that the $A\hy B$-bimodule $B$ is contained in the (compactly generated) cocomplete subcategory $\cC\subset A\hy\Mod\hy B,$ generated by $\Rep_{\mk}(A,\Perf(\mk))\tens{\mk} \Perf(B).$ In fact, the stronger statement holds: $\cC$ coincides with the full subcategory $\cC'\subset A\hy\Mod\hy B,$
 which consists of bimodules $N$ such that as an $A$-module $N$ is contained in $\Ind(\Rep_{\mk}(A,\Perf(\mk))).$ Indeed, we have $\cC\subset\cC'.$ The right orthogonal to $\Rep_{\mk}(A,\Perf(\mk))\tens{\mk} \Perf(B)$ in $\cC'$ is zero, hence $\cC=\cC'.$ This proves the proposition.	 
\end{proof}

\begin{prop}\label{prop:trace_class_functors_for_fp_approx}
We keep the notation from Theorem \ref{th:comparison_Hom^dual_with_ML}. For each $n\geq 0$ there exists $k\geq n$ such that the functor $\cA_n\to\cA_k$ is trace-class as a morphism in $\Cat_{\mk}^{\perf}.$
\end{prop}

\begin{proof} 
In the forthcoming paper \cite{E} we will give a more formal proof of this result. Here we give an ``explicit'' proof.

It suffices to prove this statement for the sequence $(\cA_n)_{n\geq 0}$ constructed in the proof of Proposition \ref{prop:diag_bimodules_ess_const}. We claim that in this case for each even $n$ the functor $\cA_n\to\cA_{n+1}$ is trace-class. Note that $\cA_n$ is generated by a single object, and its image in $\cA_{n+1}$ is also a generator. Hence, it suffices to prove the following statement.

{\noindent{\bf Claim.}} {\it Let $C$ be a smooth $\bE_1$-$\mk$-algebra, and let $M$ be a $C\hy C$-bimodule with a bimodule map $\alpha:M\to C[1]$ such that $\Fiber(\alpha)$ is perfect as a $\mk$-module. Consider the corresponding deformed tensor algebra $T_C^{\deff}(M).$ Then the functor $\Perf(C)\to\Perf(T_C^{\deff}(M))$ is trace-class as a morphism in $\Cat_{\mk}^{\perf}.$}

\begin{proof}[Proof of Claim] By Lemma \ref{lem:trace_class_functors_from_smooth} it suffices to prove that the left $C$-module $T_C^{\deff}(M)$ is contained in the full subcategory $\Ind(\Rep_{\mk}(C,\Perf(\mk)))\subset C\hy\Mod.$ To see this, consider the filtration $\Fil_{\bullet}T_C^{\deff}(M)$ as in Definition \ref{def:deformed_tensor_algebra}.
It suffices to prove that each quotient $\Fil_{2n+1}/\Fil_{2n-1}$ for $n\geq 0$ is perfect as a $\mk$-module (where $\Fil_{-1}=0$).

Put $N=\Fiber(\alpha).$ Since $N$ is perfect as a $\mk$-module and $C$ is smooth, it follows that $N$ is also perfect as a left $C$-module, as a right $C$-module and as a bimodule. It follows that all the tensor powers $N^{\otimes_C m}$ are perfect over $\mk$ for $m>0.$ Now, for each $k>0$ the $C\hy C$-bimodule $\Fiber(M^{\otimes_C k}\to  C[1]^{\otimes_C k}\cong C[k])$ has a finite filtration with subquotients being finite direct sums of copies of $N^{\otimes m}[k-m]$ for $1\leq m\leq k.$ Hence, the map $ M^{\otimes_C k}\to C[k]$ becomes an isomorphism in the quotient $\Perf(C\tens{\mk}C^{op})/\Rep_{\mk}(C^{op}\tens{\mk}C,\Perf(\mk)).$ It remains to observe 
$\Fil_{2n+1}/\Fil_{2n-1}$ is an extension of $M^{\otimes_C (2n+1)}$ by $M^{\otimes_C 2n},$ and the image of the corresponding map $M^{\otimes_C (2n+1)}\to M^{\otimes_C 2n}[1]$ in $\Perf(C\tens{\mk}C^{op})/\Rep_{\mk}(C^{op}\tens{\mk}C,\Perf(\mk))$ is isomorphic to the identity map of $C[2n+1].$ This proves that the quotient $\Fil_{2n+1}/\Fil_{2n-1}$ is perfect over $\mk.$
\end{proof}
The proposition is proved.
\end{proof}

\begin{remark}
The statement of Proposition \ref{prop:trace_class_functors_for_fp_approx} plays an important role in working with the category $\Mot_{\mk}^{\loc},$ as we will explain in \cite{E}.
\end{remark}

\begin{proof}[Proof of Theorem \ref{th:comparison_Hom^dual_with_ML}] 
We first prove \ref{isom_of_ses}. By Proposition \ref{prop:trace_class_functors_for_fp_approx}, we may and will assume that all the functors $\cA_n\to\cA_{n+1}$ are trace-class in $\Cat^{\perf}_{\mk}.$ For $n\geq 0$ we have inclusions $\Fun_{\mk}(\cA_n,\Perf(\mk))\subset\cA_n^{op},$ and our assumption implies that the (right adjoint) functor $\Ind(\cA_{n+1}^{op})\to \Ind(\cA_n^{op})$ factors through $\cC_n=\Ind(\Fun_{\mk}(\cA_n,\Perf(\mk))).$ Since the categories $\cA_n$ are finitely presented, these right adjoint functors preserve $\omega_1$-compact objects. In other words, we obtain a pro-equivalence
\begin{equation}\label{eq:pro_equivalence_Pr^L}
\proolim[n]\Ind(\cA_n^{op})\xto{\sim} \proolim[n]\cC_n\quad \text{in }\Pro(\Prr^L_{\mk,\omega_1}).
\end{equation} 
Therefore, within the notation of the theorem, we obtain
\begin{multline*}
\prolim[n](\cC_n\tens{\mk}\cD)^{\omega_1}\simeq\prolim[n](\Ind(\cA_n^{op})\tens{\mk}\cD)^{\omega_1}\simeq \prolim[n]\Fun_{\mk}(\cA_n,\cD)^{\omega_1}\\
\simeq\prolim[n]\Fun_{\mk}(\cA_n,\cD^{\omega_1})\simeq \Fun_{\mk}(\cA,\cD^{\omega_1}).
\end{multline*}
This gives the middle vertical equivalence in \eqref{eq:isom_of_ses}.

Similarly, using the assumption that the functors $\cA_n\to\cA_{n+1}$ are trace-class, we obtain the factorizations in $\Cat_{\mk}^{\perf}:$
\begin{multline*}
\Fun_{\mk}(\cA_{n+1},\Calk_{\omega_1}^{\cont}(\cD))\to \Fun_{\mk}(\cA_n,\Perf(\mk))\tens{\mk} \Calk_{\omega_1}^{\cont}(\cD)\to \Calk_{\omega_1}^{\cont}(\cC_n\tens{\mk}\cD)\\
\to \Fun_{\mk}(\cA_n,\Calk_{\omega_1}^{\cont}(\cD)).
\end{multline*}
We obtain
\begin{equation*}
\Fun_{\mk}(\cA,\Calk_{\omega_1}^{\cont}(\cD))\simeq \prolim[n]\Fun_{\mk}(\cA_n,\Calk_{\omega_1}^{\cont}(\cD))\simeq\prolim[n]\Calk_{\omega_1}^{\cont}(\cC_n\tens{\mk}\cD).
\end{equation*}
This gives the right vertical equivalence in \eqref{eq:isom_of_ses}.

The left vertical equivalence in \eqref{eq:isom_of_ses} is obtained again from the assumption that the functors $\cA_n\to\cA_{n+1}$ are trace-class in $\Cat_{\mk}^{\perf}.$ More precisely, we obtain the factorizations
\begin{equation*}
\un{\Hom}_{\mk}^{\dual}(\Ind(\cA_{n+1}),\cD)\to \cC_n\tens{\mk}\cD\to \un{\Hom}_{\mk}^{\dual}(\Ind(\cA_n),\cD).
\end{equation*}
This gives an equivalence
\begin{equation*}
\Hom_{\mk}^{\dual}(\Ind(\cA),\cD)\simeq \prolim[n]^{\dual}\Hom_{\mk}^{\dual}(\Ind(\cA_n),\cD)\simeq \prolim[n]^{\dual}(\cC_n\tens{\mk}\cD).
\end{equation*}

It is clear from the descriptions of vertical equivalences in \eqref{eq:isom_of_ses} that we in fact obtain an equivalence of short exact sequences, as required.

We now prove \ref{Mittag-Leffler}. We denote by $F_{nm}:\Ind(\cA_n)\to\Ind(\cA_m),$ $n\leq m,$ the (ind-completions of) the transition functors. We denote by $F_n:\Ind(\cA_n)\to\Ind(\cA)$ the functors to the colimit. Recall from Remark \ref{rem:remarks_on_ML} that the first condition in Definition \ref{def:strong_ML} depends only on the pro-object of $\Pr^{L}_{\mk}.$ Using the pro-equivalence \eqref{eq:pro_equivalence_Pr^L}, we reduce the question to the inverse sequence $(\Ind(\cA_n^{op}))_{n\geq 0}$ with transition functors $F_{nm}^{\vee}:\Ind(\cA_m^{op})\to\Ind(\cA_n^{op}).$ By Proposition \ref{prop:diag_bimodules_ess_const} for each $k\geq 0$ the direct sequence
$(F_{kn}^{R}F_{kn})_{n\geq k}$ is essentially constant in $\Fun_{\mk}^L(\Ind(\cA_k),\Ind(\cA_k)),$ and its colimit is given by $F_k^R F_k.$  Passing to the dual functors and taking the right adjoints, we obtain that the inverse sequence $(F_{kn}^{\vee}F_{kn}^{\vee,R})_{n\geq k}$ is essentially constant in $\Fun^{\acc}_{\mk}(\Ind(\cA_k^{op}),\Ind(\cA_k^{op})).$ The limit of the latter sequence is given by $F_k^{\vee}F_k^{\vee,R}.$ This proves that the first condition of Definition \ref{def:strong_ML} holds for the sequence $(\cC_n)_{n\geq 0}.$

Now, denote by $G_{mn}:\cC_m\to\cC_n$ the transition functors, where $m\geq n.$ Take some $n,k\geq 0,$ and consider the limit $\Phi_{nk}\in\Fun^L(\cC_n,\cC_k)$ of the (essentially constant) sequence $(G_{mk}G_{mn}^R)_{m\geq n,k}.$ By the above, we see that the restriction of $\Phi_{nk}$ to $\cC_n^{\omega}$ is given by the composition
\begin{equation*}
\cC_n^{\omega}=\Fun_{\mk}(\cA_n,\Perf(\mk))\xto{F_n^{\vee,R}}\Fun_{\mk}(\cA,\Perf(\mk))\xto{F_k^{\vee}}\Fun(\cA_k,\Perf(\mk))=\cC_k^{\omega}\subset\cC_k. 
\end{equation*}
Here the first functor is well-defined by the smoothness of $\cA_n$ and the properness of $\cA.$ We conclude that the functor $\Phi_{nk}$ is strongly continuous. We also see that the left adjoint to $\Phi_{nk}$ exists, and it is given on compact objects by the composition
\begin{equation*}
\cC_k^{\omega}=\Fun_{\mk}(\cA_k,\Perf(\mk))\to \cA_k^{op}\xto{F_k^{\vee,L}}\cA^{op}\xto{F_n^{\vee}}\Fun(\cA_n,\Perf(\mk))=\cC_n^{\omega}.
\end{equation*}
This proves the theorem.
\end{proof}

\begin{remark}\label{rem:avoiding_Mittag_Leffler}
Note that for the proof of part \ref{isom_of_ses} in Theorem \ref{th:comparison_Hom^dual_with_ML} we did not need to know part \ref{Mittag-Leffler}. In particular, using only Corollary \ref{cor:hom_epi_internal_hom} and Theorem \ref{th:comparison_Hom^dual_with_ML} \ref{isom_of_ses} we deduce that within the notation of loc. cit. we have a homological epimorphism
\begin{equation*}
\prolim[n](\cC_n\tens{\mk}\cD)^{\omega_1}\to \prolim[n]\Calk_{\omega_1}^{\cont}(\cC_n\tens{\mk}\cD).
\end{equation*}	
\end{remark}

\subsection{The category of nuclear modules as a dualizable inverse limit}

Again, we use the notation of Subsection \ref{ssec:Nuc_as_Hom^dual}, i.e. let $R$ be a discrete commutative ring, $I=(a_1,\dots,a_m)\subset R$ a finitely generated ideal, and we denote by $R^{\wedge}_I$ the derived $I$-completion of $R.$ 

\begin{prop}\label{prop:Nuc_as_limit}
Within the above notation, we put
\begin{equation*}
R_n=\Kos(R;a_1^n,\dots,a_m^n),\quad n\geq 1,
\end{equation*}
and consider $(R_n)_{n\geq 1}$ as an inverse sequence of $\bE_1$-$R$-algebras.
\begin{enumerate}[label=(\roman*),ref=(\roman*)]
\item The inverse sequence $(D(R_n))_{n\geq 1}$ is strongly Mittag-Leffler over $R.$ \label{Mittag_Leffler_derived_quotients}
\item We have an $R$-linear equivalence
\begin{equation*}
\Nuc(R^{\wedge}_I)\simeq\prolim[n]D(R_n).
\end{equation*} \label{Nuc as limit}
\end{enumerate}
\end{prop}

\begin{proof}
The statement \ref{Mittag_Leffler_derived_quotients} can be deduced from Theorem \ref{th:comparison_Hom^dual_with_ML} using Proposition \ref{prop:derived_quotients_pro_equiv_to_functors_from_approx} below. However, it is more instructive to give a direct proof.

Using Proposition \ref{prop:basic_properties_of_ML}, we reduce the general case to the special case $R=\Z[x],$ $I=(x),$ $m=1,$ $a_1=x.$ It suffices to prove that for each $k\geq 1,$ we have an isomorphism of pro-objects:
\begin{equation}\label{eq:ess_const_tensor_products}
(\Z[x]/x^k)\Ltens{\Z[x]}(\Z[x]/x^k)\xto{\sim}\proolim[n\geq k](\Z[x]/x^k)\Ltens{\Z[x]/x^n}(\Z[x]/x^k)
\end{equation}
in $D((\Z[x]/x^k)\Ltens{\Z[x]}(\Z[x]/x^k)).$ Indeed, the source of \eqref{eq:ess_const_tensor_products} is perfect over $\Z[x]/x^k$ via each of the two restrictions of scalars, hence \eqref{eq:ess_const_tensor_products} implies both conditions from Definition \ref{def:strong_ML}.

Using the base change $x\mapsto x^k,$ we reduce the proof of \eqref{eq:ess_const_tensor_products} to the case $k=1.$ We claim that for all $n\geq 1$ the map
\begin{equation}\label{eq:zero_map_in_tensor_products}
\Cone(\Z\Ltens{\Z[x]}\Z\to\Z\Ltens{\Z[x]/x^{n+1}}\Z)\to \Cone(\Z\Ltens{\Z[x]}\Z\to\Z\Ltens{\Z[x]/x^n}\Z)\end{equation}
is zero in $D(\Z\Ltens{\Z[x]}\Z).$ 

Put $A=\Z\Ltens{\Z[x]}\Z.$ We first observe that for each $n\geq 2$ the $H_{\bullet}(A)$-module $H_{\bullet}(\Z\Ltens{\Z[x]/x^n}\Z)$ is free on a sequence of generators of degrees $0,2,4,\dots.$ Indeed, we can compute the graded algebra $\Tor_{\bullet}^{\Z[x]/x^n}(\Z,\Z)$ using simplicial bar resolution, i.e.
\begin{equation*}
\Tor_{\bullet}^{\Z[x]/x^n}(\Z,\Z)\cong\pi_{\bullet}(S_{\bullet}),\quad S_k=(\Z[x]/x^n)^{\otimes k}.
\end{equation*}
We write the elements of $S_k$ as $[f_1|\dots | f_k],$ where $f_1,\dots,f_k\in\Z[x]/x^n.$ Then the homogeneous basis of $\Tor_{\bullet}^{\Z[x]/x^n}(\Z,\Z)$
is given by the classes of the following elements
\begin{equation*}
e_k\in S_k,\quad e_k = \begin{cases}[x^{n-1}\mid x\mid\dots\mid x^{n-1}\mid x] & \text{for }k\text{ even}\\
[x\mid x^{n-1}\mid x\mid\dots\mid x^{n-1}\mid x] & \text{for }k\text{ odd} \end{cases}
\end{equation*}
Now the relevant shuffle products are computed easily: we have $[x]\cdot e_{2k}=e_{2k+1}$ for $k\geq 0.$ Hence, the classes $e_{2k},$ $k\geq 0,$ form a basis of $\Tor_{\bullet}^{\Z[x]/x^n}(\Z,\Z)$ as a module over $\Tor_{\bullet}^{\Z[x]}(\Z,\Z),$ as required.

It follows that for $n\geq 1$ the source of \eqref{eq:zero_map_in_tensor_products} is isomorphic as an $A$-module to the direct sum $\biggplus[k\geq 1]A[2k].$ Thus, it suffices to prove that the map \eqref{eq:zero_map_in_tensor_products} is zero on the level of homology. This is elementary: the natural map from the standard (minimal) projective resolution of $\Z$ over $\Z[x]/x^{n+1}$ to the similar resolution over $\Z[x]/x^n$ is zero modulo $x$ in degrees $\geq 2:$
\begin{equation*}
\begin{tikzcd}
\dots \ar[r] & \Z[x]/x^{n+1}\ar[r, "x"]\ar[d, "x"] & \Z[x]/x^{n+1}\ar[r, "x^n"]\ar[d, "x"] & \Z[x]/x^{n+1} \ar[r, "x"]\ar[d, "1"] & \Z[x]/x^{n+1}\ar[d, "1"]\\
\dots \ar[r] & \Z[x]/x^{n}\ar[r, "x"] & \Z[x]/x^{n}\ar[r, "x^{n-1}"] & \Z[x]/x^{n} \ar[r, "x"] & \Z[x]/x^{n}
\end{tikzcd}
\end{equation*}
This proves \ref{Mittag_Leffler_derived_quotients}. 

To prove \ref{Nuc as limit}, we use Proposition \ref{prop:Nuc_via_rigidification} and the fact that the forgetful functor from the category of rigid symmetric monoidal categories to $\Cat_{\st}^{\dual}$ commutes with limits, see \cite[Corollary 4.85]{Ram24b}. We obtain
\begin{equation*}
\Nuc(R^{\wedge}_I)\simeq (D_{I\hy\compl}(R))^{\rig}\simeq (\prolim[n]D(R_n))^{\rig}\simeq \prolim[n]^{\dual}D(R_n).\qedhere
\end{equation*}
\end{proof}

Note that we have two approaches to represent the category $\Nuc(R^{\wedge}_I)$ as a dualizable inverse limit of a strongly Mittag-Leffler sequence over $R.$ The first one is given by Proposition \ref{prop:Nuc_as_limit}. The second one is via Theorem \ref{th:comparison_Hom^dual_with_ML}, using the description of $\Nuc(R^{\wedge}_I)$ as a dualizable internal $\Hom$ from Definition \ref{def:nuclear_via_internal_Hom}. We claim that these two approaches give the same pro-object in $\Cat_R^{\dual}.$ This essentially reduces to the following interesting observation.

\begin{lemma}\label{lem:extension_restriction_of_scalars_trace_class}
\begin{enumerate}[label=(\roman*), ref=(\roman*)]
\item For $n\geq 1$ the extension of scalars functor $\Perf(\Z[x]/x^{2n})\to\Perf(\Z[x]/x^n)$ is trace-class in $\Cat_{\Z[x]}^{\perf}.$ \label{pullback_is_trace_class}

\item For $n\geq 1$ the restriction of scalars functor $D^b_{\coh}(\Z[x]/x^n)\to D^b_{\coh}(\Z[x]/x^{2n})$ is trace-class in $\Cat_{\Z[x]}^{\perf}.$  \label{pushforward_is_trace_class}
\end{enumerate}
\end{lemma}

\begin{proof}
We first observe that \ref{pushforward_is_trace_class} follows from \ref{pullback_is_trace_class}  by applying the internal $\Hom$ in $\Cat_{\Z[x]}^{\perf}$ to the unit object.

We prove \ref{pullback_is_trace_class}. Using the base change via $x\mapsto x^n,$ we reduce to the case $n=1.$ In this case it suffices to show that the module $\Z$ over $\Z[x]/x^2\Ltens{\Z[x]}\Z$ is a direct summand of the module $\Z\Ltens{\Z[x]}\Z.$ 

To see this, we replace the $\Z[x]$-algebra $\Z$ with the quasi-isomorphic DG $\Z[x]$-algebra $A=\Kos(\Z[x];x).$ Put $$B=\Z[x]/x^2\tens{\Z[x]}A\cong\Kos(\Z[x]/x^2;x).$$ We identify left and right $B$-modules. We have $\Z\tens{\Z[x]}A\cong B/xB$ (the non-derived quotient). Denote by $y\in B_1$ the basis element, so that $d(y)=x.$ Then the $B$-module $\Z$ is quasi-isomorphic to a semi-free $B$-module $P$ with the basis $(z_k)_{k\geq 0},$ where $\deg(z_k)=2k$ and $d(z_{k+1})=xyz_k.$ Consider the unique map of DG $B$-modules $P\to B/xB,$ sending $z_0$ to $1$ and $z_k$ to $0$ for $k>0.$ This map gives a section of the map $B/xB\to\Z$ in $D(B).$ This proves the lemma. 
\end{proof}

We obtain the following comparison between the two strongly Mittag-Leffler sequences approximating the category $\Nuc(R^{\wedge}_I).$

\begin{prop}\label{prop:derived_quotients_pro_equiv_to_functors_from_approx} We keep the notation from Proposition \ref{prop:Nuc_as_limit}. Choose a direct sequence $(\cA_n)_{n\geq 0}$ of finitely presented categories in $\Cat_R^{\perf}$ such that $\indlim[n]\cA_n\simeq \Perf_{I\hy\tors}(R).$ Then we have an equivalence of pro-objects:
\begin{equation}\label{eq:pro_equiv_of_ML}
\proolim[n]\Fun_R(\cA_n,\Perf(R))\simeq \proolim[n]\Perf(R_n).
\end{equation}
in $\Cat_R^{\perf}.$
\end{prop}

\begin{proof}Our first observation is that the statement does not depend on the choice of the sequence $(\cA_n)_{n\geq 0}.$ Below we construct this sequence in such a way that the pro-equivalence \eqref{eq:pro_equiv_of_ML} will become straightforward.
	
First, consider the special case $R=\Z[x],$ $I=(x).$ It follows from Lemma \ref{lem:extension_restriction_of_scalars_trace_class} that for each $n\geq 0$ we can choose a factorization
\begin{equation*}
D^b_{coh}(\Z[x]/x^{2^n})\to \cB_n\to D^b_{coh}(\Z[x]/x^{2^{n+1}})
\end{equation*}
in $\Cat_{\Z[x]}^{\perf},$ such that $\cB_n$ is finitely presented over $\Z[x].$ Then we have $\indlim[n]\cB_n\simeq\Perf_{x\hy\tors}(\Z[x]).$ We obtain pro-equivalences
\begin{equation*}
\proolim[n]\Fun_{\Z[x]}(\cB_n,\Perf(\Z[x]))\simeq \proolim[n]\Fun_{\Z[x]}(D^b_{\coh}(\Z[x]/x^n),\Perf(\Z[x]))\simeq\proolim[n]\Perf(\Z[x]/x^n)
\end{equation*}
in $\Cat_{\Z[x]}^{\perf},$ as required.

In the general case, take the homomorphism $\Z[x_1,\dots,x_m]\to R,$ sending $x_i$ to $a_i.$ Consider the category $\cB_n^{\otimes_{\Z}m}$ as a $\Z[x_1,\dots,x_m]$-linear stable category, and put
\begin{equation*}
\cA_n = \Perf(R)\tens{\Z[x_1,\dots,x_m]}\cB_n^{\otimes_{\Z}m}.
\end{equation*}   
We obtain pro-equivalences
\begin{multline*}
\proolim[n]\Fun_R(\cA_n,\Perf(R))\simeq\proolim[n]\Fun_{\Z[x_1,\dots,x_m]}(\cB_n^{\otimes_{\Z}m},\Perf(R))\\
\simeq \proolim[n](\Fun_{\Z[x_1,\dots,x_m]}(\cB_n^{\otimes_{\Z}m},\Perf(\Z[x_1,\dots,x_m]))\tens{\Z[x_1,\dots,x_m]}\Perf(R))\\
\simeq\proolim[n](\Fun_{\Z[x]}(\cB_n,\Perf(\Z[x]))^{\otimes_{\Z}m}\tens{\Z[x_1,\dots,x_m]}\Perf(R))\\
\simeq\proolim[n]\Perf(\Z[x_1,\dots,x_m]/(x_1^n,\dots,x_m^n))\tens{\Z[x_1,\dots,x_m]}\Perf(R)\\
\simeq\proolim[n]\Perf(R_n).
\end{multline*}
This proves the proposition.
\end{proof}

In the case when $R$ is noetherian we can replace the inverse sequence $(R_n)_{n\geq 1}$ with the sequence $(R/I^n),$ as we show in the following lemma. This is a stronger version of \cite[Lemma 2.5]{Har}, with essentially the same proof.

\begin{lemma}\label{lem:pro_equivalence_noetherian} Let $R$ be a noetherian ring, and let $I=(a_1,\dots,a_m)\subset R$ be an ideal. As above, we put $R_n=\Kos(R;a_1^n,\dots,a_m^n)$ for $n\geq 1.$
Then the natural map of pro-objects $$\proolim[n]R_n\to \proolim[n]R/I^n$$ is an isomorphism in the $\infty$-category $\Pro(\Alg_{\bE_1}(D(R))).$\end{lemma}

\begin{proof}
	For $n>0$ we have inclusions of ideals $$(a_1^n,\dots,a_m^n)\subset I^n\subset (a_1^{\lceil \frac{n}{m}\rceil},\dots,a_m^{\lceil \frac{n}{m}\rceil}),$$
	hence it suffices to prove that the map
	\begin{equation}\label{eq:pro_equivalence}
		\proolim[n] R_n\to \proolim[n]R/(a_1^n,\dots,a_m^n)
	\end{equation}
	is an isomorphism. We proceed by induction on $m.$
	
	Let $m=1.$ Since $R$ is noetherian, it has bounded $a_1$-torsion, i.e. there exists $k>0$ such that $R[a_1^{\infty}]=R[a_1^k].$ Then for $n>0$ we have a well-defined map 
	\begin{equation*}
		\varphi_n:a_1^{n+k}R\to R,\quad \varphi_n(a_1^{n+k}g)=a_1^k g. 
	\end{equation*}
	The map $\varphi_n$ gives a map of DG $R$-algebras
	\begin{equation*}
		\Cone(a_1^{n+k}R\to R)\to \Cone(R\xto{a_1^n}R)=\Kos(R;a_1^n).
	\end{equation*}
	The source of this map is quasi-isomorphic to $R/(a_1^{n+k}),$ hence in the $\infty$-category $\Alg_{\bE_1}(D(R))$ we obtain the factorization
	\begin{equation*}
		\Kos(R;a_1^{n+k})\to R/a_1^{n+k}\to \Kos(R;a_1^n).
	\end{equation*}
	This proves that in this case the map \eqref{eq:pro_equivalence} is an isomorphism.
	
	Now let $m>1$ and suppose that the assertion is proved for smaller values of $m.$ It follows from the induction hypothesis that there exists $n'\geq n$ such that we have a factorization
	\begin{equation}\label{eq:factorization1}
		\Kos(R;a_1^{n'},\dots,a_m^{n'})\to \Kos(R/(a_1^{n'},\dots,a_{m-1}^{n'});a_m^{n'})\to \Kos(R;a_1^{n},\dots,a_m^{n})
	\end{equation}
	in $\Alg_{\bE_1}(D(R)).$ Arguing as in the case $m=1,$ we find $n''\geq n'$ and a factorization
	\begin{equation}\label{eq:factorization2}
		\Kos(R/(a_1^{n'},\dots,a_{m-1}^{n'});a_m^{n''})\to R/(a_1^{n'},\dots,a_{m-1}^{n'},a_m^{n''}) \to \Kos(R/(a_1^{n'},\dots,a_{m-1}^{n'});a_m^{n'}).
	\end{equation}
	Finally, using \eqref{eq:factorization1} and \eqref{eq:factorization2} we obtain a factorization
	\begin{equation*}
		\Kos(R;a_1^{n''},\dots,a_m^{n''})\to R/(a_1^{n''},\dots,a_m^{n''})\to \Kos(R;a_1^{n},\dots,a_m^{n}).
	\end{equation*}
	This proves that the map \eqref{eq:pro_equivalence} is an isomorphism.
\end{proof}

\begin{remark}
	It is not difficult to prove a stronger version of Lemma \ref{lem:pro_equivalence_noetherian}, replacing the category $\Alg_{\bE_1}(D(R))$ by the $\infty$-category of simplicial commutative $R$-algebras.
\end{remark}

\section{Localizing invariants of inverse limits}
\label{sec:loc_invar_inverse_limits}

The goal of this section is to prove the following result.

\begin{theo}\label{th:local_invar_of_inverse_limits}
Let $\cE$ be a rigid compactly generated $\bE_1$-monoidal category, and let $\Phi:\Cat_{\cE^{\omega}}^{\perf}\to\cT$ be an accessible localizing invariant, where $\cT$ is an accessible stable category. Let $(\cC_n)_{n\geq 0}$ be an inverse sequence in $\Cat_{\cE}^{\cg}.$ Denote by $F_{mn}:\cC_m\to\cC_n$ the transition functors, $m\geq n.$ 

Suppose that for some uncountable regular cardinal $\kappa$ the functor
\begin{equation}\label{eq:hom_epi_assumption}
\prolim[n]\cC_n^{\kappa}\to\prolim[n]\Calk^{\cont}_{\kappa}(\cC_n)
\end{equation} is a homological epimorphism. 

\begin{enumerate}[label=(\roman*),ref=(\roman*)]
	\item We have a canonical cofiber sequence
	\begin{equation*}
		\Phi^{\cont}(\prolim[n]^{\dual}\cC_n)\to \Phi(\prodd[n\geq 0]\cC_n^{\omega})\xto{\id-\Phi(F)} \Phi(\prodd[n\geq 0]\cC_n^{\omega})
	\end{equation*}
	in $\cT,$ where $F:\prodd[n\geq 0]\cC_n^{\omega}\to \prodd[n\geq 0]\cC_n^{\omega}$ is the product of the functors $F_{n+1,n}^{\omega}:\cC_{n+1}^{\omega}\to\cC_n^{\omega}.$ \label{Phi^cont_of_lim^dual_exact_triangle}
	\item Suppose that $\cT$ has countable products and the map $\Phi(\prodd[n\geq 0]\cC_n^{\omega})\to\prodd[n\geq 0]\Phi(\cC_n^{\omega})$ is an isomorphism. Then we have an isomorphism
	\begin{equation*}
		\Phi^{\cont}(\prolim[n]^{\dual}\cC_n)\xto{\sim}\prolim[n]\Phi^{\cont}(\cC_n)\cong \prolim[n]\Phi(\cC_n^{\omega}).
	\end{equation*} \label{Phi^cont_commute_with_the_limit}
\end{enumerate}
\end{theo}

The assumption that the functor \eqref{eq:hom_epi_assumption} is a homological epimorphism is very strong, and the examples of such sequences $(\cC_n)_{n\geq 0}$ are provided by Theorems \ref{th:hom_epi_for_ML}, \ref{th:comparison_Hom^dual_with_ML}, and Proposition \ref{prop:Nuc_as_limit}. We mention the following special cases.

\begin{cor}\label{cor:K_theory_of_limit_of_ML} Let $(\cC_n)_{n\geq 0}$ be an inverse sequence in $\Cat_{\st}^{\dual},$ which is strongly Mittag-Leffler over $\Sp$ (see Definition \ref{def:strong_ML}). Then we have an isomorphism
\begin{equation}\label{eq:K^cont_of_limit_of_ML}
K^{\cont}(\prolim[n]^{\dual}\cC_n)\cong \prolim[n] K^{\cont}(\cC_n). 
\end{equation}
\end{cor}

\begin{proof}
By Theorem \ref{th:hom_epi_for_ML}, we have a short exact sequence
\begin{equation*}
0\to \prolim[n]^{\dual}\cC_n\to \Ind(\prolim[n]\cC_n^{\omega_1})\to \Ind(\prolim[n]\Calk_{\omega_1}^{\cont}(\cC_n))\to 0.
\end{equation*}
Since the category $\prolim[n]\cC_n^{\omega_1}$ has countable coproducts, we obtain an isomorphism
\begin{equation}\label{eq:reduction_to_Calkins}
K^{\cont}(\prolim[n]^{\dual}\cC_n)\cong \Omega K(\prolim[n]\Calk_{\omega_1}^{\cont}(\cC_n)).
\end{equation}
By Proposition \ref{prop:Ind_and_Calk_of_ML}, the inverse sequence $\Ind(\Calk_{\omega_1}^{\cont}(\cC_n))_{n\geq 0}$ is also strongly Mittag-Leffler. Applying Theorem \ref{th:hom_epi_for_ML}, we see that it satisfies the assumptions of Theorem \ref{th:local_invar_of_inverse_limits}. By the proof of Theorem \ref{th:hom_epi_for_ML}, we have an equivalence
\begin{equation*}
\prolim[n]^{\dual}\Ind(\Calk_{\omega_1}^{\cont}(\cC_n))\simeq\Ind(\prolim[n]\Calk_{\omega_1}^{\cont}(\cC_n)).
\end{equation*} 
Recall that the functor $K:\Cat^{\perf}\to\Sp$ commutes with infinite products by \cite[Theorem 1.3]{KW19} (see also \cite[Theorem 4.28]{E24} for a different proof). Applying Theorem \ref{th:local_invar_of_inverse_limits}, we obtain
\begin{multline*}
K(\prolim[n]\Calk_{\omega_1}(\cC_n))\cong K^{\cont}(\prolim[n]^{\dual}\Ind(\Calk_{\omega_1}^{\cont}(\cC_n)))\cong\prolim[n]K(\Calk_{\omega_1}^{\cont}(\cC_n))\\
\cong \Sigma \prolim[n]K^{\cont}(\cC_n).
\end{multline*}
Combining this with \eqref{eq:reduction_to_Calkins}, we obtain the isomorphism \eqref{eq:K^cont_of_limit_of_ML}.
\end{proof}

\begin{cor}\label{cor:K_theory_of_Nuc}
Let $R$ be a commutative ring, and let $I=(a_1,\dots,a_m)\subset R$ be a finitely generated ideal. Consider the DG $R$-algebras $R_n=\Kos(R;a_1^n,\dots,a_m^n),$ $n\geq 1.$ Then we have an isomorphism
\begin{equation}\label{eq:first_isom_K_of_Nuc}
K^{\cont}(\Nuc(R^{\wedge}_I))\cong \prolim[n] K(R_n).	
\end{equation}
In particular, if $R$ is noetherian, then we have an isomorphism
\begin{equation}\label{eq:second_isom_K_of_Nuc}
K^{\cont}(\Nuc(R^{\wedge}_I))\cong\prolim[n]K(R/I^n).
\end{equation}
\end{cor}

\begin{proof}
By Proposition \ref{prop:Nuc_as_limit}, the inverse sequence $(D(R_n))_{n\geq 0}$ is strongly Mittag-Leffler over $D(R),$ hence also over $\Sp.$ Thus, the isomorphism \eqref{eq:first_isom_K_of_Nuc} is a special case of Corollary \ref{cor:K_theory_of_limit_of_ML}.

Alternatively, by the observation in Remark \ref{rem:avoiding_Mittag_Leffler} we can deduce that the functor $\prolim[n]D(R_n)^{\omega_1}\to\prolim[n]\Calk_{\omega_1}(R_n)$ is a homological epimorphism using only Corollary \ref{cor:hom_epi_internal_hom}, Theorem \ref{th:comparison_Hom^dual_with_ML} \ref{isom_of_ses} and Proposition \ref{prop:derived_quotients_pro_equiv_to_functors_from_approx}. Applying Theorem \ref{th:loc_invar_of_Nuc_CS} and the commutation of $K$-theory with infinite products, we obtain the isomorphism \eqref{eq:first_isom_K_of_Nuc}.

If $R$ is noetherian, then the isomorphism \eqref{eq:second_isom_K_of_Nuc} follows from \eqref{eq:first_isom_K_of_Nuc} by Lemma \ref{lem:pro_equivalence_noetherian}. 
\end{proof}

\begin{remark}\label{rmk:U_loc_of_inverse_limits}
Within the notation of Theorem \ref{th:local_invar_of_inverse_limits}, take $\Phi$ to be the universal localizing invariant $\cU_{\loc}:\Cat_{\cE^{\omega}}^{\perf}\to\Mot^{\loc},$ commuting with filtered colimits. We will prove in \cite{E} that the functor $\cU_{\loc}$ commutes with infinite products. This will imply the isomorphism
\begin{equation*}
\cU_{\loc}^{\cont}(\prolim[n]^{\dual}\cC_n)\to \prolim[n]\cU_{\loc}(\cC_n^{\omega})
\end{equation*}
under the assumption from Theorem \ref{th:local_invar_of_inverse_limits}.
\end{remark}

Now we prove Theorem \ref{th:local_invar_of_inverse_limits}. We first recall the sequential oplax limits.

\begin{defi}
Let $(\cD_n)_{n\geq 0}$ be an inverse sequence of (small or large) $\infty$-categories, and denote by $G_{mn}:\cD_m\to\cD_n,$ $m\geq n,$ the transition functors. We denote by $\prolim[n]^{\oplax}\cD_n$ the category of sections of the associated cartesian fibration over $\N.$ In other words, the objects of $\prolim[n]^{\oplax}\cD_n$ are given by sequences $(x_n;\varphi_n)_{n\geq 0},$ where $x_n\in\cD_n$ and $\varphi_n:x_n\to G_{n+1,n}(x_{n+1}).$
\end{defi}

Note that if $(\cD_n)_{n\geq 0}$ is an inverse sequence of left $\cE^{\omega}$-modules resp. left $\cE$-modules, then $\prolim[n]^{\oplax}\cD_n$ is also an $\cE^{\omega}$-module resp. an $\cE$-module. Next, we recall the notion of a relative $K$-equivalence.

\begin{defi}\label{def:K-equivalence}
Let $G:\cC\to\cD$ be a functor in $\Cat_{\cE^{\omega}}^{\perf}.$ Then $G$ is called a $K$-equivalence over $\cE^{\omega}$ if there exists an $\cE^{\omega}$-linear functor $H:\cD\to\cC,$ such that we have equalities
\begin{equation*}
[G\circ H]=[\id_{\cD}]\text{ in }K_0(\Fun_{\cE^{\omega}}(\cD,\cD)),\quad [H\circ G]=[\id_{\cC}]\text{ in }K_0(\Fun_{\cE^{\omega}}(\cC,\cC)).
\end{equation*}
\end{defi}

We recall the following relation between the oplax limit and the product.

\begin{prop}\label{prop:oplax_limit_K_equiv}
Let $(\cD_n)_{n\geq 0}$ be an inverse sequence in $\Cat_{\cE^{\omega}}^{\perf}.$ Then the natural (forgetful) functor $\prolim[n]^{\oplax}\cD_n\to\prodd[n]\cD_n$ is a $K$-equivalence over $\cE^{\omega}.$
\end{prop}

\begin{proof}
This is a special case of \cite[Proposition 4.29]{E24}. More precisely, the result in loc. cit. is over the absolute base, but the proof over $\cE^{\omega}$ is the same.  
\end{proof}

The following result is an important step towards the proof of Theorem \ref{th:local_invar_of_inverse_limits}.

\begin{prop}\label{prop:quotients_of_oplax_limits} Let $(\cD_n)_{n\geq 0}$ be an inverse sequence in $\Cat^{\perf}.$ Then the natural functor
\begin{equation}\label{eq:key_fully_faithful_functor}
H:(\prolim[n]^{\oplax}\Ind^{\kappa}(\cD_n))/(\prolim[n]^{\oplax}\cD_n)\to \prolim[n]^{\oplax}\Calk_{\kappa}(\cD_n)	
\end{equation}
is fully faithful.
\end{prop}

To prove Proposition \ref{prop:quotients_of_oplax_limits}, we need the following statement on the Beck-Chevalley condition.

\begin{lemma}\label{lem:Beck_Chevalley_for_oplax}
Within the notation of Proposition \ref{prop:quotients_of_oplax_limits}, the following square commutes:
\begin{equation}\label{eq:Beck_Chevalley_for_oplax}
\begin{CD}
\prolim[n]^{\oplax}\Ind(\cD_n) @>>> \prodd[n\geq 0]\Ind(\cD_n)\\
@VVV @VVV\\
\Ind(\prolim[n]^{\oplax}\cD_n) @>>> \Ind(\prodd[n\geq 0]\cD_n).
\end{CD}
\end{equation}
Here the vertical functors are the right adjoints to the natural functors.
\end{lemma}

\begin{proof}
Denote by $G_{mn}:\Ind(\cD_m)\to\Ind(\cD_n)$ the transition functors, $m\geq n.$ Take an object $x=(x_n;\varphi_n)_{n\geq 0}\in \prolim[n]^{\oplax}\Ind(\cD_n),$ where $x_n\in\Ind(\cD_n),$ $\varphi_n:x_n\to G_{n+1,n}(x_{n+1}).$ Consider the filtered categories $I_n=(\cD_n)_{/ x_n},$ $n\geq 0,$ and $J_n=(\cD_{n-1})_{/G_{n,n-1}(x_n)},$ $n\geq 1.$ Each map $\varphi_n$ induces the functor $I_n\to J_{n+1},$ $n\geq 0.$ Further, each functor $G_{n,n-1}$ induces the cofinal functor $I_n\to J_n,$ $n\geq 1.$ As in Subsection \ref{ssec:lemmas_about_filtered}, we can form the lax equalizer $\LEq(\prodd[n\geq 0]I_n\toto \prodd[n\geq 1]J_n).$ Then we have (tautological) equivalences
\begin{equation*}
(\prolim[n]^{\oplax}\cD_n)_{/x}\simeq \LEq(\prodd[n\geq 0]I_n\toto \prodd[n\geq 1]J_n),\quad (\prodd[n\geq 0]\cD_n)_{/(x_n)_{n\geq 0}}\simeq \prodd[n\geq 0]I_n.
\end{equation*}
By Lemma \ref{lem:lax_equalizer_of_filtered_properties} the functor
\begin{equation*}
\LEq(\prodd[n\geq 0]I_n\toto \prodd[n\geq 1]J_n)\to \prodd[n\geq 0]I_n
\end{equation*}
is cofinal. We conclude that the square \eqref{eq:Beck_Chevalley_for_oplax} commutes.
\end{proof}

\begin{proof}[Proof of Proposition \ref{prop:quotients_of_oplax_limits}]
We denote by $\cA$ and $\cB$ respectively the source and the target of the functor \eqref{eq:key_fully_faithful_functor}. As in the proof of Lemma \ref{lem:Beck_Chevalley_for_oplax}, we denote by $G_{mn}:\Ind(\cD_m)\to\Ind(\cD_n)$ the transition functors, $m\geq n.$ 

Take some objects $x,y\in\prolim[n]^{\oplax}\Ind^{\kappa}(\cD_n),$ where $x=(x_n;\varphi_n)_{n\geq 0},$ $y=(y_n;\psi_n)_{n\geq 0},$ We choose a directed poset $I$ with a cofinal functor $I\to (\prolim[n]^{\oplax}\cD_n)_{/y},$ $i\mapsto z_i=(z_{i,n};\psi_{i,n})_{n\geq 0}.$ Then we have
\begin{multline}\label{eq:Fiber_on_Homs1}
\Fiber(\Hom_{\prolim[n]^{\oplax}\Ind^{\kappa}(\cD_n)}(x,y)\to \Hom_{\cA}(x,y))\cong \indlim[i]\Hom_{\prolim[n]^{\oplax}\Ind^{\kappa}(\cD_n)}(x,z_i)\\
\cong \indlim[i]\Fiber(\prodd[n\geq 0]\Hom_{\Ind(\cD_n)}(x_n,z_{i,n})\to \prodd[n\geq 0]\Hom_{\Ind(\cD_n)}(x_n,G_{n+1,n}(z_{i,n})))\\
\cong \Fiber(\indlim[i]\prodd[n\geq 0]\Hom_{\Ind(\cD_n)}(x_n,z_{i,n})\to \indlim[i]\prodd[n\geq 0]\Hom_{\Ind(\cD_n)}(x_n,G_{n+1,n}(z_{i,n+1})))
\end{multline}
By Lemma \ref{lem:Beck_Chevalley_for_oplax}, the functor $I\to \prodd[n\geq 0](\cD_n)_{/y_n}$ is cofinal. As in \cite[Proposition 1.101]{E24} (which is essentially a straightforward application of (AB6) in the category $\Sp$), we obtain
\begin{equation*}
\indlim[i]\prodd[n\geq 0]\Hom_{\Ind(\cD_n)}(x_n,z_{i,n})\cong \indlim[i]\prodd[n\geq 0](z_{i,n}\tens{\cD_n}x_n^{\vee})\cong \prodd[n\geq 0](y_n\tens{\cD_n}x_n^{\vee}),
\end{equation*}
and similarly
\begin{equation*}
	\indlim[i]\prodd[n\geq 0]\Hom_{\Ind(\cD_n)}(x_n,G_{n+1,n}(z_{i,n+1}))\cong \prodd[n\geq 0](G_{n+1,n}(y_n)\tens{\cD_n}x_n^{\vee}).
\end{equation*}
Therefore, we obtain the isomorphisms
\begin{multline}\label{eq:Fiber_on_Homs2}
\Fiber(\indlim[i]\prodd[n\geq 0]\Hom_{\Ind(\cD_n)}(x_n,z_{i,n})\to \indlim[i]\prodd[n\geq 0]\Hom_{\Ind(\cD_n)}(x_n,G_{n+1,n}(z_{i,n+1})))\\
\cong \Fiber(\prodd[n\geq 0](y_n\tens{\cD_n}x_n^{\vee})\to \prodd[n\geq 0](G_{n+1,n}(y_n)\tens{\cD_n}x_n^{\vee}))\\
\cong \Fiber(\Hom_{\prolim[n]^{\oplax}\Ind^{\kappa}(\cD_n)}(x,y)\to \Hom_{\cB}(x,y)). 
\end{multline}
The isomorphisms \eqref{eq:Fiber_on_Homs1} and \eqref{eq:Fiber_on_Homs2} imply the fully faithfulness of the functor $H:\cA\to\cB.$ This proves the proposition.
\end{proof}

We obtain the following crucial corollary.

\begin{cor}\label{cor:diagonal_arrow}
Within the notation of Theorem \ref{th:local_invar_of_inverse_limits}, there exists a unique $\cE^{\omega}$-linear functor $\prolim[n]\Calk_{\kappa}(\cC_n^{\omega})\to (\prolim[n]^{\oplax}\cC_n^{\kappa})/(\prolim[n]^{\oplax}\cC_n^{\omega}),$ such that the following diagram commutes:
\begin{equation}\label{eq:diagonal_arrow}
\begin{tikzcd}
\prolim[n]\cC_n^{\kappa} \ar[r]\ar[d]  & \prolim[n]\Calk_{\kappa}^{\cont}(\cC_k) \ar[d]\ar[ld] \\
(\prolim[n]^{\oplax}\cC_n^{\kappa})/(\prolim[n]^{\oplax}\cC_n^{\omega})\ar[r] & \prolim[n]^{\oplax}\Calk_{\kappa}^{\cont}(\cC_k)
\end{tikzcd}
\end{equation}
Moreover, in this diagram the diagonal functor is fully faithful.
\end{cor}

\begin{proof}
Indeed, in the commutative square
\begin{equation*}
\begin{tikzcd}
	\prolim[n]\cC_n^{\kappa} \ar[r]\ar[d]  & \prolim[n]\Calk_{\kappa}^{\cont}(\cC_k) \ar[d]\\
	(\prolim[n]^{\oplax}\cC_n^{\kappa})/(\prolim[n]^{\oplax}\cC_n^{\omega})\ar[r] & \prolim[n]^{\oplax}\Calk_{\kappa}^{\cont}(\cC_k)
\end{tikzcd}
\end{equation*}
the upper horizontal functor is a homological epimorphism by our assumption, and the lower horizontal functor is fully faithful by Proposition \ref{prop:quotients_of_oplax_limits}. This implies the existence and uniqueness of the diagonal arrow which makes the diagram \eqref{eq:diagonal_arrow} commute. Since the right vertical functor and the lower horizontal functors in \eqref{eq:diagonal_arrow} are fully faithful, it follows that the diagonal functor is also fully faithful. 
\end{proof}

We now want to analyze the cokernel (cofiber) of the diagonal arrow in \eqref{eq:diagonal_arrow}, which is the same as the cokernel of the left vertical arrow (since the upper horizontal arrow is a homological epimorphism by assumption).
We will use the following notation: given an idempotent-complete small stable category $\cA$ and two full stable subcategories $\cB,\cC\subset\cA,$ we denote by $\cB\backslash\cA/\cC$ the (idempotent completion of) the quotient of $\cA$ by the stable subcategory generated by $\cB$ and $\cC.$

\begin{prop}\label{prop:double_quotient_K_equiv} Within the notation of Theorem \ref{th:local_invar_of_inverse_limits}, the functor 
\begin{equation*}
\Psi:\prolim[n]^{\oplax}\cC_n^{\kappa}\to \prodd[n\geq 0]\cC_n^{\kappa},\quad \Psi((x_n;\varphi_n)_{n\geq 0})=(\Fiber(\varphi_n:x_n\to F_{n+1,n}(x_{n+1})))_{n\geq 0},\end{equation*}
induces a $K$-equivalence over $\cE^{\omega}:$
\begin{equation*}
\bbar{\Psi}:(\prolim[n]\cC_n^{\kappa})\backslash(\prolim[n]^{\oplax}\cC_n^{\kappa})/(\prolim[n]^{\oplax}\cC_n^{\omega})\to \prodd[n\geq 0]\Calk_{\kappa}^{\cont}(\cC_n).
\end{equation*}\end{prop}
\begin{proof}
For convenience, we put $$\cD=(\prolim[n]\cC_n^{\kappa})\backslash(\prolim[n]^{\oplax}\cC_n^{\kappa})/(\prolim[n]^{\oplax}\cC_n^{\omega}).$$ We first construct an ($\cE^{\omega}$-linear) right inverse to $\bbar{\Psi}.$ For each $k\geq 0$ consider the functor
\begin{equation*}
\Theta_k:\cC_k^{\kappa}\to\prolim[n]^{\oplax}\cC_n^{\kappa},\quad \Theta_k(x)=(F_{k,0}(x),\dots,F_{k,k-1}(x),x,0,\dots).
\end{equation*}
Note that $\Theta_k$ is simply the right adjoint to the projection functor $(x_n;\varphi_n)_{n\geq 0}\mapsto x_k.$
 Consider the functor
\begin{equation*}
\Theta:\prodd[n\geq 0]\cC_n^{\kappa}\to\prolim[n]^{\oplax}\cC_n^{\kappa},\quad \Theta((x_n)_{n\geq 0})=\biggplus[n]\Theta_n(x_n).
\end{equation*}
For an object $(x_n)_{n\geq 0}\in\prodd[n\geq 0]\cC_n^{\kappa}$ and for each $k\geq 0$ we have
\begin{multline}
\label{eq:Theta_is_right_inverse}
\Fiber(\Theta((x_n)_{n\geq 0})_k\to F_{k+1,k}(\Theta((x_n)_{n\geq 0})_{k+1}))\\
\cong \Fiber(\biggplus[n\geq k]F_{n,k}(x_k)\to \biggplus[n\geq k+1]F_{n,k}(x_k))\cong  x_k.	
\end{multline}
Hence, we have $\Psi\circ\Theta\cong\id.$

Now let $(x_n\in\cC_n^{\omega})_{n\geq 0}$ be a sequence of compact objects. We want to prove that the image of $\Theta((x_n)_{n\geq 0})$ in $\cD$ is zero. It follows from \eqref{eq:Theta_is_right_inverse} that the image of $\Theta((x_n)_{n\geq 0})$ in the category $\prolim[n]^{\oplax}\Calk_{\kappa}^{\cont}(\cC_n)$ is contained in the full subcategory $\prolim[n]\Calk_{\kappa}^{\cont}(\cC_n).$ From the commutative diagram \eqref{eq:diagonal_arrow} we see that the image of $\Theta((x_n)_{n\geq 0})$ in the quotient $(\prolim[n]^{\oplax}\cC_n^{\kappa})/(\prolim[n]^{\oplax}\cC_n^{\omega})$ is contained in the image of the diagonal arrow. Since the upper horizontal arrow of \eqref{eq:diagonal_arrow} is a homological epimorphism (by assumption), we conclude that the image of $\Theta((x_n)_{n\geq 0})$ in $\cD$ is zero.

Recall that we have a short exact sequence
\begin{equation*}
0\to \prodd[n\geq 0]\cC_n^{\omega}\to \prodd[n\geq 0]\cC_n^{\kappa}\to \prodd[n\geq 0]\Calk_{\kappa}^{\cont}(\cC_n)\to 0,
\end{equation*}
which follows from \cite[Lemma 5.2]{KW19}. We conclude that the functor $\Theta$ induces an $\cE^{\omega}$-linear functor
\begin{equation*}
\bbar{\Theta}:\prodd[n\geq 0]\Calk_{\kappa}^{\cont}(\cC_n)\to \cD.
\end{equation*}
Since $\Psi\circ\Theta\cong\id,$ we also have $\bbar{\Psi}\circ\bbar{\Theta}\cong \id.$ It remains to prove that
\begin{equation}
[\bbar{\Theta}\circ\bbar{\Psi}]=[\id]\text{ in }K_0(\Fun_{\cE^{\omega}}(\cD,\cD)).
\end{equation}
For $k\geq 0$ consider the endofunctor 
\begin{equation*}
G_k:\prolim[n]^{\oplax}\cC_n^{\kappa}\to\prolim[n]^{\oplax}\cC_n^{\kappa},\quad G_k(x_0,x_1,\dots)=(F_{k,0}(x_k),\dots,F_{k,k-1}(x_k),x_k,x_{k+1},\dots).
\end{equation*}
We have natural exact triangles
\begin{equation}\label{eq:exact_traingle_G_k_G_k+1}
G_k((x_n;\varphi_n)_{n\geq 0})\to G_{k+1}((x_n;\varphi_n)_{n\geq 0})\to \Theta_k(\Fiber(\varphi_k:x_k\to F_{k+1,k}(x_{k+1})))
\end{equation}
for each $k\geq 0.$ Taking the direct sum of \eqref{eq:exact_traingle_G_k_G_k+1} over $k\geq 0,$ we obtain an exact triangle of functors
\begin{equation}\label{eq:exact_traingle_sum_G_k}
\biggplus[k\geq 0] G_k\to \biggplus[k\geq 1] G_k\to \Theta\circ\Psi
\end{equation}
in the category $\Fun_{\cE^{\omega}}(\prolim[n]^{\oplax}\cC_n^{\kappa},\prolim[n]^{\oplax}\cC_n^{\kappa}).$

We claim that the endofunctor $G=\biggplus[k\geq 1] G_k$ preserves the idempotent-complete stable subcategory generated by $\prolim[n]\cC_n^{\kappa}$ and $\prolim[n]^{\oplax}\cC_n^{\omega}.$ For $x\in \prolim[n]\cC_n^{\kappa}$ we have $G_k(x)\cong x$ for $k\geq 0,$ hence 
\begin{equation*}
G(x)\cong \biggplus[k\geq 1] x\in \prolim[n]\cC_n^{\kappa}.
\end{equation*} 
Let now $x=(x_n;\varphi_n)_{n\geq 0}\in\prolim[n]^{\oplax}\cC_n^{\omega}.$ For $k\geq 0$ we have a natural map $G_k(x)\to\Theta_k(x_k),$ which induces an isomorphism on the first $(k+1)$ components. Taking the direct sum over $k\geq 0,$ we obtain a map $x\oplus G(x)\to\Theta((x_n)_{n\geq 0}),$ since $G_0\cong\id.$ We have
\begin{equation*}
\Cone(x\oplus G(x)\to \Theta((x_n)_{n\geq 0}))\in\prolim[n]^{\oplax}\cC_n^{\omega}.
\end{equation*}
We know from the above that the image of $\Theta((x_n)_{n\geq 0}))$ in $\cD$ vanishes, hence the same holds for $G(x),$ as required. 

Now, denote by $\bbar{G}:\cD\to\cD$ the functor induced by $G.$ The exact triangle \eqref{eq:exact_traingle_sum_G_k} induces an exact triangle in $\Fun_{\cE^{\omega}}(\cD,\cD):$
\begin{equation*}
\id\oplus\bbar{G}\to\bbar{G}\to \bbar{\Theta}\circ\bbar{\Psi}.
\end{equation*}
We conclude that in $K_0(\Fun(\cD,\cD))$ we have
\begin{equation*}
[\bbar{\Theta}\circ\bbar{\Psi}]=[\id\oplus\bbar{G}]-[\bbar{G}]=[\id].
\end{equation*}
This proves the proposition.
\end{proof}

\begin{proof}[Proof of Theorem \ref{th:local_invar_of_inverse_limits}]
Our assumption on the sequence $(\cC_n)_{n\geq 0}$ implies that we have a short exact sequence in $\Cat_{\cE}^{\dual}:$
\begin{equation*}
0\to \prolim[n]^{\dual}\cC_n\to \Ind(\prolim[n]\cC_n^{\kappa})\to\Ind(\prolim[n]\Calk_{\kappa}^{\cont}(\cC_n))\to 0.
\end{equation*}
Since the category $\prolim[n]\cC_n^{\kappa}$ has countable coproducts, we obtain an isomorphism
\begin{equation}\label{eq:reduction_to_limit_of_Calkin}
\Phi^{\cont}(\prolim[n]^{\dual}\cC_n)\cong \Omega \Phi(\prolim[n]\Calk_{\kappa}^{\cont}(\cC_n)).	
\end{equation}
Now, from Corollary \ref{cor:diagonal_arrow} (and the discussion after its proof) we obtain a short exact sequence
\begin{equation}
\label{eq:short_exact_seq_limit_of_Calkin}
0\to \prolim[n]\Calk_{\kappa}^{\cont}(\cC_n)\to (\prolim[n]^{\oplax}\cC_n^{\kappa})/(\prolim[n]^{\oplax}\cC_n^{\omega})\to (\prolim[n]\cC_n^{\kappa})\backslash(\prolim[n]^{\oplax}\cC_n^{\kappa})/(\prolim[n]^{\oplax}\cC_n^{\omega})\to 0.
\end{equation}
By Proposition \ref{prop:oplax_limit_K_equiv}, we have
\begin{equation*}
\Phi((\prolim[n]^{\oplax}\cC_n^{\kappa})/(\prolim[n]^{\oplax}\cC_n^{\omega}))\cong \Sigma \Phi(\prodd[n\geq 0]\cC_n^{\omega}).
\end{equation*}
By Proposition \ref{prop:double_quotient_K_equiv}, we obtain
\begin{equation*}
\Phi((\prolim[n]\cC_n^{\kappa})\backslash(\prolim[n]^{\oplax}\cC_n^{\kappa})/(\prolim[n]^{\oplax}\cC_n^{\omega}))\cong \Phi(\prodd[n\geq 0]\Calk_{\kappa}^{\cont}(\cC_n))\cong \Sigma\Phi(\prodd[n\geq 0]\cC_n^{\omega}).
\end{equation*}
Moreover, we obtain a commutative square
\begin{equation}
\label{eq:comm_square_Phi_of_product}
\begin{CD}
\Omega \Phi((\prolim[n]^{\oplax}\cC_n^{\kappa})/(\prolim[n]^{\oplax}\cC_n^{\omega})) @>{\sim}>> \Phi(\prodd[n\geq 0]\cC_n^{\omega})\\
@VVV @V{\id-\Phi(F)}VV\\
\Omega \Phi((\prolim[n]\cC_n^{\kappa})\backslash(\prolim[n]^{\oplax}\cC_n^{\kappa})/(\prolim[n]^{\oplax}\cC_n^{\omega})) @>{\sim}>> \Phi(\prodd[n\geq 0]\cC_n^{\omega}),
\end{CD}
\end{equation}
where $F:\prodd[n\geq 0]\cC_n^{\omega}\to \prodd[n\geq 0]\cC_n^{\omega}$ is the product of functors $F_{n+1,n}:\cC_{n+1}^{\omega}\to\cC_n^{\omega}.$

Finally, using the isomorphism \eqref{eq:reduction_to_limit_of_Calkin}, the short exact sequence \eqref{eq:short_exact_seq_limit_of_Calkin} and the commutative square \eqref{eq:comm_square_Phi_of_product}, we obtain
\begin{equation*}
\Phi^{\cont}(\prolim[n]^{\dual}\cC_n)\cong \Omega \Phi(\prolim[n]\Calk_{\kappa}^{\cont}(\cC_n))\cong \Fiber(\Phi(\prodd[n\geq 0]\cC_n^{\omega})\xto{\id-\Phi(F)}\Phi(\prodd[n\geq 0]\cC_n^{\omega})).
\end{equation*}
 This proves \ref{Phi^cont_of_lim^dual_exact_triangle}, which clearly implies \ref{Phi^cont_commute_with_the_limit}.
\end{proof}

\section{Original version of the category of nuclear modules}
\label{sec:original_Nuc}

Let $R$ be a commutative ring with a finitely generated ideal $I=(a_1,\dots,a_m)\subset R.$ In this section we consider the original version of the category nuclear solid modules over $R^{\wedge}_{I}$ as defined in \cite{CS20}, see also \cite{CS24, And23, AM24}. We denote this category by $\Nuc^{CS}(R^{\wedge}_I).$ The main goal is to prove Theorem \ref{th:loc_invar_of_Nuc_CS} on the continuous $K$-theory and more general localizing invariants of the category $\Nuc^{CS}(R^{\wedge}_I).$

\subsection{Some notation}\label{ssec:some_notation}
 We will use the following notation. For a small additive $\infty$-category $\cA$ we denote by $\Stab(\cA)$ its stabilization. For pair of integers $a\leq b$ we denote by $\Stab(\cA)_{[a,b]}$ the full subcategory generated via extensions by $\cA[c],$ $a\leq c\leq b.$
 
 For a connective $\bE_1$-ring $S$ we denote by $\Proj^{\omega_1}\hy S\subset D(S)$ the full subcategory formed by direct summands of the object $\biggplus[\N]S.$ We will denote by $D^b(\Proj^{\omega_1}\hy S)\subset D(S)$ the stable subcategory generated by $\biggplus[\N]S.$ Note that we have $D^b(\Proj^{\omega_1}\hy S)\simeq\Stab(\Proj^{\omega_1}\hy S).$ We say that an object $M\in D^b(\Proj^{\omega_1}\hy S)$ has projective amplitude in $[a,b]$ if $M\in \Stab(\Proj^{\omega_1}\hy S)_{[a,b]}$

We will also denote by $\Proj^{\omega}\hy S$ the idempotent completion of the category of finitely generated free $S$-modules. We have $\Perf(S)\simeq\Stab(\Proj^{\omega}\hy S).$



The following is essentially a tautology.

\begin{prop}\label{prop:uniformly_bounded_amplitude}
Consider a family of connective $\bE_1$-rings $(S_j)_{j\in J}.$ For a family of right $S_j$-modules $(M_j\in D^b(\Proj^{\omega_1}\hy S_j))_{j\in J},$ the following are equivalent.
\begin{enumerate}[label=(\roman*),ref=(\roman*)]
\item The modules $M_j$ have uniformly bounded projective amplitude, i.e. there exist $a\leq b$ such that each $M_j$ has projective amplitude in $[a,b]$ over $S_j.$ \label{uniformly_bounded_proj_amplitude}
\item We have
\begin{equation*}
(M_j)_{j\in J}\in \Stab(\prodd[j\in J]\Proj^{\omega_1}\hy S_j)\subset \prodd[j\in J]D^b(\Proj^{\omega_1}\hy S_j).
\end{equation*} 
\end{enumerate}
\end{prop}


\begin{remark}
Recall that a (homologically) bounded below right $S$-module $M$ has a $\Tor$-amplitude in $[a,b]$ (in the homological grading) if for any discrete left $S$-module $N$ we have $M\tens{S}N\in \Sp_{[a,b]}$ (both $a$ and $b$ are finite). For $M\in D^b(\Proj^{\omega_1}\hy S)$ one has the following implications 
\begin{multline*}
(M\text{ has projective amplitude in }[a,b])\Rightarrow (M\text{ has }\Tor\hy\text{amplitude in }[a,b])\\
\Rightarrow(M\text{ has projective amplitude in }[a,b+1]).
\end{multline*}
This follows from the observation that a flat $\omega_1$-compact right $S$-module module $F$ has projective amplitude in $[0,1].$

In particular, in Proposition \ref{prop:uniformly_bounded_amplitude} \ref{uniformly_bounded_proj_amplitude} we can equivalently require that the modules $M_j$ have uniformly bounded $\Tor$-amplitude.
\end{remark}

Let $R$ and $I=(a_1,\dots,a_m)\subset R$ be as above. We denote by $D^b_{I\hy\tors}(\Proj^{\omega_1}\hy R)\subset D^b(\Proj^{\omega_1}\hy R)$ the full subcategory formed by objects with locally $I$-torsion homology. Note that this category is generated as a stable subcategory by the single object $\biggplus[\N]\bR\Gamma_I(R),$ where $\bR\Gamma_I:D(R)\to D_{I\hy \tors}(R)$ is the right adjoint to the inclusion (the local cohomology functor). It is convenient to introduce an equivalent category $D^b_{I\hy\compl}(\Proj^{\omega_1}\hy R^{\wedge}_I),$ which we define to be the essential image of the functor
\begin{equation*}
(-)^{\wedge}_I:D^b(\Proj^{\omega_1}\hy R)\to D_{I\hy\compl}(R).
\end{equation*}
The category $D^b_{I\hy\compl}(\Proj^{\omega_1}\hy R^{\wedge}_I)$ is generated by the object $(\biggplus[\N]R)^{\wedge}_I.$
We have equivalences
\begin{equation*}
D^b_{I\hy\tors}(\Proj^{\omega_1}\hy R)\stackrel{(-)^{\wedge}_I}{\xto{\sim}} D^b_{I\hy\compl}(\Proj^{\omega_1}\hy R^{\wedge}_I)\simeq \prolim[n]D^b(\Proj^{\omega_1}\hy R_n),
\end{equation*}
where $R_n=\Kos(R;a_1^n,\dots,a_m^n).$

\subsection{The category $\Nuc^{CS}(R^{\wedge}_{I})$} In this subsection we will very briefly recall the category of nuclear solid modules on the affine formal scheme $\Spf(R^{\wedge}_I)$ as defined in \cite{CS20}. We refer to \cite{CS20, And23, AM24} for details. We will give a presentation of this category via a short exact sequence in Corollary \ref{cor:Nuc^CS_resolution}. We also discuss the relation with our category $\Nuc(R^{\wedge}_I)$ from Definition \ref{def:nuclear_via_internal_Hom}.   

We will work within the light condensed topos \cite{CS24} (this choice does not affect the relevant categories of nuclear solid modules). Denote by $D(\Solid_{\Z})$ the derived category of solid abelian groups. The abelian category $\Solid_{\Z}$ has a single compact projective generator $\prodd[\N]\Z,$ and have a symmetric monoidal equivalence
\begin{equation*}
D(\Solid_{\Z})\xto{\sim}\Ind(D^b(\Proj^{\omega_1}\hy\Z)^{op}),
\end{equation*}
sending $\prodd[\N]\Z$ to $(\biggplus[\N]\Z)^{op}.$

We consider the ring $R^{\wedge}_{I}$ as an $\bE_{\infty}$-algebra in the category $D(\Solid_{\Z}).$ We denote the tensor product in the category $D(\Solid_{\Z})$ by $-\tens{\Z^{\bs}}-.$ Taking the induced analytic structure on $R^{\wedge}_I,$ we obtain the category
$$D(\Solid_{(R^{\wedge}_I,\Z)})=\Mod_{R^{\wedge}_I}(D(\Solid_{\Z})).$$ This category is generated by the single compact object $R^{\wedge}_I\tens{\Z^{\bs}}\prodd[\N]\Z.$ 
One of the equivalent definitions of the category of nuclear solid modules is given by
\begin{equation*}
\Nuc^{CS}(R^{\wedge}_I)=\Nuc(D(\Solid_{(R^{\wedge}_I,\Z)})).
\end{equation*}
More generally, for any subset $S\subset H_0(R^{\wedge}_I)(\ast)$ one has the associated analytic structure on $R^{\wedge}_{I},$ which depends only on the smallest open integrally closed subring containing $S.$ We denote by $D(\Solid_{(R^{\wedge}_I,S)})$ the associated derived category of complete modules. As explained in \cite[Example 3.34]{AM24}, the category $\Nuc(D(\Solid_{(R^{\wedge}_I,S)}))$ does not depend on $S.$

We use the notation $D(\Solid_{R^{\wedge}_I})=D(\Solid_{(R^{\wedge}_I,H_0(R^{\wedge}_I)(\ast))}).$ By the above, we have an equivalence
\begin{equation*}
\Nuc^{CS}(R^{\wedge}_I)\simeq \Nuc(D(\Solid_{R^{\wedge}_I})).
\end{equation*} 
If the ring $R/I$ is finitely generated over $\Z,$ then the category $D(\Solid_{R^{\wedge}_I})$ has a single compact generator $\prodd[\N]R^{\wedge}_I,$ and we have an equivalence
\begin{equation*}
D(\Solid_{R^{\wedge}_I})\xto{\sim}\Ind(D^b_{I\hy\compl}(\Proj^{\omega_1}\hy R^{\wedge}_I)^{op}),
\end{equation*}
sending $\prodd[\N]R^{\wedge}_I$ to $((\biggplus[\N]R)^{\wedge}_I)^{op}.$

To deal with the general case (without the finite generation assumption), it is convenient to define the following category.




\begin{defi}
We denote by $D(\wt{\Solid}_{R^{\wedge}_I})$ the category $\Ind(D^b_{I\hy\compl}(\Proj^{\omega_1}\hy R^{\wedge}_I)^{op}).$
\end{defi}

We claim that the category $\Nuc^{CS}(R^{\wedge}_I)$ is naturally identified with the category of nuclear objects in $D(\wt{\Solid}_{R^{\wedge}_I}).$ This essentially follows from the results of \cite{And23}, but the issue is that the natural functor $D(\Solid_{R^{\wedge}_I})\to D(\wt{\Solid}_{R^{\wedge}_I})$ is in general not a quotient functor. Equivalently, its right adjoint can be non-fully faithful, hence we need an argument showing that on nuclear objects we still have an equivalence.

We now give the details. We have a fully faithful functor $D(R^{\wedge}_I)\hto D(\wt{\Solid}_{R^{\wedge}_I}),$ sending $R^{\wedge}_I$ to the unit object, which we denote by the same symbol. Note that for any compact object $X$ in the category $D(\wt{\Solid}_{R^{\wedge}_I}),$ the countable product $\prodd[\N]X$ is again compact. Moreover, the category $D(\wt{\Solid}_{R^{\wedge}_I})$ is generated by the single compact object $\prodd[\N]R^{\wedge}_I.$ Consider the composition functor
\begin{equation*}
	\Lambda_I:D(R)\to D(R^{\wedge}_I)\to D(\wt{\Solid}_{R^{\wedge}_I})\xto{(-)^{\wedge}_I}  D(\wt{\Solid}_{R^{\wedge}_I}),
\end{equation*}  
where the latter functor is the $I$-completion.

\begin{prop}\label{prop:good_rigidification_synthetic_solid}
For a compact object $M\in D(\wt{\Solid}_{R^{\wedge}_I})^{\omega},$ the predual object $M^{\vee}$ is nuclear. In particular, the category $\Nuc(D(\wt{\Solid}_{R^{\wedge}_I}))$ is generated by the object $(\prodd[\N]R^{\wedge}_I)^{\vee},$ and we have an equivalence
\begin{equation*}
\Nuc(D(\wt{\Solid}_{R^{\wedge}_I}))\simeq D(\wt{\Solid}_{R^{\wedge}_I})^{\rig}.
\end{equation*}
\end{prop} 

\begin{proof}
Assuming the first statement, the assertions about the category of nuclear objects follow from Propositions \ref{prop:conditions_for_good_rigidification} and \ref{prop:right_adjoint_description_for_rigidification}.

Now, putting $M=\prodd[\N]R^{\wedge}_I,$ we need to prove that $M^{\vee}$ is nuclear. This can be formally deduced from \cite[Satz 3.8]{And23}, similarly to the proof of \cite[Korollar 3.14]{And23}. Alternatively, the proof of \cite[Satz 3.8]{And23} can be straightforwardly generalized.

Namely, we have $M^{\vee}\cong\Lambda_I(\biggplus[\N]R).$ Thus, we need to prove that we have an isomorphism
\begin{equation*}
\Lambda_I(\biggplus[\N]R)\otimes \Lambda_I(\biggplus[\N]R)\cong \Lambda_I(\biggplus[\N\times\N]R).
\end{equation*}
This is proved in the same way as \cite[Satz 3.8]{And23}.
\end{proof}

\begin{prop}\label{prop:nuclear_via_synthetic}
The natural symmetric monoidal functor $f^*:D(\Solid_{(R^{\wedge}_I,\Z)})\to D(\wt{\Solid}_{R^{\wedge}_I})$ induces an equivalence on the full subcategories of nuclear objects.
\end{prop}

\begin{proof}
Denote by $f_*$ the right adjoint to $f^*.$ We claim that the functor $f_*$ preserves nuclear objects. Indeed, by Proposition \ref{prop:good_rigidification_synthetic_solid} it suffices to check that the object $f_*((\prodd[\N]R^{\wedge}_I)^{\vee})$ is nuclear. But this object is $I$-adically complete, and it is discrete modulo $I$ (in the derived sense). Therefore, it is nuclear by \cite[Lemma 2.18]{AM24}.

The functor $f_*$ is conservative, so it suffices to prove that for any nuclear object $X\in \Nuc(D(\Solid_{(R^{\wedge}_I,\Z)}))$ the adjunction unit $X\to f_*f^*(X)$ is an isomorphism. We may assume that $X$ is basic nuclear, i.e. $X\cong\indlim[n\in\N]P_n,$ where each $P_n$ is compact and each transition map $P_n\to P_{n+1}$ is trace-class. But then for each $n$ the choice of a trace-class witness $1\to P_n^{\vee}\otimes P_{n+1}$ induces a map $f_*f^*(P_n)\to P_{n+1},$ such that the following diagram commutes:
\begin{equation*}
\begin{tikzcd}
P_n\ar[r]\ar[d] & P_{n+1}\ar[d]\\
f_* f^*(P_n)\ar[r]\ar[ru] & f_*f^*(P_{n+1}).
\end{tikzcd}
\end{equation*}
This allows to construct the map $f_*f^*(X)\to X,$ inverse to the adjunction unit. This proves the proposition.
\end{proof}

We obtain the following relation between the categories $\Nuc^{CS}(R^{\wedge}_I)$ and $\Nuc(R^{\wedge}_I).$

\begin{cor}\label{cor:from_Nuc^CS_to_Nuc}
We have a natural fully faithful strongly continuous symmetric monoidal functor $\Nuc^{CS}(R^{\wedge}_I)\hto \Nuc(R^{\wedge}_I),$ given by the composition
\begin{multline}\label{eq:from_Nuc^CS_to_Nuc}
\Nuc^{CS}(R^{\wedge}_I)\simeq D(\wt{\Solid}_{R^{\wedge}_I})^{\rig}\simeq \Ind(D^b_{I\hy\compl}(\Proj^{\omega_1}\hy R^{\wedge}_I))^{\rig}\\
\hto \Ind(D_{I\hy\compl}(R)^{\omega_1})^{\rig}
\simeq D_{I\hy\compl}(R)^{\rig}\simeq \Nuc(R^{\wedge}_I). 
\end{multline}
\end{cor}

\begin{proof} The first equivalence in \eqref{eq:from_Nuc^CS_to_Nuc} is given by Proposition \ref{prop:nuclear_via_synthetic}, and the second equivalence follows from Proposition \ref{prop:from_Ind_D_to_Ind_D^op}. The final equivalence follows from Proposition \ref{prop:Nuc_via_rigidification}.
\end{proof}

To be able to compute the localizing invariants of the category $\Nuc^{CS}(R^{\wedge}_I),$ we need to describe the quotient $D(\wt{\Solid}_{R^{\wedge}_I})/\Nuc^{CS}(R^{\wedge}_I).$ More precisely, it will be convenient to pass to the dual categories. 

For an object $M\in D(R),$ we put
\begin{equation*}
\Calk_R^{\topp}(M)=\End_{\Calk(R)}(M)^{\wedge}_I\cong \prolim[n]\End_{\Calk(R_n)}(R_n\tens{R}M),
\end{equation*}
where as above $R_n=\Kos(R;a_1^n,\dots,a_m^n).$
Note that we have
\begin{equation*}
\Calk_R^{\topp}(M)\cong \Calk_R^{\topp}(M^{\wedge}_I),
\end{equation*} 
for any $M\in D(R).$

\begin{cor}\label{cor:Nuc^CS_resolution} Consider the $I$-complete $R$-module $P=(\biggplus[\N]R)^{\wedge}_I.$
Then we have a short exact sequence
\begin{equation*}
0\to \Nuc^{CS}(R^{\wedge}_I)\to D(\End_R(P))\to D(\Calk_R^{\topp}(P))\to 0.
\end{equation*}  
\end{cor}

\begin{proof}
By Proposition \ref{prop:nuclear_via_synthetic}, we have an equivalence $\Nuc^{CS}(R^{\wedge}_I)\simeq \Nuc(D(\wt{\Solid}_{R^{\wedge}_I})).$ From Propositions \ref{prop:good_rigidification_synthetic_solid} and \ref{prop:conditions_for_good_rigidification}, we see that the category $\Nuc^{CS}(R^{\wedge}_I)$ is also the category of nuclear objects in $$\Ind(D^b_{I\hy\compl}(\Proj^{\omega_1}\hy R^{\wedge}_I))\simeq D(\End_R(P)).$$ The right adjoint to the inclusion is given by the functor $(-)^{\tr}.$ Moreover, for our object $P$ we have
\begin{multline*}
\Hom_{\Ind(D^b_{I\hy\compl}(\Proj^{\omega_1}\hy R^{\wedge}_I))}(P,P^{\tr})\cong ((\prodd[\N]R)\otimes (\biggplus[\N] R))^{\wedge}_I\\
\cong\Fiber(\End_R(P)\to\Calk_R^{\topp}(P)).
\end{multline*}
This proves the proposition.
\end{proof}

\begin{remark}
Let $P=(\biggplus[\N]R)^{\wedge}_I$ be as in Corollary \ref{cor:Nuc^CS_resolution}. It follows from loc. cit. and Corollary \ref{cor:from_Nuc^CS_to_Nuc} that we have a map of short exact sequences
\begin{equation*}
\begin{tikzcd}
0\ar[r] & \Nuc^{CS}(R^{\wedge}_I)\ar[r]\ar[d] & D(\End_R(P))\ar[r]\ar[d] &\ar[r]\ar[d] D(\Calk_R^{\topp}(P))\ar[r]\ar[d] & 0\\
0\ar[r] & \Nuc(R^{\wedge}_I)\ar[r] & \Ind(D_{I\hy\compl}(R)^{\omega_1})\ar[r] & \Ind(\prolim[n]\Calk_{\omega_1}(R_n))\ar[r] & 0,
\end{tikzcd}
\end{equation*}
where all vertical maps are fully faithful (and strongly continuous).
\end{remark}




\subsection{Localizing invariants of the category $\Nuc^{CS}(R^{\wedge}_I)$} We prove the following result, which in the case of $K$-theory was conjectured by Clausen and Scholze. We use the notation from Subsection \ref{ssec:some_notation}.

\begin{theo}\label{th:loc_invar_of_Nuc_CS}
Let $R$ be a commutative ring with a finitely generated ideal $I=(a_1,\dots,a_n).$ We put $R_n = \Kos(R;a_1^n,\dots,a_m^n),$ and consider $R_n$ as a DG $R$-algebra. Let $\Phi:\Cat_R^{\perf}\to\cT$ be an accessible localizing invariant, where $\cT$ is an accessible stable category.
\begin{enumerate}[label=(\roman*),ref=(\roman*)]
	\item We have a canonical cofiber sequence
	\begin{equation*}
		\Phi^{\cont}(\Nuc^{CS}(R^{\wedge}_I))\to \Phi(\Stab(\prodd[n\geq 1]\Proj^{\omega}\hy R_n))\xto{\id-\Phi(F)} \Phi(\Stab(\prodd[n\geq 1]\Proj^{\omega}\hy R_n)),
	\end{equation*}
	where $F$ is the endofunctor of the category $\Stab(\prodd[n\geq 1]\Proj^{\omega}\hy R_n),$ which is obtained by applying $\Stab$ to the product of functors $\Proj^{\omega}\hy R_{n+1}\to\Proj^{\omega}\hy R_n.$ \label{Phi^cont_of_Nuc^CS_via_exact_triangle}
\item Suppose that $\cT$ has countable products and we have an isomorphism 
\begin{equation*}
	\Phi(\Stab(\prodd[n\geq 1]\Proj^{\omega}\hy R_n))\xto{\sim} \prodd[n\geq 1]\Phi(\Perf(R_n)).
\end{equation*}
Then we obtain an isomorphism
\begin{equation*}
	\Phi^{\cont}(\Nuc^{CS}(R^{\wedge}_I))\cong \prolim[n]\Phi(\Perf(R_n)).
\end{equation*} \label{Phi^cont_of_Nuc^CS_as_a_limit}
\end{enumerate}
\end{theo}

\begin{cor}\label{cor:K_theory_of_Nuc^CS}
Within the notation of Theorem \ref{th:loc_invar_of_Nuc_CS}, we have an isomorphism
\begin{equation*}
K^{\cont}(\Nuc^{CS}(R^{\wedge}_I))\cong\prolim[n] K(R_n).
\end{equation*}
In particular, if $R$ is noetherian, we have an isomorphism
\begin{equation*}
K^{\cont}(\Nuc^{CS}(R^{\wedge}_I))\cong \prolim[n] K(R/I^n).
\end{equation*}
\end{cor}

\begin{proof}
The first isomorphism follows from Theorem \ref{th:loc_invar_of_Nuc_CS} and from the commutation of $K$-theory with infinite products of additive $\infty$-categories \cite[Proposition 2.11]{Cor}. If $R$ is noetherian, then the second isomorphism follows from the first one by Lemma \ref{lem:pro_equivalence_noetherian}.
\end{proof}

The following corollary was pointed out to me by Adriano C\'ordova Fedeli.

\begin{cor}\label{cor:U_loc_of_Nuc^CS_and_Nuc}
Consider the universal localizing invariant $\cU_{\loc}:\Cat_R^{\perf}\to\Mot^{\loc}_R,$ commuting with filtered colimits. Then we have an isomorphism
\begin{equation}\label{eq:U_loc_of_Nuc^CS_and_Nuc}
\cU_{\loc}^{\cont}(\Nuc^{CS}(R^{\wedge}_I))\xto{\sim}\cU_{\loc}^{\cont}(\Nuc(R^{\wedge}_I)).
\end{equation}
\end{cor}

\begin{proof}
By the relative version of \cite[Proposition 2.10]{Cor} (over $R$), we have an isomorphism
\begin{equation*}
\cU_{\loc}(\Stab(\prodd[n\geq 1]\Proj^{\omega}\hy R_n))\xto{\sim}\cU_{\loc}(\prodd[n\geq 1]\Perf(R_n)).
\end{equation*}
The isomorphism \eqref{eq:U_loc_of_Nuc^CS_and_Nuc} now follows from Theorem \ref{th:loc_invar_of_Nuc_CS} and Theorem \ref{th:local_invar_of_inverse_limits} applied to the sequence $(D(R_n))_{n\geq 1}.$
\end{proof}

The proof of Theorem \ref{th:loc_invar_of_Nuc_CS} is a modification of the proof of Theorem \ref{th:local_invar_of_inverse_limits}. First, we introduce the following notation. We denote by 
\begin{equation*}
\prodd[n\geq 1]^{\bnd}\Perf(R_n)\simeq \Stab(\prodd[n\geq 1]\Proj^{\omega}\hy R_n)\subset \prodd[n\geq 1]\Perf(R_n)
\end{equation*}
the full subcategory formed by objects $(M_n)_{n\geq 1}$ such that the objects $M_n\in\Perf(R_n)$ have uniformly bounded projective amplitude. Similarly, we introduce the subcategory
\begin{equation*}
\prodd[n\geq 1]^{\bnd}D^b(\Proj^{\omega_1}\hy R_n)\simeq \Stab(\prodd[n\geq 1]\Proj^{\omega_1}\hy R_n)\subset \prodd[n\geq 1]D^b(\Proj^{\omega_1}\hy R_n).
\end{equation*}

For each $n\geq 1$ we denote by $\Calk_{\omega_1}^b(R_n)\subset \Calk_{\omega_1}(R_n)$ the idempotent completion of the essential image of $D^b(\Proj^{\omega_1}\hy R_n).$ We define the full subcategory
\begin{equation*}
\prodd[n\geq 1]^{\bnd}\Calk_{\omega_1}^b(R_n)\subset\prodd[n\geq 1]\Calk_{\omega_1}^b(R_n)
\end{equation*}
to be the idempotent completion of the essential image of $\prodd[n\geq 1]^{\bnd}D^b(\Proj^{\omega_1}\hy R_n).$

\begin{prop}\label{prop:ses_of_bounded_products}
We have a short exact sequence in $\Cat_R^{\perf}:$
\begin{equation*}
0\to \prodd[n\geq 1]^{\bnd}\Perf(R_n)\to \prodd[n\geq 1]^{\bnd}D^b(\Proj^{\omega_1}\hy R_n)\to \prodd[n\geq 1]^{\bnd}\Calk_{\omega_1}^b(R_n)\to 0.
\end{equation*}
\end{prop}

\begin{proof}
We have an ``unbounded'' version of this short exact sequence by \cite[Lemma 5.2]{KW19}. Hence, it suffices to prove that the functor
\begin{equation*}
(\prodd[n\geq 1]^{\bnd}D^b(\Proj^{\omega_1}\hy R_n))/(\prodd[n\geq 1]^{\bnd}\Perf(R_n))\to (\prodd[n\geq 1]D(R_n)^{\omega_1})/(\prodd[n\geq 1]\Perf(R_n))
\end{equation*}
is fully faithful. By the standard description of morphisms in a quotient category, it suffices to prove that the following square commutes:
\begin{equation*}
\begin{CD}
\prodd[n\geq 1]^{\bnd}D^b(\Proj^{\omega_1}\hy R_n) @>>> \prodd[n\geq 1]D(R_n)^{\omega_1}\\
@VVV @VVV\\
\Ind(\prodd[n\geq 1]^{\bnd}\Perf(R_n)) @>>> \Ind(\prodd[n\geq 1]\Perf(R_n)). 
\end{CD}
\end{equation*}
Here the vertical functors are the right adjoints to the natural functors. In other words, we need to show that for an object $M=(M_n)_{n\geq 1}\in \prodd[n\geq 1]^{\bnd}D^b(\Proj^{\omega_1}\hy R_n)$ the functor
\begin{equation}\label{eq:cofinal_functor_for_products}
(\prodd[n\geq 1]^{\bnd}\Perf(R_n))_{/M}\to (\prodd[n\geq 1]\Perf(R_n))_{/M}
\end{equation}
is cofinal. We may and will assume that $M_n\in \Proj^{\omega_1}\hy R_n$ for $n\geq 1.$ Then the functor
\begin{equation*}
(\Proj^{\omega}\hy R_n)_{/M_n}\to (\Perf(R_n))_{/M_n} 
\end{equation*}
is cofinal for $n\geq 1.$ The class of cofinal functors is closed under infinite products, hence the functor
\begin{equation*}
(\prodd[n\geq 1]\Proj^{\omega}\hy R_n)_{/M}\to (\prodd[n\geq 1]\Perf(R_n))_{/M}
\end{equation*}
is cofinal. It follows that the functor \eqref{eq:cofinal_functor_for_products} is also cofinal, as required.
\end{proof}

Next, we denote by 
\begin{equation*}
\prolim[n]^{\oplax,\bnd}\Perf(R_n)=\bigcup\limits_{k\geq 0}\prolim[n]^{\oplax}\Stab(\Proj^{\omega}\hy R_n)_{[-k,k]}\subset \prolim[n]^{\oplax}\Perf(R_n)
\end{equation*} 
the full subcategory formed by objects $(M_n;\varphi_n)_{n\geq 1},$ such that the objects $M_n\in\Perf(R_n)$ have uniformly bounded projective amplitude. Similarly, we introduce the full subcategory
\begin{equation*}
	\prolim[n]^{\oplax,\bnd}D^b(\Proj^{\omega_1}\hy R_n)=\bigcup\limits_{k\geq 0}\prolim[n]^{\oplax}\Stab(\Proj^{\omega_1}\hy R_n)_{[-k,k]}\subset \prolim[n]^{\oplax}D(R_n)^{\omega_1}.
\end{equation*} 

It will be convenient to consider the full subcategory of the bounded oplax limit formed by objects with projective components
\begin{equation*}
\prolim[n]^{\oplax}\Proj^{\omega_1}\hy R_n\subset 	\prolim[n]^{\oplax,\bnd}D^b(\Proj^{\omega_1}\hy R_n).
\end{equation*}

Our next observation is that this subcategory generates the bounded oplax limit as a stable subcategory.

\begin{prop}\label{prop:oplax_limit_of_proj_generates}
The category $\prolim[n]^{\oplax,\bnd}D^b(\Proj^{\omega_1}\hy R_n)$ is generated by the full subcategory $\prolim[n]^{\oplax}\Proj^{\omega_1}\hy R_n$ as a stable subcategory.
\end{prop}


\begin{proof}
Consider an object 
\begin{equation*}
M=(M_n;\varphi_n)_{n\geq 1}\in \prolim[n]^{\oplax}\Stab(\Proj^{\omega_1}\hy R)_{[a,b]},
\end{equation*}
where $a\leq b$ are integers.
We prove by induction on $(b-a)$ that $M$ is contained in the stable subcategory generated by $\prolim[n]^{\oplax}\Proj^{\omega_1}\hy R_n.$ For $a=b$ there is nothing to prove.

Suppose that $a<b.$ For each $n\geq 1$ we put $N_n=\biggplus[\N]R_n,$ and choose a map $N_n[a]\to M_n$ inducing an epimorphism on $H_a(-).$ Then for $n\geq 1$ the map $$H_a(N_{n+1}[a]\tens{R_{n+1}}R_n)\to H_a(M_{n+1}\tens{R_{n+1}}R_n)$$ is also an epimorphism, hence we can find a map $\psi_n:N_n\to N_{n+1}\tens{R_{n+1}}R_n$ such that the following square commutes:
\begin{equation*}
\begin{CD}
N_n[a] @>{\psi_n[a]}>> N_{n+1}[a]\tens{R_{n+1}}R_n\\
@VVV @VVV\\  
M_n @>>> M_{n+1}\tens{R_{n+1}}R_n.
\end{CD}
\end{equation*}     
We obtain an object $N=(N_n;\psi_n)_{n\geq 1}\in \prolim[n]^{\oplax}\Proj^{\omega_1}\hy R_n$ with a map $N[a]\to M$ such that
\begin{equation*}
\Cone(N[a]\to M)\in \prolim[n]^{\oplax}\Stab(\Proj^{\omega_1}\hy R)_{[a+1,b]}.
\end{equation*} 
By the induction hypothesis, this object is contained in the stable subcategory generated by $\prolim[n]^{\oplax}\Proj^{\omega_1}\hy R_n.$ This proves the proposition. 
\end{proof}

Next, we observe that a ``bounded'' version of Proposition \ref{prop:oplax_limit_K_equiv} holds. For $k\geq 1,$ denote by $\pi_k:\prolim[n]^{\oplax,\bnd}\Perf(R_n)\to\Perf(R_k)$ the functor, sending an an object $(M_n;\varphi_n)_{n\geq 1}$ to $M_k.$ We denote by $\pi:\prolim[n]^{\oplax,\bnd}\Perf(R_n)\to\prodd[n\geq 1]^{\bnd}\Perf(R_n)$ the functor with components $\pi_k,$ $k\geq 1.$

\begin{prop}\label{prop:oplax_limit_K_equiv_bounded}
The functor $\pi:\prolim[n]^{\oplax,\bnd}\Perf(R_n)\to\prodd[n\geq 1]^{\bnd}\Perf(R_n)$ is a $K$-equivalence over $R.$  
\end{prop}

\begin{proof}
The proof is the same as in \cite[Proposition 4.29]{E24}. For completeness, we give the details. We put $\cA=\prolim[n]^{\oplax,\bnd}\Perf(R_n)$ and $\cB=\prodd[n\geq 1]^{\bnd}\Perf(R_n).$

For $k\geq 1,$ we denote by $\iota_k:\Perf(R_k)\to\cA$ the functor given by
\begin{equation*}
\iota_k(M)=(0,\dots,0,M,0,\dots).
\end{equation*} 
Consider the functor $\iota:\cB\to \cA,$ sending an object $(M_n)_{n\geq 1}$ to $\biggplus[n\geq 1]\iota_n(M_n)$ (this coproduct is well-defined). Then we have $\pi\circ\iota\cong \id.$ It remains to show that $[\iota\circ \pi]=[\id]$ in $K_0(\Fun_R(\cA,\cA)).$ 

We define the sequence of $R$-linear endofunctors $\Psi_k:\cA\to\cA,$ $k\geq 1,$ given by
\begin{equation*}
\Psi_k((M_n;\varphi_n)_{n\geq 1})=(0,\dots,0,M_k,M_{k+1},\dots),
\end{equation*} 
where the non-trivial transition maps in the right hand side are given by $\varphi_n,$ $n\geq k.$ Denote by $u_k:\Psi_{k+1}\to\Psi_k$ the natural map. Then we have an exact triangle in $\Fun_R(\cA,\cA):$
\begin{equation*}
\Psi_{k+1}\xto{u_k} \Psi_k\to \iota_k\pi_k,\quad k\geq 1. 
\end{equation*}
The direct sum of these exact triangles is well-defined, hence we obtain exact triangles
\begin{equation*}
\biggplus[k\geq 2]\Psi_k\xto{(u_k)_{k\geq 1}}\biggplus[k\geq 1]\Psi_k\to\iota\pi.
\end{equation*}
This gives the equalities in $K_0(\Fun_R(\cA,\cA)):$
\begin{equation*}
[\iota\circ\pi] = [\biggplus[k\geq 1]\Psi_k] - [\biggplus[k\geq 2]\Psi_k] = [\Psi_1] = [\id],
\end{equation*}
as required.
\end{proof}

We need the following comparison statement about the quotients of bounded and unbounded oplax limits.

\begin{prop}\label{prop:comparison_of_quotients_of_oplax_limits}
The natural functor
\begin{equation*}
(\prolim[n]^{\oplax,\bnd}D^b(\Proj^{\omega_1}\hy R_n))/(\prolim[n]^{\oplax,\bnd}\Perf(R_n))\to (\prolim[n]^{\oplax}D(R_n)^{\omega_1})/(\prolim[n]^{\oplax}\Perf(R_n)) 
\end{equation*}
is fully faithful.
\end{prop}

\begin{proof}
As in the proof of Proposition \ref{prop:ses_of_bounded_products}, it suffices to prove that the following square commutes:
\begin{equation*}
\begin{CD}
\prolim[n]^{\oplax,\bnd} D^b(\Proj^{\omega_1}\hy R_n) @>>> \prolim[n]^{\oplax} D(R_n)\\
@VVV @VVV\\
\Ind(\prolim[n]^{\oplax,\bnd}\Perf(R_n)) @>>> \Ind(\prolim[n]^{\oplax}\Perf(R_n))
\end{CD}
\end{equation*}
Here the vertical functors are the right adjoints to the natural functors. Equivalently, we need to prove that for an object $M=(M_n;\varphi_n)_{n\geq 1}\in \prolim[n]^{\oplax,\bnd} D^b(\Proj^{\omega_1}\hy R_n)$ the functor
\begin{equation*}
(\prolim[n]^{\oplax,\bnd}\Perf(R_n))_{/M}\to (\prolim[n]^{\oplax}\Perf(R_n))_{/M}
\end{equation*}
is cofinal. By Proposition \ref{prop:oplax_limit_of_proj_generates} we may and will assume that $M_n\in\Proj^{\omega_1}\hy R_n$ for $n\geq 1.$ Then it will suffice to prove that the functor
\begin{equation*}
(\prolim[n]^{\oplax}\Proj^{\omega}\hy R_n)_{/M}\to (\prolim[n]^{\oplax}\Perf(R_n))_{/M}
\end{equation*}
is cofinal. This follows directly from Lemma \ref{lem:functor_between_oplax_equalizers}. Indeed, we have equivalences
\begin{equation*}
(\prolim[n]^{\oplax}\Proj^{\omega}\hy R_n)_{/M}\simeq \LEq(\prodd[n\geq 1](\Proj^{\omega}\hy R_n)_{/M_n}\toto \prodd[n\geq 2](\Proj^{\omega}\hy R_{n-1})_{/M_n\tens{R_n}R_{n-1}}),
\end{equation*}
\begin{equation*}
(\prolim[n]^{\oplax}\Perf(R_n))_{/M}\simeq \LEq(\prodd[n\geq 1](\Perf(R_n))_{/M_n}\toto \prodd[n\geq 2](\Perf(R_{n-1}))_{/M_n\tens{R_n}R_{n-1}}),
\end{equation*}
and the functors
\begin{equation*}
(\Proj^{\omega}\hy R_n)_{/M_n}\to (\Perf(R_n))_{/M_n}
\end{equation*}
are cofinal. This proves the proposition.
\end{proof}

We obtain the following statement, analogous to Proposition \ref{prop:quotients_of_oplax_limits}.

\begin{cor}\label{cor:fully_faithful_bounded_oplax_into_Calkins}
The natural functor 
\begin{equation}\label{eq:fully_faithful_embedding_for_bounded_oplax}
(\prolim[n]^{\oplax,\bnd}D^b(\Proj^{\omega_1}\hy R_n))/(\prolim[n]^{\oplax,\bnd}\Perf(R_n))\to\prolim[n]^{\oplax}\Calk_{\omega_1}(R_n)
\end{equation} is fully faithful.
\end{cor}

\begin{proof}
Indeed, by Propositions \ref{prop:comparison_of_quotients_of_oplax_limits} and \ref{prop:quotients_of_oplax_limits} the functor \eqref{eq:fully_faithful_embedding_for_bounded_oplax} is the composition of fully faithful functors
\begin{multline*}
(\prolim[n]^{\oplax,\bnd}D^b(\Proj^{\omega_1}\hy R_n))/(\prolim[n]^{\oplax,\bnd}\Perf(R_n))\to (\prolim[n]^{\oplax}D^b(\Proj^{\omega_1}\hy R_n))/(\prolim[n]^{\oplax}\Perf(R_n))\\
\to \prolim[n]^{\oplax}\Calk_{\omega_1}(R_n).\qedhere
\end{multline*}
\end{proof}

We obtain the following key statement, analogous to Corollary \ref{cor:diagonal_arrow}.

\begin{cor}\label{cor:diagonal_arrow_bounded} We use the notation of Corollary \ref{cor:Nuc^CS_resolution}. In particular, we put $P=(\biggplus[\N]R)_I^{\wedge},$ and $\Calk_R^{\topp}(P)=\End_{\Calk(R)}(P)_I^{\wedge}.$
Then there exists a unique $R$-linear functor $\Perf(\Calk_R^{\topp})\to (\prolim[n]^{\oplax,\bnd}D^b(\Proj^{\omega_1}\hy R_n))/(\prolim[n]^{\oplax,\bnd}\Perf(R_n))$ such that the following diagram commutes:
\begin{equation}\label{eq:diagonal_arrow_bounded}
\begin{tikzcd}
\prolim[n]D^b(\Proj^{\omega_1}\hy R_n)\ar[r]\ar[d] & \Perf(\Calk_R^{\topp}(P))\ar[d]\ar[ld]\\
(\prolim[n]^{\oplax,\bnd}D^b(\Proj^{\omega_1}\hy R_n))/(\prolim[n]^{\oplax,\bnd}\Perf(R_n))\ar[r] & \prolim[n]^{\oplax}\Calk_{\omega_1}(R_n).
\end{tikzcd}
\end{equation}
Moreover, in this diagram the diagonal arrow is fully faithful.
\end{cor}

\begin{proof}
This is analogous to Corollary \ref{cor:diagonal_arrow}. Recall that we have equivalences
\begin{equation*}
\prolim[n]D^b(\Proj^{\omega_1}\hy R_n) \simeq D^b_{I\hy\compl}(\Proj^{\omega_1}\hy R^{\wedge}_I)\simeq\Perf(\End_R(P)).
\end{equation*} Hence, the upper horizontal arrow in \eqref{eq:diagonal_arrow_bounded} is a homological epimorphism by Corollary \ref{cor:Nuc^CS_resolution}. The lower horizontal arrow is fully faithful by Corollary \ref{cor:fully_faithful_bounded_oplax_into_Calkins}. This implies the existence and uniqueness of the diagonal arrow making the diagram commutative. Moreover, the diagonal arrow is fully faithful, because the lower horizontal arrow and the right vertical arrows are fully faithful. 
\end{proof}

Now to prove Theorem \ref{th:loc_invar_of_Nuc_CS} we only need the following analogue of Proposition \ref{prop:double_quotient_K_equiv}. We use the same notation for the double quotient as in loc. cit.

\begin{prop}\label{prop:double_quotient_K_equiv_bounded}
The functor
\begin{equation*}
\Psi:\prolim[n]^{\oplax,\bnd}D^b(\Proj^{\omega_1}\hy R_n)\to \prodd[n\geq 1]^{\bnd}D^b(\Proj^{\omega_1}\hy R_n),
\end{equation*}
given by
\begin{equation*}
\quad \Psi((M_n;\varphi_n)_{n\geq 1}) = (\Fiber(\varphi_n:M_n\to M_{n+1}\tens{R_{n+1}}R_n))_{n\geq 1},
\end{equation*}
induces a $K$-equivalence over $R:$
\begin{equation*}
\bbar{\Psi}:(\prolim[n] D^b(\Proj^{\omega_1}\hy R_n))\backslash (\prolim[n]^{\oplax,\bnd}D^b(\Proj^{\omega_1}\hy R_n))/(\prolim[n]^{\oplax,\bnd}\Perf(R_n))\to \prodd[n\geq 1]^{\bnd}\Calk_{\omega_1}^b(R_n).
\end{equation*}
\end{prop}

\begin{proof}
The proof is almost completely analogous to the proof of Proposition \ref{prop:double_quotient_K_equiv}, with some changes which we explain below. We put
\begin{equation}\label{eq:D_for_the_double_quotient_bounded}
\cD = (\prolim[n] D^b(\Proj^{\omega_1}\hy R_n))\backslash (\prolim[n]^{\oplax,\bnd}D^b(\Proj^{\omega_1}\hy R_n))/(\prolim[n]^{\oplax,\bnd}\Perf(R_n)).
\end{equation}

For $k\geq 1$ consider the functor
\begin{equation*}
\Theta_k:D^b(\Proj^{\omega_1}\hy R_k)\to \prolim[n]^{\oplax,\bnd} D^b(\Proj^{\omega_1}\hy R_n),
\end{equation*}
given by
\begin{equation*}
\quad \Theta_k(M)=(M\tens{R_k}R_1,\dots,M\tens{R_k}R_{k-1},M,0,\dots).
\end{equation*}
Again, $\Theta_k$ is the right adjoint to the projection functor. We define the functor
\begin{equation*}
\Theta:\prodd[n\geq 1]^{\bnd}D^b(\Proj^{\omega_1}\hy R_n)\to \prolim[n]^{\oplax,\bnd} D^b(\Proj^{\omega_1}\hy R_n),\quad \Theta((M_n)_{n\geq 1})=\biggplus[n\geq 1]\Theta_n(M_n).
\end{equation*}
Here the infinite direct sum is well-defined. As in the proof of Proposition \ref{prop:double_quotient_K_equiv}, we have $\Psi\circ\Theta\cong \id.$ 

We need to show that for $M=(M_n)_{n\geq 1}\in\prodd[n\geq 1]^{\bnd}\Perf(R_n),$ the image of the object $\Theta(M)$ in $\cD$ is zero.
For this we may and will assume that each $M_n$ is a finitely generated free module, i.e. $M_n\cong R_n^{\oplus k_n},$ for some $k_n\geq 0.$ Identifying the category $\prolim[n]D^b(\Proj^{\omega_1}\hy R_n)$ with $D^b_{I\hy\compl}(\Proj^{\omega_1}\hy R^{\wedge}_I),$ consider the natural morphisms
\begin{equation*}
u_n:(R^{\wedge}_I)^{\oplus k_n}\to \Theta_n(M_n),\quad n\geq 1.
\end{equation*}
The first $n$ components of the object $\Cone(u_n)$ vanish, and we have $\Cone(u_n)\in \prolim[n]^{\oplax,\bnd}\Perf(R_n).$ Taking the direct sum of the morphisms $u_n$ (which is well-defined), we obtain the morphism
\begin{equation*}
u:(\biggplus[n\geq 1]R^{\oplus k_n})^{\wedge}_I\to \Theta(M),
\end{equation*}
such that $\Cone(u)\in  \prolim[n]^{\oplax,\bnd}\Perf(R_n).$ This proves the vanishing of the image of $\Theta(M)$ in $\cD,$ as required.

Using Proposition \ref{prop:ses_of_bounded_products}, we see that the functor $\Theta$ induces the functor
\begin{equation*}
\bbar{\Theta}:\prodd[n\geq 1]^{\bnd}\Calk_{\omega_1}^b(R_n)\to \cD,
\end{equation*}
and we have $\bbar{\Psi}\circ\bbar{\Theta}\cong\id.$

Now, it remains to prove that
\begin{equation*}
[\bbar{\Theta}\circ\bbar{\Psi}]=[\id]\text{ in }K_0(\Fun_R(\cD,\cD)).
\end{equation*}
For $k\geq 1$ consider the endofunctor
\begin{equation*}
G_k:\prolim[n]^{\oplax,\bnd} D^b(\Proj^{\omega_1}\hy R_n)\to \prolim[n]^{\oplax,\bnd} D^b(\Proj^{\omega_1}\hy R_n),
\end{equation*}
given by
\begin{equation*}
G_k(M_1,M_2,\dots) = (M_k\tens{R_k}R_1,\dots,M_k\tens{R_k}R_{k-1},M_k,M_{k+1},\dots).
\end{equation*}
As in the proof of Proposition \ref{prop:double_quotient_K_equiv}, the direct sum $\biggplus[k\geq 1]G_k$ is well-defined,
and we have an exact triangle
\begin{equation*}
\biggplus[k\geq 1]G_k\to \biggplus[k\geq 2] G_k\to \Theta\circ\Psi
\end{equation*}
in the category $\Fun_R(\prolim[n]^{\oplax,\bnd}D^b(\Proj^{\omega_1}\hy R_n),\prolim[n]^{\oplax,\bnd}D^b(\Proj^{\omega_1}\hy R_n)).$ 
The rest of the proof is the same as for Proposition \ref{prop:double_quotient_K_equiv}.


\end{proof}

\begin{proof}[Proof of Theorem \ref{th:loc_invar_of_Nuc_CS}] The proof is now completely analogous to the proof of Theorem \ref{th:local_invar_of_inverse_limits}.
By Corollary \ref{cor:Nuc^CS_resolution}, we have a short exact sequence
\begin{equation*}
0\to \Nuc^{CS}(R^{\wedge}_I)\to D(\End_R(P))\to D(\Calk_R^{\topp}(P))\to 0,
\end{equation*}
where $P=(\biggplus[\N]R)^{\wedge}_I.$ The category $\Perf(\End_R(P))\simeq D^b_{I\hy\compl}(\Proj^{\omega_1}\hy R^{\wedge}_I)$ has a countable coproduct of copies of any object, hence its additive invariants vanish. We obtain an isomorphism
\begin{equation*}
\Phi^{\cont}(\Nuc^{CS}(R^{\wedge}_I))\cong \Omega\Phi(\Perf(\Calk_R^{\topp}(P))).
\end{equation*}
From Corollary \ref{cor:diagonal_arrow_bounded} we obtain a short exact sequence
\begin{equation*}
0\to \Perf(\Calk_R^{\topp}(P))\to (\prolim[n]^{\oplax,\bnd} D^b(\Proj^{\omega_1}\hy R_n))/(\prolim[n]^{\oplax,\bnd} \Perf(R_n))\to \cD\to 0,
\end{equation*}
where $\cD$ is the double quotient from \eqref{eq:D_for_the_double_quotient_bounded}. Since the category $\prolim[n]^{\oplax,\bnd} D^b(\Proj^{\omega_1}\hy R_n)$ also has a countable coproduct of copies of any object, Proposition \ref{prop:oplax_limit_K_equiv_bounded}  implies the isomorphism \begin{equation*}
	\Phi((\prolim[n]^{\oplax,\bnd} D^b(\Proj^{\omega_1}\hy R_n))/(\prolim[n]^{\oplax,\bnd} \Perf(R_n)))\cong \Sigma\Phi(\prodd[n\geq 1]^{\bnd}\Perf(R_n)),
\end{equation*}
Using Propositions \ref{prop:ses_of_bounded_products} and \ref{prop:double_quotient_K_equiv_bounded}, we obtain the isomorphisms
\begin{equation*}
\Phi(\cD)\cong \Phi(\prodd[n\geq 1]^{\bnd}\Calk_{\omega_1}^b(R_n))\cong \Sigma\Phi(\prodd[n\geq 1]^{\bnd}\Perf(R_n)).
\end{equation*}
Here we use the fact that the category $\prodd[n\geq 1]^{\bnd}D^b(\Proj^{\omega_1}\hy R_n)$ also has a countable coproduct of copies of any object.

Now, the following square naturally commutes
\begin{equation*}
\begin{CD}
\Omega\Phi((\prolim[n]^{\oplax,\bnd} D^b(\Proj^{\omega_1}\hy R_n))/(\prolim[n]^{\oplax,\bnd} \Perf(R_n))) @>{\sim}>> \Phi(\prodd[n\geq 1]^{\bnd}\Perf(R_n))\\
@VVV @V{\id-\Phi(F)}VV\\
\Omega\Phi(\cD) @>{\sim}>>  \Phi(\prodd[n\geq 1]^{\bnd}\Perf(R_n)),
\end{CD}
\end{equation*}
where $F$ is as in the formulation of the theorem. Therefore, we obtain the isomorphisms
\begin{multline*}
\Phi^{\cont}(\Nuc^{CS}(R^{\wedge}_I))\cong \Omega\Phi(\Perf(\Calk_R^{\topp}(P)))\\
\cong \Fiber(\Phi(\prodd[n\geq 1]^{\bnd}\Perf(R_n))\to \Phi(\prodd[n\geq 1]^{\bnd}\Perf(R_n))).
\end{multline*}
This proves \ref{Phi^cont_of_Nuc^CS_via_exact_triangle}, which implies \ref{Phi^cont_of_Nuc^CS_as_a_limit}.
\end{proof}

\subsection{Applications to Hochschild homology}
\label{ssec:applications_to_HH}

We mention surprising applications to Hochschild homology, which can be deduced from Theorems \ref{th:local_invar_of_inverse_limits} and \ref{th:loc_invar_of_Nuc_CS}, using the results of C\'ordova Fedeli \cite{Cor}. Recall that for a connective $\bE_1$-ring $S$ there is a notion of an almost perfect right $S$-module \cite[Definition 7.2.4.10, Proposition 7.2.4.11]{Lur17}. One of the equivalent characterizations is the following: $M\in D(S)$ is almost perfect if $M$ is bounded below (homologically) and for any $n$ there exists a perfect $S$-module $N$ with a map $N\to M$ such that $\Cone(N\to M)\in D_{\geq n}(S).$ 

\begin{prop}\label{prop:THH_of_products_over_good_base}
Let $\mk$ be a connective $\bE_{\infty}$-ring such that $\mk$ is almost perfect over $\THH(\mk).$ Then for any family of $\mk$-linear additive $\infty$-categories $(\cA_j)_{j\in J},$ we have the isomorphisms
\begin{equation}\label{eq:THH_products_of_k_linear_additive}
\THH(\prodd[j]\Stab(\cA_j)/\mk)\cong\THH(\Stab(\prodd[j]\cA_j)/\mk)\cong \prodd[j]\THH(\Stab(\cA_j)/\mk).
\end{equation}
\end{prop}

\begin{proof}
Recall that for a small $\mk$-linear stable category $\cC$ we have an isomorphism
\begin{equation*}
\THH(\cC/\mk)\cong \THH(\cC)\tens{\THH(\mk)}\mk.
\end{equation*}
Now the first isomorphism in \eqref{eq:THH_products_of_k_linear_additive} follows from \cite[Proposition 2.10]{Cor}. The second isomorphism in the case $\mk=\bS$ is given by \cite[Proposition 2.15]{Cor}. In general, using the assumption that $\mk$ is almost perfect over $\THH(\mk)$ we obtain
\begin{multline*}
\THH(\Stab(\prodd[j]\cA_j)/\mk)\cong \THH(\Stab(\prodd[j]\cA_j))\tens{\THH(\mk)}\mk\cong (\prodd[j]\THH(\Stab(\cA_j)))\tens{\THH(\mk)}\mk\\
\cong \prodd[j](\THH(\Stab(\cA_j))\tens{\THH(\mk)}\mk)\cong \prodd[j]\THH(\Stab(\cA_j)/\mk).
\end{multline*}
This proves the proposition.
\end{proof}

Using this statement, we can compute the Hochschild homology of the categories of nuclear modules in many situations. In the following corollary we consider two examples. We write $\HH$ instead of $\HH^{\cont}$ for simplicity.

\begin{cor}\label{cor:HH_Nuc_examples}
\begin{enumerate}[label=(\roman*),ref=(\roman*)]
	\item For a prime $p$ we have
	\begin{equation}\label{eq:HH_Nuc_Z_p}
	\HH(\Nuc(\Z_p)/\Z)\cong \HH(\Nuc^{CS}(\Z_p)/\Z)\cong\Z_p.
	\end{equation} \label{HH_Nuc_Z_p}
	\item Let $\mk$ be a field such that either $cha \mk=p$ and $[\mk:\mk^p]<\infty,$ or $\cha \mk=0$ and $\trdeg(\mk/\Q)<\infty.$ Then we have
	\begin{equation}\label{eq:HH_Nuc_k[[x]]}
	\HH_n(\Nuc(\mk[[x]])/\mk)\cong \HH_n(\Nuc^{CS}(\mk[[x]])/\mk)\cong\begin{cases}
	\mk[[x]] & \text{for }n=0;\\
	\mk[[x]]\cdot dx & \text{for }n=1;\\
	0 & \text{else.}	
	\end{cases}
	\end{equation} \label{HH_Nuc_k[[x]]}
\end{enumerate}
\end{cor}

\begin{proof}
Both in \eqref{eq:HH_Nuc_Z_p} and in \eqref{eq:HH_Nuc_k[[x]]} the first isomorphism follows from Corollary \ref{cor:U_loc_of_Nuc^CS_and_Nuc}.

We prove \ref{HH_Nuc_Z_p}. Recall that $\THH_0(\Z)=\Z$ and for $n>0$ the abelian group $\THH_n(\Z)$ is finite. It follows easily that $\Z$ is almost perfect over $\THH(\Z).$ Using Theorem \ref{th:loc_invar_of_Nuc_CS} and Proposition \ref{prop:THH_of_products_over_good_base} we obtain
\begin{equation*}
\HH(\Nuc^{CS}(\Z_p)/\Z)\cong\prolim[n]\HH((\Z/p^n)/\Z)\cong\Z_p.
\end{equation*}
The latter isomorphism follows from the fact that for $n,k>0$ the map $\HH_k(\Z/p^{2n})\to\HH_k(\Z/p^n)$ is zero, since the map between the cotangent complexes $L_{(\Z/p^{2n})/\Z}\to L_{(\Z/p^{n})/\Z}$ is zero.

Now we prove \ref{HH_Nuc_k[[x]]}. We claim that $\mk$ is almost perfect over $\THH(\mk).$ If $\cha \mk=0,$ then $\THH(\mk)\cong \HH(\mk/\Q),$ and our assumption on $\mk$ implies that $L_{\mk/\Q}\cong \Omega^1_{\mk/\Q}$ is a finite-dimensional $\mk$-vector space (of dimension $\trdeg(\mk/\Q)$). Hence, $\HH(\mk/\Q)\in\Perf(\mk)$ by Hochschild-Kostant-Rosenberg theorem. This easily implies that $\mk$ is almost perfect over $\HH(\mk/\Q).$

Suppose that $\cha\mk = p.$ By B\"okstedt's theorem, $\FF_p$ is perfect over $\THH(\FF_p),$ hence we only need to show that $\mk$ is almost perfect over $\HH(\mk/\FF_p).$ As above, this follows from the fact that $L_{\mk/\FF_p}\cong\Omega^1_{\mk/\FF_p}$ is a finite-dimensional $\mk$-vector space (of dimension $[\mk:\mk^p]$).

Applying Theorem \ref{th:loc_invar_of_Nuc_CS} and Proposition \ref{prop:THH_of_products_over_good_base}, we obtain
\begin{equation*}
\HH(\Nuc^{CS}(\mk[[x]])/\mk)\cong\prolim[n]\HH((\mk[x]/x^n)/\mk).
\end{equation*}
Again, an elementary analysis of the pro-system $(\HH((\mk[x]/x^n)/\mk))_n$ gives the second isomorphism in \eqref{eq:HH_Nuc_k[[x]]}.
\end{proof}

\begin{remark}
In the proof of Corollary \ref{cor:HH_Nuc_examples} we can avoid the computations with cotangent complexes. Namely, it follows from \ref{prop:derived_quotients_pro_equiv_to_functors_from_approx} and its proof that
\begin{multline*}
\prolim[n]\HH((\Z/p^n)/\Z)\cong \Hom(\indlim[n]\HH(D^b_{\coh}(\Z/p^n)/\Z),\Z)\cong \Hom(\HH(\Perf_{p\hy\tors}(\Z)),\Z)\\
\cong \Hom(\Q_p/\Z_p[-1],\Z)\cong \Z_p,
\end{multline*}
and similarly for the limit of $\HH((\mk[x]/x^n)/\mk).$
\end{remark}

Note that if in Corollary \ref{cor:HH_Nuc_examples} we replace the categories of nuclear modules with the usual derived categories $D(\Z_p)$ resp. $D(\mk[[x]]),$ then the corresponding Hochschild homology is much larger. For example, consider $\HH(\Z_p/\Z).$ Since $L_{\FF_p/\Z}\cong L_{\FF_p/\Z_p}\cong \FF_p[1],$ we see that $L_{\Z_p/\Z}\tens{\Z}\FF_p=0,$ hence  $L_{\Z_p/Z}\cong L_{\Q_p/\Q}\cong\Omega^1_{\Q_p/\Q}.$ The latter module of differentials is a $\Q_p$-vector space of dimension $2^{\aleph_0},$ with a basis $\{dx_i\}_{i\in I},$ where $\{x_i\}_{i\in I}$ is a transcendence basis of $\Q_p$ over $\Q.$ Hence, for $n>0$ the $\Z_p$-module $\HH_n(\Z_p/\Z)$ is also a $\Q_p$-vector space of dimension $2^{\aleph_0}.$ Miraculously, after adding sufficiently many ``extra objects'' to the category $D(\Z_p)$ the Hochschild homology becomes much smaller and more meaningful. 

\begin{remark}
Consider the object $P=(\biggplus[\N]\Z)^{\wedge}_p$ as a Banach $\Z_p$-module. Denote by $J\subset \End_{\Z}(P)$ the ideal of compact operators (i.e. operators such that the closure of the image is compact). An elementary computation shows that the non-unital ring $J$ is $\Tor$-unital in the sense of \cite{Sus95}. This means that $\Tor_i^{\Z\ltimes J}(\Z,\Z)=0$ for $i>0.$ More precisely, we have $J^2=J$ and $J$ is flat as a right module over $\Z\ltimes J,$ which implies the $\Tor$-unitality.

Using \cite[Theorem 26]{Tam} and Corollary \ref{cor:Nuc^CS_resolution}, we obtain the short exact sequence
\begin{equation*}
0\to\Nuc^{CS}(\Z_p)\to D(\Z\ltimes J)\to D(\Z)\to 0.
\end{equation*}
From Corollary \ref{cor:HH_Nuc_examples} \ref{HH_Nuc_Z_p} we obtain $\HH(J/\Z)\cong\Z_p.$ Inverting $p,$ we obtain a similar statement for the ideal of compact operators on a separable Banach $\Q_p$-vector space, i.e. $\HH(J[p^{-1}]/\Q)\cong \Q_p.$
\end{remark}

\appendix

\section{Algebra of limits and colimits}
\label{app:limits_colimits}

Given a cartesian resp. cocartesian fibration $p:\cC\to\cD,$ we denote by $p^{\vee}:\cC^{\vee}\to\cD^{op}$ the dual cocartesian resp. cartesian fibration. It is classified by the same functor $\cD^{op}\to\Cat_{\infty}$ resp. $\cD\to\Cat_{\infty}.$ Below we will only consider the (co)cartesian fibrations in which the source and the target are posets, hence the fibers are also posets.

We will need the following generalization of Proposition \ref{prop:seq_limits_of_filtered_colimits}.

\begin{prop}\label{prop:limit_of_colimits_for_fibrations}Let $p:\cI\to\N^{op}$ be a cocartesian fibration such that the fibers $\cI_n$ are directed posets, $n\in\N.$ Denote by $f_{m,n}:\cI_m\to\cI_n$ the transition maps, $m\geq n.$ Suppose that the maps $f_{n+1,n}$ are cofinal for all $n\geq 0.$ We write the objects of $\cI$ as pairs $(n,i_n),$ where $n\in\N,$ $i_n\in\cI_n.$
	
	Let $\cC$ be a presentable category which satisfies strong (AB5) (i.e. filtered colimits commutes with finite limits) and (AB6) for countable products (for example, $\cC$ can be any compactly assembled presentable category, such as $\cS$ or $\Sp$). Let $G:\cI\to\cC$ be a functor. Then we have the following isomorphism:
	$$\prolim[n]\indlim[i_n\in\cI_n]G(n,i_n)\cong \indlim[\varphi:\N\to\cI^{\vee}]\prolim[n\leq m]G(n,f_{m,n}(\varphi(m))).$$
	Here $\varphi$ runs through the directed poset of sections of the cartesian fibration $p^{\vee}:\cI^{\vee}\to\N.$\end{prop}

\begin{proof}Note that our cofinality assumption on the maps $f_{n+1,n}$ implies that the poset $\Fun_{\N}(\N,\cI^{\vee})$ is indeed directed. To see this, we first inductively construct a sequence of elements $(i_n\in \cI_n)_{n\geq 0},$ such that for all $n\geq 0$ we have $i_n\leq f_{n+1,n}(i_{n+1}).$ This sequence defines a section $\N\to\cI^{\vee},$ hence the set $\Fun_{\N}(\N,\cI^{\vee})$ is non-empty. Given two sections $\varphi_1,\varphi_2:\N\to \cI^{\vee},$ we can inductively construct a sequence $(i_n\in \cI_n)_{n\geq 0},$
	such that for all $n\geq 0$ we have $i_n\geq \varphi_1(n),$ $i_n\geq \varphi_2(n),$ $i_n\leq f_{n+1,n}(i_{n+1}).$ This sequence defines a section $\varphi:\N\to\cI^{\vee},$ such that $\varphi\geq \varphi_1,$ $\varphi\geq \varphi_2.$  
	
	The same argument shows that the map $$\Fun_{\N}(\N,\cI^{\vee})\to\prodd[n\geq 0]\cI_n,\quad \varphi\mapsto (\varphi(n))_n,$$ is cofinal. Similarly, the map
	$$\Fun_{\N}(\N,\cI^{\vee})\to\prodd[n\geq 0]\cI_n,\quad \varphi\mapsto (f_{n+1,n}(\varphi(n+1)))_n$$
	is cofinal. Thus, we obtain the following isomorphisms
	\begin{multline*}\prolim[n]\indlim[i_n\in\cI_n]G(n,i_n)\cong\Eq(\prodd[n]\indlim[i_n\in\cI_n]G(n,i_n)\toto \prodd[n]\indlim[i_n\in\cI_n]G(n,i_n))\cong\\ \Eq(\indlim[\varphi:\N\to\cI^{\vee}]\prodd[n]G(n,\varphi(n))\toto \indlim[\varphi:\N\to\cI^{\vee}]\prodd[n]G(n,f_{n+1,n}(\varphi(n+1))))\cong\\
		\indlim[\varphi:\N\to\cI^{\vee}]\Eq(\prodd[n]G(n,\varphi(n))\toto \prodd[n]G(n,f_{n+1,n}(\varphi(n+1))))\cong \indlim[\varphi:\N\to\cI^{\vee}]\prolim[n\leq m]G(n,f_{m,n}(\varphi(m))).\qedhere\end{multline*}\end{proof}

We prove the following statement, which directly implies Proposition \ref{prop:quadrofunctors}. For a future reference, we formulate it in a quite general form, namely for pairs of cocartesian fibrations over $\N^{op}.$

\begin{theo}\label{th:pullback_square_quadrofunctors} Let $p:\cI\to\N^{op},$ $q:\cJ\to\N^{op}$ be cocartesian fibrations, such that the fibers $\cI_n$ and $\cJ_k$ are directed, $n,k\in\N.$ We denote by $f_{mn}:\cI_m\to\cI_n$ and $g_{lk}:\cJ_l\to\cJ_k$ the transition maps, $m\geq n,$ $l\geq k.$ Suppose that the maps $f_{n+1,n}$ and $g_{k+1,k}$ are cofinal for all $n,k\in \N.$ We write the objects of $\cI\times\cJ$ as quadruples $(n,i_n,k,j_k),$ where $n,k\in\N,$ $i_n\in\cI_n,$ $j_k\in\cJ_k.$

Let $F:\cI\times\cJ\to\cC$ be a functor, where $\cC$ satisfies the assumptions of Proposition \ref{prop:limit_of_colimits_for_fibrations}. Consider the natural commutative square
\begin{equation}
\label{eq:key_pullback_square}
\begin{CD}
	X_{00}  @>>> X_{01}\\
	@VVV @VVV\\
	X_{10} @>>> X_{11},
\end{CD}	
\end{equation}
where
\begin{equation*}
X_{00} = \indlim[\varphi:\N\to\cI^{\vee}]\indlim[\psi:\N\to\cJ^{\vee}]\prolim[n\leq m]\prolim[k\leq l] F(n,f_{mn}(\varphi(m)),k,g_{lk}(\psi(l))),
\end{equation*}

\begin{equation*}
X_{01} = \prolim[n]\indlim[i_n]\prolim[k]\indlim[j_k]F(n,i_n,k,j_k),
\end{equation*}

\begin{equation*}
X_{10} = \prolim[k]\indlim[j_k]\prolim[n]\indlim[i_n]F(n,i_n,k,j_k),
\end{equation*}

\begin{equation*}
X_{11} = \prolim[n]\prolim[k]\indlim[i_n]\indlim[j_n]F(n,i_n,k,j_k).
\end{equation*}
Then \eqref{eq:key_pullback_square} is a pullback square.
\end{theo} 

\begin{remark}
The upper horizontal arrow in \eqref{eq:key_pullback_square} is given by the composition
\begin{multline*}
\indlim[\varphi:\N\to\cI^{\vee}]\indlim[\psi:\N\to\cJ^{\vee}]\prolim[n\leq m]\prolim[k\leq l] F(n,f_{mn}(\varphi(m)),k,g_{lk}(\psi(l)))\\
\to \indlim[\varphi:\N\to\cI^{\vee}]\prolim[n\leq m]\indlim[\psi:\N\to\cJ^{\vee}]\prolim[k\leq l] F(n,f_{mn}(\varphi(m)),k,g_{lk}(\psi(l)))\\
\cong \prolim[n]\indlim[i_n]\prolim[k]\indlim[j_k]F(n,i_n,k,j_k),
\end{multline*}
and similarly for the left vertical arrow.
\end{remark}

\begin{proof}[Proof of Proposition \ref{prop:quadrofunctors} assuming Theorem \ref{th:pullback_square_quadrofunctors}] 
We apply Theorem \ref{th:pullback_square_quadrofunctors} to the trivial cocartesian fibrations $\N^{op}\times I\to \N^{op}$ and $\N^{op}\times J\to\N^{op},$ and the functor $F:\N^{op}\times I\times \N^{op}\times J\to\cC.$ The assumptions \ref{assump1} and \ref{assump2} of the proposition exactly mean that in the associated pullback square \eqref{eq:key_pullback_square} the maps $X_{01}\to X_{11}$ and $X_{10}\to X_{11}$ are isomorphisms. Hence, the map $X_{00}\to X_{11}$ is also an isomorphism, as required.
\end{proof}

For a future reference we mention the following statement, which is essentially a special case of Theorem \ref{th:pullback_square_quadrofunctors}.

\begin{cor}\label{cor:pullback_square_products_and_colimits}
Let $(I_n)_{n\geq 0}$ and $(J_k)_{k\geq 0}$ be sequences of directed posets. Let $\cC$ be a presentable $\infty$-category, satisfying the assumptions of Proposition \ref{prop:limit_of_colimits_for_fibrations}. Suppose that for each $(n,k)\in\N\times\N$ we have a functor $F_{n,k}:I_n\times J_k\to\cC.$ Then we have a pullback square
\begin{equation}\label{eq:pullback_square_for_products}
\begin{tikzcd}
\indlim[(i_n)_n\in\prod\limits_{n} I_n]\indlim[(j_k)_k\in\prod\limits_{k} J_k]\prodd[n]\prodd[k]F_{n,k}(i_n,j_k)\ar[r]\ar[d] & \prodd[n]\indlim[i_n\in I_n]\prodd[k]\indlim[j_k\in J_k] F_{n,k}(i_n,j_k)\ar[d]\\
\prodd[k]\indlim[j_k\in J_k]\prodd[n]\indlim[i_n\in I_n] F_{n,k}(i_n,j_k)\ar[r] & \prodd[n]\prodd[k]\indlim[i_n\in I_n]\indlim[j_k\in J_k]F_{n,k}(i_n,j_k).
\end{tikzcd}
\end{equation}
\end{cor}

\begin{proof}
Consider the cocartesian fibrations $\cI\to\N^{op},$ $\cJ\to\N^{op},$ in which the fibers are given by
\begin{equation*}
\cI_n=\prodd[0\leq m\leq n] I_m,\quad \cJ_k=\prodd[0\leq l\leq k] J_k,
\end{equation*}
and the transition maps are the projections. Define the functor $F:\cI\times\cJ\to\cC$ as follows:
\begin{equation*}
F(n,(i_m)_{0\leq m\leq n},k,(j_l)_{0\leq l\leq k})=\prodd[0\leq m\leq n]\prodd[0\leq l\leq k]F_{m,l}(i_m,j_l).
\end{equation*}
Applying Theorem \ref{th:pullback_square_quadrofunctors}, we obtain a pullback square which is naturally isomorphic to \eqref{eq:pullback_square_for_products}.
\end{proof}

Now we prove Theorem \ref{th:pullback_square_quadrofunctors}. We introduce several relevant partially ordered sets. All the inclusions below are full, i.e. we consider the subposets with the induced partial order.
First, we denote by 
\begin{equation}\label{eq:twisted_arrow_of_N}
T = \{(n,m)\mid n\leq m\}\subset \N^{op}\times\N
\end{equation}
the twisted arrow category of $\N.$ Next, we put
\begin{equation*}
P = T\times T =\{(n,m,k,l)\mid n\leq m,\, k\leq l\}\subset\N^{op}\times\N\times\N^{op}\times\N.
\end{equation*}
We put $P_{00}=P,$ and
\begin{equation}\label{eq:poset_S_01}
P_{01} = \{(n,m,k,l)\mid n\leq m\leq k\leq l\}\subset P,
\end{equation} 

\begin{equation*}
P_{10} = \{(n,m,k,l)\mid k\leq l\leq n\leq m\}\subset P,
\end{equation*}

\begin{equation*}
P_{11} = \{(n,m,k,l)\mid n=k\leq m=l\}\subset P.
\end{equation*}

The following statement is crucial for the proof of Theorem \ref{th:pullback_square_quadrofunctors}.

\begin{lemma}\label{lem:rewriting_X_ij} Within the notation of Theorem \ref{th:pullback_square_quadrofunctors}, we have natural isomorphisms
\begin{equation}\label{eq:rewriting_X_ij}
X_{ij}\cong \indlim[\varphi:\N\to\cI^{\vee}]\indlim[\psi:\N\to\cJ^{\vee}]\prolim[(n,m,k,l)\in P_{ij}] F(n,f_{mn}(\varphi(m)),k,g_{lk}(\psi(l))),\quad i,j\in \{0,1\}.
\end{equation}
\end{lemma}

To prove Lemma \ref{lem:rewriting_X_ij} we will need several auxiliary statements.

\begin{lemma}\label{lem:trick_with_fibrations}
Within the notation of Theorem \ref{th:pullback_square_quadrofunctors}, 
consider the cartesian fibration $\cE\to\N,$ with the fibers given by $\cE_n=\Fun_{\N}(\N_{\geq n},\cJ^{\vee}),$ so each $\cE_n$ can be considered as a poset of partially defined sections of $q^{\vee}:\cJ^{\vee}\to\N.$ The transition maps $\alpha_{mn}:\cE_m\to\cE_n$ for $m\geq n$ are defined as the ``right Kan extensions'', i.e.
\begin{equation}\label{eq:right_Kan_extensions}
\alpha_{mn}(\psi_m)(k)=\begin{cases}
	\psi_m(k) & \text{for }k\geq m;\\
	g_{mk}(\psi_m(m)) & \text{for }n\leq k\leq m-1.\end{cases}
	\end{equation}
	Then the posets $\cE_n$ are directed, the maps $\alpha_{n+1,n}$ are cofinal, and the natural map \begin{equation}\label{eq:map_beta}\beta:\Fun_{\N}(\N,\cJ^{\vee})\to\Fun_{\N}(\N,\cE),\quad \beta(\psi)(n)=\psi_{\mid \N_{\geq n}},\end{equation}
	is cofinal. 
\end{lemma}

\begin{proof}
By Proposition \ref{prop:limit_of_colimits_for_fibrations}, the posets $\cE_n$ are directed. Cofinality of $g_{n+1,n}:\cJ_{n+1}\to\cJ_n$ implies that the maps $\alpha_{n+1,n}:\cE_{n+1}\to\cE_n$ are cofinal. Applying Proposition \ref{prop:limit_of_colimits_for_fibrations}, we see that the poset $\Fun_{\N}(\N,\cE)$ is directed.

To prove the cofinality of \eqref{eq:map_beta}, consider a section $s:\N\to\cE,$ given by $s(n)=\psi_n,$ $\psi_n:\N_{\geq n}\to \cJ^{\vee}.$ Using the cofinality of the maps $g_{n+1,n},$ we can inductively construct a sequence $(j_n\in\cJ_n)_{n\geq 0},$ such that $j_n\geq \psi_n(n)$ (hence also $j_n\geq \psi_k(n)$ for $k\leq n$) and $j_n\leq g_{n+1,n}(j_{n+1}).$
Then we obtain a section $\psi:\N\to\cJ^{\vee},$ given by $\psi(n)=j_n,$ such that $\beta(\psi)\geq s.$ This proves the cofinality of $\beta.$ 
\end{proof}

\begin{lemma}\label{lem:trick_with_twisted_arrow_categories}
Let $\cD$ be any $\infty$-category with countable limits. We fix a pair of natural numbers $n\leq m,$ and denote by $T_n$ the twisted arrow category of $\N_{\geq n},$ i.e. 
\begin{equation*}
T_n = \{(k,l)\mid n\leq k\leq l\}\subset \N^{op}\times\N.
\end{equation*}
Let $G:T_n\to\cD$ be a functor such that
\begin{equation*}G(k,k)\xto{\sim}G(k,k+1)\quad \text{for }n\leq k\leq m-1.\end{equation*}
Then we have an isomorphism
\begin{equation*}
\prolim[n\leq k\leq l]G(k,l)\xto{\sim}\prolim[m\leq k\leq l]G(k,l)
\end{equation*}
\end{lemma}

\begin{proof} We have
\begin{multline*}
\prolim[n\leq k\leq l]G(k,l)\cong \Eq(\prodd[k\geq n]G(k,k)\toto \prodd[k\geq n]G(k,k+1))\\
\cong G(n,n)\times_{G(n,n+1)}\cdots G(m-1,m-1)\times_{G(m-1,m)}\Eq(\prodd[k\geq m]G(k,k)\toto \prodd[k\geq m]G(k,k+1))\\
\cong\Eq(\prodd[k\geq m]G(k,k)\toto \prodd[k\geq m]G(k,k+1))\cong \prolim[m\leq k\leq l]G(k,l).
\end{multline*} 
Here the third isomorphism follows from the assumption that the maps $G(k,k)\to G(k,k+1)$ are isomorphisms for $n\leq k\leq m-1.$ This proves the lemma.
\end{proof}

\begin{prop}\label{prop:projection_is_a_fibration}
Recall the posets $T$ and $P_{01},$ introduced in \eqref{eq:twisted_arrow_of_N}, \eqref{eq:poset_S_01}. The natural map 
\begin{equation*}
P_{01}\to T,\quad (n,m,k,l)\mapsto (n,m),
\end{equation*}
is a cartesian fibration.
\end{prop}

\begin{proof}
This is clear: the transition maps on the fibers are the tautological inclusions.
\end{proof}

\begin{proof}[Proof of Lemma \ref{lem:rewriting_X_ij}]
The isomorphism \eqref{eq:rewriting_X_ij} for $X_{00}$ is tautological, and for $X_{11}$ it follows from Proposition \ref{prop:limit_of_colimits_for_fibrations}.
It remains to prove the isomorphism for $X_{01},$ and the case of $X_{10}$ would follow by symmetry.

Using Proposition \ref{prop:limit_of_colimits_for_fibrations}, we obtain the isomorphisms
\begin{multline*}
X_{01} = \prolim[n]\indlim[i_n]\prolim[k]\indlim[j_k]F(n,i_n,k,j_k)\cong \prolim[n]\indlim[i_n]\prolim[k\in\N_{\geq n}]\indlim[j_k]F(n,i_n,k,j_k)\\
\cong \prolim[n]\indlim[i_n]\indlim[\psi_n:\N_{\geq n}\to\cJ^{\vee}]\prolim[n\leq k\leq l]F(n,i_n,k,g_{lk}(\psi_n(l))).
\end{multline*}
Here in the last expression the inverse limit is taken over the twisted arrow category of $\N_{\geq n}.$ Now, recall the cartesian fibration $\cE\to\N$ from Lemma \ref{lem:trick_with_fibrations}: the fibers are $\cE_n=\Fun_{\N}(\N_{\geq n},\cJ^{\vee}),$ and the transition maps $\alpha_{mn}:\cE_m\to\cE_n$ are given by \eqref{eq:right_Kan_extensions}. Applying Proposition \ref{prop:limit_of_colimits_for_fibrations} to the cocartesian fibration $\cI\times_{\N^{op}}\cE^{\vee}\to \N^{op}$ and using Lemma \ref{lem:trick_with_fibrations}, we obtain
\begin{multline*}
\prolim[n]\indlim[i_n]\indlim[\psi_n:\N_{\geq n}\to\cJ^{\vee}]\prolim[n\leq k\leq l]F(n,i_n,k,g_{lk}(\psi_n(l)))\\
\cong \indlim[\varphi:\N\to\cI^{\vee}]\indlim[s:\N\to\cE]\prolim[n\leq m]\prolim[n\leq k\leq l]F(n,f_{mn}(\varphi(m)),k,g_{lk}(\alpha_{mn}(s(m))(l)))\\
\cong \indlim[\varphi:\N\to\cI^{\vee}]\indlim[\psi:\N\to\cJ^{\vee}]\prolim[n\leq m]\prolim[n\leq k\leq l]F(n,f_{mn}(\varphi(m)),k,g_{lk}(\alpha_{mn}(\psi_{\mid \N_{\geq m}})(l))).
\end{multline*} 
Next, we observe that for any pair $n\leq m$ and for any $k$ such that $n\leq k\leq m-1$ we have
\begin{multline*}
g_{k+1,k}(\alpha_{mn}(\psi_{\mid \N_{\geq m}})(k+1)) = g_{k+1,k}(g_{m,k+1}(\psi(m))) = g_{mk}(\psi(m))\\
= g_{kk}(\alpha_{mn}(\psi_{\mid \N_{\geq m}})(k)).
\end{multline*}
Therefore, applying Lemma \ref{lem:trick_with_twisted_arrow_categories} we obtain 
\begin{multline*}
\indlim[\varphi:\N\to\cI^{\vee}]\indlim[\psi:\N\to\cJ^{\vee}]\prolim[n\leq m]\prolim[n\leq k\leq l]F(n,f_{mn}(\varphi(m)),k,g_{lk}(\alpha_{mn}(\psi_{\mid \N_{\geq m}})(l)))\\
\cong \indlim[\varphi:\N\to\cI^{\vee}]\indlim[\psi:\N\to\cJ^{\vee}]\prolim[n\leq m]\prolim[m\leq k\leq l]F(n,f_{mn}(\varphi(m)),k,g_{lk}(\alpha_{mn}(\psi_{\mid \N_{\geq m}})(l)))\\
\cong \indlim[\varphi:\N\to\cI^{\vee}]\indlim[\psi:\N\to\cJ^{\vee}]\prolim[n\leq m]\prolim[m\leq k\leq l]F(n,f_{mn}(\varphi(m)),k,g_{lk}(\psi(l))).
\end{multline*}
Finally, using Proposition \ref{prop:projection_is_a_fibration} we obtain an isomorphism
\begin{multline*}
\indlim[\varphi:\N\to\cI^{\vee}]\indlim[\psi:\N\to\cJ^{\vee}]\prolim[n\leq m]\prolim[m\leq k\leq l]F(n,f_{mn}(\varphi(m)),k,g_{lk}(\psi(l)))\\
\cong \indlim[\varphi:\N\to\cI^{\vee}]\indlim[\psi:\N\to\cJ^{\vee}]\prolim[n\leq m\leq k\leq l]F(n,f_{mn}(\varphi(m)),k,g_{lk}(\psi(l))).
\end{multline*}
This proves the isomorphism \eqref{eq:rewriting_X_ij} for $X_{01}.$ The lemma is proved.
\end{proof}

To deduce Theorem \ref{th:pullback_square_quadrofunctors} from Lemma \ref{lem:rewriting_X_ij} we need to compute some left Kan extensions. Denote by $H_{ij}:P_{ij}\to P$ the tautological inclusions, and consider the left Kan extension functors 
\begin{equation*}
H_{ij,!}:\Fun(P_{ij},\cS)\to \Fun(P,\cS),\quad i,j\in \{0,1\}. 
\end{equation*}

In each of the functor categories we denote by $\ast$ the final object, i.e. the constant functor with value $\pt.$ For an upward closed subset $U\subset P$ we denote by $\chi_U$ the characteristic functor of $U,$ i.e. 
\begin{equation*}
\chi_U:P\to\cS,\quad \chi_U(x)=\begin{cases}
\pt & \text{for }x\in U;\\
\emptyset & \text{for }x\not\in U.
\end{cases}
\end{equation*}

\begin{lemma}\label{lem:left_Kan_extensions} Within the above notation, we have $H_{ij,!}(\ast)\cong \chi_{U_{ij}},$ $i,j\in\{0,1\},$ where the subsets $U_{ij}\subset P$ are given by $U_{00}=P,$
\begin{equation*}
U_{01}=\{(n,m,k,l)\mid n\leq m,\,k\leq l,\,n\leq l\}\subset P,	
\end{equation*}
\begin{equation*}
U_{10}=\{(n,m,k,l)\mid n\leq m,\,k\leq l,\,k\leq m\}\subset P,	
\end{equation*}
\begin{equation*}
U_{11} = \{(n,m,k,l)\mid \max(n,k)\leq \min(m,l)\}\subset P.
\end{equation*}
In particular, we have the following pushout square in $\Fun(P,\cS):$
\begin{equation}\label{eq:pushout_of_left_Kan_extensions}
\begin{CD}
	H_{11,!}(\ast) @>>> H_{01,!}(\ast)\\
	@VVV @VVV\\
	H_{10,!}(\ast) @>>> \ast.
\end{CD}
\end{equation}
\end{lemma}

\begin{proof}We first compute the functor $H_{11,!}(\ast).$ Take some element $(n,m,k,l)\in P,$ and consider the poset 
\begin{equation*}
(P_{11})_{/(n,m,k,l)}\simeq \{(n',m')\mid \max(n,k)\leq n'\leq m'\leq \min(m,l)\}\subset\N^{op}\times \N		
\end{equation*} 
If $(n,m,k,l)\not\in U_{11},$ then this poset is empty. If $(n,m,k,l)\in U_{11},$ then the poset $(P_{11})_{/(n,m,k,l)}$ has the largest element, given by $(\max(n,k),\min(m,l)).$ In particular, in the latter case this poset is weakly contractible. This proves the isomorphism $H_{11,!}(\ast)\cong \chi_{U_{11}}.$

Next, we compute the functor $H_{01,!}(\ast).$ Again, take some element $(n,m,k,l)\in P,$ and consider the poset 
\begin{multline*}
(P_{01})_{/(n,m,k,l)}\simeq \{(n',m',k',l')\mid n\leq n'\leq m'\leq m,\,\,k\leq k'\leq l'\leq l,\,\,m'\leq k'\}\\
\subset\N^{op}\times \N\times\N^{op}\times\N.		
\end{multline*} 
If this poset is non-empty, then choosing some element $(n',m',k',l')$ we obtain $n\leq m'\leq k'\leq l,$ hence $(n,m,k,l)\in U_{01}.$ We need to show that in this case the poset $(P_{01})_{/(n,m,k,l)}$ is weakly contractible. Consider the subposet
\begin{equation*}
Q=\{(n',m',k',l')\in (P_{01})_{/(n,m,k,l)}\mid n'=n,\,\,l'=l \}\subset (P_{01})_{/(n,m,k,l)}.
\end{equation*}
Then the inclusion $Q\to (P_{01})_{/(n,m,k,l)}$ has a left adjoint, given by $(n',m',k',l')\mapsto (n,m',k',l).$ Hence, the inclusion $Q\to (P_{01})_{/(n,m,k,l)}$ is a weak homotopy equivalence. Now, the poset $Q$ has the smallest element, given by $(n,n,l,l).$ Hence, $Q$ is weakly contractible and so is $(P_{01})_{/(n,m,k,l)}.$ This proves the isomorphism $H_{01,!}(\ast)\cong \chi_{U_{01}}.$

By symmetry, we also obtain an isomorphism  $H_{10,!}(\ast)\cong \chi_{U_{10}}.$ Finally, to deduce the pushout square \eqref{eq:pushout_of_left_Kan_extensions} we simply observe that \begin{equation*}
U_{01}\cup U_{10} = P,\quad U_{01}\cap U_{10} = U_{11}.\qedhere
\end{equation*}
\end{proof}

\begin{proof}[Proof of Theorem \ref{th:pullback_square_quadrofunctors}]
For any sections $\varphi:\N\to\cI^{\vee},$ $\psi:\N\to\cJ^{\vee},$ it follows from Lemma \ref{lem:left_Kan_extensions} that we have a pullback square
\begin{equation}\label{eq:pullback_square_for_chosen_phi_psi}
\begin{CD}
	\prolim[(n,m,k,l)\in P_{00}]F(n,f_{mn}(\varphi(m)),k,g_{lk}(\psi(l))) @>>> \prolim[(n,m,k,l)\in P_{01}]F(n,f_{mn}(\varphi(m)),k,g_{lk}(\psi(l)))\\
	@VVV @VVV\\
	\prolim[(n,m,k,l)\in P_{10}]F(n,f_{mn}(\varphi(m)),k,g_{lk}(\psi(l))) @>>> \prolim[(n,m,k,l)\in P_{11}]F(n,f_{mn}(\varphi(m)),k,g_{lk}(\psi(l))).
\end{CD}
\end{equation}
Taking the colimit of \eqref{eq:pullback_square_for_chosen_phi_psi} over $\varphi:\N\to \cI^{\vee}$ and $\psi:\N\to\cJ^{\vee},$ we obtain a pullback square, which is identified with \eqref{eq:key_pullback_square} by Lemma \ref{lem:rewriting_X_ij}. This proves the theorem.
\end{proof}

\section{Proof of Lemma \ref{lem:lax_equalizer_of_filtered_properties} \ref{filtered}}
\label{app:proof_of_lemma_on_filtered}

The proof is quite direct, and goes as follows. First, take some pair of objects 
\begin{equation}\label{eq:pair_of_objects}
	x=(x_n;\varphi_n)_{n\geq 0},\,y=(y_n;\psi_n)_{n\geq 0}\in\LEq(\prodd[n\geq 0]I_n\toto\prodd[n\geq 1]J_n).
\end{equation}
We will construct an object $z\in\LEq(\prodd[n\geq 0]I_n\toto\prodd[n\geq 1]J_n)$ together with morphisms $x\to z,$ $y\to z.$ Since $I_0$ is filtered, we can find an object $z_0\in I_0$ and a pair of maps $t_0:x_0\to z_0,$ $s_0:y_0\to z_0.$ Since $J_1$ is filtered, we can find an object $w_1\in J_1$ together with the maps $\theta_0':F_0(z_0)\to w_1,$ $p_1:G_1(x_1)\to w_1,$ $q_1:G_1(y_1)\to w_1,$ which fit into commutative squares
	\begin{equation*}
		\begin{tikzcd}
			F_0(x_0) \ar[r, "\varphi_0"] \ar[d, "F_0(t_0)"] & G_1(x_1) \ar[d, "p_1"] & & F_0(y_0) \ar[r, "\psi_0"] \ar[d, "F_0(s_0)"] & G_1(y_1) \ar[d, "q_1"] \\
			F_0(z_0)\ar[r, "\theta_0'"] & w_1 & & F_0(z_0)\ar[r, "\theta_0'"] & w_1.
		\end{tikzcd}
	\end{equation*}
	Using the cofinality of $G_1,$ we find an object $z_1\in I_1$ with the maps $r_1:w_1\to G_1(z_1),$ $t_1:x_1\to z_1,$ $s_1:y_1\to z_1,$ together with homotopies $r_1\circ p_1\sim G_1(t_1),$ $r_1\circ q_1\sim G_1(s_1).$ Denoting by $\theta_0$ the composition $F_0(z_0)\xto{\theta_0'} w_1\xto{r_1} G_1(z_1),$ we obtain commutative squares
	\begin{equation*}
		\begin{tikzcd}
			F_0(x_0) \ar[r, "\varphi_0"] \ar[d, "F_0(t_0)"] & G_1(x_1) \ar[d, "G_1(t_1)"] & & F_0(y_0) \ar[r, "\psi_0"] \ar[d, "F_0(s_0)"] & G_1(x_1) \ar[d, "G_1(s_1)"] \\
			F_0(z_0)\ar[r, "\theta_0"] & G_1(z_1) & & F_0(z_0)\ar[r, "\theta_0"] & G_1(z_1).
		\end{tikzcd}
	\end{equation*}
	Proceeding inductively, we construct an object $(z_n;\theta_n)_{n\geq 0}$ of $\LEq(\prodd[n\geq 0]I_n\toto\prodd[n\geq 1]J_n)$ together with the morphisms $t:x\to z,$ $s:y\to z.$
	
	Next, consider again a pair of objects as in \eqref{eq:pair_of_objects}, and suppose that we are given with a map $u:S^k\to\Map(x,y)$ for some $k\geq 0.$ More precisely, $u$ is given by a sequence of maps $u_n:S^k\to \Map(x_n,y_n),$ $n\geq 0$ together with a sequence of homotopies between the compositions
	\begin{equation*}
		S^k\xto{u_n}\Map(x_n,y_n)\to\Map(F_n(x_n),F_n(y_n))\to\Map(F_n(x_n),G_{n+1}(y_{n+1})),
	\end{equation*} 
	\begin{equation*}
		S^k\xto{u_{n+1}}\Map(x_{n+1},y_{n+1})\to\Map(G_{n+1}(x_n),G_{n+1}(y_n))\to\Map(F_n(x_n),G_{n+1}(y_{n+1})),
	\end{equation*}
	for $n\geq 0.$
	
	Since $I_0$ is filtered, we can find an object $z_0\in I_0$ with a map $s_0:y_0\to z_0$ such that the composition
	\begin{equation*}
		S^k\xto{u_0}\Map(x_0,y_0)\xto{s_0\circ-}\Map(x_0,z_0)\end{equation*} 
	is null-homotopic. Arguing as above, we find an object $z_1\in I_1,$ a map $\theta_0:F_0(z_0)\to G_1(z_1)$ and a map $s_1:y_1\to z_1,$ together with a commutative square
	\begin{equation}\label{eq:structural_comm_squares}
		\begin{tikzcd}
			F_0(y_0) \ar[r, "\psi_0"] \ar[d, "F_0(s_0)"] & G_1(y_1) \ar[d, "G_1(s_1)"] \\
			F_0(z_0)\ar[r, "\theta_0"] & G_1(z_1),
		\end{tikzcd}
	\end{equation}
	such that the composition 
	\begin{equation*}
		S^k\xto{u_1}\Map(x_1,y_1)\xto{s_1\circ-}\Map(x_1,z_1)\end{equation*} 
	is null-homotopic. Proceeding inductively, we construct an object $z=(z_n;\theta_n)_{n\geq 0}$ in $\LEq(\prodd[n\geq 0]I_n\toto\prodd[n\geq 1]J_n)$ together with a morphism $s:y\to z$ such that all the compositions
	\begin{equation*}
		S^k\xto{u_n}\Map(x_n,y_n)\xto{s_n\circ-}\Map(x_n,z_n)\end{equation*}
	are null-homotopic. Therefore, we may and will assume that the maps $u_n$ are null-homotopic for $n\geq 0.$
	
	Choosing these null-homotopies, we obtain the maps 
	\begin{equation*}v_n:S^{k+1}\to \Map(F_n(x_n),G_{n+1}(y_{n+1})),\quad n\geq 0.
	\end{equation*}
	If all $v_n$ would be null-homotopic, then $u$ itself would be null-homotopic. Arguing as above, we construct a sequence of objects $z_n\in I_n,$ $n\geq 0,$ with $z_0=y_0,$ together with the maps $\theta_n:F_n(z_n)\to G_{n+1}(z_{n+1}),$ $s_n:y_n\to z_n$ and commutative squares
	\eqref{eq:structural_comm_squares}, such that the compositions
	\begin{equation*}
		S^{k+1}\xto{v_n}\Map(F_n(x_n),G_{n+1}(y_{n+1}))\xto{G_{n+1}(s_{n+1})\circ-}\Map(F_n(x_n),G_{n+1}(z_{n+1}))
	\end{equation*}
	are null-homotopic, $n\geq 0.$ This gives an object $z=(z_n;\theta_n)_{n\geq 0}$ in $\LEq(\prodd[n\geq 0]I_n\toto\prodd[n\geq 1]J_n)$ together with the morphism $s:y\to z$ such that the composition
	\begin{equation*}
		S^k\xto{u}\Map(x,y)\xto{s\circ-}\Map(x,z)
	\end{equation*}
	is null-homotopic. This proves that the category $\LEq(\prodd[n\geq 0]I_n\toto\prodd[n\geq 1]J_n)$ is filtered.

\end{document}